\documentclass[11pt]{amsart}
\usepackage{amssymb}
\usepackage{graphicx}
\usepackage{xcolor} 
\usepackage{tensor}
\usepackage{fullpage} 
\usepackage{amsmath}
\usepackage{amsthm}
\usepackage{verbatim}
\usepackage{hyperref}	

\usepackage{enumitem}
\setlist[enumerate]{leftmargin=2em}
\setlist[itemize]{leftmargin=2em}

\definecolor{green}{rgb}{0,0.8,0} 

\newtheorem{maintheorem}{Theorem}
\newtheorem{maincorollary}[maintheorem]{Corollary}

\newtheorem{theorem}{Theorem}[section]

\newtheorem{lemma}[theorem]{Lemma}
\newtheorem{proposition}[theorem]{Proposition}

\theoremstyle{definition}

\theoremstyle{remark}
\newtheorem{remark}[theorem]{Remark}

\numberwithin{equation}{section}
\newcommand{\nrm}[1]{\Vert#1\Vert}
\newcommand{\abs}[1]{\vert#1\vert}
\newcommand{\brk}[1]{\langle#1\rangle}
\newcommand{\set}[1]{\{#1\}}

\newcommand{\nnrm}[1]{{\vert\kern-0.25ex\vert\kern-0.25ex\vert #1 
    \vert\kern-0.25ex\vert\kern-0.25ex\vert}}

\newcommand{\supp}{{\mathrm{supp}}\,}

\renewcommand{\Re}{\mathrm{Re}}

\newcommand{\aeq}{\simeq}
\newcommand{\aleq}{\lesssim}
\newcommand{\ageq}{\gtrsim}

\newcommand{\lap}{\Delta}

\newcommand{\ud}{\mathrm{d}}
\newcommand{\rd}{\partial}
\newcommand{\nb}{\nabla}

\newcommand{\alp}{\alpha}
\newcommand{\bt}{\beta}

\newcommand{\dlt}{\delta}

\newcommand{\eps}{\epsilon}
\newcommand{\veps}{\varepsilon}

\newcommand{\lmb}{\lambda}

\newcommand{\sgm}{\sigma}

\newcommand{\tht}{\theta}

\newcommand{\omg}{\omega}

\newcommand{\bfb}{{\bf b}}

\newcommand{\bfp}{{\bf p}}

\newcommand{\bfu}{{\bf u}}
\newcommand{\bfv}{{\bf v}}

\newcommand{\bfx}{{\bf x}}
\newcommand{\bfy}{{\bf y}}

\newcommand{\bfB}{{\bf B}}

\newcommand{\bfE}{{\bf E}}

\newcommand{\bfG}{{\bf G}}
\newcommand{\bfH}{{\bf H}}

\newcommand{\bfJ}{{\bf J}}

\newcommand{\bfU}{{\bf U}}
\newcommand{\bfV}{{\bf V}}
\newcommand{\bfW}{{\bf W}}

\newcommand{\bfZ}{{\bf Z}}

\newcommand{\bfdlt}{\boldsymbol{\dlt}}

\newcommand{\bbN}{\mathbb N}

\newcommand{\bbP}{\mathbb P}

\newcommand{\bbR}{\mathbb R}

\newcommand{\bbT}{\mathbb T}

\newcommand{\bbZ}{\mathbb Z}

\newcommand{\calB}{\mathcal B}
\newcommand{\calC}{\mathcal C}

\newcommand{\calE}{\mathcal E}
\newcommand{\calF}{\mathcal F}

\newcommand{\calL}{\mathcal L}

\newcommand{\calR}{\mathcal R}
\newcommand{\calS}{\mathcal S}

\newcommand{\frkb}{\mathfrak b}

\newcommand{\pfstep}[1]{\vskip.5em \noindent {\it #1.}}

\newcommand{\ackn}[1]{
\addtocontents{toc}{\protect\setcounter{tocdepth}{1}}
\subsection*{Acknowledgements} {#1}
\addtocontents{toc}{\protect\setcounter{tocdepth}{2}} }

\newcommand{\Abs}[1]{\left\vert#1\right\vert}		
\newcommand{\bgu}{\mathring{\bfu}}				
\newcommand{\bgB}{\mathring{\bfB}}				
\newcommand{\bfomg}{\boldsymbol{\omg}}		
\newcommand{\err}{\boldsymbol{\epsilon}}		
\newcommand{\errh}{\boldsymbol{\delta}}			
\newcommand{\errwp}{\boldsymbol{e}}		

\newcommand{\tu}{\tilde{u}}					
\newcommand{\tb}{\tilde{b}}					
\newcommand{\tpsi}{\tilde{\psi}}					
\newcommand{\tomg}{\tilde{\omg}}				

\newcommand{\dfrm}[1]{{}^{(#1)} \pi}				
\newcommand{\mean}[1]{\bar{#1}}				

\begin{document}

\title[]{On the Cauchy problem for the Hall and electron magnetohydrodynamic equations without resistivity I: illposedness near degenerate stationary solutions}
\author{In-Jee Jeong}%
\address{KIAS, Seoul, Korea 02455}%
\email{ijeong@kias.re.kr}%

\author{Sung-Jin Oh}%
\address{UC Berkeley, CA, USA 94720 and KIAS, Seoul, Korea 02455}%
\email{sjoh@math.berkeley.edu}%

\begin{abstract}
In this article, we prove various illposedness results for the Cauchy problem for the incompressible Hall- and electron-magnetohydrodynamic (MHD) equations without resistivity. These PDEs are fluid descriptions of plasmas, where the effect of collisions is neglected (no resistivity), while the motion of the electrons relative to the ions (Hall current term) is taken into account. The Hall current term endows the magnetic field equation with a quasilinear dispersive character, which is key to our mechanism for illposedness.

Perhaps the most striking conclusion of this article is that the Cauchy problems for the Hall-MHD (either viscous or inviscid) and the electron-MHD equations, under one translational symmetry, are ill-posed near the trivial solution in any sufficiently high regularity Sobolev space $H^{s}$ and even in any Gevrey spaces. This result holds despite obvious wellposedness of the linearized equations near the trivial solution, as well as conservation of the nonlinear energy, by which the $L^{2}$ norm (energy) of the solution stays constant in time. The core illposedness (or instability) mechanism is degeneration of certain high frequency wave packet solutions to the linearization around a class of linearly degenerate stationary solutions of these equations, which are essentially dispersive equations with degenerate principal symbols. The method developed in this work is sharp and robust, in that we also prove nonlinear $H^{s}$-illposedness (for $s$ arbitrarily high) in the presence of fractional dissipation of any order less than $1$, matching the previously known wellposedness results.

The results in this article are complemented by a companion work, where we provide geometric conditions on the initial magnetic field that ensure wellposedness(!) of the Cauchy problems for the incompressible Hall and electron-MHD equations. In particular, in stark contrast to the results here, it is shown in the companion work that the nonlinear Cauchy problems are well-posed near any nonzero constant magnetic field.
\end{abstract}

\maketitle

\tableofcontents

\section{Introduction}
In magnetohydrodynamics (MHD), a plasma is described as a single electrically conducting fluid interacting with a magnetic field. In the incompressible case, the equation of motion takes the form
\begin{equation} \label{eq:mhd} \tag{MHD}
\left\{
\begin{aligned}
	& \rd_{t} \bfu + \bfu \cdot \nb \bfu + \nb \bfp - \nu \lap \bfu = \bfJ \times \bfB,\\
	& \rd_{t} \bfB + \nb \times \bfE = 0, \\
	& \nb \cdot \bfu = \nb \cdot \bfB = 0,
\end{aligned}
\right.
\end{equation}
where $\bfu(t) : \bbR^{3} \to \bbR^{3}$ is the bulk plasma velocity field, $\bfp(t) : \bbR^{3} \to \bbR$ is the plasma pressure, $\nu \geq 0$ is the plasma viscosity, $\bfB(t), \bfE(t) : \bbR^{3} \to \bbR^{3}$ are the magnetic and electric fields, and $\bfJ$ is the current density. The celebrated \emph{ideal} (resp. \emph{resistive}) \emph{MHD equation} is obtained by additionally assuming no viscosity $\nu = 0$, Amp\'ere's law $\bfJ = \nb \times \bfB$ and ideal Ohm's law $\bfE + \bfu \times \bfB = 0$ (resp. Ohm's law $\bfE + \bfu \times \bfB = \eta \bfJ$, where $\eta > 0$ is the resistivity); the latter two effective laws close the system in terms of $(\bfu, \bfB)$.

Actual plasmas, however, are made up of at least two distinct species, namely, negatively-charged, lighter electrons and positively-charged, heavier ions. When the motion of the electrons is significantly faster compared to the bulk plasma, which is the case in many settings of astrophysical importance, Ohm's law attains a correction proportional to $\bfJ \times \bfB$, called the \emph{Hall current term}; see \cite{Light, Pecseli, ADF, JM} for formal derivations. The resulting system, first introduced by M.~J.~Lighthill \cite{Light}, is referred to as the \emph{Hall-MHD equation}. 

The subject of this paper and its companion \cite{JO2} is \emph{the incompressible Hall-MHD equation without resistivity}, i.e., \eqref{eq:mhd} with Amp\`ere's law, but with Ohm's law supplanted by  (normalized) \emph{generalized ideal Ohm's law}
\begin{equation*}
\bfE + \bfu \times \bfB = \bfJ \times \bfB.
\end{equation*}
In terms of $(\bfu, \bfB)$, the system takes the form
\begin{equation} \tag{Hall-MHD} \label{eq:hall-mhd}
\left\{
\begin{aligned}
&\rd_{t} \bfu + \bfu \cdot \nb \bfu + \nb \bfp  - \nu \lap \bfu = (\nb \times \bfB) \times \bfB,  \\
&\rd_{t} \bfB - \nb \times (\bfu \times \bfB) + \nb \times ((\nb \times \bfB) \times \bfB)= 0, \\
&\nb \cdot \bfu = \nb \cdot \bfB = 0.
\end{aligned}
\right.
\end{equation}
In the special case $\nu = 0$, the resulting system is called the \emph{ideal Hall-MHD equation}. 

The Hall current term $\nb \times ((\nb \times \bfB) \times \bfB)$ is both quasilinear\footnote{By which we mean that the Hall current term is nonlinear, but is linear in the highest order (i.e., second order) derivatives of $\bfB$.} and of the highest order; a priori, it may incur derivative losses. For this reason, previous mathematically rigorous investigations of the Hall-MHD equation were mostly carried out either in the presence of resistivity \cite{CDL, CL, CWo0, CWo1, CWo2, CWW, DaiLiu, Dai1, Dai2, Dai3}, which gives rise to a strong dissipative term $\eta \lap \bfB$ compensating for this loss, or in axisymmetry \cite{CWe, JKL}, in which the second order terms vanish. In the absence of resistivity and symmetries, even the basic question of (local) wellposedness of the Cauchy problem for \eqref{eq:hall-mhd} had been open. The answer to this question, as we show in this paper and its companion \cite{JO2}, turns out to be strikingly rich and markedly different compared to both the resistive Hall-MHD and the ideal MHD equations.

In the present paper we confirm that the derivative loss in the Hall current term cannot be avoided in general by establishing a number of illposedness results. In the central result (Theorem~\ref{thm:norm-growth}), we identify a strong instability mechanism for the linearized \eqref{eq:hall-mhd} around a stationary magnetic field with a degeneracy (i.e., vanishing) along a hypersurface, by which the energy of the initial perturbation is transferred to extremely small scales at a rate proportional to the frequency of the initial perturbation. Various linear and nonlinear illposedness results are proved as a consequence of this instability mechanism; see Sections~\ref{subsec:illposed-linear}--\ref{subsec:illposed-gevrey} below for their statements. In particular, we show that the Cauchy problem for \eqref{eq:hall-mhd} is ill-posed near the trivial solution $(\bfu, \bfB) = (0, 0)$ in any high regularity Sobolev space on any domain of the form $M = \bbT^{k} \times \bbR^{3-k}$ with $1 \leq k \leq 3$ (Corollary~\ref{cor:illposed-zero} and Theorem~\ref{thm:illposed-strong2}); this result is on the contrary to the cases of the resistive Hall-MHD equation \cite{CDL} and the ideal MHD equation \cite{Sch}. The Cauchy problem remains illposed even in Gevrey spaces (Theorem~\ref{thm:illposed-gevrey}), which is in stark difference to, for instance, the reverse heat equation $\rd_{t} f = - \lap f$ or the ill-posed problems underlying some classical hydrodynamical instabilities; see Section~\ref{subsec:discussions}. Our method also extends to the fractionally dissipative case, thereby establishing local illposedness in $H^{s}$ with arbitrarily high $s$ as long as the order of the dissipative term in the $\bfB$-equation is strictly less than $1$ (Theorem~\ref{thm:illposed-fradiss}). This result is sharp, exactly matching the wellposedness results previously obtained by Chae--Wan--Wu \cite{CWW} in the case the order is at least $1$.

In the companion work \cite{JO2}, we complement the illposedness results in this paper by providing geometric conditions on the initial magnetic field that ensure wellposedness(!) of the Cauchy problems for \eqref{eq:hall-mhd}. For instance, in contrast to the aforementioned illposedness result near the trivial solution, we prove that the Cauchy problem for \eqref{eq:hall-mhd} is well-posed near any \emph{nontrivial} constant magnetic field. We note that the latter setting is the more physically relevant one, going back to the original work of M.~J.~Lighthill \cite{Light}. For a short (and partial) summary of the results proved in \cite{JO2}, see Section~\ref{subsec:wellposed} below.

The essential features of \eqref{eq:hall-mhd} relevant in the issue of local ill- or wellposedness are more clearly seen in the simpler system
\begin{equation} \tag{E-MHD} \label{eq:e-mhd}
\left\{
\begin{aligned}
&\rd_{t} \bfB + \nb \times ((\nb \times \bfB) \times \bfB)= 0, \\
&\nb \cdot \bfB = 0,
\end{aligned}
\right.
\end{equation}
which is called the \emph{electron-MHD equation} (or the Hall equation) \cite[Section~10.7]{Pecseli}. It corresponds to the case when the bulk plasma is essentially at rest compared to the motion of the electrons. All the results in this paper and \cite{JO2} apply to both \eqref{eq:hall-mhd} and \eqref{eq:e-mhd}\footnote{This is with the exception of Theorem \ref{thm:illposed-strong2}, which works in a somewhat more restrictive setting for the \eqref{eq:hall-mhd} case.}. In fact, all the proofs will proceed by first handling the case of \eqref{eq:e-mhd}, and then extending the argument to \eqref{eq:hall-mhd}.

In both this paper and its companion \cite{JO2}, our basic insight is that the Hall current term endows the magnetic field equation with a \emph{quasilinear dispersive} (i.e., Schr\"odinger-like) character. The main ideas behind both the ill- and wellposedness results are most natural with such a viewpoint. In particular, the instability mechanism revealed in this paper is qualitatively distinct from the more classical examples of hydrodynamic instabilities (Kelvin--Helmholtz, Rayleigh--Taylor, boundary layer etc.), but it seems to be a prevalent phenomenon for degenerate dispersive equations.  We refer to Sections~\ref{subsec:ideas} and \ref{subsec:discussions} below for further discussion.

\subsection{Basic properties of \eqref{eq:hall-mhd} and \eqref{eq:e-mhd}} \label{subsec:hall-e-mhds}
To set the stage for the precise formulation of our main results, we begin with a discussion of some basic properties of \eqref{eq:hall-mhd} and \eqref{eq:e-mhd}. 
\subsubsection*{Energy identities}
A fundamental property of \eqref{eq:hall-mhd} and \eqref{eq:e-mhd}, of both mathematical and physical importance, is the energy identity. 
\begin{proposition} \label{prop:nonlin-en}
For a solution $(\bfu, \bfB)$ to \eqref{eq:hall-mhd} on $M = \bbT^{k} \times \bbR^{3-k}$ with sufficiently regularity and spatial decay, we have
\begin{equation*}
		\frac{\ud}{\ud t} \left( \frac{1}{2} \int_{M} (\abs{\bfu}^{2} + \abs{\bfB}^{2})(t) \, \ud x \ud y \ud z \right) = - \nu \int_{M} \abs{\nb \bfu}^{2}(t) \, \ud x \ud y \ud z.
\end{equation*}
Similarly, for a solution $\bfB$ to \eqref{eq:e-mhd} on $M = \bbT^{k} \times \bbR^{3-k}$ with sufficiently regularity and spatial decay, we have
\begin{equation*}
	\frac{\ud}{\ud t} \left( \frac{1}{2} \int_{M} \abs{\bfB}^{2}(t) \, \ud x \ud y \ud z \right)= 0.
\end{equation*}
\end{proposition}
The expressions inside the parentheses on the LHS are the \emph{energies} for \eqref{eq:hall-mhd} and \eqref{eq:e-mhd}, respectively. 

We only sketch the proof for \eqref{eq:e-mhd} and leave to the reader the (slightly more involved but similar) case of \eqref{eq:hall-mhd}. Multiplying \eqref{eq:e-mhd} by $\bfB$ and integrating on $M$, we have
\begin{equation*}
	\frac{1}{2} \frac{\ud}{\ud t} \int_{M} \abs{\bfB}^{2} \, \ud x \ud y \ud z
	= - \int_{M} \bfB \cdot (\nb \times ((\nb \times \bfB) \times \bfB) \, \ud x \ud y \ud z.
\end{equation*}
The Hall term multiplied by $\bfB$ disappears since the operator $\nb \times$ is symmetric:
\begin{equation*}
	\int_{M} \bfB \cdot (\nb \times ((\nb \times \bfB) \times \bfB) \, \ud x \ud y \ud z
	= \int_{M} (\nb \times \bfB) \cdot ((\nb \times \bfB) \times \bfB) \, \ud x \ud y \ud z = 0,
\end{equation*}
which completes the proof.

However, the situation is different when one tries to control higher Sobolev norms. For concreteness, consider the task of controlling $\nrm{\rd^{(N)} \bfB(t)}_{L^{2}}$ for a solution $\bfB$ to \eqref{eq:e-mhd}, where $\rd^{(N)}$ refers to an $N$-th order spatial derivative.  Performing a similar computation as above, we have from the Hall term a contribution of the form
\begin{equation} \label{eq:d-loss}
	\frac{1}{2} \frac{\ud}{\ud t} \int_{M} \abs{\rd^{(N)} \bfB}^{2} \, \ud x \ud y \ud z
	= - \int_{M} (\nb \times \rd^{(N)} \bfB) \cdot ((\nb \times \bfB) \times \rd^{(N)} \bfB) \, \ud x \ud y \ud z + \cdots
\end{equation}
where the other terms only involve up to $N$ derivatives of $\bfB$. It is not clear at all how to bound the integral on the right-hand side using $N$ derivatives of $\bfB$ only, and indeed, the results of this paper show that this loss of one derivative is unavoidable in certain cases.

\subsubsection*{Continuous and discrete symmetries}
Next, we describe some continuous and discrete symmetries of \eqref{eq:hall-mhd} and \eqref{eq:e-mhd}, which will be used in this paper.
\begin{itemize}
\item (Translational symmetry) For any $(t_{0}, x_{0},y_{0}, z_{0}) \in \bbR \times M$, \eqref{eq:hall-mhd} and \eqref{eq:e-mhd} are invariant under the translation $(\bfu, \bfB) \mapsto (\bfu, \bfB)(t - t_{0}, x - x_{0}, y - y_{0}, z - z_{0})$ and $\bfB \mapsto \bfB(t - t_{0}, x - x_{0}, y - y_{0}, z - z_{0})$, respectively.
\item (Rotational symmetry)
For any rotation matrix $O$, \eqref{eq:hall-mhd} and \eqref{eq:e-mhd} are invariant under the rotation $(\bfu, \bfB) \mapsto (O^{\top} \bfu, O^{\top} \bfB)(O(x, y, z)^{\top})$ and $\bfB \mapsto O^{\top} \bfB (O(x, y, z)^{\top})$, respectively.
\item (Reflection symmetry)
\eqref{eq:hall-mhd} and \eqref{eq:e-mhd} are invariant under the reflection about any hyperplane. For instance, the reflection about $\set{y = 0}$ for \eqref{eq:hall-mhd} is:
\begin{equation*}
	(\bfu, \bfB) \mapsto \calR (\bfu, \bfB)(x, y, z) = \left( \begin{pmatrix} \bfu^{x}(x, -y, z) \\ -\bfu^{y}(x, -y, z) \\ \bfu^{z}(x, -y, z) \end{pmatrix}, \begin{pmatrix} - \bfB^{x}(x, -y, z) \\ \bfB^{y}(x, -y, z) \\ - \bfB^{z}(x, -y, z) \end{pmatrix} \right),
\end{equation*}
and for \eqref{eq:e-mhd} is:
\begin{equation*} 
	\bfB \mapsto \calR \bfB(x, y, z) = \begin{pmatrix} - \bfB^{x}(x, -y, z) \\ \bfB^{y}(x, -y, z) \\ - \bfB^{z}(x, -y, z) \end{pmatrix}.
\end{equation*}
\item (Time reversal symmetry) In the ideal case $\nu = 0$, \eqref{eq:hall-mhd} is invariant under the time reversal $(\bfu, \bfB) \mapsto (-\bfu, -\bfB)(-t, x, y, z)$, and similarly \eqref{eq:e-mhd} is invariant under $\bfB \mapsto -\bfB(-t, x, y, z)$.

\item  {(Scaling symmetries for \eqref{eq:e-mhd}) For any $\alp \in \bbR$, \eqref{eq:e-mhd} on $M = \bbR^{3}$ is invariant under $\bfB \mapsto \lmb^{2-\alp} \bfB(\lmb^{-\alp} t, \lmb^{-1}(x, y, z))$. There is no exact scaling symmetry for \eqref{eq:hall-mhd}.}
\end{itemize}

\begin{remark} \label{rem:other-symm}
We mention in passing the following additional  {symmetry}, which will not be used in this paper, but  {is} used in \cite{JO2}:
\begin{itemize}
\item Galilean symmetry for \eqref{eq:hall-mhd}, $(\bfu, \bfB) \mapsto (\bfu - \mean{\bfU}, \bfB)(t, (x, y, z) + t \mean{\bfU})$. 
\end{itemize}
\end{remark}

\subsubsection*{Stationary solutions}
As is typical in (magneto)hydrodynamics, \eqref{eq:hall-mhd} and \eqref{eq:e-mhd} possess rich families of stationary solutions. A special class of highly symmetric stationary solutions, namely \emph{planar stationary magnetic fields with an additional symmetry}, will play a central role in this paper (see Section~\ref{subsec:ideas}). These solutions are characterized by the following properties:
\begin{itemize}
\item (Stationary magnetic field) The solution is of the form $\bfB = \bgB$ for \eqref{eq:e-mhd}, and $(\bfu, \bfB) = (0, \bgB)$ for \eqref{eq:hall-mhd}, where $\bgB$ is a $t$-independent vector field on $\bbR^{3}$ such that $\nb \cdot \bgB = 0$ (divergence-free) and $(\nb \times \bgB) \times \bgB$ is a pure gradient.

\item (Planarity) $\bgB$ is independent of the $z$-coordinate and $\bgB^{z}$ = 0. 
\item (Additional symmetry) $\bgB = \bgB^{x} \rd_{x} + \bgB^{y} \rd_{y}$, viewed as a vector field on $\bbR^{2}_{x, y}$, is invariant under a one-parameter family of isometries of $\bbR^{2}_{x, y}$.
\end{itemize}
The first property implies that $(0, \bgB)$ and $\bgB$ solve \eqref{eq:hall-mhd} and \eqref{eq:e-mhd}, respectively\footnote{Indeed, for the $\bfB$-equation in both \eqref{eq:hall-mhd} and \eqref{eq:e-mhd}, one uses the fact that gradient is curl-free. For the $\bfu$-equation in \eqref{eq:hall-mhd}, the contribution of $\bgB$ can be put into the pressure.}. In the third property, note that there are only two distinct possibilities up to symmetries: Either $\bgB$ is independent of one of the coordinates (say $x$) or it is axi-symmetric in $\bbR^{2}_{x,y}$.  

A complete classification of such stationary solutions is possible:
\begin{proposition} \label{prop:planar-stat}
A smooth planar stationary magnetic field with an additional symmetry is, up to symmetries, one of the following forms: ($f, g$ are smooth and $c_{0}, c_{1}, d \in \bbR$)
\begin{equation*}
	\bgB = f(y) \rd_{x}, \quad (c_{1} y + c_{0}) \rd_{x} + d \rd_{y}, \quad g(x^{2} + y^{2}) (x \rd_{y} - y \rd_{x}).
\end{equation*}
\end{proposition}
We postpone the proof until Section~\ref{subsec:planar-stat}.

\subsubsection*{Linearization around stationary solutions}
Returning to a general stationary solution to \eqref{eq:hall-mhd} of the form $(0, \bgB)$, let us consider perturbations of the form $(\bfu, \bfB) = (u, \bgB + b)$. The linearized equation satisfied by $(u, b)$ (i.e., the linearization of \eqref{eq:hall-mhd} around $\bgB$) is:
\begin{equation} \label{eq:hall-mhd-lin}
	\left\{
\begin{aligned}
	& \rd_{t} u - \nu \lap u = \bbP ((\nb \times \bgB) \times b + (\nb \times b) \times \bgB) \\
	& \rd_{t} b + \nb \times (u \times \bgB) + \nb \times ((\nb \times b) \times \bgB) + \nb \times ((\nb \times \bgB) \times b) = 0, \\
	& \nb \cdot u = \nb \cdot b = 0,
\end{aligned}
\right.
\end{equation}
where $\bbP$ is the Leray projection operator onto divergence-free vector fields.

In the case of \eqref{eq:e-mhd}, the linearization around a stationary solution of the form $\bgB$ takes the form
\begin{equation} \label{eq:e-mhd-lin}
	\left\{
\begin{aligned}
	& \rd_{t} b + \nb \times ((\nb \times b) \times \bgB) + \nb \times ((\nb \times \bgB) \times b) = 0, \\
	& \nb \cdot b = 0.
\end{aligned}
\right.
\end{equation}

For the linearized equations, the $L^{2}$ norm of the perturbation is still under control. Indeed, we have the following linearized energy identities:
\begin{proposition} \label{prop:lin-en}
For a sufficiently regular and decaying solution $(u, b)$ to \eqref{eq:hall-mhd-lin}, we have
\begin{equation} \label{eq:lin-en-hall}
\begin{aligned}
	& \frac{\ud}{\ud t} \left(\frac{1}{2} \int_{M} \abs{u}^{2}(t) + \abs{b}^{2}(t) \, \ud x \ud y \ud z \right) + \nu \int_{M} \abs{\nb u}^{2}(t) \, \ud x \ud y \ud z \\
	& =  \int_{M} ((b \cdot \nb) \bgB_{j}) u^{j} - ((u \cdot \nb) \bgB_{j}) b^{j} \, \ud x \ud y \ud z 
	+ \int_{M} ((b \cdot \nb) (\nb \times \bgB)_{j}) b^{j} \, \ud x \ud y \ud z.
\end{aligned}
\end{equation}

Similarly, for a sufficiently regular and decaying solution $b$ to \eqref{eq:e-mhd-lin}, we have
\begin{equation} \label{eq:lin-en-e}
	\frac{\ud}{\ud t} \left(\frac{1}{2} \int_{M} \abs{b}^{2}(t) \, \ud x \ud y \ud z \right)
	=  \int_{M} ((b \cdot \nb) (\nb \times \bgB)_{j}) b^{j} \, \ud x \ud y \ud z.
\end{equation}
\end{proposition}
We omit the proof, which is a simple exercise in vector calculus.

\subsection{Linear illposedness results in Sobolev spaces} \label{subsec:illposed-linear}
The energy identities in Proposition~\ref{prop:lin-en} suggest that for any ``reasonable'' solutions to the linearized equations \eqref{eq:hall-mhd-lin} and \eqref{eq:e-mhd-lin} around a sufficiently regular $\bgB$, the $L^{2}$ norm (energy) would enjoy good local-in-time bounds. Nevertheless, our results demonstrate that around certain stationary solutions, the linearized equation is ill-posed(!) in any higher Sobolev spaces.

Our first main result concerns the linearization of \eqref{eq:hall-mhd} and \eqref{eq:e-mhd} around a \emph{linearly degenerate} (to be defined below) planar stationary magnetic field with an additional symmetry. It asserts the existence of a sequence of initial data sets with frequencies $\lmb \in 2^{\bbN_{0}}$, such that the $H^{s}$ norms of the corresponding solutions for any $s > 0$ grow at rates that are sharp in view of the loss of one derivative observed in \eqref{eq:d-loss}.

In what follows, by an \emph{$L^{2}$-solution} on an interval $I$, we mean:
\begin{itemize}
\item ({linearized \eqref{eq:hall-mhd} with $\nu > 0$}) a pair of vector fields $(u, b)$ such that $u \in C_{w} (I; L^{2}) \cap L^{2}_{t}(I; \dot{H}^{1})$ and $b \in C_{w}(I; L^{2})$ that satisfies \eqref{eq:hall-mhd-lin} in the sense of distributions;
\item ({linearized \eqref{eq:hall-mhd} with $\nu = 0$}) a pair of vector fields $(u, b) \in C_{w}(I; L^{2})$ that satisfies \eqref{eq:hall-mhd-lin} with $\nu = 0$ in the sense of distributions; or
\item ({linearized \eqref{eq:e-mhd}}) a vector field $b \in C_{w}(I; L^{2})$ that satisfies \eqref{eq:e-mhd-lin} in the sense of distributions.
\end{itemize} 
Here, $C_w(I;L^2)$ is a subspace of $L^\infty(I;L^2)$ consisting of functions weakly continuous in time with values in $L^2$. Moreover, in the case $M = \bbT^{3}$, we assume\footnote{The interpretation of this assumption is that the constant part in $(u, b)$ should not be considered a perturbation, but rather should be put in the background.} in addition that 
\begin{equation} \label{eq:mean-zero}
	\int_{M} u(t) = \int_{M} b(t) = 0 \quad \hbox{ for all } t \in I,
\end{equation}
where we ignore the condition for $u$ in the case of \eqref{eq:e-mhd}. Note that, with the regularity assumptions above, the mean of any solution in the sense of distributions is preserved; thus it suffices to ensure \eqref{eq:mean-zero} for the initial data.

\begin{maintheorem}[Sharp norm growth] \label{thm:norm-growth}
Consider a stationary planar magnetic field $\bgB$ on $M$ of one of the following forms\footnote{We use $\bbT_{z}$ for convenience, but  it is not crucial for topological or algebraic reasons; note that both the stationary solution and the perturbations (i.e., solution to the linearized equation) are independent of $z$.  {See Section~\ref{subsec:discussions} for a further discussion on the issue of $z$-independence.}}:
\begin{enumerate}[label=(\alph*)]
\item (linearly degenerate, translationally-symmetric) On $M = (\bbT, \bbR)_{x} \times (\bbT, \bbR)_{y} \times \bbT_{z}$, $\bgB = f(y) \rd_{x}$ where $f$ is uniformly smooth (i.e., $f$ and its derivatives are bounded) and $f(y_{0}) = 0$, $\ud f (y_{0}) \neq 0$ for some $y_{0} \in (\bbT, \bbR)_{y}$;
\item (linearly degenerate, axi-symmetric) On $M = \bbR^{2}_{x, y} \times \bbT_{z}$, $\bgB = f(r) \rd_{\tht} = f(\sqrt{x^{2} + y^{2}})(x \rd_{y} - y \rd_{x})$ where $\bgB$ is uniformly smooth\footnote{In terms of $f$, it is equivalent to the condition that the odd extension of $r f$ to $\bbR$ is uniformly smooth.}, and $f(r_{0}) = 0$, $\ud f(r_{0}) \neq 0$ for some $r_{0} > 0$;
\end{enumerate}
where by the notation $(\bbT, \bbR)_{x}$ we mean that both $\bbT_{x}$ and $\bbR_{x}$ are allowed.
Then the following statements hold.
\begin{enumerate}
\item Consider the linearized \eqref{eq:hall-mhd} with $\nu > 0$ around the stationary solution $\bgB$ on a time interval $I \ni 0$. For each $\lmb \in \bbN$ sufficiently large depending on $\bgB$, there exists an initial data set of the form
\begin{itemize}
\item (Case (a): translationally-symmetric background)
\begin{equation*}
	u_{0} = 0, \quad b_{0} = \Re (e^{i (\lmb x + \lmb G(y))}) \frkb(x,y)
\end{equation*}
where $G(y) \in C^{\infty}((\bbT, \bbR)_{y})$ and $\frkb(x, y) \in \calS((\bbT, \bbR)_{x} \times (\bbT, \bbR)_{y})$ with compact support in $y$ and either compact support in $x$ or real-analyticity in $x$; or
\item (Case (b): axi-symmetric background)
\begin{equation*}
	u_{0} = 0, \quad b_{0} = \Re (e^{i (\lmb \tht + \lmb G(r))}) \frkb(r)
\end{equation*}
where $G(r) \in C^{\infty}((0, \infty))$ and $\frkb(r) \in \calC^{\infty}((0, \infty))$ with compact support in $r$,
\end{itemize}
such that any corresponding $z$-independent $L^{2}$-solution $(u, b)$ exhibits norm growth of the form
\begin{align*}
	\nrm{b(t)}_{W^{s, p}(M)} 
	\geq c \left(s, p, \bgB, \tfrac{\nrm{(u, b)}_{L^{\infty} (I; L^{2})} + \nrm{\nb u}_{L^{2}(I; L^{2})}}{\nrm{(u_{0}, b_{0})}_{L^{2}}}\right) \nrm{b_{0}}_{W^{s, p}(M)} e^{c_{0}(\bgB) \cdot (s + \frac{1}{2} - \frac{1}{p}) \lmb t},
\end{align*}
for $t \in I$ satisfying $0 \leq t < \dlt ( \tfrac{\nrm{(u, b)}_{L^{\infty} (I; L^{2})} + \nrm{\nb u}_{L^{2}(I; L^{2})} }{\nrm{(u_{0}, b_{0})}_{L^{2}}} )$ and for any $p \in [1, \infty]$, $s \in \bbR$ such that $s + \frac{1}{2} - \frac{1}{p} \geq 0$.

\item Consider the linearized \eqref{eq:hall-mhd} with $\nu = 0$ around the stationary solution $\bgB$ on a time interval $I \ni 0$. For each $\lmb \in \bbN$ sufficiently large depending on $\bgB$, there exists an initial data set of the form
\begin{itemize}
\item (Case (a): translationally-symmetric background)
\begin{equation*}
	 {u_{0}} = 0, \quad b_{0} = \Re (e^{i (\lmb x + \lmb G(y))}) \frkb(x,y)
\end{equation*}
where $G(y)$ and $\frkb(x, y)$ are as in part~(1); or
\item (Case (b): axi-symmetric background)
\begin{equation*}
	u_{0} = 0, \quad b_{0} = \Re (e^{i (\lmb \tht + \lmb G(r))}) \frkb(r)
\end{equation*}
where $G(r)$ and $\frkb(r)$ are as in part~(1),
\end{itemize}
such that any corresponding $z$-independent $L^{2}$-solution $(u, b)$ on $I$ exhibits norm growth of the form
\begin{equation*}
	\nrm{b(t)}_{W^{s, p}(M)} \geq c \left(\bgB, \tfrac{\nrm{(u, b)}_{L^{\infty} (I; L^{2})}}{\nrm{(u_{0}, b_{0})}_{L^{2}}}\right) \nrm{b_{0}}_{W^{s, p}(M)} e^{c_{0}(\bgB) \cdot (s + \frac{1}{2} - \frac{1}{p}) \lmb t}
\end{equation*}
for $t \in I$ satisfying $0 \leq t < \dlt ( \tfrac{\nrm{(u, b)}_{L^{\infty} (I; L^{2})}}{\nrm{(u_{0}, b_{0})}_{L^{2}}} )$ and  for any $p \in [1, \infty]$, $s \in \bbR$ such that $s + \frac{1}{2} - \frac{1}{p} \geq 0$.

\item Consider the linearized \eqref{eq:e-mhd} around the stationary solution $\bgB$ on a time interval $I \ni 0$. For each $\lmb \in \bbN$ sufficiently large depending on $\bgB$, there exists an initial data set of the form
\begin{itemize}
\item (Case (a): translationally-symmetric background)
\begin{equation*}
	\quad b_{0} = \Re (e^{i (\lmb x + \lmb G(y)}) \frkb
\end{equation*}
where $G(y)$ and $\frkb(x, y)$ are as in part~(1); or
\item (Case (b): axi-symmetric background)
\begin{equation*}
	b_{0} = \Re (e^{i (\lmb \tht + \lmb G(r))}) \frkb(r)
\end{equation*}
where $G(r)$ and $\frkb(r)$ are as in part~(1),
\end{itemize}
such that any corresponding $z$-independent $L^{2}$-solution $b$ on $I$ exhibits norm growth of the form
\begin{equation*}
	\nrm{b(t)}_{W^{s, p}(M)} \geq c \left(s, p, \bgB, \tfrac{\nrm{b}_{L^{\infty} (I; L^{2})}}{\nrm{b_{0}}_{L^{2}}}\right)  {\nrm{b_0}_{W^{s,p}(M)}} e^{c_{0}(\bgB) \cdot (s + \frac{1}{2} - \frac{1}{p})\lmb t}
\end{equation*}
for $t \in I$ satisfying $0 \leq t < \dlt ( \tfrac{\nrm{b}_{L^{\infty} (I; L^{2})}}{\nrm{b_{0}}_{L^{2}}} )$ and  for any $p \in [1, \infty]$, $s \in \bbR$ such that $s + \frac{1}{2} - \frac{1}{p} \geq 0$.
\end{enumerate}
\end{maintheorem}
\begin{remark} \label{rem:L2-sol}
Theorem~\ref{thm:norm-growth} is carefully formulated so that it does not rely on any wellposedness theory for the linearized \eqref{eq:hall-mhd} and \eqref{eq:e-mhd} equations, whose validity seems to be a delicate question precisely due to the illposedness issues considered here. 

For a reasonable notion of a solution for the linearized \eqref{eq:hall-mhd} and \eqref{eq:e-mhd} equations, it is expected that the $z$-independence property follows from uniqueness, and that the ratios \begin{equation*}
\begin{split}
\frac{\nrm{(u, b)}_{L^{\infty} L^{2}} + \nrm{\nb u}_{L^{2} L^{2}}}{\nrm{(u_{0}, b_{0})}_{L^{2}}},   \quad \frac{\nrm{b}_{L^{\infty} L^{2}}}{\nrm{b_{0}}_{L^{2}}}
\end{split}
\end{equation*}  are uniformly bounded by a constant that only depends $\bgB$ in view of the energy identities in Proposition~\ref{prop:lin-en}.

By appealing to a standard argument based on the Aubin--Lions lemma, one can show at least the \emph{existence} of such an $L^2$ solution, for any $L^2$ initial data; see Appendix~\ref{sec:L2-exist} for details. In particular, the class of solutions to which Theorem \ref{thm:norm-growth} applies is not vacuous.
\end{remark}
 
\begin{remark} \label{rem:norm-growth-cascade}
Due to the boundedness of energy, the norm growth in Theorem~\ref{thm:norm-growth} necessarily involves a rapid transfer of energy from larger to smaller scales. Such a phenomenon is reflected in the $s$-dependence of the growth rate $e^{c_{0} s \lmb t}$ of the $H^{s}$ norm (indeed, for $s \geq s_{0} > 0$, the $s$-dependent growth rate is a quick consequence of $L^{2}$ boundedness, $H^{s_{0}}$ growth and interpolation). It is also the key mechanism behind Theorem~\ref{thm:illposed-gevrey} below, which asserts illposedness of the linearized equations in all Gevrey spaces.

This phenomenon is clearly impossible for constant coefficient linear PDEs, where there are no energy transfers between different Fourier modes. Moreover, it is qualitatively different compared to well-known examples of illposedness in hydrodynamics such as the Kelvin--Helmholtz instability, Rayleigh--Taylor instability and boundary layer instability, in all of which the growth rate of the $H^{s}$ norm is independent of $s$ and wellposedness is recovered in a strong enough Gevrey space (at least in the linearized case). We refer to Section~\ref{subsec:discussions} for further discussion.
\end{remark}
 
\begin{remark} \label{rem:norm-growth-sharpness}
When $s$, $p$ and $\bgB$ are fixed, the norm growth inequality asserted in Theorem~\ref{thm:norm-growth} is optimal in that the both sides are comparable (uniformly in $\lmb$) for a suitably constructed wave packet approximate solution with the same initial data (see Section~\ref{sec:wavepackets} for the construction). This optimality is crucial for our proof of the illposedness results in the fractionally dissipative case (Theorem~\ref{thm:illposed-fradiss}). Moreover, with $p = 2$ and with a suitable definition of the norms $H^{s}$, the constant $c$ may be chosen to be independent on $s$; this observation is used in our proof of illposedness in Gevrey spaces (Theorem~\ref{thm:illposed-gevrey}). While we expect this property to generalize to all other values of $p$, we do not investigate this issue further in this paper.
\end{remark}

Our next result is a conditional refinement of Theorem~\ref{thm:norm-growth}. We assume that \eqref{eq:hall-mhd-lin} and \eqref{eq:e-mhd-lin} are well-posed in $L^2$, and show the existence of initial data sets with arbitrarily high regularity and decay, such that the corresponding solutions (unique by assumption) immediately exits any Sobolev space above $L^{2}$. More precisely, by \emph{$L^{2}$-wellposedness} of the linearized equation on an interval $I$, we mean the existence of a bounded linear solution map from $L^{2}$ into the energy class $\calE(I)$, where
\begin{equation*}
\calE(I) = 
\begin{cases}
\left( C_{w}(I; L^{2}) \cap L^{2}_{t} (I; \dot{H}^{1}) \right) \times C_{w}(I; L^{2}) 
& \hbox{ for linearized } \eqref{eq:hall-mhd}, \nu > 0; \\
C_{w}(I; L^{2}) \times C_{w} (I; L^{2})
& \hbox{ for linearized } \eqref{eq:hall-mhd}, \nu = 0; \\
C_{w} (I; L^{2})			 & \hbox{ for linearized } \eqref{eq:e-mhd}.
\end{cases}
\end{equation*}  
Since we know the existence of at least one $L^2$ solution, the following result maybe rephrased as follows: either the linearized system does not have a unique $L^2$ solution for some $L^2$ data, or there is \textit{nonexistence} in any higher regularity Sobolev spaces. 
 
\begin{maintheorem}[Instantaneous instability in $H^{s}$ with $s > 0$] \label{thm:inst}
Let $\bgB$ and $M$ be as in Theorem~\ref{thm:norm-growth}, and suppose that \eqref{eq:hall-mhd-lin} with $\nu \geq 0$ around $(0, \bgB)$ (resp. \eqref{eq:e-mhd-lin} around $\bgB$) is $L^{2}$-well-posed on $[0, 1]$. 
\begin{enumerate}
\item ($C^{\infty}$, polynomially decaying data) There exists an initial data set $(u_{0}, b_{0}) \in \set{0} \times \calS$ (resp. $b_{0} \in \calS$) and $0 < \dlt \leq 1$ such that the $L^{2}$-solution $(u, b)(t)$ (resp. $b(t)$) fails to be in any local Sobolev space $L^{2} \times H^{s'}_{loc}$ (resp. $H^{s'}_{loc}$) for any $s' > 0$ and $0 < t < \dlt$.

\item (Arbitrarily regular data with compact support) 
For any $s > 0$, there exists an initial data set $(u_{0}, b_{0}) \in \set{0} \times H^{s}_{comp}$ (resp. $b_{0} \in H^{s}_{comp}$) and $0 < \dlt \leq 1$ such that the $L^{2}$-solution $(u, b)(t)$ (resp. $b(t)$) fails to be in any local Sobolev space $L^{2} \times H^{s'}_{loc}$ (resp. $H^{s'}_{loc}$) for any $s' > 0$ and $0 < t < \dlt$.
\end{enumerate}
\end{maintheorem}

\subsection{Nonlinear illposedness results in Sobolev spaces} \label{subsec:illposed-nonlinear}
Given the preceding illposedness results for the linearized equations, it is natural to ask whether the corresponding statements are still valid for the nonlinear Cauchy problem. We show that the nonlinear Cauchy problems for \eqref{eq:hall-mhd} and \eqref{eq:e-mhd} are ill-posed in the sense of Hadamard \cite{Had} with $ {(\bfu,\bfB)} \in L^\infty_tH^2 \times L^\infty_tH^3$ and $ \bfB \in L^\infty_t H^3$ in the Hall- and electron-MHD cases, respectively. Moreover, we establish \textit{nonexistence} for certain initial data close to the trivial solution, which may be regarded as the strongest notion of illposedness. 

To describe the results we need some  {notation and conventions}. In what follows, we denote a function-space ball of radius $\eps$ with respect to a norm $\nrm{\cdot}_{X}$ centered at $\bfx$ by
\begin{equation*}
\calB_{\eps}(\bfx; X) = \set{\bfy \in \bfx + X : \nrm{\bfy - \bfx}_{X} < \eps},
\end{equation*}
and its restriction to compactly supported functions by
\begin{equation*}
\calB_{\eps}(\bfx; X_{comp}) = \set{\bfy \in \calB_{\eps}(\bfx; X) : \bfy - \bfx \hbox{ has compact support in $M$}}.
\end{equation*}
 {For any interval $I$ and $s_{0} \leq 1$, the notion of an $L^{\infty}_{t}(I; H^{s_{0}}(M))$ solution $\bfb$ to \eqref{eq:e-mhd} is formulated in the sense of distributions. For \eqref{eq:hall-mhd}, we need to also specify the pressure gradient; we say that $(\bfu,  \bfB) \in L^{\infty}_{t} ([0, \dlt]; H^{s_{0}-1}(M)) \times L^{\infty}_{t} ([0, \dlt]; H^{s_{0}}(M))$ is a (weak) solution to \eqref{eq:hall-mhd} if the equation is satisfied with
\begin{equation} \label{eq:p-grad}
	\nb_{j} \bfp = R_{k} R_{\ell} \nb_{j} (-\bfu^{k} \bfu^{\ell} + \bfB^{k} \bfB^{\ell}),
\end{equation}
where $R_{j} = (-\lap)^{-\frac{1}{2}} \rd_{j}$ is the Riesz transform.
}

The first nonlinear illposedness result shows that the solution map near the degenerate stationary solutions in high enough Sobolev spaces, even if it exists, must be unbounded.

\begin{maintheorem}[Unboundedness of the solution map] \label{thm:illposed-strong}
Let $M = (\bbT, \bbR)_{x} \times (\bbT, \bbR)_{y} \times \bbT_{z}$ and the stationary magnetic field $\bgB$ is given either by $f(y) \rd_{x}$ or $f(r)\rd_{\theta}$ as in Theorem \ref{thm:norm-growth}. Assume that for some $\eps, \dlt, r, s, s_{0} > 0$, the solution map for \eqref{eq:hall-mhd} (resp. \eqref{eq:e-mhd}) exists as a map 
 \begin{gather*}
\calB_{\eps}((0, \bgB); H^{r}_{comp} \times H^{s}_{comp})  \to L^{\infty}_{t} ([0, \dlt]; H^{s_{0}-1}) \times L^{\infty}_{t} ([0, \dlt]; H^{s_{0}})  \\
\left( \hbox{resp. } \calB_{\eps}(\bgB; H^{s}_{comp}) \to L^{\infty}_{t} ([0, \dlt]; H^{s_{0}}) \right).
\end{gather*}Then this solution map is unbounded for  $s_0 \ge 3$, and is not $\alpha$-H\"older continuous $(0 < \alpha \le 1)$ for $s_0 > \max\{ 2, 3(1-\alpha) \}$.  
\end{maintheorem} 
\begin{remark}
	Let us comment on the statement regarding the absence of H\"older continuity of the solution map. This notion of illposedness is analogous to that in a theorem of G.~M\'etivier \cite[Theorem~3.2]{Met}, which applies to any first order $n \times n$ nonlinear PDE of the form $\rd_{t} u = F(t, x, u, \rd_{x} u)$ in $\bbR^{d}$, where $F$ is real-analytic and \emph{nonhyperbolic}, i.e., $\rd_{v}F(0, x_{0}, u_{0}, v_{0})$ has a nonreal eigenvalue for some $(x_{0}, u_{0}, v_{0}) \in \bbR^{d} \times \bbR^{n} \times \bbR^{n\times d}$. See also a work of F.~John (\cite{Fr}) which discusses relevance of H\"older continuity for evolutionary systems of physical origin. 
	
	The case $\alpha = 1$ corresponds to the notion of Lipschitz continuity of the solution map, which  {has} been considered by many authors. We note that in a work of Guo and Tice \cite{GT} (see also \cite{GN}), a somewhat general argument was presented, which enables one to pass from a strong ill-posedness result for a linearized system to the failure of Lipschitz continuity for the corresponding nonlinear system. See \cite{GT} for details.
\end{remark}

Observe that the background stationary magnetic field $\bgB$ in Theorem~\ref{thm:illposed-strong} may be arbitrarily close to $0$ in quite strong topologies, or more specifically, in $H^{s}_{comp}(\bbR^{2})$ using axi-symmetric ones and $H^{s}_{comp}(\bbT_{x} \times \bbR_{y})$ using translationally-symmetric ones for any $s > 0$. Hence, the preceding result immediately implies that nonlinear Cauchy problems for \eqref{eq:hall-mhd} and \eqref{eq:e-mhd} are ill-posed near the trivial solution in the same sense.
\begin{maincorollary}[Unboundedness of the solution map near $0$] \label{cor:illposed-zero} 
The results of Theorem \ref{thm:illposed-strong} holds with $\bgB \equiv 0$.
\end{maincorollary}

Moreover, we show that there exist initial data, which are compactly supported arbitrarily close to the trivial solution in a high regularity Sobolev space, for which no solution can be found in the same Sobolev space. We emphasize that uniqueness does \emph{not} have to be assumed. 
\begin{maintheorem}[Nonexistence near $0$]\label{thm:illposed-strong2}
Let $s > 3+\frac{1}{2}$ and  {$M = (\bbT, \bbR)_{x} \times \bbR_{y} \times \bbT_{z}$ in the case of \eqref{eq:hall-mhd} and $M = (\bbT, \bbR)_{x} \times (\bbT, \bbR)_{y} \times \bbT_{z}$ for \eqref{eq:e-mhd}}. Given any $\epsilon > 0$, there exist initial data $(\mathbf{u}_0,\mathbf{B}_0) \in H^{s-1}_{comp} \times  {H^s}$ for  \eqref{eq:hall-mhd} satisfying $\nrm{\mathbf{u}_0}_{H^{s-1}} + \nrm{\mathbf{B}_0}_{H^s} < \epsilon$ (resp.~$\mathbf{B}_0 \in H^s_{comp} $ for \eqref{eq:e-mhd} satisfying $\nrm{\mathbf{B}_0}_{H^s} < \epsilon$) such that for any $\delta > 0$, there is no corresponding $ L^\infty_t([0,\delta];H^{s-1} \times H^s)$ solution to \eqref{eq:hall-mhd} (resp.~$L^\infty_t([0,\delta];H^s)$ solution to \eqref{eq:e-mhd}). \end{maintheorem}

 { \begin{remark}
		The nonlinear ill-posedness result stated in this section can be established for the scale of $C^{k,\alpha}$-spaces (with a straightforward modification of the arguments); for \eqref{eq:hall-mhd} and \eqref{eq:e-mhd}, one respectively needs the assumption that $(\bfu,\bfB) \in C^{k-1,\alpha}(M) \times C^{k,\alpha}(M)$ and $\bfB \in C^{k,\alpha}(M)$ with $k + \alpha \ge 2$. Note that this level of regularity corresponds exactly to the threshold for classical solutions -- solutions for which every term in the system can be identified with a continuous function. 
\end{remark} }

\subsection{Illposedness results in Gevrey spaces} \label{subsec:illposed-gevrey}
As discussed in Remark~\ref{rem:norm-growth-cascade}, the $s$-depen\-dence of the growth rate of the $H^{s}$ norm in Theorem~\ref{thm:norm-growth} hints at illposedness of \eqref{eq:hall-mhd} and \eqref{eq:e-mhd} in all Gevrey spaces; this behavior is in stark contrast to the ill-posed constant coefficients PDEs, as well as many traditional examples of ill-posed problems in hydrodynamics. Our goal is to rigorously illustrate this property; to avoid technical nuisances, we contend ourselves with a linear illposedness result on the domain\footnote{For consideration of other domains, see Remark~\ref{rem:gevrey-other-domain}} $M = \bbT^{3}$.
Before we describe the statements, let us briefly review the notion of Gevrey regularity classes and some basic properties. We will follow the illuminating work of Levermore--Oliver \cite{LO}.

A function $b \in C^\infty(\mathbb{T}^3)$ belongs to the \emph{Gevrey class $\sigma$} for some $\sigma > 0$ if there exist constants $\rho > 0, A < \infty$ such that for any $\alpha \in \mathbb{N}^3$, \begin{equation}\label{eq:gevrey-uniform} 
\begin{split}
\sup_{x \in \mathbb{T}^{3}}|\rd^\alpha b(x)| & \le A \left(\frac{\alpha!}{\rho^{|\alpha|}}\right)^\sigma
\end{split}
\end{equation} with $\rd^\alp= \rd_x^{\alp_1}\rd_y^{\alp_2}\rd_z^{\alp_3}$ and $|\alp|=\alp_1+\alp_2+\alp_3$. We denote $G^\sigma(\mathbb{T}^{3})$ to be the space of Gevrey class $\sigma$ functions. It is closed under multiplication and differentiation for all $\sgm >0$, and under composition as well for $\sgm \geq 1$. It is clear that for $0 < \sigma_1 < \sigma_2 < \infty$, we have $G^{\sigma_1}(\mathbb{T}^{3}) \subset G^{\sigma_2}(\mathbb{T}^{3}) \subset C^\infty(\mathbb{T}^{3})$ and that the containments are proper. It is well-known that $G^1(\mathbb{T}^{3})$ coincides with the space of real analytic functions. 

As in \cite{LO}, it will be convenient to characterize Gevrey classes in terms of the Sobolev norms. Then using Sobolev embedding, it is not difficult to show that $b \in G^\sigma(\mathbb{T}^{3})$ if and only if there are constants $0 < \rho, A < \infty$ such that \begin{equation}\label{eq:gevrey-Sobolev}
\begin{split}
\nrm{\rd^{\alp} b}_{L^2} & \le A\left(\frac{\alpha!}{\rho^{|\alpha|}}\right)^\sigma
\end{split}
\end{equation} where $\alp!=\alp_1!\alp_2!\alp_3!$. Furthermore, with the operator $|\nabla| := \sqrt{-\Delta}$, we define a family of normed spaces \begin{equation*} 
\begin{split}
&\mathcal{D}(e^{\tau |\nabla|^{1/\sigma}} : L^2(\mathbb{T}^{3})) = \{ b \in L^2(\mathbb{T}^{3}) : \nrm{ e^{\tau |\nabla|^{1/\sigma}} b }_{L^2} < \infty \}. 
\end{split}
\end{equation*} Then \cite[Theorem 4]{LO} states that for any $\sigma > 0$, \begin{equation*} 
\begin{split}
G^\sigma(\mathbb{T}^{3}) = \bigcup_{\tau > 0} \mathcal{D}(e^{\tau |\nabla|^{1/\sigma}} : L^2(\mathbb{T}^{3})) . 
\end{split}
\end{equation*} Here $\tau > 0$ corresponds precisely to the radius of Gevrey regularity, namely \begin{equation} \label{eq:gevrey-rad} 
\begin{split}
\frac{1}{\tau} = \limsup_{ |\alp| \rightarrow \infty} \left( \frac{\nrm{\rd^\alp b}_{L^2}^{1/\sigma}}{\alp!} \right)^{\frac{1}{\abs{\alp}}} . 
\end{split}
\end{equation} In the case of analytic functions ($\sigma = 1$), $\tau$ is simply the radius of analyticity.  

We are ready to state our linear illposedness results in Gevrey spaces for \eqref{eq:hall-mhd} and \eqref{eq:e-mhd}.

\begin{maintheorem}[Gevrey space illposedness]\label{thm:illposed-gevrey}
	Consider the stationary magnetic field $\bgB = f(y)\rd_x$ on $\mathbb{T}^3$ where $f(y)$ is a smooth function on $\mathbb{T}_y$ with $f(y_0) = 0$ and $f'(y_0) \ne 0$ for some $y_0$. Then the linearized \eqref{eq:hall-mhd} and \eqref{eq:e-mhd} systems at $(0, \bgB)$ and $\bgB$ respectively are  illposed in any $G^\sigma(\mathbb{T}^3)$ with $\sigma > 0$ when $f \in G^\sigma(\mathbb{T})$. To be more precise, assuming $L^2$-wellposedness,  {for $\sgm \geq 1$ (resp.~$0 < \sgm < 1$) there exist initial data in $G^\sigma(\mathbb{T}^3)$ whose corresponding unique solution escapes $C^\infty(\mathbb{T}^3)$ (resp.~$\cup_{\sgm' > 0} G^{\sgm'}(\bbT^{3})$) instantaneously for $t > 0$. }
\end{maintheorem}

In the statement of the above theorem, one can simply take $f(y) = \sin(y)$, which belongs to $\cap_{\sigma > 0} G^\sigma$ trivially and thus the associated linearized \eqref{eq:hall-mhd} and \eqref{eq:e-mhd} systems are illposed in every Gevrey class $G^{\sgm}$.

\begin{remark}[Nonlinear illposedness in Gevrey spaces]
Using the same methods involved in extending the linear result to the nonlinear one in the Sobolev case, one can easily obtain illposedness statements in Gevrey regularity. Let us just state the results which can be obtained, restricting ourselves to the $\bbT^3$-case. When $\sigma>1$, there exists $\bfB_0\in G^\sigma$ such that there is no local-in-time $C^\infty$ solution to \eqref{eq:e-mhd} with initial data $\bfB_0$. On the other hand, when $0<\sigma\le1$, we can prove the following norm-inflation type statement: for any $\eps>0$, there exists a data $\bfB_0\in G^\sigma$ such that any corresponding solution to \eqref{eq:e-mhd} in $L^\infty([0,\dlt];G^\sigma)$, if exists, satisfies $\tau(\bfB(t(\eps)))<\eps$ for some $0<t(\eps)$ with $t(\eps)\rightarrow 0$ as $\eps\rightarrow 0$. Here, $\tau(\bfB(t(\eps)))$ denotes the radius of Gevrey-$\sigma$ regularity for $\bfB(t(\eps))$. Observe that this statement contradicts the usual well-posedness statement in Gevrey spaces; see \cite{LO}.
\end{remark} 

\subsection{Illposedness results in the fractionally dissipative case}\label{rem:norm-growth-order}

The power of $\lmb$ in the growth rates $e^{c_{0} (s + \frac{1}{2} - \frac{1}{p}) \lmb t}$ in Theorem~\ref{thm:norm-growth} is sharp in view of the loss of one derivative seen in \eqref{eq:d-loss}. These rates are also consistent with the previously proved wellposedness due to Chae--Wan--Wu \cite{CWW}\footnote{In \cite{CWW} the domain is $\bbR^{3}$, but the result easily extends to any of $M = (\bbT, \bbR)_{x} \times (\bbT, \bbR)_{y} \times (\bbT, \bbR)_{z}$.} for the fractionally dissipative system 
 \begin{equation}\label{eq:e-mhd-fradiss}
\left\{
\begin{aligned}
& \rd_t \mathbf{B} + \nabla \times ((\nabla\times \mathbf{B})\times\mathbf{B}) = -\eta (-\Delta)^\alpha \mathbf{B},  \\
& \nabla\cdot\bfB=0 , 
\end{aligned}
\right.
\end{equation}
with $\eta\ge 0$ and $\alpha > 1/2$; see also \cite{CDL}. In view of the instability observed in this paper, one can expect this system to be illposed in the range $0\le \alpha < 1/2$ and similarly for the Hall-MHD system with a fractional dissipation in the magnetic field: \begin{equation}  \label{eq:hall-mhd-fradiss}
\left\{
\begin{aligned}
&\rd_{t} \bfu + \bfu \cdot \nb \bfu + \nb \bfp - (\nb \times \bfB) \times \bfB  = -\nu (-\Delta)^{1+\beta }\bfu,  \\
&\rd_{t} \bfB - \nb \times (\bfu \times \bfB) + \nb \times ((\nb \times \bfB) \times \bfB)= -\eta (-\Delta)^\alpha \mathbf{B}, \\
&\nb \cdot \bfu = \nb \cdot \bfB = 0,
\end{aligned}
\right.
\end{equation} where $0\le \alpha, \beta<\frac{1}{2}$. In the critical case where $\alp =\frac{1}{2}$, it is not difficult to show that the system \eqref{eq:e-mhd-fradiss} is globally well-posed in $H^s$ with $s$ large enough for \textit{small} (relative to $\eta$) $H^s$ initial data\footnote{The proof of this statement boils down to the inequality (with $|\nb|=(-\lap)^{\frac{1}{2}}$ and $s$ large)\begin{equation*}
\begin{split}
\left| \brk{|\nb|^{s} \nb\times ((\nb\times \bfB)\times\bfB), |\nb|^s\bfB }\right| &\le C\nrm{|\nb|^{s+\frac{1}{2}}\bfB}_{L^2} \nrm{|\nb|^{\frac{1}{2}} (\nb\bfB |\nb|^s\bfB) }_{L^2} \\
&\le C \nrm{\bfB}_{H^s} \nrm{|\nb|^{s+\frac{1}{2}}\bfB}_{L^2}^2 ,
\end{split}
\end{equation*} which gives \begin{equation*}
\begin{split}
\frac{1}{2}\frac{\ud }{\ud t} \nrm{|\nb|^s \bfB}_{L^2}^2 + (\eta-C\nrm{\bfB}_{H^s}) \nrm{|\nb|^{s+\frac{1}{2}}\bfB}_{L^2}^2 \le 0. 
\end{split}
\end{equation*} Therefore, there exists a universal constant $m_0>0$ such that if $\nrm{\bfB_0}_{H^s}\le \eta m_0$ then the solution still satisfies $\nrm{\bfB(t)}_{H^s}\le \eta m_0$ for all $t\ge0$.}; a similar argument . 

Indeed, we confirm that the fractionally dissipative systems \eqref{eq:e-mhd-fradiss} and \eqref{eq:hall-mhd-fradiss} are strongly ill-posed in Sobolev spaces in the expected range $0 \leq \alp < \frac{1}{2}$, in the sense that the solution map cannot be bounded near the trivial solution (similarly as in Theorem \ref{thm:illposed-strong}).

To avoid excessive technicalities, we contend ourselves with the case $M = \bbT^{3}$. However, with the ideas already in this paper, it is straightforward to extend what follows to the case $M = (\bbT, \bbR)_{x} \times (\bbT, \bbR)_{y} \times \bbT_{z}$. 

\begin{maintheorem}[Illposedness for fractionally dissipative systems]\label{thm:illposed-fradiss}
	Let $M = \bbT^{3}$, and consider the fractionally dissipative system \eqref{eq:hall-mhd-fradiss} with $0 \leq \alp, \bt < \frac{1}{2}$, $\eta > 0$ and $\nu \geq 0$ (resp.~\eqref{eq:e-mhd-fradiss} with $0 \leq \alp < \frac{1}{2}$ and $\eta > 0$). Assume that for some $\eps, \dlt,r > 0$ and $s_0\le s < \min\{ \frac{1}{2\alp}, \frac{1}{2\beta}, 3\}s_0$ (resp.~$s_{0} \leq s < \min \set{\frac{1}{2\alp}, 3}s_{0}$), the solution operator for \eqref{eq:hall-mhd-fradiss} (resp.~\eqref{eq:e-mhd-fradiss}) exists as a map 
		\begin{gather*}
		\calB_{\eps}((0,0); H^{r} \times H^{s})  \to L^{\infty}_{t} ([0, \dlt]; H^{s_{0}-1}) \times L^{\infty}_{t} ([0, \dlt]; H^{s_{0}})  \\
		\left( \hbox{resp. } \calB_{\eps}(0; H^{s}) \to L^{\infty}_{t} ([0, \dlt]; H^{s_{0}}) \right).
		\end{gather*} Then this solution map is unbounded for  $s_0 \ge 3$. 
\end{maintheorem} 

\begin{remark}[Critical dissipation]\label{rem:fradiss-crit}
By a suitable adaptation of our proof of Theorem~\ref{thm:illposed-fradiss} to the critical case $\alp = \frac{1}{2}$ of \eqref{eq:e-mhd-fradiss}, the following statement may be proved (see Remark~\ref{rem:fradiss-crit-pf}):

{\it Assume that for some $s\ge 3$, the solution operator for \eqref{eq:e-mhd-fradiss} exists  as a map $\mathbf{\Phi}: H^s \rightarrow L^\infty_{t}([0,\dlt);H^s)$, where $\dlt = \dlt(\bfB_0)\in (0,\infty]$ is the maximal lifespan of $\bfB_0$; that is, for any $\eps>0$, there is no $L^\infty_t([0,\dlt+\eps);H^s)$-solution to \eqref{eq:e-mhd-fradiss} with initial data $\bfB_0$. Then, there exists an absolute constant $A_0>0$ such that for any $A\ge A_0$, there exists a sequence of initial data $\bfB_0^{(\lmb)}$  such that 
\begin{equation*}
		\begin{split}
		\nrm{\bfB_0^{(\lmb)}}_{H^s} \le \eta A  \mbox{ for any } \lmb\ge 1 
		\end{split}
\end{equation*} 
and at least one of the following holds: 
\begin{itemize}
			\item $\dlt(\bfB_0^{(\lmb)}) \rightarrow 0$ as $\lmb\rightarrow+\infty$, 
			\item for any $\veps > 0$, there exists a constant $c_{\veps} > 0$ and a sequence $t^{(\lmb)} \in (0, \dlt(\bfB_0^{(\lmb)}))$ satisfying  $t^{(\lmb)}\rightarrow 0$ as $\lmb\rightarrow+\infty$ and $\nrm{\mathbf{\Phi}[\bfB_0^{(\lmb)}](t^{(\lmb)})}_{H^s} \ge c_{\veps} \eta A^{(1-\veps)s+1}$.
\end{itemize}
}
While this statement does \emph{not} rule out the possibility of $H^{s}$-wellposedness for $s$ large, it nevertheless shows that the modulus of continuity of $t \mapsto \nrm{\mathbf{\Phi}[\bfB_0](t)}_{H^{s}}$ at $t=0$ cannot be not uniform for $\bfB_0 \in \set{\bfB_0 \in H^{s} : \nrm{\bfB_0}_{H^{s}} \leq \eta A}$, which is in contrast to the case $\alp > 1/2$.
\end{remark}

Compared to the statement of Theorem \ref{thm:illposed-strong}, there are additional restrictions on the range of $s$ relative to $s_0$, which become more strict as $\alp$ approaches $\frac{1}{2}$. Similarly with Theorem \ref{thm:illposed-strong}, this nonlinear illposedness is based on a linear one which we do not state. For the linearized electron-MHD equation with fractional dissipation (i.e., \eqref{eq:e-mhd-lin} with $-\eta(-\lap)^\alp b$ on the right hand side), norm growth statements similar to those given in Theorem \ref{thm:norm-growth} can be proved but they are valid only for $O(\lmb^{-1}\ln \lmb)$ timescale.

\subsection{Brief summary of the results in \cite{JO2}} \label{subsec:wellposed}
We give a short, partial summary of the results in \cite{JO2}, which are various linear and nonlinear wellposedness statements for \eqref{eq:hall-mhd} and \eqref{eq:e-mhd}. They are complementary to the illposedness results proved in the present paper.

\subsubsection*{Linear case} The main linear result in \cite{JO2} is a local geometric condition on the stationary magnetic field that (together with some regularity and uniformity conditions) that ensures above-energy wellposedness for the linearized equations.

\begin{theorem} [{\cite{JO2}}] \label{thm:lin-stab}
Let $\bgB$ be a stationary magnetic field on $M = (\bbT, \bbR)_{x} \times (\bbT, \bbR)_{y} \times (\bbT, \bbR)_{z}$, and consider the deformation tensor associated to $\bgB$:
\begin{equation*}
\dfrm{\bgB}^{jk} = \frac{1}{2} (\nb^{j} \bgB^{k} + \nb^{k} \bgB^{j}).
\end{equation*}
Assume that $\bgB$ is uniformly smooth, $\abs{\bgB}^{-1} \abs{\dfrm{\bgB}}$ is essentially bounded and the following \emph{no-orthogonal-deformation} condition hold everywhere on $M$:
\begin{equation} \label{eq:shear-cond}
	\dfrm{\bgB} \vert_{\bgB^{\perp}}	 = 0 \quad (\hbox{i.e., } \dfrm{\bgB}^{jk} v_{j} w_{k} = 0 \hbox{ if } \bgB^{k} v_{k} = \bgB^{k} w_{k} = 0).
\end{equation}
Then the Cauchy problems for the linearized \eqref{eq:hall-mhd} and \eqref{eq:e-mhd} equations around the stationary solutions $(0, \bgB)$ and $\bgB$, respectively, are well-posed for $H^{\infty}$ (i.e., all derivatives are square integrable) data.
\end{theorem}

This result, when combined with Theorems~\ref{thm:norm-growth} and \ref{thm:inst}, provides a fairly comprehensive description of well- and illposedness of the linearized equations around magnetic fields of the form $f(y) \rd_{x}$ or $f(r) \rd_{\tht}$. 

Indeed, let us first consider $\bgB = f(y) \rd_{x}$, where $f$ is uniformly smooth. When $f$ has a zero of order $1$, the linear illposedness results (Theorems~\ref{thm:norm-growth} and \ref{thm:inst}) apply. On the other hand, in general we have
\begin{equation*}
	\dfrm{\bgB} = \begin{pmatrix} 0 & \frac{1}{2} \rd_{y} f & 0 \\ \frac{1}{2} \rd_{y} f & 0 & 0\\ 0 & 0 & 0 \end{pmatrix}.
\end{equation*}
The no-orthogonal-deformation condition \eqref{eq:shear-cond} obviously fails at a zero of order $1$ of $f$ since $\bgB^{\perp} = \bbR^{3}$ and $\dfrm{\bgB} \neq 0$ there. On the other hand, at points where $f \neq 0$, we have $\bgB^{\perp} = \mathrm{span}\set{\rd_{y}, \rd_{z}}$. Thus, the wellposedness result (Theorem~\ref{thm:lin-stab}) applies whenever $f$ does not vanish anywhere and $\abs{f}^{-1} \abs{f'}$ is essentially bounded.

Similar statements hold for $\bgB = f(r) \rd_{\tht} = f(r) (x \rd_{y} - y \rd_{x})$, for which
\begin{equation*}
	\dfrm{\bgB} = \begin{pmatrix} - f' \frac{x y}{r} & \frac{1}{2} f' \frac{x^{2} - y^{2}}{r} & 0 \\ \frac{1}{2} f' \frac{x^{2} - y^{2}}{r} & f' \frac{xy}{r} & 0 \\ 0 & 0 & 0 \end{pmatrix}.
\end{equation*}
At a zero of order $1$ of $f$, \eqref{eq:shear-cond} clearly fails. However, when $f \neq 0$ and $(x, y) \neq (0, 0)$,
\begin{equation*}
	\bgB^{\perp} = \mathrm{span} \left\{ \frac{x}{r} \rd_{x} + \frac{y}{r} \rd_{y}, \rd_{z} \right\},
\end{equation*}
so it can be checked that $\dfrm{\bgB} \vert_{\bgB^{\perp}} = 0$. In conclusion, Theorems~\ref{thm:norm-growth} and \ref{thm:inst} apply when $f$ has a zero of order $1$, whereas Theorem~\ref{thm:lin-stab} applies whenever $f$ does not vanish anywhere and $\abs{f}^{-1} \abs{f'}$ is essentially bounded.

\begin{remark} \label{rem:planar-stat-excep-0}
Note that the no-orthogonal-deformation condition \eqref{eq:shear-cond} fails for $\bgB = c y \rd_{x} + d \rd_{y}$ with $c \neq 0$, which is the remaining class of planar stationary magnetic fields with an additional symmetry according to Proposition~\ref{prop:planar-stat}. An analysis at the level of bicharacteristics (to be explained below) suggests illposedness of the linearized equations around such a $\bgB$; see Remark~\ref{rem:planar-stat-excep}.
\end{remark}
\subsubsection*{Nonlinear case}
The main nonlinear result in \cite{JO2} is a set of global geometric conditions on the initial magnetic field that imply wellposedness of the nonlinear Cauchy problems for \eqref{eq:hall-mhd} and \eqref{eq:e-mhd}. The precise statement of the conditions and the results requires more preparation, which would take us too far from the subject of this paper. Here we will be content with giving a rough statement of a corollary of the nonlinear results in \cite{JO2}, which illustrates a remarkably different behavior of \eqref{eq:hall-mhd} and \eqref{eq:e-mhd} near nonzero constant magnetic fields, as opposed to the zero magnetic field (cf. Corollary~\ref{cor:illposed-zero} and Theorem~\ref{thm:illposed-strong2}).
\begin{theorem}[{{A special case of the nonlinear wellposedness result in \cite{JO2}}}] \label{thm:parallel-mag-wp}
Let $\bgB$ be any non\-zero constant vector field on $M = \bbR_{x} \times (\bbT, \bbR)_{y} \times (\bbT, \bbR)_{z}$, whose integral curves are noncompact. The Cauchy problem for \eqref{eq:hall-mhd} (resp. \eqref{eq:e-mhd}) is locally well-posed on the unit time interval for initial data $(\bfu_{0}, \bfB_{0})$ such that $(\bfu_{0}, \bfB_{0} - \bgB)$ is sufficiently regular, decaying and small  {(resp. $\bfB_{0}$ such that $\bfB_{0} - \bgB$ is sufficiently regular, decaying and small)}.
\end{theorem}
The key conceptual difference between $\bgB = 0$ and the constant magnetic field considered in Theorem~\ref{thm:parallel-mag-wp} is that in the latter case one can establish a \emph{local smoothing estimate} for the linearized equation, which is a robust (dispersive) smoothing mechanism that overcomes the loss of one derivative seen in \eqref{eq:d-loss}. For further details and more general results, in particular for possibly large perturbations of the constant magnetic field, we refer to \cite{JO2}.

\begin{remark}
We emphasize that the present article and \cite{JO2} are logically independent of each other; neither is a prerequisite for the other. However, since \cite{JO2} carries out a more general analysis of \eqref{eq:hall-mhd} and \eqref{eq:e-mhd}, a concurrent reading of \cite{JO2} may be useful for placing the results and the specialized analysis of the present article in a broader context.
\end{remark}

\subsection{Main ideas} \label{subsec:ideas}
Here, we explain the main ideas of the proofs of our results. We note in advance that Theorem~\ref{thm:norm-growth}, which gives the sharp rate of growth for solutions of the linearized systems, is the basic building block in the proof of all the others.

\subsubsection*{Dispersive character of the Hall current term}
The dispersive character of the Hall current term is most directly seen by linearizing \eqref{eq:e-mhd} around a constant magnetic field $\bgB = \bar{\bfB} \rd_{x}$, where $\bar{\bfB} \neq 0$. The resulting waveform is called the \emph{whistler wave}, which is well-known in the the plasma physics literature (see, for instance, \cite[Section~10.7.1]{Pecseli}). 

Using the vector calculus identities in Section~\ref{subsec:notation}, the linearized system \eqref{eq:e-mhd-lin} reduces to
\begin{equation} \label{eq:e-mhd-lin-const}
	\rd_{t} b + \bar{\bfB} \rd_{x} \nb \times b = 0, \quad \nb \cdot b = 0.
\end{equation}
To diagonalize this system, we take the Fourier transform, which gives
\begin{equation*}
	i \tau \hat{b}(\tau, \xi) - \bar{\bfB} \xi_{x} \xi \times \hat{b}(\tau, \xi) = 0, \quad \xi \cdot \hat{b} (\tau, \xi)= 0.
\end{equation*}
Note that $\bar{\bfB} \xi_{x} \xi \times$ is an anti-symmetric matrix, and $\left( \bar{\bfB} \xi_{x} \xi \times \right)^{2} \hat{b} = - \bar{\bfB}^{2} \xi_{x}^{2} \abs{\xi}^{2} \hat{b}$ (cf. \eqref{eq:curl-curl}) when $\xi \cdot \hat{b} = 0$. Thus $\bar{\bfB} \xi_{x} \xi \times$ restricted to $\set{\hat{b} \in \bbR^{3}: \xi \cdot \hat{b} = 0}$ is diagonalizble with eigenvalues $\pm i \bar{\bfB} \xi_{x} \abs{\xi}$. It follows that \eqref{eq:e-mhd-lin-const} splits into two constant coefficient dispersive PDEs $\rd_{t} b_{\pm} \pm i \omg_{\bgB} (i^{-1}\nb) b_{\pm} = 0$ with the dispersion relations $\pm \omg_{\bgB}$, where
\begin{equation*}
\omg_{\bgB} = \bar{\bfB} \xi_{x} \abs{\xi}.
\end{equation*}

Key to the analysis of \eqref{eq:e-mhd-lin-const} is the \emph{group velocity} $\pm \nb_{\xi} \omg_{\bgB}$, which describes the physical space trajectory of a wave packet solution\footnote{In this heuristic discussion, by a \emph{wave packet solution}, we mean a solution that is well-localized in both the physical and frequency spaces around certain points at each time $t$, which we call $X(t)$ and $\Xi(t)$, respectively. By the physical (resp. frequency or phase) space trajectory, we mean the trajectory of $X$ (resp. $\Xi$ or $(X, \Xi)$).}, at least for a short time. For a further discussion of this case, see \cite{JO2}, in which wellposedness near such a $\bgB$ is established.

\subsubsection*{Diagonalization of the principal symbol and the bicharacteristics}
To look for a mechanism for illposedness, we need an analogous way to analyze more general linearized systems. For an arbitrary stationary magnetic field $\bgB$, the linearized system \eqref{eq:e-mhd-lin} takes the form
\begin{equation*}
	\rd_{t} b + (\bgB \cdot \nb) \nb \times b = \hbox{first or lower order in $b$}.
\end{equation*}
For each $\xi \in \bbR^{3} \setminus \set{0}$, the matrix-valued principal symbol $\bfp_{\bgB} = - (\bgB \cdot \xi) \xi \times$ may be diagonalized on the subspace $\set{u \in \bbR^{3} :  \xi \cdot u = 0}$ in the same fashion as above. The eigenvalues $\pm i p_{\bgB}(x, \xi) = \pm i \bgB(x) \cdot \xi \abs{\xi}$ are analogous to the dispersion relations $\pm i \omg_{\bgB}(\xi)$. The analogue of the group velocity $\nb_{\xi} \omg_{\bgB}$ is the Hamiltonian vector field $(\nb_{\xi} p_{\bgB}, - \nb_{x} p_{\bgB})$ on $T^{\ast} M = M \times \bbR^{3}_{\xi}$; the associated ODE 
\begin{equation*}
\left\{
\begin{aligned}
\dot{X} & = \nb_{\xi} p_{\bgB}(X, \Xi), \\
\dot{\Xi} & = - \nb_{x} p_{\bgB}(X, \Xi),
\end{aligned}
\right.
\end{equation*}
is called the \emph{Hamiltonian ODE}, and its solution $(X, \Xi)(t)$ is called a \emph{bicharacteristic}. A bicharacteristic describes the phase space trajectory of a wave packet solution, at least for a short time. 

\subsubsection*{Stationary magnetic fields with symmetries; complete integrability of the Hamiltonian ODE}
With the above ideas, the natural first goal is to find a stationary magnetic field $\bgB$ with associated bicharacteristics exhibiting a rapid growth of $\abs{\Xi(t)}$. To simplify the problem, it is desirable to restrict to $\bgB$ whose associated Hamiltonian ODE is easily solved. Therefore, we are motivated to look for stationary magnetic fields with \emph{two} independent continuous symmetries, which makes the \emph{three}-dimensional Hamiltonian ODE completely integrable.

Such considerations lead to the idea of using \emph{planar\footnote{We note that the assumption $\bgB^{z} = 0$ may seem excessive in this regard, as it is not related to symmetry. However, this restriction ensures a finer cancellation that is important for the construction of an approximate solution to the PDE; see {\it Construction of ...} below.} stationary magnetic fields with an additional symmetry}. The restrictions make possible a complete classification of all such stationary magnetic fields; see Proposition~\ref{prop:planar-stat}. It is remarkable that the resulting family is still rich enough to allow for many stationary $\bgB$ with an instability mechanism.

\subsubsection*{Instability mechanism at the level of bicharacteristics}
We are ready to describe the key instability mechanism at the level of bicharacteristics in the model case $\bgB = y \rd_{x}$ near $y = 0$. Essentially the same mechanism is present for $\bgB = f(y) \rd_{x}$ or $f(r) \rd_{\tht}$ near any linear degeneracy of $f$ (i.e., $f(y_{0}) = 0$ but $f'(y_{0}) \neq 0$).

We begin by observing that, in addition to the Hamiltonian $p_{\bgB}(X, \Xi) = y(X) \Xi_{x} \abs{\Xi}$, the quantities $\Xi_{x}$ and $\Xi_{z}$ are conserved along the bicharacteristics by the $x$- and $z$-invariance of $\bgB$, respectively. Thus $y(X) \abs{\Xi}$ is conserved, which suggests that bicharacteristics starting from $\set{y = 1}$ and traveling to $\set{y = 0}$ would exhibit a blow-up of $\abs{\Xi}$. 

Motivated by such considerations, we take the bicharacteristic $(X, \Xi)(t)$ with the initial data $X(0) = (0, 1, 0)$ and $\Xi(0) = (\lmb, - \lmb, 0)$ for $\lmb > 0$, so that $\dot{X}(0)$ points towards $\set{y = 0}$; see Figure~\ref{fig:bicharacteristic}. Then the ODEs for $\Xi_{y}$ and $y = y(X)$ become
\begin{equation*}
\dot{\Xi}_{y} = - \Xi_{x} \abs{\Xi} = - \lmb \sqrt{\lmb^{2} + \Xi_{y}^{2}}, \quad
\dot{y} = y \frac{\Xi_{x} \Xi_{y}}{\abs{\Xi}} = \lmb y \frac{\Xi_{y}}{\sqrt{\lmb^{2} + \Xi_{y}^{2}}}
\end{equation*}
which may be explicitly integrated to
\begin{equation*}
	\Xi_{y} = - \lmb \sinh(\lmb t + \tht_{0}), \quad
	y = \frac{\cosh(\tht_{0})}{\cosh(\lmb t + \tht_{0})}.
\end{equation*}
where $\tht_{0} = \sinh^{-1} 1 = \log(1 + \sqrt{2})$. In particular, $y(t) \aeq e^{-\lmb t}$ and $\abs{\Xi(t)} \aeq \lmb e^{\lmb t}$ as desired.

\begin{figure}[h]
\begin{center}
\def\svgwidth{200px}
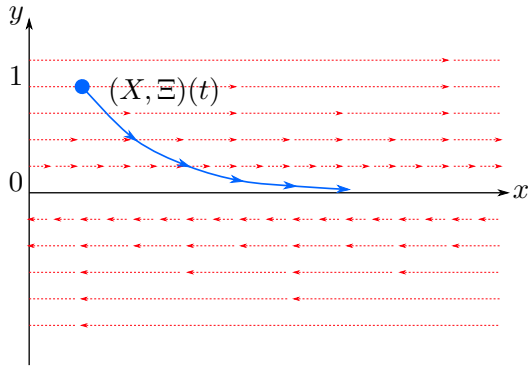 
\caption{Plot of the bicharacteristic $(X, \Xi)(t)$ (blue solid curve) on the background magnetic field $\bgB = y \rd_{x}$ (red dashed lines)} \label{fig:bicharacteristic}
\end{center}
\end{figure}

\begin{remark}[{Instability at the level of bicharacteristics for $\bgB = c y \rd_{x} + d \rd_{y}$, $c \neq 0$}] \label{rem:planar-stat-excep}
In the case of $cy \rd_{x} + d \rd_{y}$ with $c \neq 0$, the ODE for $\Xi_{y}$ is again $\dot{\Xi}_{y} = - c \Xi_{x} \abs{\Xi}$, where $\Xi_{x}$ and $\Xi_{z}$ are conserved. So for $(\Xi_{x}, \Xi_{y}, \Xi_{z})(0) = (\lmb, - \lmb, 0)$, $\abs{\Xi_{y}}$ grows as in the case of $\bgB = y \rd_{x}$! However, the physical space behavior of the bicharacteristic is quite different when $d \neq 0$, as it escapes to $y \to \infty$ at a speed that increases with $\lmb$. We leave open the interesting question of whether this mechanism can be made rigorous at the level of PDEs.
\end{remark}

\subsubsection*{Construction of degenerating wave packets; $(2+\frac{1}{2})$-dimensional reduction and renormalization}
The next step is to construct an actual, or at least approximate, solution to the linearized \eqref{eq:e-mhd} around $\bgB$ that follows the behavior of such a bicharacteristic; this is the basic idea of geometric optics (or semiclassical analysis). However, the standard construction stops short at the so-called \emph{semiclassical time scale} $\abs{t} \ll \lmb^{-1}$, which is just not enough for the growth rate $e^{c \lmb t}$ to take effect. To elongate the construction, we need favorable properties of not only the principal terms, but also the subprincipal (first order) terms of \eqref{eq:e-mhd}.

It is in this analysis, which is carried out in Section~\ref{sec:wavepackets} below, that we use the deepest structural properties of the Hall current term. First, we exploit planarity of $\bgB$ to make the $(2+\frac{1}{2})$-dimensional reduction, which leads to a remarkable simplification of the first order terms\footnote{A precise description of this simplification requires an adequate reformulation of the linearized \eqref{eq:e-mhd} as a system of dispersive equations; we refer the interested reader to \cite{JO2}. Here, we contend ourselves with just pointing out that it is analogous to the vanishing of the vortex-stretching term for the $(2+\frac{1}{2})$-dimensional Euler equation.}; see \eqref{eq:e-mhd-2.5d-lin-second} and \eqref{eq:e-mhd-2.5d-lin-second-axisym}. Moreover, we make a suitable change-of-variables and conjugation, which renormalize the second and first order terms, respectively, to more favorable forms; see \eqref{eq:e-mhd-eta-conj2} and \eqref{eq:e-mhd-eta-conj2-axisym}.  As a result, for any $\lmb > 0$, which corresponds to the initial frequency, we construct a wave packet approximate solution $\tb_{(\lmb)}$ for $\abs{t} \aleq 1$ (as opposed to $\abs{t} \ll \lmb^{-1}$) that is \emph{degenerating} in the sense that all of its $\dot{H}^{s}$ norm (resp. its $L^{p}$ norm) diminish or grow (possibly up to a small error) depending on the sign of $s$ (resp. whether $p < 2$ or $p > 2$) at the rate consistent with the $L^{2}$ boundedness and the growth rate of $\abs{\Xi(t)}$; see Proposition~\ref{prop:wavepackets}.

\subsubsection*{Generalized energy identity and testing by degenerating wave packets}
One way to conclude the proof of the norm growth (Theorem~\ref{thm:norm-growth}) for the linearized \eqref{eq:e-mhd} would be to show that an actual solution $b_{(\lmb)}$ with the same initial data $b_{(\lmb)}(0) = \tb_{(\lmb)}(0)$ is well-approximated by the degenenerating wave packet $\tb_{(\lmb)}$ constructed above, by explicitly estimating the error in the same Sobolev space that we want to see growth. While this approach is possible, it involves cumbersome technicalities, such as careful commutations and additional degenerate elliptic estimates; moreover, it is unclear how to handle in this way an arbitrary $L^{2}$-solution without additionally assuming uniqueness (cf. Remark~\ref{rem:L2-sol}).

Instead, we introduce what we call the method of \emph{testing by degenerating wave packets}, which curtails technicalities and is very robust (the latter property is most clearly demonstrated by the applications to \eqref{eq:hall-mhd} and to nonlinear settings below). Inspired by the work of Ifrim--Tataru \cite{IfTa}, we seek to capture the leading order behavior of the actual solution $b_{(\lmb)}$ by the (energy) inner product $\brk{\tb_{(\lmb)}, b_{(\lmb)}}$ with the wave packet approximate solution $\tb_{(\lmb)}$. By the bilinear generalization of the energy identity, we can control
\begin{equation*}
	\Abs{\frac{\ud}{\ud t}\brk{\tilde{b}_{(\lmb)}, b_{(\lmb)}}} \aleq \nrm{\tb_{(\lmb)}(0)}_{L^{2}} \nrm{b_{(\lmb)}(0)}_{L^{2}} \quad \hbox{ for } 0 < t \aleq 1.
\end{equation*}
Thus, for a sufficiently small $T > 0$ (independent of $\lmb$), $\brk{\tb_{(\lmb)}, b_{(\lmb)}}(t) \geq \frac{1}{2} \nrm{b_{(\lmb)}(0)}_{L^{2}}^{2}$ for $0 < t < T$.
By H\"older's inequality and the simple $L^{p}$-degeneration property of $\tilde{b}_{(\lmb)}$,
\begin{equation*}
	\frac{1}{2} \nrm{b_{(\lmb)}(0)}_{L^{2}}^{2} \leq \nrm{b_{(\lmb)}(t)}_{L^{p}} \nrm{\tilde{b}_{(\lmb)}(t)}_{L^{p'}} \aleq \nrm{b_{(\lmb)}}_{L^{p}} \nrm{b_{(\lmb)}(0)}_{L^{2}} e^{-c (\frac{1}{p} - \frac{1}{2})\lmb t},
\end{equation*}
which implies the desired growth of all $L^{p}$ norms with $p > 2$. By a similar argument using duality and the sharper form of the degeneration property, we also obtain the growth of all $W^{s, p}$ norms with $s + \frac{1}{2} - \frac{1}{p} > 0$, as claimed in Theorem~\ref{thm:norm-growth}.

\subsubsection*{Incorporation of the fluid component}
Now we describe the ideas behind the proof of the norm growth (Theorem~\ref{thm:norm-growth}) for the linearized \eqref{eq:hall-mhd}. The starting point is the bilinear generalization of the energy identity for the linearized \eqref{eq:hall-mhd}. Then, through the method of testing by degenerating wave packets, the problem is reduced to that of finding suitable approximate solutions $(\tu_{(\lmb)}, \tb_{(\lmb)})$, where $\lmb$ corresponds to the initial frequency.  To solve the latter problem, we exploit the remarkable structure of the ideal \eqref{eq:hall-mhd} (i.e., $\nu = 0$) that the combination $\bfZ := \bfB + \nb \times \bfu$ is simply transported by $\bfu$ (cf.~Remark~\ref{rem:wavepackets-hall-err}). Working with the ``good variables'' $(\bfZ, \bfB)$, it follows that $\bfu$ is smoother by one order compared to $\bfB$, if it is initially so. Motivated by this consideration, we consider a pair $(\tu_{(\lmb)}, \tb_{(\lmb)})$ corresponding to taking the $\bfZ$-perturbation to be zero and the $\bfB$-perturbation to be a degenerating wave packet for the linearized \eqref{eq:e-mhd}. Naturally, $(\tu_{(\lmb)}, \tb_{(\lmb)})$ is then a suitable approximate solution to the linearized \eqref{eq:hall-mhd} such that $\tu_{(\lmb)}$ is smoother by one order (i.e., smaller by a factor of $\lmb^{-1}$) compared to $\tb_{(\lmb)}$. Amusingly, the same choice works even for $\nu > 0$, since the improved smoothness of $\tu_{(\lmb)}$ allows us to treat the dissipative term $\nu \lap \bfu$ perturbatively in the method of testing by degenerating wave packets.

\subsubsection*{Superposition of instabilities in frequency space: Proof of Theorem~\ref{thm:inst}}
After the construction of norm-growing solutions (cf. Theorem~\ref{thm:norm-growth}), one may simply superpose (via linearity) a sequence of such solutions with increasing initial frequencies to obtain a solution that is smooth initially, but instantaneously exits any Sobolev space above $L^{2}$. This is the basic idea of the proof of Theorem~\ref{thm:inst}. We note that in order to carry it out, however, uniqueness of $L^{2}$-solutions needs to be assumed.

\subsubsection*{Contradiction argument for nonlinear illposedness: Proof of Theorem~\ref{thm:illposed-strong}}
We now turn to the first nonlinear illposedness result, Theorem~\ref{thm:illposed-strong}. Assuming that the solution map exists (i.e., existence and uniqueness) near a stationary solution $\bgB$ as in the statement of Theorem \ref{thm:illposed-strong}, we need to prove its unboundedness and absence of H\"older continuity depending on the range of $s_0$. In both cases, the idea is to treat the nonlinear terms as a perturbation in the context of testing by degenerating wave packets, using the hypothesis together with the energy identity. For instance, the contribution of a typical nonlinear term $\nb b \nb b$ in $\frac{\ud}{\ud t} \brk{\tb_{(\lmb)}, b}$ (where $b = \bfB - \bgB$) obeys
\begin{equation*}
	\abs{\brk{\tb_{(\lmb)}, \nb b \nb b}(t)} \aleq \nrm{\tb_{(\lmb)}(t)}_{L^{2}} \nrm{b(t)}_{L^{2}} \nrm{b(t)}_{H^{3}}
	\aleq \nrm{\tb_{(\lmb)}(0)}_{L^{2}}\nrm{b(0)}_{L^{2}} \nrm{b}_{L^{\infty}_{t} H^{3}},
\end{equation*}
where we used interpolation in the first inequality, and the energy identity (cf.~Proposition~\ref{prop:nonlin-en}) for the second inequality. This contribution is acceptable thanks to the contradiction assumption $\nrm{b}_{L^{\infty}_{t} H^{s_{0}}} < \infty$ if $s_{0} \geq 3$. Under the assumption of H\"older continuity, one obtains a better estimate for the nonlinear terms, which allows for a lower range of $s_0$. 

\subsubsection*{Superposition of instabilities in physical space: Proof of Theorem~\ref{thm:illposed-strong2}}
As in the proof of Theorem~\ref{thm:inst}, one idea for improving the nonlinear illposedness result is to superpose perturbations of different initial frequencies. Unfortunately, this strategy becomes daunting in the nonlinear case, as the low frequency part may strongly influence the high frequency part. Instead, inspired by an idea in Bourgain--Li \cite{BL1,BL2}, we exploit the nonlinear structure of \eqref{eq:hall-mhd} and \eqref{eq:e-mhd} to superpose disjoint sources of instability in physical space. As a result, we prove \emph{nonexistence} of the solution in high regularity Sobolev spaces.

 {More precisely, the idea is to start with a compactly supported stationary solution $\widetilde{\bfB} = \tilde{f}(y)\rd_x$ (or similarly $\tilde{f}(r)\rd_\theta$) with a linear degeneracy, and consider the following superposition: 
\begin{equation*}
\begin{split}
\bgB = \sum_{k=k_{0}}^{ \infty} \bgB_{k} := \sum_{k=k_{0}}^{\infty} a_{k} \widetilde{\bfB}(L_{k}^{-1} x, L_{k}^{-1} (y - y_{k})) 
\end{split}
\end{equation*} where $a_k > 0$ is decaying sufficiently fast in $k$ so that $\nrm{\bgB}_{H^{s}} \to 0$ as $k_{0} \to \infty$, and the sequence of center $(0, y_{k})$ and the scales $L_{k} > 0$ are chosen so that $\bgB_{k}$'s have disjoint supports; as a result, $\bgB$ is a planar stationary magnetic field. Then we perturb each $\bgB_{k}$ with a very high frequency $\lmb_{k}$, so that the instability induced by the perturbation dominates the decay of the coefficients $a_k$. 

The key ingredient of the proof is localization of the usual and generalized energy identities (i.e., $L^{2}$-inner product of the solution with each degenerating wave packet). The latter task is straightforward since the degenerating wave packets already have good physical space localization properties. The former task is at the heart of the matter; it is in this aspect that \eqref{eq:e-mhd} seems to behave much better than \eqref{eq:hall-mhd} (which may be guessed from the presence of the pressure in \eqref{eq:hall-mhd}), and therefore our proof can cover $M = \bbT^{3}$ only in the case of \eqref{eq:e-mhd}. We refer to Sections~\ref{subsec:illposedness-strong-e} and \ref{subsec:illposedness-strong-hall} for more details.
 }

\subsubsection*{Proof of Gevrey illposedness: Proof of Theorem~\ref{thm:illposed-gevrey}}
In the proof of Theorem \ref{thm:norm-growth}, the initial perturbation is chosen to be supported away from the degenerate point for $\bgB$, which allows for the freedom of choosing a $C^{\infty}$ phase $G(y)$ at our convenience (see \eqref{eq:wkb_G_sol} below). However, such choices are clearly not allowed once the perturbation is required to be in at least the analytic class $G^1$. The main step in proving Theorem~\ref{thm:illposed-gevrey} therefore consists of adapting the construction of degenerating wave packets for Gevrey class initial perturbations with the phase $y$ instead of $G(y)$.

\subsubsection*{Illposedness for the fractionally dissipative systems: Proof of Theorem~\ref{thm:illposed-fradiss}}
The strategy of the proof is identical to that for Theorem~\ref{thm:illposed-strong}; we first prove norm inflation of the linearized system and use a contradiction argument to handle the nonlinearity. Since continuous-in-time loss of one full derivative is explicit in the growth rates given in Theorem~\ref{thm:norm-growth}, we are able to treat the dissipation terms perturbatively. An additional complication arises due to the fact that now the background magnetic field $\bgB$ is time-dependent. Replacing $\bgB(t)$ by $\bgB_0$, one obtains another error term, which is handled by exploiting its smallness in time.

\subsection{Discussions} \label{subsec:discussions}
Further discussion of related papers and subjects is in order.
\subsubsection*{Work of Chae--Weng \cite{CWe} in axisymmetry}
In their intriguing paper \cite{CWe}, D.~Chae and S.~Weng showed that \eqref{eq:hall-mhd} is well-posed\footnote{This observation is implicit in \cite{CWe}, and was later made explicit in \cite{JKL} in the ideal case. To be more precise, one furthermore requires $\bfu^{\tht} = \bfB^{r} = \bfB^{z} = 0$.} within axisymmetry, and moreover that it admits finite time blow-up solutions with regular initial data. These properties are evident for \eqref{eq:e-mhd}, which reduces exactly to inviscid Burger's equation under axisymmetry and $\bfB^{r} = \bfB^{z} = 0$ (see Section~\ref{subsec:notation} for our conventions in cylindrical coordinates):
\begin{equation*}
	\rd_{t} \bfB^{\tht} + \rd_{z} (\bfB^{\tht})^{2} = 0.
\end{equation*}

Our results show severe illposedness for arbitrarily small perturbation away from axisymmetry, if the initial magnetic field either has a linear degeneracy or vanishes in an open annulus. Whether there is a regime in which both local wellposedness (outside symmetry) and the finite time blow-up of Chae--Weng hold remains an interesting open problem.

\subsubsection*{Comparison with the ideal MHD and Alfv\'en waves}
As is well-known, the linearized ideal MHD around a constant magnetic field $\bgB = \bar{\bfB} \rd_{x}$ ($\bar{\bfB} \neq 0$) exhibits a wave propagation phenomenon; the waveform is called the Alfv\'en wave \cite[Section~10.5]{Pecseli}. Unlike whistler waves, whose group speed is proportional to the frequency (dispersive), the group speed of Alfv\'en waves is independent of the frequency (hyperbolic). Moreover, while whistler waves can propagate in directions transversal to the magnetic field (which in fact plays a key role in our instability mechanism; see Figure~\ref{fig:bicharacteristic}), Alfv\'en waves travel only along the magnetic field lines. 

\subsubsection*{Comparison with instabilities in hydrodynamics}
The instability mechanism presented in this work is drastically different and much stronger compared to more traditional hydrodynamical instabilities, such as the Kelvin--Helmholtz, Rayleigh--Taylor, and boundary layer instabilities. They can be respectively described by the Birkhoff--Rott, (compressible or incompressible) Euler with variable density, and Prandtl equations. 

In the case of the Kelvin--Helmholtz instability (see \cite[Chap.~9]{MB}), the linearization around a steady solution explicitly takes the form $\rd_t b= |\nabla|b$ where $b$ is the perturbation. This shows the growth rate of $e^{\lmb t}$ for initial data with frequency $ {\aeq} \lmb$. Next, the growth rate of $e^{c \lmb^{1/2}t}$ is classically known for the linearized systems describing Rayleigh--Taylor (both for compressible and incompressible models; see for instance \cite{Eb, GT} and the references therein). While far less trivial, the same growth rate in frequency was established for the linearized Prandtl equations near certain shear flows \cite{GD}. 

From such growth rates, it follows immediately that these linear equations are ill-posed in $H^s$ with all $s \ge 0$, while the Hall-MHD and electron-MHD equations enjoy stability in $L^2$. On the other hand, these growth rates are not so catastrophic in the sense that as long as the initial data have Fourier spectrum decaying exponentially fast, such a decay property should propagate at least for some time interval. Indeed, local wellposedness in the analytic regularity class for these models (both for linear and nonlinear cases) have been established; see, for instance, \cite{CO1,Su,SC1,SC2,GM,BL}. The propagation of analytic regularity can be reformulated in terms of the growth of the sequence of Sobolev norms $H^s$ ($s \ge 0$), and as far as frequency localized perturbations are concerned, the above growth rates show that all the $H^s$ norms grow at a rate uniform in $s$. In stark contrast to this observation, Theorem \ref{thm:norm-growth} explicitly shows the growth rate of $e^{c s \lmb t}$ for the $H^s$ norm of a perturbation whose frequency is initially localized near $\lmb$, which is not compatible with local wellposedness in the analytic class. The difference in the growth mechanism in our systems can be summarized as follows: rather than simple amplitude growth of Fourier modes, instability is due to transfer of energy to higher Fourier modes with speed proportional to the initial frequency. We also note that unlike our situation, where an $L^2$ bound allows one to treat the nonlinear terms perturbatively, the passage from a linear $H^s$ illposedness result to a nonlinear one in the above problems is highly nontrivial (but see the works \cite{CO2, Eb, GT, GD, GuN}). 

On a different note, we point out that geometric optics techniques, which form the basis of our approach in this paper, have been employed to study localized instabilities of ideal fluids; see the review article \cite{FriLip} and the references therein.

\subsubsection*{On the illposedness result of Brushlinskii--Morozov \cite{BrMo} in the compresible case}
We note an insightful early investigation of Brushlinskii--Morozov \cite{BrMo} that demonstrated illposedness (or in their terminology, ``nonevolutionarity'') of the ideal Hall-MHD in the compressible case. The instability is also due to the degeneracy of the Hall current term; however, it is based on compressibility and is closer to the traditional instabilities discussed earlier (for instance, it is proved by finding highly oscillating plane wave solutions whose amplitudes grow).

\subsubsection*{Instabilities in degenerate dispersive equations}
While our instability mechanism is qualitatively different from the traditional instabilities discussed above, it is prevalent in \emph{degenerate dispersive equations}, which arise from a diverse range of physical and mathematical sources (see the papers cited below and references therein).

An instructive example that is closest to \eqref{eq:hall-mhd} is the two-dimensional variable coefficient ultrahyperbolic Schr\"odinger equation:
\begin{equation} \label{eq:uh-sch}
\rd_t b + if(y)\rd_x\rd_y b = 0, \quad (t, x, y) \in \bbR \times \bbR^{2}.
\end{equation}
Away from the hyperplanes on which $f$ vanishes, this equation is explicitly solvable by essentially the same procedure as in Section~\ref{sec:wavepackets} below: take the Fourier transform in $x$; make the change of variables $(t, y) \mapsto (\tau, \eta)$ so that $\rd_{\tau} = \xi_{x} \rd_{t}$, $\rd_{\eta} = f(y) \rd_{y}$; then observe that the resulting equation is the simple transport equation with the operator $\rd_{\tau} - \rd_{\eta}$.  From such an explicit solution, it may be checked that \eqref{eq:uh-sch} exhibits a qualitatively similar illposedness as that in Theorem~\ref{thm:norm-growth} near a linearly degenerate point $y_{0}$ of $f$, i.e., $f(y_{0}) = 0$ with $f'(y_{0}) \neq 0$. Moreover, we note that the ideas in this paper allow for a straightforward generalization of such an illposedness to nonlinear perturbations of \eqref{eq:uh-sch}.

Another illuminating exercise is to re-examine the classical example of an ill-posed degenerate dispersive equation due to Craig--Goodman \cite{CrGm}. Its behavior seems a bit different from that of \eqref{eq:hall-mhd} at first sight, but it may ultimately be understood by the methodology in this paper. The equation is
\begin{equation} \label{eq:crgm}
	\rd_{t} u \pm x \rd_{x}^{3} u =0, \quad (t, x) \in \bbR \times \bbR.
\end{equation}
By explicitly solving the equation, Craig--Goodman showed that \eqref{eq:crgm} is well-posed in the direction $\pm t > 0$, and ill-posed in the opposite direction. The unidirectionality, which is different from our case, may already be observed at the level of the (formal) energy identity for \eqref{eq:crgm}, which is
\begin{equation} \label{eq:crgm-en}
	\frac{1}{2} \nrm{u(t)}_{L^{2}}^{2}  \pm \frac{3}{2} \nrm{\rd_{x} u}_{L^{2}((0, t);L^{2})}^{2}= \frac{1}{2} \nrm{u(0)}_{L^{2}}^{2} .
\end{equation}
The deeper reason for the unidirectionality is the the behavior of bicharacteristics, all of which propagate towards (resp.~away from) the degeneracy in the direction of ill-(resp.~well-)posedness; see \cite[Section~2]{CrGm}. 

To exhibit illposedness in this example with the methodology of this paper, one begins by constructing degenerating wave packets for $\pm t \leq 0$ based on the bicharacteristics propagating towards the degeneracy (cf.~Section~\ref{sec:wavepackets}). Next, the key ingredient needed to upgrade the behavior of degenerating wave packets to that of actual solutions is a generalized energy identity (cf.~Section~\ref{sec:pf-lin}). The energy identity \eqref{eq:crgm-en} for $u$ is unsuitable for this purpose due to the presence of the term involving $\rd_{x} u$. Nevertheless, the problematic term may be removed by considering a suitable conjugation $v = T u$ up to acceptable lower order terms (e.g., a Fourier multiplier $\widehat{T u} = m(\xi) \widehat{u}$ with $m(\xi)$ smooth and $m(\xi) = \mathrm{sgn}\, \xi \abs{\xi}^{\frac{1}{2}}$ for $\abs{\xi} > 1$ would do). Such an approach has the advantage of being far more robust compared to the explicit solution method in \cite{CrGm}.

For further discussion and results in this direction, we refer to our follow-up work \cite{JO-disp}, in which we extend the methods developed in this paper to establish illposedness of the Cauchy problem in standard function spaces, such as Sobolev spaces with arbitrary high regularities, for a wide class of one-dimensional nonlinear degenerate dispersive equations (in particular, degenerate KdV-type equations, for which \eqref{eq:crgm} serves as a model). See also the prior works \cite{ASWY, CrGm} in the direction of illposedness for degenerate dispersive equations. Concerning the existence and uniqueness of solutions with degeneracies\footnote{The illposedness results in \cite{JO-disp} also apply to the equations considered in \cite{GHGM1, GHGM2, HGM}, which seem contradictory at first sight. Rather, these results are complementary. To wit, while \cite{JO-disp} shows that even the existence of the solution map fails with respect to standard function spaces (e.g., high regularity Sobolev spaces), \cite{GHGM2, HGM} prove existence and uniqueness in certain function spaces adapted to the degeneracies of the solution.}, we also note the interesting recent works \cite{GHGM1, GHGM2, HGM}.

\subsubsection*{Removing $z$-independence}
Since our domain $M$ is always taken to be periodic in $z$ and we work exclusively with $z$-independent solutions (with the exception of Theorem \ref{thm:illposed-strong2}), the reader might wonder whether this is essential. However, this is not the case. First of all, let us point out that in Theorem \ref{thm:norm-growth}, the requirement that the $L^2$-solution to the linearized equations be $z$-independent may be easily lifted. (This is rather trivial, as the extra terms appearing in the linearization for a $z$-dependent solution come with $\rd_z$, so that they disappear after integrating by parts against a degenerating wave packet which is $z$-independent.) In a similar vein, the linearized equations themselves can be considered in $\mathbb{R}^3$, and using the method developed in this paper, it is not difficult to prove the same rate of growth for initial data which has either compact support or decaying fast in the $z$-direction. 

However, the preceding discussion is not entirely satisfactory, as the background magnetic field $\bgB$ is still kept $z$-independent. More interestingly, by considering the linearized systems against a \emph{$z$-dependent background magnetic field} solving \eqref{eq:hall-mhd} or \eqref{eq:e-mhd}, which may be compactly supported in $\bbR^3$, it is possible to prove nonlinear illposedness results for compactly supported data in $\bbR^{3}$. For more details, see our follow-up work \cite{JO-z-dep}, in which this strategy is carried out.

\subsection{Notation, conventions and some useful vector calculus identities} \label{subsec:notation}
Here, we collect some notation, conventions and vector calculus identities that will be used freely in the remainder of the paper.

\subsubsection*{Notation and conventions}
By $A \aleq B$, we mean that there exists some positive constant $C > 0$ such that $\abs{A} \leq C B$. The dependency of the implicit constant $C$ is specified by subscripts, e.g. $A \aleq_{E} B$. By $A \aeq B$, we mean $A \aleq B$ and $A \ageq B$.

We denote by $\bbR$ the real line, $\bbZ$ the set of integers, $\bbT = \bbR /  {2 \pi \bbZ}$ the  {torus with length $2 \pi$}, $\bbN_{0} = \set{0, 1, 2, \ldots}$ the set of nonnegative integers and $\bbN = \set{1, 2, \ldots}$ the set of positive integers.

We write $M$ for the $3$-dimensional domain of the form $\bbT^{k} \times \bbR^{3-k}$ ($0 \leq k \leq 3$) equipped with the rectangular coordinates $(x, y, z)$, and $M^{2} = M^{2}_{x, y}$ for the two-dimensional projection of $M$ along the $z$-axis. We use the notation $\brk{u, v}_{M}$ and $\brk{u, v}$ for the standard $L^{2}$-inner product for vector fields on $M$ and $M^{2}$, respectively; i.e.,
\begin{equation*}
\brk{\bfu, \bfv}_{M} = \int_{M} \bfu \cdot \bfv \, \ud x \ud y \ud z, \quad \brk{u, v} = \int_{M^{2}} u \cdot v \, \ud x \ud y.
\end{equation*}
Given a vector $u$ on $M^{2}$, we define its perpendicular $u^{\perp}$ by
\begin{equation*} 
	u^{\perp} = \begin{pmatrix} - u^{y} \\ u^{x} \end{pmatrix},
\end{equation*}
and accordingly, we introduce the perpendicular gradient operator
\begin{equation} \label{eq:grad-perp}
\nb^{\perp} = \begin{pmatrix} - \rd_{y} \\ \rd_{x} \end{pmatrix}.
\end{equation}

We use the usual notation $W^{s, p}$ for the $L^{p}$-based Sobolev space of regularity $s$; when $p = 2$, we write $H^{s} = W^{s, 2}$. The mixed Lebesgue norm $L^{p}_{x} L^{q}_{y}$ is defined as 
\begin{equation*}
	\nrm{u}_{L^{p}_{x} L^{q}_{y}} = \nrm{\nrm{u(x, y)}_{L^{q}_{y}}}_{L^{p}_{x}}.
\end{equation*}
The norm $L^{p}_{t} H^{s}$ is defined similarly. 

Given any space $X$ of functions on $M$, we denote by $X_{comp}(M)$ the subspace of compactly supported elements of $X$, and by $X_{loc}(M)$ the space of functions $u$ such that $\chi u \in X_{loc}(M)$ for any smooth compactly supported function $\chi$ on $M$.

\subsubsection*{Vector calculus identities}
We recall some useful vector calculus identities:
\begin{align} 
	\bfU \times (\bfV \times \bfW) & = \bfV (\bfU \cdot \bfW) - \bfW (\bfU \cdot \bfV) , \label{eq:cross-cross} \\
%
%
	\nb \times (\bfU \times \bfV) 
	& = (\bfV \cdot \nb) \bfU + \bfU (\nb \cdot \bfV) - (\bfU \cdot \nb) \bfV - \bfV (\nb \cdot \bfU) , \label{eq:curl-cross} \\
%
%
	(\nb \times \bfU) \times \bfV 
	& = (\bfV \cdot \nb) \bfU - \bfV_{j} \nb \bfU^{j} , \label{eq:cross-curl} \\
%
%
	\nb \times (\nb \times \bfU) 
	& = - \lap \bfU + \nb (\nb \cdot \bfU) . \label{eq:curl-curl}
\end{align}

\subsubsection*{Vector calculus in cylindrical coordinates}
 {The cylindrical coordinates $(r, \theta, z)$ are defined by}
\begin{equation*}
	r = \sqrt{x^{2} + y^{2}}, \quad \theta = \tan^{-1} \frac{y}{x}.
\end{equation*}
In this paper, we use the \emph{coordinate derivative basis} $(\rd_{r}, \rd_{\tht}, \rd_{z})$ to decompose vectors into components, i.e., given a vector $\bfU$ on $M$, we define its components $\bfU^{r}, \bfU^{\tht}, \bfU^{z}$ by
\begin{equation*}
	\bfU = \bfU^{r} \rd_{r} + \bfU^{\tht} \rd_{\tht} + \bfU^{z} \rd_{z}.
\end{equation*}
Another widespread choice is its normalization $(e_{r}, e_{\tht}, e_{z}) = (\rd_{r}, r^{-1} \rd_{\tht}, \rd_{z})$, which differs from our choice by factors of $r$. The advantage of our choice is that the change of coordinates formulas are simpler; the disadvantage is that the inner product takes the inconvenient form
\begin{equation} \label{eq:dot-cylin}
	\bfU \cdot \bfV = \bfU^{r} \bfV^{r} + r^{2} \bfU^{\tht} \bfV^{\tht} + \bfU^{z} \bfV^{z}.
\end{equation}
The gradient, curl and divergence in the cylindrical coordinates are
\begin{equation} \label{eq:grad-cylin}
	\nb a = (\rd_{r} a) \rd_{r} + (r^{-2} \rd_{\tht} a) \rd_{\tht}  {+ (\rd_z a)\rd_z },
\end{equation}
\begin{equation}\label{eq:curl-cylin}
\begin{aligned}
\nabla \times (\bfU^r{\rd_r}+\bfU^\theta {\rd_\theta} + \bfU^z{\rd_{z}}) =& (r^{-1}\rd_\theta \bfU^z - r \rd_{z} \bfU^{\tht}){\rd_r} + (r^{-1} \rd_{z} \bfU^{r} - r^{-1} \rd_r \bfU^z) {\rd_\theta} \\
& + r^{-1}(\rd_r(r^{2}\bfU^\theta) - \rd_\theta \bfU^r){\rd_{z}},
\end{aligned}\end{equation}
\begin{equation} \label{eq:div-cylin}
\nabla \cdot (\bfU^r{\rd_r}+\bfU^\theta {\rd_\theta} + \bfU^z{\rd_{z}}) = r^{-1}\rd_r(r\bfU^r) + \rd_\theta \bfU^\theta + \rd_{z} \bfU^{z}.
\end{equation}
 {Assuming that $a$ is independent of $z$, the} perpendicular gradient takes the form
\begin{equation} \label{eq:grad-perp-cylin}
	\nb^{\perp} a = - (r^{-1} \rd_{\tht} a) \rd_{r} + (r^{-1} \rd_{r} a) \rd_{\tht}.
\end{equation}

\subsection{Organization of the paper}
The rest of the paper is organized as follows.
\begin{itemize}
\item In {\bf Section~\ref{sec:2.5d}}, we carry out some basic algebraic manipulation and derive the energy identities for the linearized \eqref{eq:hall-mhd} and \eqref{eq:e-mhd} under the $(2+\frac{1}{2})$-dimensional reduction, which is a particularly simple reformulation of these equations assuming $z$-independence.
\item {\bf Section~\ref{sec:wavepackets}} is the heart of this paper, where we construct degenerating wave packet approximate solutions to the linearized \eqref{eq:e-mhd} and \eqref{eq:hall-mhd}, under the $(2+\frac{1}{2})$-dimensional reduction and around a planar stationary magnetic field with an additional symmetry.
\item In {\bf Section~\ref{sec:pf-lin}}, the energy identities in Section~\ref{sec:2.5d} and the degenerating wave packets constructed in Section~\ref{sec:wavepackets} are combined to prove the linear Sobolev illposedness results, Theorems~\ref{thm:norm-growth} and \ref{thm:inst}. 
\item In {\bf Section~\ref{sec:pf-nonlin}}, we establish the nonlinear illposedness results, Theorems~\ref{thm:illposed-strong} and \ref{thm:illposed-strong2}. 
\item In {\bf Section~\ref{sec:gevrey}}, we establish the (linear) Gevrey illposedness result, Theorem~\ref{thm:illposed-gevrey}.
\item Finally, in {\bf Section~\ref{sec:fradiss}}, we establish the  illposedness result for the fractionally dissipative systems, Theorem~\ref{thm:illposed-fradiss}.
\end{itemize}
The paper is supplemented with {\bf Appendix~\ref{sec:L2-exist}}, where we sketch the proof of existence of an $L^{2}$-solution for the linearized systems.

\ackn{Both authors are grateful to Dongho Chae for bringing the Hall magnetohydrodynamic equation to their attention, and for his interest in this work. 
I.-J.~Jeong thanks Tarek Elgindi for suggesting several references related to the magnetohydrodynamic systems. 
I.-J.~Jeong has been supported by the POSCO Science Fellowship of POSCO TJ Park Foundation.  {S.-J.~Oh was supported by the Samsung Science and Technology Foundation under Project Number SSTF-BA1702-02, a Sloan Research Fellowship and a National Science Foundation CAREER Grant under NSF-DMS-1945615.}}

\section{The $(2+\frac{1}{2})$-dimensional reduction and linearized energy identities} \label{sec:2.5d}
The purpose of this section is to record the basic algebraic manipulations and energy identities for our proof of the illposedness results.

\subsection{The $(2+\frac{1}{2})$-dimensional reduction of \eqref{eq:hall-mhd} and \eqref{eq:e-mhd}} \label{subsec:2.5d-nonlin}
Here we derive a simpler reformulation of \eqref{eq:hall-mhd} and \eqref{eq:e-mhd} under one translational symmetry (or, more concretely, independence on the $z$-coordinate), which involves the $z$-component of the solution and of its curl. Following the usual terminology in fluid mechanics, we refer to this procedure as the \emph{$(2+\frac{1}{2})$-dimensional reduction}. The derivation in this subsection is formal; for justification in cases that arise in applications, we refer to Propositions~\ref{prop:justify-eq} and \ref{prop:justify}.

\subsubsection*{The $(2+\frac{1}{2})$-dimensional reduction of \eqref{eq:hall-mhd}}
We take the system \eqref{eq:hall-mhd} and simplify it under the assumption of $z$-independence. Dealing with the equation for $\bfB$ first, we have \begin{equation*}
\begin{split}
\rd_t \bfB - (\bfB\cdot\nabla) \bfu +(\bfu\cdot\nabla)\bfB + (\bfB\cdot\nabla)(\nabla\times \bfB) - ((\nabla\times \bfB)\cdot\nabla)\bfB = 0.
\end{split}
\end{equation*} Taking the $z$-component, we obtain \begin{equation*}
\begin{split}
\rd_t \bfB^{z} - (\bfB\cdot\nabla) \bfu^{z} +(\bfu\cdot\nabla)\bfB^{z} + (\bfB\cdot\nabla) (\nabla\times \bfB)^{z}  = 0.
\end{split}
\end{equation*} with the observation that
\begin{equation*}
	((\nb \times \bfB) \cdot \nb) \bfB^{z} = \rd_{y} \bfB^{z} \rd_{x} \bfB^{z} - \rd_{x} \bfB^{z} \rd_{y} \bfB^{z} = 0.
\end{equation*} On the other hand, returning to the form \begin{equation*}
\begin{split}
\rd_t \bfB - \nabla\times ( \bfu\times \bfB) + \nabla\times((\nabla\times \bfB)\times \bfB) = 0
\end{split}
\end{equation*} then using \eqref{eq:curl-curl}, \eqref{eq:cross-curl} and $z$-independence, we obtain \begin{equation} \label{eq:hall-mhd-2.5d-curlB}
\rd_t (\nabla\times \bfB)^{z} + \lap \left( ( \bfu \times \bfB)^{z} - (\bfB\cdot\nabla)\bfB^{z} \right) = 0.
\end{equation} Turning to the equation for $\bfu$, taking the $z$-component gives \begin{equation*}
\begin{split}
\rd_t \bfu^{z} + (\bfu\cdot\nabla) \bfu^{z} - \nu \lap \bfu^{z} = (\bfB \cdot \nb) \bfB^{z},
\end{split}
\end{equation*} where we note that the pressure term vanishes by $z$-independence. 
Taking the $z$-component of the curl gives \begin{equation*}
\begin{split}
\rd_t(\nabla\times\bfu)^{z} + (\bfu \cdot \nb) (\nb \times \bfu)^{z} - \nu \lap (\nb \times \bfu)^{z} = (\bfB \cdot \nb) (\nb \times \bfB)^{z}.
\end{split}
\end{equation*}  
Note that the divergence-free condition reads $$\rd_x \bfB^{x} + \rd_y \bfB^{y} = 0,\qquad \rd_x \bfu^{x} + \rd_y \bfu^{y} = 0.$$ We introduce the notation $\bfomg = (\nb \times \bfu)^{z}$ for the $z$-component of the vorticity. Then we arrive at the following closed system of four scalar quantities depending on two variables $(x,y)$: \begin{equation} \label{eq:hall-mhd-2.5d}
\left\{
\begin{aligned}
&\rd_{t} \bfu^{z} + (\bfu\cdot\nabla) \bfu^{z} - (\bfB\cdot\nabla)\bfB^{z} - \nu \lap \bfu^{z} =0 ,  \\
&\rd_{t} \bfomg - (\bfu \cdot \nb) \bfomg - (\bfB \cdot \nb) (\nb \times \bfB)^{z}  - \nu \lap \bfomg = 0 , \\
&\rd_t  \bfB^{z} - (\bfB\cdot\nabla) \bfu^{z} +(\bfu\cdot\nabla)\bfB^{z} + (\bfB\cdot\nabla) (\nb \times \bfB)^{z}  = 0 ,\\
&\rd_t (\nabla\times \bfB)^{z} + \lap \left( ( \bfu \times \bfB)^{z} - (\bfB\cdot\nabla)\bfB^{z} \right) = 0, 
\end{aligned}
\right.
\end{equation} where the system is to be supplemented with the div-curl relations 
\begin{equation}\label{eq:hall-mhd-2.5d-rel-u}
\left\{
\begin{aligned}
	&\rd_{x} \bfu^{x} + \rd_{y} \bfu^{y} = 0, \\
	&\rd_{x} \bfu^{y} - \rd_{y} \bfu^{x} = \bfomg,
\end{aligned}
\right.\end{equation} 
as well as
\begin{equation} \label{eq:hall-mhd-2.5d-rel-b}
\left\{
\begin{aligned}
	&\rd_{x} \bfB^{x} + \rd_{y} \bfB^{y} = 0, \\
	&\rd_{x} \bfB^{y} - \rd_{y} \bfB^{x} = (\nb \times \bfB)^{z}.
\end{aligned}
\right.\end{equation} 

\subsubsection*{The $(2+\frac{1}{2})$-dimensional reduction of \eqref{eq:e-mhd}}
The case of \eqref{eq:e-mhd} is easily obtained from the preceding computation by formally setting $\bfu^{z}$ and $\bfomg$ equal to zero. 
The $(2+\frac{1}{2})$-dimensional reduction of \eqref{eq:e-mhd} is the closed system of two scalar quantities
\begin{equation*}
	\bfB^{z}, \quad (\nb \times \bfB)^{z},
\end{equation*}
depending only on two variables $(x, y)$, of the following form:
\begin{equation} \label{eq:e-mhd-2.5d}
\left\{
\begin{aligned}
&\rd_t  \bfB^{z}  + (\bfB\cdot\nabla) (\nb \times \bfB)^{z}  = 0 ,\\
&\rd_t (\nb \times \bfB)^{z}  - \lap (\bfB\cdot\nabla)\bfB^{z} = 0.
\end{aligned}
\right.
\end{equation} 
As before, the system is to be supplemented with the div-curl relation
\begin{equation} \label{eq:e-mhd-2.5d-rel}
\left\{
\begin{aligned}
	&\rd_{x} \bfB^{x} + \rd_{y} \bfB^{y} = 0, \\
	&\rd_{x} \bfB^{y} - \rd_{y} \bfB^{x} = (\nb \times \bfB)^{z}.
\end{aligned}
\right.
\end{equation}
where the mean $\mean{\bfB}$ only arises in the case $M^{2} = \bbT^{2}$.

\subsection{Stationary planar magnetic fields with an additional symmetry} \label{subsec:planar-stat}
In this short subsection, we provide a quick proof of Proposition~\ref{prop:planar-stat} using the preceding computation.

\begin{proof}[Proof of Proposition~\ref{prop:planar-stat}]
In what follows, we write $\bfB$ instead of $\bgB$ for simplicity. By planarity, $\bfB^{z} = 0$, and by stationarity, $\rd_{t} \bfB = 0$. Thus, it follows from \eqref{eq:e-mhd-2.5d} that
\begin{equation} \label{eq:planar-stat-curl}
	\bfB \cdot \nb (\nb \times \bfB)^{z} = 0 \quad \hbox{ in } M.
\end{equation}
As remarked above, there are two possibilities for an additional symmetry: Either (1) $\bfB$ is independent of one of the coordinates, which may be taken to be $x$ without loss of generality, or (2) $\bfB^{x} \rd_{x} + \bfB^{y} \rd_{y}$ is axi-symmetric in $\bbR^{2}_{x, y}$, where the axis may be taken to be the origin $(0, 0)$ without loss of generality.

In case~(1), \eqref{eq:planar-stat-curl} and the divergence-free condition amount to:
\begin{equation*}
	\bfB^{y} \rd_{y}^{2} \bfB^{x} = 0, \quad \rd_{y} \bfB^{y}(y) = 0,
\end{equation*}
whose general solution has the form $\bfB = f(y) \rd_{x}$ or $\bfB = (c_{1} y + c_{0}) \rd_{x} + d \rd_{y}$, as desired.

In case~(2), we write $\bfB$ in the cylindrical coordinates $(r, \tht, z)$ as
\begin{equation*}
	\bfB = \bfB^{r}(r) \rd_{r} + \bfB^{\tht}(r) \rd_{\tht}.
\end{equation*}
Then \eqref{eq:planar-stat-curl} and the divergence-free condition become
\begin{equation*}
	\bfB^{r} \rd_{r} (r^{-1} \rd_{r} (r^{2} \bfB^{\tht})) = 0, \qquad \rd_{r} (r \bfB^{r}) = 0.
\end{equation*}
The second equation (divergence-free) and the requirement of smoothness of $\bfB$ at the origin force $\bfB^{r} = 0$; thus we are left with a general solution of the form $\bfB = f(r) \rd_{\tht}$. \qedhere
\end{proof}

\subsection{Perturbed and linearized equations under the $(2+\frac{1}{2})$-dimensional reduction} \label{subsec:2.5d-lin}
Here we derive the full and linearized equations satisfied by the perturbation $(\bfu, \bfB) = (0, \bgB) + (u, b)$ of the stationary solutions $(0, \bgB)$ considered in Theorem~\ref{thm:norm-growth}, under the $(2+\frac{1}{2})$-dimensional reduction. We first present formal derivations, and then describe the precise sense in which the reduced equations hold for $L^{2}$-solutions in Proposition~\ref{prop:justify-eq}.

In our derivation, as in the definition of an $L^{2}$-solution for the linearized equations, \emph{we assume the mean-zero condition \eqref{eq:mean-zero} for the perturbations when $M^{2} = \bbT^{2}_{x, y}$}, which ensure that Biot--Savart-type identities hold; see \eqref{eq:lin-ub-rel} and \eqref{eq:lin-b-rel} below.

\subsubsection*{Translationally-symmetric background, \eqref{eq:hall-mhd}}
 {Consider} a translationally-symmetric background magnetic field of the form $\bgB =  f(y)\rd_x$. Recall that $\bfu =u$, $\bfB = \bgB + b$; the vector fields $u$ and $b$ are divergence-free.  Introducing now \begin{equation*}
\begin{split}
(\nabla\times b)^{z} = -\lap \psi, \quad (\nabla\times u)^{z} = \omg
\end{split}
\end{equation*} 
we formally have the Biot--Savart-type identity
\begin{equation} \label{eq:lin-ub-rel}
	b^{x, y} = - \nb^{\perp} \psi, \qquad u^{x, y} = - \nb^{\perp} (-\lap)^{-1} \omg.
\end{equation}
Since 
\begin{equation*}
(\nb \times \bgB)^{z} = - f'(y),
\end{equation*}
we have
\begin{equation*}
\begin{split}
  \psi = (-\lap)^{-1}((\nb \times \bfB)^{z} + f'(y)) ,\quad \omg = \bfomg.
\end{split}
\end{equation*} 
Therefore, from \eqref{eq:hall-mhd-2.5d}, we may derive the following perturbation equation satisfied by the quadruple $(u^{z}, \omg, b^{z}, \psi)$:
\begin{equation} \label{eq:hall-mhd-2.5d-pert}
\left\{
\begin{aligned}
&\rd_{t} u^{z} - f(y) \rd_{x} b^{z} - \nu \lap u^{z} = - u \cdot \nb u^{z}, \\
&\rd_{t} \omg - f''(y) \rd_{x} \psi + f(y) \rd_{x} \lap \psi - \nu \lap \omg = - u \cdot \nb \omg + \nb^{\perp} \psi \cdot \nb \lap \psi,  \\
&\rd_t  b^{z} - f(y)\rd_x u^{z}   + f''(y)\rd_x\psi  - f(y)\rd_x \lap\psi = - u \cdot \nb b^{z} - \nb^{\perp} \psi \cdot \nb u^{z} \\
&\phantom{\rd_t  b^{z} - f(y)\rd_x u^{z}   + f''(y)\rd_x\psi  - f(y)\rd_x \lap\psi =}
	- \nb^{\perp} \psi \cdot \nb \lap \psi ,\\
&\rd_t \psi - f(y) \rd_{x} (-\lap)^{-1} \omg + f(y)\rd_xb^{z} = - u \cdot \nb \psi + \nb^{\perp} \psi \cdot \nb b^{z}.
\end{aligned}
\right.
\end{equation}
Removing all the quadratic terms in $u$ and $b$, we arrive at the linearized system: 
\begin{equation} \label{eq:hall-mhd-2.5d-lin}
\left\{
\begin{aligned}
&\rd_{t} u^{z} - f(y) \rd_{x} b^{z} - \nu \lap u^{z}= 0, \\
&\rd_{t} \omg - f''(y) \rd_{x} \psi + f(y) \rd_{x} \lap \psi - \nu \lap \omg= 0,  \\
&\rd_t  b^{z} - f(y)\rd_{x} u^{z}   + f''(y)\rd_x\psi  - f(y)\rd_x \lap\psi = 0, \\
&\rd_t \psi - f(y) \rd_{x} (-\lap)^{-1} \omg + f(y)\rd_xb^{z} = 0 {.} 
\end{aligned}
\right.
\end{equation}
We note that the LHS of the equation for $\rd_{t} \omg$ in \eqref{eq:hall-mhd-2.5d-pert} and \eqref{eq:hall-mhd-2.5d-lin} may be rewritten in the divergence form
\begin{equation} \label{eq:omg-eq-div}
\rd_{t} \omg - \nb \cdot (f' \nb^{\perp} \psi) + \rd_{x} (f \lap \psi) - \nu \lap \omg.
\end{equation}

\subsubsection*{Translationally-symmetric background, \eqref{eq:e-mhd}}
The counterpart of the perturbation equation for the electron-MHD case is simply obtained by setting $\omg = u^{z} = 0$. Thus, the Bio--Savart-type identity is
\begin{equation} \label{eq:lin-b-rel}
	b^{x, y} = - \nb^{\perp} \psi, 
\end{equation}
the perturbation equation is of the form
\begin{equation} \label{eq:e-mhd-2.5d-pert}
\left\{
\begin{aligned}
&\rd_t  b^{z} + f''(y)\rd_x\psi  - f(y)\rd_x \lap\psi = - \nb^{\perp} \psi \cdot \nb \lap \psi, \\
&\rd_t \psi + f(y)\rd_x b^{z} = \nb^{\perp} \psi \cdot \nb b^{z},
\end{aligned}
\right.
\end{equation}
and linearized equation is \begin{equation} \label{eq:e-mhd-2.5d-lin}
\left\{
\begin{aligned}
&\rd_t b^{z} - f(y)\rd_x \lap\psi  + f''(y)\rd_x\psi   = 0 ,\\
&\rd_t \psi  + f(y)\rd_x b^{z} = 0.
\end{aligned}
\right.
\end{equation}

\subsubsection*{Axisymmetric background, \eqref{eq:hall-mhd}}
Let us take the systems \eqref{eq:hall-mhd-2.5d} and \eqref{eq:e-mhd-2.5d} and write down the linearized equations around $\bgB = f(r) \rd_\theta$. We shall use the standard cylindrical coordinates system $(r,\theta,z)$ with coordinate vectors $(\rd_r, \rd_\theta, \rd_z)$, and denote the components of a vector $\bfU$ in the following form: \begin{equation*}
\begin{split}
\bfU = \bfU^{r} \rd_r + \bfU^{\theta} \rd_\theta + \bfU^{z} \rd_z .
\end{split}
\end{equation*} Assuming $z$-independence component functions in cylindrical coordinates, the formulas \eqref{eq:curl-cylin} and \eqref{eq:div-cylin} simplify to: \begin{equation}\label{eq:del_in_cylindrical_z_indep}
\begin{split}
\nabla \times (\bfU^r{\rd_r}+\bfU^\theta {\rd_\theta} + \bfU^z{\rd_{z}}) &= r^{-1}\rd_\theta \bfU^z{\rd_r} - r^{-1} \rd_r\bfU^z {\rd_\theta} + r^{-1}(\rd_r(r^{2}\bfU^\theta) - \rd_\theta \bfU^r){\rd_{z}}, \\
\nabla \cdot (\bfU^r{\rd_r}+\bfU^\theta {\rd_\theta} + \bfU^z{\rd_{z}}) &= r^{-1}\rd_r(r\bfU^r) + \rd_\theta \bfU^\theta  .
\end{split}
\end{equation} Moreover, for a scalar function $a$ independent of $z$, a combination of \eqref{eq:div-cylin} and \eqref{eq:grad-cylin} imply\begin{equation}\label{eq:lap_in_cylindrical_z_indep}
\begin{split}
\lap a = \frac{1}{r} \rd_r(r\rd_r a) + \frac{1}{r^2}\rd_\theta^2 a. 
\end{split}
\end{equation}

Equipped with the above preliminaries, we are ready to derive the perturbation and linearized equations. We write $(\bfu, \bfB) = (u, \bgB + b)$ and introduce \begin{equation*}
\begin{split}
(\nabla\times b)^{z} = -\lap\psi, \qquad (\nabla\times u)^{z} = \omg.
\end{split}
\end{equation*} 
Since
\begin{equation*}
(\nabla\times\bgB)^{z} = r^{-1}\rd_r(r^{2} f),
\end{equation*}
we obtain that \begin{equation*}
\begin{split}
\psi = (-\lap)^{-1}((\nb \times \bfB)^{z} - r^{-1}\rd_r(r^{2} f)),\quad \omg = \bfomg.
\end{split}
\end{equation*}
Then, using \begin{equation*}
\begin{split}
b^{r} = \frac{1}{r}\rd_\theta \psi, \quad b^{\theta} = -\frac{1}{r} \rd_r\psi, 
\end{split}
\end{equation*} which follows from the formulas $b = -\nb^{\perp} \psi$ and \eqref{eq:grad-perp-cylin}, we obtain the following perturbation equation:
\begin{equation} \label{eq:hall-mhd-2.5d-pert-axisym}
\left\{
\begin{aligned}
&\rd_{t} u^{z} - f(r) \rd_{\tht} b^{z} - \nu \lap u^{z} = - u \cdot \nb u^{z}, \\
&\rd_{t} \omg  - \left(f''(r) + \frac{3}{r} f'(r)\right) \rd_{\tht} \psi + f(r) \rd_{\tht} \lap \psi - \nu \lap \omg = - u \cdot \nb \omg + \nb^{\perp} \psi \cdot \nb \lap \psi,  \\
&\rd_t  b^{z} - f(r)\rd_{\theta} u^{z} + \left( f''(r) + \frac{3}{r}f'(r) \right)\rd_{\theta}\psi- f(r)\rd_{\theta} \lap\psi = - u \cdot \nb b^{z} - \nb^{\perp} \psi \cdot \nb u^{z} \\
&\phantom{\rd_t  b^{z} - f(r)\rd_{\theta} u^{z} + \left( f''(r) + \frac{3}{r}f'(r) \right)\rd_{\theta}\psi- f(r)\rd_{\theta} \lap\psi =}
	- \nb^{\perp} \psi \cdot \nb \lap \psi ,\\
&\rd_t \psi - f(r) \rd_{\tht} (-\lap)^{-1} \omg + f(r)\rd_\theta b^{z}  = - u \cdot \nb \psi + \nb^{\perp} \psi \cdot \nb b^{z}.
\end{aligned}
\right.
\end{equation}
Removing all the quadratic terms in $u$ and $b$, we arrive at the linearized system: 
\begin{equation} \label{eq:hall-mhd-2.5d-lin-axisym}
\left\{
\begin{aligned}
&\rd_{t} u^{z} - f(r) \rd_{\tht} b^{z} - \nu \lap u^{z}= 0 , \\
&\rd_{t} \omg  - \left(f''(r) + \frac{3}{r} f'(r)\right) \rd_{\tht} \psi + f(r) \rd_{\tht} \lap \psi - \nu \lap \omg =0 ,  \\
&\rd_t  b^{z} - f(r)\rd_{\theta} u^{z} + \left( f''(r) + \frac{3}{r}f'(r) \right)\rd_{\theta}\psi- f(r)\rd_{\theta} \lap\psi = 0 ,\\
&\rd_t \psi - f(r) \rd_{\tht} (-\lap)^{-1} \omg + f(r)\rd_\theta b^{z}  = 0.
\end{aligned}
\right.
\end{equation}
Moreover, by the formulas $u = - \nb^{\perp} (-\lap)^{-1} \omg$ and \eqref{eq:grad-perp-cylin},
\begin{equation*}
	u^{r} = \frac{1}{r} \rd_{\tht} (-\lap)^{-1} \omg, \quad
	u^{\tht} = - \frac{1}{r} \rd_{r} (-\lap)^{-1} \omg.
\end{equation*}
As before, the LHS of the equation for $\rd_{t} \omg$ can be rewritten in the divergence form
\begin{equation} \label{eq:omg-eq-div-axisym}
\rd_{t} \omg - \nb \cdot \left( (r f' + 2 f) \nb^{\perp} \psi\right) +r^{-1} \rd_{\tht} (r f \lap \psi) - \nu \lap \omg.
\end{equation}

Note the similarity of the form with linearized systems around a translationally-symmetric planar stationary solution.

\subsubsection*{Axisymmetric background, \eqref{eq:e-mhd}}
The $(2+\frac{1}{2})$-dimensional perturbation equation in the case of \eqref{eq:e-mhd} are simply obtained by formally setting $u = 0$ in \eqref{eq:hall-mhd-2.5d-pert-axisym}: 
\begin{equation} \label{eq:e-mhd-2.5d-pert-axisym}
\left\{
\begin{aligned}
&\rd_t b^{z}   - f(r)\rd_\theta \lap\psi  + \left( f''(r) + \frac{3}{r}f'(r) \right) \rd_\theta\psi = - \nb^{\perp} \psi \cdot \nb \lap \psi,\\
&\rd_t \psi  + f(r)\rd_\theta b^{z} = \nb^{\perp} \psi \cdot \nb b^{z}.
\end{aligned}
\right.
\end{equation} 
The corresponding linearized equation is
\begin{equation} \label{eq:e-mhd-2.5d-lin-axisym}
\left\{
\begin{aligned}
&\rd_t b^{z}   - f(r)\rd_\theta \lap\psi  + \left( f''(r) + \frac{3}{r}f'(r) \right) \rd_\theta\psi = 0 ,\\
&\rd_t \psi  + f(r)\rd_\theta b^{z} = 0.
\end{aligned}
\right.
\end{equation} 

\subsubsection*{Justification for $L^{2}$-solutions}
The linearized equations derived in this subsection hold for $L^{2}$-solutions to \eqref{eq:hall-mhd-lin} and \eqref{eq:e-mhd-lin} in the following sense.
\begin{proposition} \label{prop:justify-eq}
Let $(u, b)$ be a $z$-independent $L^{2}$-solution to \eqref{eq:hall-mhd-lin} around $(0, \bgB)$, where either (a)~$\bgB = f(y) \rd_{x}$ or (b)~$\bgB = f(r) \rd_{\tht}$ as in Theorem~\ref{thm:norm-growth}. Then $(u^{z}, \omg, b^{z}, \psi)$, defined from $(u, b)$ as above, is well-defined up to addition of a space-independent distribution\footnote{That is, a distribution on $I \times M$ whose spatial gradient vanishes.} for $\psi$, and the Biot--Savart-type identity \eqref{eq:lin-ub-rel} holds. Moreover, \eqref{eq:hall-mhd-2.5d-lin} or \eqref{eq:hall-mhd-2.5d-lin-axisym}, respectively in cases~(a) or (b), holds when tested against vector-valued functions of the form
\begin{equation*}
(\phi_{u}, \nb \cdot (-\lap)^{-1} \phi^{x,y}_{\omg}, \phi_{b}, \nb \phi_{\psi})
\end{equation*}
where $\phi_{u}, \phi^{x,y}_{\omg}, \phi_{b}, \phi_{\psi} \in C^{\infty}_{c}(I \times M)$.

Analogously, for a $z$-independent $L^{2}$-solution $b$ to \eqref{eq:e-mhd-lin} around the same $\bgB$, $(b^{z}, \psi)$ is well-defined up to addition of a space-independent distribution for $\psi$, and the Biot--Savart-type identity \eqref{eq:lin-b-rel} holds. Moreover, \eqref{eq:e-mhd-2.5d-lin} or \eqref{eq:e-mhd-2.5d-lin-axisym}, respectively in cases~(a) or (b), holds when tested against vector-valued functions of the form
\begin{equation*}
(\phi_{b}, \nb \phi_{\psi})
\end{equation*}
where $\phi_{b}, \phi_{\psi} \in C^{\infty}_{c}(I \times M)$.
\end{proposition}
Note that the ambiguity of $\psi$ is harmless, in view of the fact that $\rd_{t} \psi$ is tested against $\nb \phi_{\psi}$ and all other occurrences of $\psi$ in \eqref{eq:lin-ub-rel}, \eqref{eq:hall-mhd-2.5d-lin} and \eqref{eq:hall-mhd-2.5d-lin-axisym} in the case of \eqref{eq:hall-mhd-lin} (resp. \eqref{eq:lin-b-rel}, \eqref{eq:e-mhd-2.5d-lin} and \eqref{eq:e-mhd-2.5d-lin-axisym} in the case of \eqref{eq:e-mhd-lin}) come with a spatial derivative; in every instance the space-independent distribution is annihilated.

The proof is straightforward, so we only sketch the main points. We focus on the case of \eqref{eq:hall-mhd-lin}, as \eqref{eq:e-mhd-lin} is entirely analogous. In the derivation, the only place where one has to be careful is when inverting $-\lap$ on $(\nb \times b)^{z}$ and $\omg = (\nb \times u)^{z}$, for which we rely on the following result:
\begin{lemma} \label{lem:lap-invert}
In $M^{2} = (\bbT, \bbR)_{x} \times (\bbT, \bbR)_{y}$, consider the Poisson equation
\begin{equation*}
	-\lap w = \rd_{x} g^{x} + \rd_{y} g^{y}, 
\end{equation*}
where $g^{x}, g^{y} \in L^{2}$. Then there exists a solution $w \in L^{1}_{loc} \cap \dot{H}^{1}$ such that $\nrm{\nb w}_{L^{2}} \aleq \nrm{g^{x,y}}_{L^{2}}$, which is unique up to addition of a constant.
\end{lemma}
The key point in the proof of this lemma is the quantitative estimate $\nrm{\nb w}_{L^{2}} \aleq \nrm{g^{x,y}}_{L^{2}}$, which is a consequence of $L^{2}$-boundedness of Riesz transforms on $(\bbT, \bbR)_{x} \times (\bbT, \bbR)_{y}$; this estimate allows one to solve the equation by approximating $g^{x}, g^{y}$ with smooth functions. We omit the obvious details.

By Lemma~\ref{lem:lap-invert}, $\psi$ and $(-\lap)^{-1} \omg$ are well-defined as a distribution on $I \times M^{2}$ up to addition of a space-independent distribution. Regardless of the ambiguity, the Biot--Savart-type identities in \eqref{eq:lin-ub-rel} are justified, where the mean-zero condition is needed when $M^{2} = \bbT^{2}_{x,y}$. The rest of the derivation can be followed without change, and the property that $(u,b)$ solves \eqref{eq:hall-mhd-lin} in the sense of distributions translates to \eqref{eq:hall-mhd-2.5d-lin} and \eqref{eq:hall-mhd-2.5d-lin-axisym} (in the respective cases) with the test functions as above.

\subsection{Energy identities under the $(2+\frac{1}{2})$-dimensional reduction} \label{subsec:2.5d-en}
Here, we first formally derive energy-type identities for inhomogeneous solutions to the linearized equations computed in Section~\ref{subsec:2.5d-lin}, which play a central role in our paper. These identities are then justified in two important cases that arise in this paper, namely for a pair of $L^{2}$-solutions with an additional $H^{\frac{1}{2}}$ regularity for $b$, or for a pair of an $L^{2}$-solution and a suitable test function (Proposition~\ref{prop:justify}).

\subsubsection*{Translationally-symmetric case, \eqref{eq:hall-mhd}}
We first consider the case $\bgB = f(y) \rd_{x}$ for \eqref{eq:hall-mhd}. Motivated by the form of \eqref{eq:hall-mhd-2.5d-lin}, we introduce the error terms
\begin{equation} \label{eq:hall-2.5d-err-parallel}
\begin{aligned}
\errh_{u}^{(\nu)}[u^{z}, \omg, b^{z}, \psi]=& \rd_{t} u^{z} - f(y) \rd_{x} b^{z} - \nu \lap u^{z}, \\
 \errh_{\omg}^{(\nu)}[u^{z}, \omg, b^{z}, \psi]=& \rd_{t} \omg + f(y)\rd_x \lap\psi - f''(y)\rd_x\psi - \nu \lap \omg^{z},  \\
\errh_{b}[u^{z}, \omg, b^{z}, \psi] =& \rd_t  b^{z} - f(y)\rd_x \lap\psi + f''(y)\rd_x\psi - f(y)\rd_x u^{z},\\
\errh_{\psi}[u^{z}, \omg, b^{z}, \psi]=&\rd_t \psi + f(y)\rd_x b^{z} - f(y) \rd_{x} (-\lap)^{-1} \omg.
\end{aligned}
\end{equation}

Consider two quadruples of scalar functions
\begin{equation*}
	(\tu^{z}, \tomg, \tb^{z}, \tpsi), \quad
	(u^{z}, \omg, b^{z}, \psi).
\end{equation*}
and the two associated pairs of planar vector fields $(\tu^{x, y}, \tb^{x, y})$ and $(u^{x, y}, b^{x, y})$ given by
\begin{align*}
	(\tu^{x, y}, \tb^{x, y}) =& (-\nb^{\perp} (-\lap)^{-1} \tomg, - \nb^{\perp} \tpsi), \\
	(u^{x, y}, b^{x, y}) =& (-\nb^{\perp} (-\lap)^{-1} \omg, - \nb^{\perp} \psi).
\end{align*}
Introducing the shorthands 
\begin{align*}
	(\errh_{\tu}^{(\nu)}, \errh_{\tomg}^{(\nu)}, \errh_{\tb},\errh_{\tpsi})  
	=& (\errh_{u}^{(\nu)}, \errh_{\omg}^{(\nu)}, \errh_{b},\errh_{\psi})[\tu^{z}, \tomg, \tb^{z}, \tpsi],
\end{align*}
and
\begin{align*}
	(\errh_{u}^{(\nu)}, \errh_{\omg}^{(\nu)}, \errh_{b},\errh_{\psi})  
	=& (\errh_{u}^{(\nu)}, \errh_{\omg}^{(\nu)}, \errh_{b},\errh_{\psi})[u^{z}, \omg, b^{z}, \psi], \\
\end{align*}
the desired (bilinear) energy identity is given by 
\begin{equation} \label{eq:en-hall-mhd-parallel}
\begin{aligned}
& \frac{\ud}{\ud t} \left(\brk{\tb, b} + \brk{\tu, u} \right) + 2 \nu \brk{\nb \tu, \nb u} \\
& = - \brk{f'' \rd_{x} \tpsi, b^{z}} - \brk{\tb^{z}, f'' \rd_{x} \psi} 
	- \brk{f' \nb \tpsi, u^{x, y}} 
	- \brk{\tu^{x,y}, f' \nb \psi} \\
& \phantom{=}
	+\brk{\nb^{\perp} \errh_{\tpsi}, \nb^{\perp} \psi} + \brk{\nb^{\perp} \tpsi, \nb^{\perp} \errh_{\psi}} 
	+ \brk{\errh_{\tb}, b^{z}} + \brk{\tb^{z}, \errh_{b}} \\
& \phantom{=}
 	- \brk{\nb^{\perp} (-\lap)^{-1} \errh_{\tomg}^{(\nu)}, u^{x,y}} - \brk{\tu^{x,y}, \nb^{\perp} (-\lap)^{-1} \errh_{\omg}^{(\nu)}} 
	+ \brk{\errh_{\tu}^{(\nu)}, u^{z}} + \brk{\tu^{z}, \errh_{u}^{(\nu)}}.
\end{aligned}
\end{equation}
This identity is essentially \eqref{eq:lin-en-hall}, but allowing for errors on the RHS of the linearized equations. It is the precise form of the errors given in \eqref{eq:hall-2.5d-err-parallel} that is important here.

To prove \eqref{eq:en-hall-mhd-parallel}, we use \eqref{eq:hall-2.5d-err-parallel} (and \eqref{eq:omg-eq-div} for $\rd_{t} \tomg$, $\rd_{t} \omg$) to compute
\begin{align*}
&\frac{\ud}{\ud t} \brk{\tb, b} + \frac{\ud}{\ud t} \brk{\tu, u} + 2 \nu \brk{\nb \tu, \nb u} \\
&= \brk{\rd_{t} \nb^{\perp} \tpsi, \nb^{\perp} \psi} + \brk{\nb^{\perp} \tpsi, \rd_{t} \nb^{\perp} \psi} + \brk{\rd_{t} \tb^{z}, b^{z}} + \brk{\tb^{z}, \rd_{t} b^{z}} \\
&\phantom{=}
 	+ \brk{\rd_{t} \nb^{\perp} (-\lap)^{-1} \tomg, \nb^{\perp} (-\lap)^{-1} \omg} + \brk{\nb^{\perp} (-\lap)^{-1} \tomg, \rd_{t} \nb^{\perp} (-\lap)^{-1} \omg} + 2 \nu \brk{\tomg, \omg} \\
&\phantom{=}
	+ \brk{\rd_{t} \tu^{z}, u^{z}} + \brk{\tu^{z}, \rd_{t} u^{z}} + 2 \nu \brk{\tu^{z}, u^{z}} \\
&= \brk{- f \rd_{x} \tb^{z} + f \rd_{x} (-\lap)^{-1} \tomg, - \lap \psi} 
	+ \brk{-\lap \tpsi, - f \rd_{x} b^{z} + f \rd_{x} (-\lap)^{-1} \omg}  \\
&\phantom{=}
	+ \brk{f \rd_{x} \lap \tpsi + f \rd_{x} \tu^{z} - f'' \rd_{x} \tpsi, b^{z}} + \brk{\tb^{z}, f \rd_{x} \lap \psi + f \rd_{x} u^{z} - f'' \rd_{x} \psi} \\
&\phantom{=}
 	+ \brk{\nb^{\perp} (-\lap)^{-1} \left( - \rd_{x} (f \lap \tpsi) + \nb \cdot (f' \nb^{\perp} \tpsi) \right), \nb^{\perp} (-\lap)^{-1} \omg} \\
&\phantom{=}
	+ \brk{\nb^{\perp} (-\lap)^{-1} \tomg, \nb^{\perp} (-\lap)^{-1} \left( - \rd_{x} (f \lap \psi) + \nb \cdot (f' \nb^{\perp} \psi) \right)} \\
&\phantom{=}
	+ \brk{f \rd_{x} \tb^{z}, u^{z}} + \brk{\tu^{z}, f \rd_{x} b^{z}} \\
&\phantom{=}
	+\brk{\nb^{\perp} \errh_{\tpsi}, \nb^{\perp} \psi} + \brk{\nb^{\perp} \tpsi, \nb^{\perp} \errh_{\psi}} + \brk{\errh_{\tb}, b^{z}} + \brk{\tb^{z}, \errh_{b}} \\
&\phantom{=}
 	+ \brk{\nb^{\perp} (-\lap)^{-1} \errh_{\tomg}^{(\nu)}, \nb^{\perp} (-\lap)^{-1} \omg} + \brk{\nb^{\perp} (-\lap)^{-1} \tomg, \nb^{\perp} (-\lap)^{-1} \errh_{\omg}^{(\nu)}} \\
&\phantom{=}
	+ \brk{\errh_{\tu}^{(\nu)}, u^{z}} + \brk{\tu^{z}, \errh_{u}^{(\nu)}}.
\end{align*}
Many high order terms cancel, essentially from the same energy structure as in Proposition~\ref{prop:lin-en}. In this process, the formal identity $((\nb^{\perp})(-\lap)^{-1})^{\ast} (\nb^{\perp})(-\lap)^{-1} = (-\lap)^{-1}$ is used. After a suitable distribution of derivatives, we arrive at
\begin{align*}
&\frac{\ud}{\ud t} \brk{\tb, b} + \frac{\ud}{\ud t} \brk{\tu, u} + 2 \nu \brk{\nb \tu, \nb u} \\
&= \brk{- f'' \rd_{x} \tpsi, b^{z}} + \brk{\tb^{z}, - f'' \rd_{x} \psi} \\
&\phantom{=}
 	- \brk{f' \nb^{\perp} \tpsi, \nb (-\lap)^{-1} \omg} 
 	- \brk{\nb (-\lap)^{-1} \tomg, f' \nb^{\perp} \psi} \\
&\phantom{=}
	+\brk{\nb^{\perp} \errh_{\tpsi}, \nb^{\perp} \psi} + \brk{\nb^{\perp} \tpsi, \nb^{\perp} \errh_{\psi}} + \brk{\errh_{\tb}, b^{z}} + \brk{\tb^{z}, \errh_{b}} \\
&\phantom{=}
 	+ \brk{\nb^{\perp} (-\lap)^{-1} \errh_{\tomg}^{(\nu)}, \nb^{\perp} (-\lap)^{-1} \omg} 
	+ \brk{\nb^{\perp} (-\lap)^{-1} \tomg, \nb^{\perp} (-\lap)^{-1} \errh_{\omg}^{(\nu)}} \\
&\phantom{=}
	+ \brk{\errh_{\tu}^{(\nu)}, u^{z}} + \brk{\tu^{z}, \errh_{u}^{(\nu)}}.
\end{align*}
Then switching $\nb$ and $\nb^{\perp}$ in the third and fourth terms on the RHS (which incurs a sign change), and using $u^{x,y} =  - \nb^{\perp} (-\lap)^{-1} \omg$ (as well as the counterpart for $\tu^{x,y}$), we obtain \eqref{eq:en-hall-mhd-parallel}.

\subsubsection*{Translationally-symmetric case, \eqref{eq:e-mhd}}
Next, we consider the case $\bgB = f(y) \rd_{x}$ for \eqref{eq:e-mhd}. In view of \eqref{eq:e-mhd-2.5d-lin}, we introduce the error terms
\begin{equation} \label{eq:e-2.5d-err-parallel}
\begin{aligned}
\err_{b}[b^{z}, \psi]=&\rd_t b^{z}   + f''(y)\rd_x\psi  - f(y)\rd_x \lap\psi,\\
\err_{\psi}[b^{z}, \psi]=&\rd_t \psi  + f(y)\rd_x b^{z}.
\end{aligned}
\end{equation}

Consider two pairs of scalar functions
\begin{equation*}
	(\tb^{z}, \tpsi) \quad
	(b^{z}, \psi),
\end{equation*}
and the associated planar vector fields $\tb^{x, y} = - \nb^{\perp} \tpsi$ and $b^{x, y} = - \nb^{\perp} \psi$. As before, we introduce the shorthands
\begin{equation*}
	(\err_{\tb},\err_{\tpsi})  = (\err_{b}, \err_{\psi})[\tb^{z}, \tpsi], \quad
	(\err_{b},\err_{\psi})  = (\err_{b}, \err_{\psi})[b^{z}, \psi].
\end{equation*}
Then
\begin{equation} \label{eq:en-e-mhd-parallel}
\begin{aligned}
\frac{\ud}{\ud t} \brk{\tb, b}
=& - \brk{f'' \rd_{x} \tpsi, b^{z}} - \brk{\tb^{z}, f'' \rd_{x} \psi} \\
&	+\brk{\nb^{\perp} \err_{\tpsi}, \nb^{\perp} \psi} + \brk{\nb^{\perp} \tpsi, \nb^{\perp} \err_{\psi}} \\
&	+ \brk{\err_{\tb}, b^{z}} + \brk{\tb^{z}, \err_{b}} .
\end{aligned}
\end{equation}
This identity is obtained from \eqref{eq:en-hall-mhd-parallel} by formally setting $\tu^{z}, \tomg, u^{z}, \omg$ equal to zero.

\subsubsection*{Axisymmetric case, \eqref{eq:hall-mhd}}
Now we consider the  {case} $\bgB = f(r) \rd_{\tht}$ for \eqref{eq:hall-mhd}. In view of \eqref{eq:hall-mhd-2.5d-lin-axisym}, we introduce
\begin{equation} \label{eq:hall-2.5d-err-axi}
\begin{aligned}
\errh_{u}^{(\nu)}[u^{z}, \omg, b^{z}, \psi] =&\rd_{t} u^{z} - f \rd_{\tht} b^{z} - \nu \lap u^{z}, \\
\errh_{\omg}^{(\nu)}[u^{z}, \omg, b^{z}, \psi] =&\rd_{t} \omg + f(r) \rd_{\tht} \lap \psi - (f''(r) - \frac{3}{r} f'(r)) \rd_{\tht} \psi - \nu \lap \omg,  \\
\errh_{b}[u^{z}, \omg, b^{z}, \psi] =&\rd_t  b^{z}   - f(r)\rd_\theta \lap \psi  + (f''(r) + \frac{3}{r}f'(r))\rd_\theta \psi - f(r)\rd_{\theta}  {u^{z}}, \\
\errh_{\psi}[u^{z}, \omg, b^{z}, \psi] =&\rd_t \psi  + f(r)\rd_\theta b^{z} - f(r) \rd_{\tht} (-\lap)^{-1}  {\omg}. 
\end{aligned}
\end{equation}

Consider two quadruples of scalar functions
\begin{equation*}
	(\tu^{z}, \tomg, \tb^{z}, \tpsi), \quad
	(u^{z}, \omg, b^{z}, \psi),
\end{equation*}
and define the associated pairs of planar vector fields $(\tu, \tb)$, $(u, b)$, respectively, and the error terms $(\errh_{\tu}^{(\nu)}, \errh_{\tomg}^{(\nu)}, \errh_{\tb}, \errh_{\tpsi})$, $(\errh_{u}^{(\nu)}, \errh_{\omg}^{(\nu)}, \errh_{b}, \errh_{\psi})$ as before. The energy identity in this case is
\begin{equation} \label{eq:en-hall-mhd-axi}
\begin{aligned}
& \frac{\ud}{\ud t} \left(\brk{\tb, b} + \brk{\tu, u} \right) + 2 \nu \brk{\nb \tu, \nb u} \\
&= 	-\brk{(r f'' + 3 f') r^{-1} \rd_{\tht} \tpsi, b^{z}} 
	- \brk{\tb^{z}, (r f'' + 3 f') r^{-1} \rd_{\tht} \psi} \\
&\phantom{=}
	- \brk{(r f' + 2 f) \nb \tpsi, u^{r, \tht}}
	- \brk{\tu^{r, \tht}, (r f' + 2f) \nb \psi} \\
&\phantom{=}
	+\brk{\nb^{\perp} \errh_{\tpsi}, \nb^{\perp} \psi} + \brk{\nb^{\perp} \tpsi, \nb^{\perp} \errh_{\psi}} \\
&\phantom{=}
	+ \brk{\errh_{\tb}, b^{z}} + \brk{\tb^{z}, \errh_{b}} \\
&\phantom{=}
 	- \brk{\nb^{\perp} (-\lap)^{-1} \errh_{\tomg}, u^{r, \tht}} 
	- \brk{\tu^{r, \tht}, \nb^{\perp} (-\lap)^{-1} \errh_{\omg}} \\
&\phantom{=}
	+ \brk{\errh_{\tu}^{(\nu)}, u^{z}} + \brk{\tu^{z}, \errh_{u}^{(\nu)}}.
\end{aligned}
\end{equation}
The derivation is similar to \eqref{eq:en-hall-mhd-parallel}; we leave the details to the reader.

\subsubsection*{Axisymmetric case, \eqref{eq:e-mhd}}
Finally, we consider the case $\bgB = f(r) \rd_{\tht}$ for \eqref{eq:e-mhd}. From \eqref{eq:e-mhd-2.5d-lin-axisym}, we introduce the error terms
\begin{equation} \label{eq:e-2.5d-err-axi}
\begin{aligned}
\err_{b}[b^{z}, \psi] =&\rd_t  b^{z}   - f(r)\rd_\theta \lap\psi  + (f''(r) + \frac{3}{r}f'(r))\rd_\theta\psi, \\
\err_{\psi}[b^{z}, \psi] =&\rd_t \psi + f(r)\rd_\theta b^{z}. 
\end{aligned}
\end{equation}

Consider two pairs of scalar functions
\begin{equation*}
	(\tb^{z}, \tpsi), \quad
	(b^{z}, \psi),
\end{equation*}
and define the associated planar vector fields $\tb^{x,y}$, $b^{x, y}$, respectively, and the error terms $(\err_{\tb}, \err_{\tpsi})$, $(\err_{b}, \err_{\psi})$, respectively, as before. The energy identity in this case is
\begin{equation} \label{eq:en-e-mhd-axi}
\begin{aligned}
\frac{\ud}{\ud t} \brk{\tb, b}
=& -\brk{(r f'' + 3 f') r^{-1} \rd_{\tht} \tpsi, b^{z}} - \brk{\tb^{z}, (r f'' + 3 f') r^{-1} \rd_{\tht} \psi} \\
&	+\brk{\nb^{\perp} \err_{\tpsi}, \nb^{\perp} \psi} + \brk{\nb^{\perp} \tpsi, \nb^{\perp} \err_{\psi}} \\
&	+ \brk{\err_{\tb}, b^{z}} + \brk{\tb^{z}, \err_{b}},
\end{aligned}
\end{equation}
which is obtained by formally setting $\tu^{z}, \tomg, u^{z}, \omg$ equal to zero in \eqref{eq:en-hall-mhd-axi}.

\subsubsection*{Justification of the energy identities}
The above energy identities can be rigorously justified under the following conditions:
\begin{proposition} \label{prop:justify}
The energy identities \eqref{eq:en-hall-mhd-parallel} and \eqref{eq:en-hall-mhd-axi} hold in the following two cases:
\begin{itemize}
\item $(\tu^{z}, \tomg, \tb^{z}, \tpsi)$ and $(u^{z}, \omg, b^{z}, \psi)$ are derived from $L^{2}$-solutions $(\tu, \tb)$ and $(u, b)$ as in Proposition~\ref{prop:justify-eq}, respectively, under the additional conditions $(u, b)\in C_t(I;L^2)$ and $\tb, b \in L^{2}_{t}(I; H^{\frac{1}{2}})$; or
\item $(u^{z}, \omg, b^{z}, \psi)$ is derived from an $L^{2}$-solution $(u, b)$ as in Proposition~\ref{prop:justify-eq}, and $(\tu^{z}, \tomg , \tb^{z}, \tpsi)$ obeys\footnote{Here, by the assertion $\nb (-\lap)^{-1} \tomg \in X$, we mean $\tomg$ is of the form $-\lap w$ where $\nb w \in X$.}
\begin{equation*}
	(\tu^{z}, \nb (-\lap)^{-1} \tomg) \in C_{t} (I; L^{2}), \quad
	(\tb^{z}, \nb \tpsi) \in C_{t}(I; L^{2}) \cap L^{1}_{t}(I; H^{1}),
\end{equation*}
and the error terms obey
\begin{equation*}
	\errh_{\tu}^{(\nu)}, \nb (-\lap)^{-1} \errh_{\tomg}^{(\nu)}, \errh_{\tb}, \nb \errh_{\tpsi} \in L^{1}_{t}(I; L^{2}),
\end{equation*}
and when $\nu > 0$, also
\begin{equation*}
	\nb \tu^{z}, \tomg \in L^{2}_{t}(I; L^{2}).
\end{equation*}
\end{itemize}

Analogously, the energy identities \eqref{eq:en-e-mhd-parallel} and \eqref{eq:en-e-mhd-axi} hold in the following two cases:
\begin{itemize}
\item $(\tb^{z}, \tpsi)$ and $(b^{z}, \psi)$ are derived from $L^{2}$-solutions $\tb$ and $b$ as in Proposition~\ref{prop:justify-eq}, respectively, under the additional conditions $b \in C_t(I;L^2)$ and $\tb, b \in L^{2}_{t}(I; H^{\frac{1}{2}})$; or
\item $(b^{z}, \psi)$ is derived from an $L^{2}$-solution $b$ as in Proposition~\ref{prop:justify-eq}, and $(\tb^{z}, \tpsi)$ obeys
\begin{equation*}
	(\tb^{z}, \nb \tpsi) \in C_{t}(I; L^{2}) \cap L^{1}_{t}(I; H^{1})
\end{equation*}
and the error terms obey
\begin{equation*}
	\err_{\tb}, \nb \err_{\tpsi} \in L^{1}_{t}(I; L^{2}).
\end{equation*}
\end{itemize}
\end{proposition}
The idea is to first mollify $(\tu^{z}, \tomg, \tb^{z}, \tpsi)$ and $(u^{z}, \omg, b^{z}, \psi)$ in space; then the derivation of the energy identities go through, with additional errors generated from the mollification. Next, one checks that the above conditions allow one to take the mollification parameter to zero in the energy identities, while the mollification errors vanish. For instance, in the case of \eqref{eq:hall-mhd}, under the condition that $(\tu, \tb), (u, b) \in \calE(I)$ one can show, using standard commutator estimates for mollifiers, that all mollifications errors go to zero except $\errh_{\tb}, \errh_{\tpsi}$ and $\errh_{b}, \errh_{\psi}$, which lose one derivative in the commutator with the Hall term. Roughly speaking, distributing this loss equally to $b$ and $\tb$ results in the first case, and shifting it to $\tb$ results in the second case. We omit the straightforward details.

\section{Construction of degenerating wave packets} \label{sec:wavepackets}
The goal of this section is to carry out the construction of a degenerating wave packet approximate solutions for the linearized \eqref{eq:e-mhd} and \eqref{eq:hall-mhd} equations around stationary solutions as in Theorem~\ref{thm:norm-growth}. 
\subsection{Statement of the main propositions} \label{subsec:wavepackets-results}
The aim of this subsection is to state precisely the main properties of the construction in this section. 

We begin with some preparations. In what follows, we write $\rd_{x}^{-1}$ for a right inverse of $\rd_{x}$ that is formally defined as follows:
\begin{equation*}
	\rd_{x}^{-1} g = \begin{cases} \int_{0}^{x} g(x') \, \ud x' + \int_{0}^{ {2 \pi}} x' g(x')  \, \ud x' & \hbox{ when } (\bbT, \bbR)_{x} = \bbT_{x}, \\
	\int_{-\infty}^{x} g(x') \, \ud x' & \hbox{ when } (\bbT, \bbR)_{x} = \bbR_{x}. \end{cases}
\end{equation*}
When $(\bbT, \bbR)_{x} = \bbT_{x}$, $\rd_{x}^{-1} g$ is well-defined only when $\int g \, \ud x = 0$.  {In this case, note that $\int \rd_{x}^{-1} g \, \ud x = 0$.}
When $(\bbT, \bbR)_{x} = \bbR_{x}$, $\rd_{x}^{-1} g$ stays in $\calS(\bbR_{x})$ if $g \in \calS(\bbR_{x})$ and $\int g \, \ud x = 0$.

Assuming that $\rd_x^{-1}$ is well-defined for $e^{i \lmb x }g $, we introduce the notation
\begin{equation*}
	g^{(-1; \lmb)} = i \lmb e^{-i\lmb x} \rd_{x}^{-1} (e^{i \lmb x} g), \quad
	g^{(-2; \lmb)} = i \lmb e^{-i\lmb x} \rd_{x}^{-1} (e^{i \lmb x} g^{(-1; \lmb)}), \hbox{ etc.}
\end{equation*}
The factor $i \lmb$ is inserted to compensate for the effect of $\rd_{x}^{-1}$ on $e^{i \lmb x} g$; see Lemma~\ref{lem:bump} below, where the advantage of this normalization is most evident.

For any $x_{0} \in (\bbT, \bbR)_{x}$, we introduce the $x$-translation operator 
\begin{equation*}
T_{x_{0}} g(x, y, z) = g(x - x_{0}, y, z).
\end{equation*}

We first state the main result in the case of \eqref{eq:e-mhd}.
\begin{proposition}[Construction of degenerating wave packets for \eqref{eq:e-mhd}] \label{prop:wavepackets}
Let $\bgB$ and $M$ be as in Theorem~\ref{thm:norm-growth}. Then the following statements hold.
\begin{enumerate}[label=(\alph*)]
\item (translationally-symmetric case) Consider case~(a) in Theorem~\ref{thm:norm-growth}, i.e., $\bgB = f(y) \rd_{x}$ and $M^{2} = (\bbT, \bbR)_{x} \times (\bbT, \bbR)_{y}$. Assume, without loss of generality, that $f(0) = 0$ and $f'(0) > 0$, and fix $y_{1} > 0$ such that 
\begin{equation*}
	f'(y) > \frac{1}{2} f'(0), \quad 0 < f(y) < \frac{1}{2} \quad \hbox{ for } y \in [0, y_{1}].
\end{equation*}
Then to any $\lmb \in \bbN_{0}$ and a complex-valued Schwartz function $g_{0}(x, y) \in \calS(M^{2})$ such that
\begin{equation} \label{eq:wavepacket-g0}
	\supp g_{0} \subseteq (\bbT, \bbR)_{x} \times (\tfrac{1}{2} y_{1}, y_{1}), \qquad \int e^{i \lmb x} g_{0}(x, y) \, \ud x = 0 \hbox{ for all } y \in (0, y_{1}),
\end{equation}
we may associated a pair $(\tb^{z}_{(\lmb)}, \tpsi_{(\lmb)})[g_{0}]$ satisfying the following properties:
\begin{itemize}
\item (linearity) the map $g_{0} \mapsto (\tb^{z}_{(\lmb)}, \tpsi_{(\lmb)})[g_{0}]$ is (real) linear;
\item (initial data) at $t =0$, we have 
\begin{equation}\label{eq:initial-bz}
\begin{split}	
&\tb^{z}_{(\lmb)}(0) = \lmb f^{-1}\rd_x^{-1} \Re \left(\frac{1}{i} e^{i \lmb (x + G(y))} g_{0} \right) \\
&\qquad + f^{-1}\rd_x^{-1} \Re\left( e^{i\lmb(x+G)} \left( \frac{1}{2} \frac{ {f^2}\rd_y f}{\sqrt{1-f^2}} g_0 - f\sqrt{1-f^2} \rd_y g_0 - (1+f^2) \rd_x g_0 \right)  \right)
\end{split}
\end{equation} and 
\begin{equation}\label{eq:initial-psi} 
	\tpsi_{(\lmb)}(0) = \lmb^{-1} \Re \left( e^{i \lmb (x + G(y))} g_{0} \right),
\end{equation}
where $G(y)$ is a smooth function on $y \in (0, y_{1})$ determined by $f$, and
\begin{equation*}
	\nrm{\tb^{z}_{(\lmb)}(0)}_{L^{2}} + \nrm{\nb \tpsi_{(\lmb)}(0)}_{L^{2}}
	\geq c \nrm{g_{0}}_{L^{2}} - C \lmb^{-1} \nrm{g_{0}}_{H^{1}};
\end{equation*}
\item ($x$-invariance) for any $x_{0} \in (\bbT, \bbR)_{x}$,
\begin{equation*}
(\tb^{z}_{(\lmb)}, \tpsi_{(\lmb)})[e^{-i\lmb x_{0}} T_{x_{0}} g_{0}]
= T_{x_{0}} (\tb^{z}_{(\lmb)}, \tpsi_{(\lmb)})[g_{0}];
\end{equation*}
\item (regularity estimates) for any $m \in \bbN_{0}$ and $ {t \geq 0}$, 
\begin{align*}
	\max_{0 \leq k, \ell, k + \ell \leq m} 
		\nrm{(\lmb^{-2} \rd_{t})^{k} (\lmb^{-1} \rd_{x})^{\ell} (\lmb^{-1} f \rd_{y})^{m - k - \ell} \tb^{z}_{(\lmb)}(t)}_{L^{2}} 
	\aleq & \nrm{g_{0}^{(-1; \lmb)}}_{H^{m+1}}, \\
	\max_{0 \leq k, \ell, k + \ell \leq m} \nrm{(\lmb^{-2} \rd_{t})^{k} (\lmb^{-1} \rd_{x})^{\ell} (\lmb^{-1} f \rd_{y})^{m - k - \ell} \nb \tpsi_{(\lmb)}(t)}_{L^{2}} 
	\aleq & \nrm{g_{0}}_{H^{m+1}};
\end{align*}
\item (degeneration) there exists $0 < c_{f} < C_{f}$ such that the following holds:
\begin{itemize}
\item For $1 \leq p \leq \infty$ and $s \in \bbR$ obeying $\frac{1}{p} - s \leq \frac{1}{2}$, we have
\begin{equation} \label{eq:wavepackets-degen-upper}
	\nrm{\tb_{(\lmb)}(t)}_{L^{p}_{x} W^{s, p}_{y}}
	\aleq_{s} \lmb^{s} e^{C_{f} (s - \frac{1}{p} + \frac{1}{2})\lmb t} \nrm{(g_{0}, g_{0}^{(-1; \lmb)})}_{W^{\lfloor s \rfloor +2, p}},
\end{equation}
where $\tb_{(\lmb)}^{x, y}(t) = - \nb^{\perp} \tpsi_{(\lmb)}$;
\item There exists a decomposition $\tb_{(\lmb)}(t) = \tb_{(\lmb)}^{main}(t) + \tb_{(\lmb)}^{small}(t)$
such that for any $1 \leq p \leq \infty$ and $s \in \bbR$ be obeying $s - \frac{1}{p} + \frac{1}{2} \leq 0$, we have
\begin{equation}  \label{eq:wavepackets-degen} 
	\nrm{\tb_{(\lmb)}^{main}(t)}_{L^{p}_{x} W^{s, p}_{y}}
	\aleq_{s} \lmb^{s} e^{c_{f} (s - \frac{1}{p} + \frac{1}{2})\lmb t} \nrm{(g_{0}, g_{0}^{(-1; \lmb)})}_{W^{\lfloor -s \rfloor +2, p}},
\end{equation}
and for any $1 \leq p \leq \infty$ and $s \in \bbR$ such that $-\frac{1}{2} < s - \frac{1}{p} + \frac{1}{2} \leq 0$, we have
\begin{equation} \label{eq:wavepackets-degen-small}  
	\nrm{\tb_{(\lmb)}^{small}(t)}_{L^{p}_{x} W^{s, p}_{y}}
	\aleq_{s} \lmb^{-1} e^{c_{f} (s - \frac{1}{p} + \frac{1}{2}) \lmb t} \nrm{(g_{0}, g_{0}^{(-1; \lmb)})}_{W^{2, p}}; 
\end{equation} 
\end{itemize}
and analogous estimates hold for $(\lmb^{-1} \rd_{x})^{m} \tb_{(\lmb)}$ for any $m \in \bbN_{0}$;
\item (error bounds) for $ {t \geq 0}$, $\err_{\psi}[\tb^{z}_{(\lmb)}, \tpsi_{(\lmb)}](t) = 0$ and
\begin{equation*}
	\nrm{\err_{b}[\tb^{z}_{(\lmb)}, \tpsi_{(\lmb)}](t)}_{L^{2}} \aleq \nrm{g_{0}^{(-1; \lmb)}}_{H^{4}}.
\end{equation*}
\end{itemize}
In the above statements, we omitted the dependence of the implicit constants on $f$.

\item (axi-symmetric case: $\bgB = f(r) \rd_{\tht}$ and $M^{2} = \bbR^{2}_{x,y}$). Assume, without loss of generality, that $f'(r_{0}) > 0$, and fix $r_{1} > 0$ such that 
\begin{equation*}
	f'(r) > \frac{1}{2} f'(r_{0}), \quad 0 < f(r) < \frac{1}{2} \quad \hbox{ for } r \in [r_{0}, r_{1}].
\end{equation*}
Then for any $\lmb \in \bbN_{0}$ and a complex-valued smooth radial function $g_{0}(r)$ such that
\begin{equation} \label{eq:wavepacket-g0-axi}
	\supp g_{0} \subseteq (\tfrac{1}{2}(r_{0}+r_{1}), r_{1}),
\end{equation}
we may associated a pair $(\tb^{z}_{(\lmb)}, \tpsi_{(\lmb)})[g_{0}]$ satisfying the following properties:
\begin{itemize}
\item (linearity) the map $g_{0} \mapsto (\tb^{z}_{(\lmb)}, \tpsi_{(\lmb)})[g_{0}]$ is linear;
\item (initial data) at $t =0$, we have 
\begin{equation*}
\begin{split}	
&\tb^{z}_{(\lmb)}(0) = - f^{-1} \Re \left(  e^{i \lmb (\theta + G(r))} g_{0} \right) \\
&\qquad - f^{-1}\lmb^{-1} \Re\left(i e^{i\lmb(\theta+G)} ( \frac{1}{2} \frac{ {f^2}\rd_r f}{\sqrt{1-f^2}} g_0 - f\sqrt{1-f^2} \rd_r g_0 )   \right)
\end{split}
\end{equation*} and 
\begin{equation*} 
	\tpsi_{(\lmb)}(0) = \lmb^{-1} \Re \left( e^{i \lmb (\theta + G(r))} g_{0} \right),
\end{equation*}
where $G(r)$ is a smooth function on $r \in (r_{0}, r_{1})$ determined by $f$, and
\begin{equation*}
	\nrm{\tb^{z}_{(\lmb)}(0)}_{L^{2}} + \nrm{\nb \tpsi_{(\lmb)}(0)}_{L^{2}}
	\geq c \nrm{g_{0}}_{L^{2}} - C \lmb^{-1} \nrm{g_{0}}_{H^{1}},
\end{equation*}
where $c, C > 0$ are absolute constants;
\item (regularity estimates) for any $m \in \bbN_{0}$ and $ {t \geq 0}$,
\begin{align*}
	\max_{0 \leq \ell \leq m} 
		\nrm{(\lmb^{-2} \rd_{t})^{k} (\lmb^{-1} \rd_{ \tht})^{\ell} (\lmb^{-1} f \rd_{ {r}})^{m - k - \ell} \tb^{z}_{(\lmb)}(t)}_{L^{2}} 
	\aleq & \nrm{g_{0}}_{H^{m+1}}, \\
	\max_{0 \leq \ell \leq m} \nrm{(\lmb^{-2} \rd_{t})^{k} (\lmb^{-1} \rd_{ \tht})^{\ell} (\lmb^{-1} f \rd_{ {r}})^{m - k - \ell} \nb \tpsi_{(\lmb)}(t)}_{L^{2}} 
	\aleq & \nrm{g_{0}}_{H^{m+1}};
\end{align*}
\item (degeneration)
there exists $0 < c_{f} < C_{f}$ such that the following holds:
\begin{itemize}
\item For $1 \leq p \leq \infty$ and $s \in \bbR$ obeying $\frac{1}{p} - s \leq \frac{1}{2}$, we have
\begin{equation*}
	\nrm{\tb_{(\lmb)}(t)}_{L^{p}_{\tht} W^{s, p}_{r}}
	\aleq_{s} \lmb^{s} e^{C_{f} (s - \frac{1}{p} + \frac{1}{2})\lmb t} \nrm{g_{0}}_{W^{\lfloor s \rfloor +2, p}},
\end{equation*}
where $\tb_{(\lmb)}^{r} = r^{-1} \rd_{\tht} \tpsi_{(\lmb)}$ and $\tb_{(\lmb)}^{\tht} = - r^{-1} \rd_{r} \tpsi_{(\lmb)}$;
\item There exists a decomposition $\tb_{(\lmb)}(t) = \tb_{(\lmb)}^{main}(t) + \tb_{(\lmb)}^{small}(t)$
such that for any $1 \leq p \leq \infty$ and $s \in \bbR$ be obeying $s - \frac{1}{p} + \frac{1}{2} \leq 0$, we have
\begin{equation*} 
	\nrm{\tb_{(\lmb)}^{main}(t)}_{L^{p}_{\tht} W^{s, p}_{r}}
	\aleq_{s} \lmb^{s} e^{c_{f} (s - \frac{1}{p} + \frac{1}{2})\lmb t} \nrm{g_{0}}_{W^{\lfloor -s \rfloor +2, p}},
\end{equation*}
and for any $1 \leq p \leq \infty$ and $s \in \bbR$ such that $-\frac{1}{2} < s - \frac{1}{p} + \frac{1}{2} \leq 0$, we have
\begin{equation*} 
	\nrm{\tb_{(\lmb)}^{small}(t)}_{L^{p}_{\tht} W^{s, p}_{r}}
	\aleq_{s} \lmb^{-1} e^{c_{f} (s - \frac{1}{p} + \frac{1}{2}) \lmb t} \nrm{g_{0}}_{W^{2, p}}; 
\end{equation*}
\end{itemize}
and analogous estimates hold for $(\lmb^{-1} \rd_{\tht})^{\ell} \tb_{(\lmb)}$ for any $\ell \in \bbN_{0}$;
\item (error bounds) for $ {t \geq 0}$, $\err_{\psi}[\tb^{z}_{(\lmb)}, \tpsi_{(\lmb)}](t) = 0$ and
\begin{equation*}
	\nrm{\err_{b}[\tb^{z}_{(\lmb)}, \tpsi_{(\lmb)}](t)}_{L^{2}} \aleq \nrm{g_{0}}_{H^{4}}.
\end{equation*}
\end{itemize}
In the above statements, we omitted the dependence of the implicit constants on $f$, $r_{0}$ and $r_{1}$.
\end{enumerate}
\end{proposition}
Note that, in case~(a), the mean-zero property in \eqref{eq:wavepacket-g0} ensures that $g_{0}^{(-1; \lmb)} \in \calS(M^{2})$ with the same support property as $g_{0}$.

\begin{remark} \label{rem:degeneration}
Key to our instability mechanism is the degeneration property. As it will be clear in our construction, the wave packet $(\tb^{z}_{(\lmb)}, \tpsi_{(\lmb)})$ is initially supported in $(\frac{1}{2} y_{1}, y_{1})$, but travels towards the hypersurface $\set{y = 0}$ (where we focus on case~(a) for concreteness), on which $\bgB$ is linearly degenerate. In this process, its $y$-support degenerates at a rate determined by $\lmb$ while the $L^{2}$ norm remains invariant; by H\"older's inequality, we already obtain the $L^{p}$-degeneration inequality
\begin{equation} \label{eq:wavepackets-degen-Lp}
	\nrm{\tb_{(\lmb)}}_{L^{2}_{x} L^{p}_{y}} \aleq e^{-c_{f} (\frac{1}{p} - \frac{1}{2}) \lmb t} \nrm{(g_{0}, g_{0}^{(-1; \lmb)})}_{H^{1}}.
\end{equation}
We remark that the simpler inequality \eqref{eq:wavepackets-degen-Lp} may be used in place of \eqref{eq:wavepackets-degen}--\eqref{eq:wavepackets-degen-small} in the ensuing proofs to establish some norm growth inequalities, but these are not sharp in the case $s > 0$.

To ensure such a behavior, which is not time symmetric, the choice of $\tb^{z}_{(\lmb)}(0)$ is crucial. Changing its sign reverses time\footnote{Note that the time reversal symmetry for \eqref{eq:e-mhd} does not immediately induce the analogous symmetry for the linearized equation \eqref{eq:e-mhd-lin}, since the background solution $\bgB$ reverses sign. However, for a planar stationary magnetic field, we may apply an additional reflection about $\set{z = 0}$, and obtain a time reversal symmetry for \eqref{eq:e-mhd-lin}; this is what we observe here.} for \eqref{eq:e-mhd-2.5d-lin}, and the corresponding wave packet \emph{expands} its $y$-support while keeping the $L^{2}$ norm invariant.
\end{remark}

\begin{remark}\label{rem:wavepackets-const}
A careful inspection of the proof reveals that the optimal constants $c_f$ and $C_{f}$ are actually given by $f'(0) - \dlt$ and $f'(0)$ (in the axi-symmetric case, $f'(r_0) - \dlt$ and $f'(r_{0})$), respectively, where $\dlt \searrow 0$ as we move the support of $g_{0}$ closer to the degeneracy. 
\end{remark}

Remarkably, the construction in the case of \eqref{eq:hall-mhd} turns out to be a minor extension of Proposition~\ref{prop:wavepackets} for a suitable choice of $\tu^{z}_{(\lmb)}$ and $\tomg_{(\lmb)}$.
\begin{proposition}[Construction of degenerating wave packets for \eqref{eq:hall-mhd}] \label{prop:wavepackets-hall}
Let $\bgB$, $M$, $\lmb$ and $g_{0}$ be as in Proposition~\ref{prop:wavepackets}. In case~(a) and when $(\bbT, \bbR)_{x} = \bbR_{x}$, assume also that
\begin{equation} \label{eq:wavepacket-hall-g0}
\int x e^{i \lmb x} g_{0}(x, y) \, \ud x = 0 \hbox{ for all } y \in (0, y_{1}).
\end{equation}
In each case, in addition to $(\tb^{z}_{(\lmb)}, \tpsi_{(\lmb)})$, take
\begin{equation} \label{eq:wavepackets-hall-u-omg}
	\tu^{z}_{(\lmb)}[g_{0}] = - \tpsi_{(\lmb)}[g_{0}], \quad \tomg_{(\lmb)}[g_{0}] = -\tb^{z}_{(\lmb)}[g_{0}].
\end{equation}
Then the following properties hold:
\begin{itemize}
\item (smoothing for fluid components) for $ {t \geq 0}$, we have
\begin{align*}
	\nrm{\tu^{z}_{(\lmb)}(t)}_{L^{2}} + \nrm{\nb^{\perp} (-\lap)^{-1} \tomg_{(\lmb)}(t)}_{L^{2}} \aleq & \lmb^{-1} \nrm{(g_{0}, g_{0}^{(-1; \lmb)}, g_{0}^{(-2; \lmb)})}_{H^{1}} ,\\
	\nrm{\nb \tu^{z}_{(\lmb)}(t)}_{L^{2}} + \nrm{\tomg_{(\lmb)}(t)}_{L^{2}} \aleq & \nrm{(g_{0}, g_{0}^{(-1; \lmb)} {)}}_{H^{1}} ;
\end{align*}
\item (error estimates) for $ {t \geq 0}$, we have
\begin{align*}
	\errh_{u}^{(\nu)}[\tu^{z}_{(\lmb)}, \tomg_{(\lmb)}, \tb^{z}_{(\lmb)}, \tpsi_{(\lmb)}] + \nu \lap \tpsi = & 0, \\
	\nrm{\nb^{\perp} (-\lap)^{-1} (\errh_{\omg}^{(\nu)}[\tu^{z}_{(\lmb)}, \tomg_{(\lmb)}, \tb^{z}_{(\lmb)}, \tpsi_{(\lmb)}] + \nu \lap {\tb^z})(t)}_{L^{2}} \aleq & \lmb^{-1} \nrm{(g_{0}^{(-1; \lmb)}, g_{0}^{(-2; \lmb)})}_{H^{4}}, \\
	\nrm{\errh_{b}^{(\nu)}[\tu^{z}_{(\lmb)}, \tomg_{(\lmb)}, \tb^{z}_{(\lmb)}, \tpsi_{(\lmb)}](t)}_{L^{2}} \aleq & \nrm{(g_{0}, g_{0}^{(-1; \lmb)}, g_{0}^{(-2; \lmb)})}_{H^{4}}, \\
	\nrm{\nb \errh_{\psi}^{(\nu)}[\tu^{z}_{(\lmb)}, \tomg_{(\lmb)}, \tb^{z}_{(\lmb)}, \tpsi_{(\lmb)}](t)}_{L^{2}} \aleq & \nrm{g_{0}^{(-2; \lmb)}}_{H^{4}}.
\end{align*}
\end{itemize}
In case~(b), $g_{0}^{(-1; \lmb)}$ and $g_{0}^{(-2; \lmb)}$ are replaced by $g_{0}$. In the above statements, we omitted the dependence of the implicit constants on $f$, $y_{1}$ (in case~(a)), and $r_{1} - r_{0}$ (in case~(b)).
\end{proposition}
Note that, in case~(a) and when $(\bbT, \bbR)_{x} = \bbT_{x}$, the mean-zero condition in \eqref{eq:wavepacket-g0} ensures that $g_{0}^{(-2; \lmb)} \in \calS(M^{2})$ with the same support property as $g_{0}$. When $(\bbT, \bbR)_{x} = \bbR_{x}$, the additional condition \eqref{eq:wavepacket-hall-g0} implies the same properties of $g_{0}^{(-2; \lmb)}$.

\subsection{Derivation of a single second-order-in-time equation and renormalization} \label{subsec:renrm}
In this subsection, we carry out the following algebraic manipulations needed for our proof of Proposition~\ref{prop:wavepackets}:
\begin{itemize}
\item derivation of a single equation for $\psi$, which is second order in $t$;
\item introduction of a suitable change of variables $(y, \psi) \mapsto (\eta, \varphi)$ for $\bgB = f(y) \rd_{x}$ (resp. $(r, \psi) \mapsto (\eta, \varphi)$ for $\bgB = f(r) \rd_{\tht}$), which removes the degeneracy in the principal term and kills all subprincipal (i.e., third order) terms;
\item introduction of a rescaled time $\tau = \lmb t$, which puts the equation in a form where a standard WKB-type ansatz is applicable (see Section~\ref{subsec:wkb} below).
\end{itemize}
The viability of the second manipulation, which is crucial for the proof Proposition~\ref{prop:wavepackets}, is the main advantage of using the $(2+\frac{1}{2})$-dimensional reduction and working with the stationary solutions of the form $\bgB = f(y) \rd_{x}$ or $f(r) \rd_{\tht}$.

\subsubsection*{Translationally-symmetric background}
In order to construct an approximate solution for \eqref{eq:e-mhd-2.5d-lin}, we begin by noting that $\psi$ obeys the following equation:
 \begin{equation} \label{eq:e-mhd-2.5d-lin-second}
\rd_t^2 \psi + f(y)^2\rd_x^2\lap\psi   - f(y)f''(y)\rd_x^2\psi = 0 .
\end{equation}
Indeed, \eqref{eq:e-mhd-2.5d-lin-second} follows by taking $\rd_{t}$ of the second equation in \eqref{eq:e-mhd-2.5d-lin}, and using the first equation to substitute $\rd_{t} b^{z}$. Conversely, we may reconstruct $(\psi, b^{z})$ from a solution $\psi$ to \eqref{eq:e-mhd-2.5d-lin-second} by defining 
\begin{equation} \label{eq:e-mhd-2.5d-lin-bz}
b^{z} = - (f \rd_{x})^{-1} \rd_{t} \psi, 
\end{equation}
provided that $(f \rd_{x})^{-1}$ is well-defined for $\rd_{t} \psi$.

Next, we make a change of variables for \eqref{eq:e-mhd-2.5d-lin-second} to fix the degeneracy in the term $f^{2} \rd_{x}^{2} \rd_{y}^{2}$. Take the connected component of $\set{y : f(y) > 0}$ (in either $\bbT$ or $\bbR$) that intersects any neighborhood of $0$. On this component, take the maximum $y_{1} > 0$ such that $f'(y) \geq f'(0)/2$ and $f(y) \leq \frac{1}{2}$ for $y \in [0, y_{1}]$; if no such maximum exists, simply take $y_{1}$ large enough so that $\supp h \subseteq (\bbT, \bbR)_{x} \times [0, y_{1}]$. For $y \in [0, y_{1}]$, we make a change of variables $\eta = \eta(y)$, where
\begin{equation*}
	\eta'(y) = \frac{1}{f(y)}, \qquad \eta(y_{1}) = 0,
\end{equation*}
so that $\eta \to - \infty$ as $y \to 0^{+}$. Then \eqref{eq:e-mhd-2.5d-lin-second} becomes
\begin{equation*}
\rd_t^2 \psi + \rd_{x}^{2} \rd_{\eta}^{2} \psi + f^2\rd_x^4\psi
- (f^{-1} \rd_{\eta} f) \rd_{x}^{2}  \rd_{\eta} \psi    - \left( \rd_{\eta} \left(f^{-1} \rd_{\eta} f\right) \right) \rd_x^2\psi = 0 .
\end{equation*}
By construction, $h$ is supported in $(\bbT, \bbR)_{x} \times [0, y_{1}]$, on which such a change of variables is valid. 

Finally, for a parameter $\lmb > 0$ to be chosen later, we introduce
\begin{equation*}
	\tau = \lmb t, \quad \varphi = f^{-\frac{1}{2}} \psi.
\end{equation*}
The parameter $\lmb$ will be the magnitude of the space-time frequency of our approximate solution.
The role of the conjugation $\varphi = f^{-\frac{1}{2}} \psi$ is to remove the third order term $- (f^{-1} \rd_{\eta} f) \rd_{x}^{2}  \rd_{\eta} \psi$. Indeed, $\varphi$ in the $(\tau, x, \eta)$ coordinate system  solves:
\begin{equation}\label{eq:e-mhd-eta-conj2}
\begin{split}
\rd_\tau^2\varphi + (\lmb^{-1}\rd_x)^2\rd_\eta^2 \varphi + \lmb^2 f^2(\lmb^{-1}\rd_x)^4\varphi - \left[\frac{1}{2}\rd_\eta ( f^{-1} \rd_{\eta} f) + \frac{1}{4} f^{-2} (\rd_{\eta} f)^{2} \right] (\lmb^{-1}\rd_x)^2 \varphi = 0. 
\end{split}
\end{equation}
Note that $\varphi$ is related to $b$ by
\begin{equation} \label{eq:varphi->b}
b^{x} =  \frac{\rd_\eta(f^{\frac{1}{2}}\varphi)}{f}, \quad b^{y} = -f^{\frac{1}{2}}\rd_x\varphi, \quad b^{z} = - \frac{\rd_t\rd_x^{-1}\varphi}{f^{\frac{1}{2}}} .
\end{equation}

\subsubsection*{Axisymmetric background}
In the same fashion as before, from \eqref{eq:e-mhd-2.5d-lin-axisym} we derive the following single second-order-in-time equation for $\psi$:
\begin{equation} \label{eq:e-mhd-2.5d-lin-second-axisym}
\rd_t^2 \psi  + f(r)^2\rd_\theta^2\lap\psi - f(r) \left( f''(r) + \frac{3}{r}f'(r) \right)\rd_\theta^2\psi  = 0 .
\end{equation}
Conversely, a solution $(\psi, b^{z})$ to \eqref{eq:e-mhd-2.5d-lin-axisym} can be reconstructed from a solution $\psi$ to \eqref{eq:e-mhd-2.5d-lin-second-axisym} by defining 
\begin{equation} \label{eq:e-mhd-2.5d-lin-bz-axisym}
b^{z} = - (f \rd_{\tht})^{-1} \rd_{t} \psi,
\end{equation}
provided that $(f \rd_{\tht})^{-1}$ is well-defined for $\rd_{t} \psi$.

Expanding the Laplacian in the cylindrical coordinates, \eqref{eq:e-mhd-2.5d-lin-second-axisym} can be rewritten as
\begin{equation*}
\rd_t^2 \psi  + f^2\rd_\theta^2 \rd_{r}^{2}\psi 
+ \frac{1}{r} f^2 \rd_\theta^2 \rd_{r} \psi 
+ f^2\rd_\theta^4 \psi - f \left(f'' + \frac{3}{r}f' \right)\rd_\theta^2\psi  = 0 .
\end{equation*}
Fix $r_{1} > r_{0}$ so that $f' \geq \frac{1}{2} f'(r_{0})$ and $f \leq \frac{1}{2}$ on $[r_{0}, r_{1}]$, and furthermore $\supp h \subseteq \set{(\tht, r) : r_{0} \leq r \leq r_{1}}$. Make a change of variables $\eta = \eta(r)$, where
\begin{equation*}
	\eta'(r) = \frac{1}{f(r)}, \qquad \eta(r_{1}) = 0.
\end{equation*}
Note that $\eta \to - \infty$ as $r \to r_{0}^{+}$. Moreover, $\rd_{\eta} f = f \rd_{r} f$ and $\rd_{\eta} r = f$. Thus, in the $(t, \eta, \tht)$-coordinate system, we have 
\begin{equation*}
\rd_t^2 \psi  + \rd_\theta^2 \rd_{\eta}^{2} \psi 
+ f^2\rd_\theta^4 \psi 
+ \left(r^{-1} \rd_{\eta} r - f^{-1} \rd_{\eta} f \right)\rd_\theta^2 \rd_{\eta} \psi 
- \left(\rd_{\eta} \left(f^{-1} \rd_{\eta} f\right) + \frac{3}{r} \rd_{\eta} f \right)\rd_\theta^2\psi  = 0 .
\end{equation*}
Finally, for a parameter $\lmb \in \bbN_{0}$ to be fixed later, we introduce
\begin{equation*}
	\tau = \lmb t, \quad \varphi = \left(\frac{r}{f}\right)^{\frac{1}{2}} \psi.
\end{equation*}
Then $\varphi$ solves the following equation in the $(\tau,  \tht, \eta)$ coordinate system:
\begin{equation} \label{eq:e-mhd-eta-conj2-axisym}
\begin{aligned}
& \rd_{\tau}^{2} \varphi + (\lmb^{-1} \rd_{\tht})^{2} \rd_{\eta}^{2} \varphi + \lmb^{2} f^{2} (\lmb^{-1} \rd_{\tht})^{4} \varphi \\
& \qquad - \left[\frac{1}{2} \rd_{\eta} (f^{-1} \rd_{\eta} f) 
	+ \frac{1}{4} f^{-2} (\rd_{\eta} f)^{2}
	+ 3 r^{-1} \rd_{\eta} f - \frac{1}{4} r^{-2} f^{2} \right] (\lmb^{-1} \rd_{\tht})^{2} \varphi = 0.
\end{aligned}\end{equation}
Note that $\varphi$ is related to $b$ by
\begin{equation} \label{eq:varphi->b-axisym}
b^{r} 
= \left( \frac{f}{r} \right)^{\frac{1}{2}} \rd_\tht \varphi, \quad
b^{\tht} 
=  - \frac{1}{r f} \rd_{\eta}\left(\left(\frac{f}{r}\right)^{\frac{1}{2}} \varphi \right), \quad
b^{z} = - \frac{\rd_t\rd_\theta^{-1}\varphi}{f^{\frac{1}{2}}} .
\end{equation}

\subsection{A WKB-type ansatz} \label{subsec:wkb}
Here we carry out the core construction of the degenerating wave packet approximate solutions in the case of \eqref{eq:e-mhd}. In this subsection, we work exclusively in the renormalized coordinates $(\tau, x, \eta)$; correspondingly, we use the shorthand $f(\eta) = f(y(\eta))$. Moreover, we suppress the dependence of implicit constants on $f$.

\subsubsection*{A WKB-type ansatz for $\bgB = f(y) \rd_{x}$}
We start with the case $\bgB = f(y) \rd_{x}$ and $M^{2} = (\bbT, \bbR)_{x} \times (\bbT, \bbR)_{y}$.
We work with $\varphi(\tau, x, \eta)$ and use a WKB-type ansatz 
\begin{equation*}
\varphi = \lmb^{-1} e^{i \lmb (x + \Phi(\tau, \eta))} h(\tau, x, \eta),
\end{equation*}
with the initial condition $h(0, x, \eta) = h_{0}(x, \eta)$, where we assume that
\begin{equation*}
h_{0} \in \calS((\bbT, \bbR)_{x} \times \bbR_{\eta}), \quad
\supp h_{0} \subseteq (\bbT, \bbR)_{x} \times (-\infty, 0).
\end{equation*}
To obtain the equations for $\Phi$ and $h$, we simply evaluate: 
\begin{equation}\label{eq:wkb_expansion}
\begin{split}
&e^{- i\lmb(x+\Phi(\tau,\eta))}\left[ \rd_\tau^2 + (\lmb^{-1}\rd_x)^2\rd_\eta^2 + \lmb^2 f^2(\lmb^{-1}\rd_x)^4 \right](\lmb^{-1} e^{i\lmb(x+\Phi(\tau,\eta))} {h}(\tau,x,\eta))   \\
& = -\lmb(\rd_\tau\Phi)^2 h + 2i \rd_\tau\Phi \rd_\tau h + i \rd_\tau^2\Phi h + \lmb^{-1} \rd_\tau^2 h \\
&\phantom{=} + \left( - \lmb(\rd_\eta\Phi)^2 + i\rd_\eta^2\Phi + 2i\rd_\eta\Phi\rd_\eta + \lmb^{-1} \rd_\eta^2  \right)(-h + 2i(\lmb^{-1}\rd_x)h + (\lmb^{-1}\rd_x)^2h) \\
&\phantom{=} + f^2 \left( \lmb {h} - 4i \rd_x{h} - 6\lmb^{-1} \rd_x^2 h + 4i\lmb^{-2}\rd_x^3{h} + \lmb^{-3}\rd_x^4{h}  \right) \\
& =	\lmb( -(\rd_\tau\Phi)^2 + (\rd_\eta\Phi)^2 + f^2 )h \\
&\phantom{=}
+ (2i\rd_\tau\Phi\rd_\tau + i\rd_\tau^2\Phi - i\rd_\eta^2\Phi - 2i\rd_\eta\Phi\rd_\eta  - 2i(\rd_\eta\Phi)^2\rd_x - 4if^2\rd_x)h  \\
&\phantom{=}
+ \lmb^{-1} (\cdots)		
\end{split} \end{equation} and setting the first two terms on the far RHS (which are expected to be of orders $\lmb$ and $1$, respectively) to vanish, we obtain respectively the equations \begin{equation}\label{eq:wkb_phi}
\begin{split}
(\rd_\tau\Phi)^2 - (\rd_\eta\Phi)^2 = f^2, 
\end{split}
\end{equation} and \begin{equation}\label{eq:wkb_h}
\begin{split}
(\rd_\tau\Phi\rd_\tau - \rd_\eta\Phi\rd_\eta - (\rd_\eta\Phi)^2\rd_x  - 2f^2\rd_x)h = -\frac{1}{2}(\rd_\tau^2 \Phi - \rd_\eta^2\Phi) h. 
\end{split}
\end{equation} We seek a solution of \eqref{eq:wkb_phi} such that for $\eta < 0$, $h$ from \eqref{eq:wkb_h} is being transported to $\eta \rightarrow -\infty$. 

\subsubsection*{Hamilton--Jacobi equation}
We start by solving \eqref{eq:wkb_phi}. Taking $\Phi(\tau,\eta) = \tau + G(\eta)$, $G$ needs to satisfy $1 - f^2 = (G'(\eta))^2$ (recall that we have assumed from the beginning that $f < 1/2$ in $\eta \leq 0$). We choose $G$ so that $G'(\eta) > 0$ and $G(\eta) - \eta \to 0$ as $\eta \to -\infty$; thus \begin{equation} \label{eq:wkb_G_sol}
G(\eta) = \eta + \int_{-\infty}^{\eta} \left( \sqrt{1 - f^2(\eta')} - 1 \right) \ud \eta' ,
\end{equation} which fixes \begin{equation}\label{eq:wkb_phi_sol}
\begin{split}
\Phi(\tau,\eta) = \tau + \eta + \int_{-\infty}^{\eta} \left( \sqrt{1 - f^2(\eta')} - 1 \right) \ud \eta',
\end{split}
\end{equation} and leads to \begin{equation}\label{eq:wkb_h_sol}
\begin{split}
(\rd_\tau - \sqrt{1 - f^2}\rd_\eta - (1+f^2)\rd_x)h = -\frac{1}{2}\frac{f\rd_\eta f}{\sqrt{1 - f^2}} h. 
\end{split}
\end{equation}
Our choice of the sign of $G'(\eta)$ is justified by the fact that the characteristics for the LHS of  \eqref{eq:wkb_h_sol} travel towards $\eta \to - \infty$ forward in time, as we will explicitly compute below.
\begin{remark}
	The fact that we can explicitly solve the equation for the phase $\Phi$ (i.e., the Hamilton--Jacobi equation) is a manifestation of complete integrability of the bicharacteristic flow around $\bgB$ as in Theorem~\ref{thm:norm-growth}.
\end{remark}

\subsubsection*{Characteristics for the transport operator}
Our next step is to analyze the transport operator
\begin{equation} \label{eq:transport}
\calL = \rd_{\tau} - \sqrt{1 - f^{2}} \rd_{\eta} - (1+f^{2}) \rd_{x} ,
\end{equation}
towards the goal of estimating $h$ via the transport equation \eqref{eq:wkb_h_sol}.

To control the characteristics associate to $\calL$, we need information on the coefficients. Note that \begin{equation*}
\begin{split}
\eta(y) \approx c + c_0 \ln y 
\end{split}
\end{equation*} for some constant $c$ and $c_0 = \rd_y f(0) > 0$, where we use $\approx $ to denote that the ratio of both sides converges to 1 as $y \rightarrow 0$ (or equivalently, $\eta \rightarrow -\infty$). This implies that \begin{equation} \label{eq:wkb-f}
f(\eta) \approx c' e^{c_0\eta} 
\end{equation} and in particular we obtain $0 < f(\eta)\le C e^{c_0\eta}$ for all $\eta \le 0$ for some constant $C > 0$ independent of $\eta$ (but depending on $f$). Similarly, $\rd_\eta f = f \rd_y f$ and $\rd_\eta^2 f = f^2 \rd_y^2 f + f \left( \rd_yf \right)^2$ imply
\begin{equation*}
\begin{split}
|\rd_\eta f|(\eta) \le C \nrm{ \rd_y f }_{L^\infty_y} f(\eta) \lesssim e^{c_0\eta}, \quad 
|\rd_\eta^2 f|(\eta) \le (\nrm{ f\rd_y^2 f}_{L^\infty_y} + \nrm{ (\rd_y f)^2 }_{L^\infty_y})f(\eta) \lesssim e^{c_0\eta},
\end{split}
\end{equation*} for $\eta \le 0$. Continuing, it is straightforward to see that \begin{equation} \label{eq:f-exp}
|\rd_\eta^{(n)}f|(\eta) \lesssim_n f(\eta) \lesssim e^{c_0\eta},\quad \forall \eta \le 0.  
\end{equation} 

With the above information, we are now ready to study the geometry of the characteristics $X(\tau), Y(\tau)$ associated to $\calL$, which are defined as
\begin{equation}\label{eq:characteristics-R3}
\begin{split}
\frac{\ud}{\ud \tau} (X(\tau,x_0,\eta_0), Y(\tau,x_0,\eta_0) ) &= \left(  -(1 + (f(Y))^2), -\sqrt{1 - (f(Y))^2} \right),\\
(X(0,x_0,\eta_0), Y(0,x_0,\eta_0))& = (x_0,\eta_0) ,
\end{split}
\end{equation} so that
\begin{equation} \label{eq:characteristic-trans}
\frac{\ud}{\ud \tau} h(\tau, X(\tau), Y(\tau)) =  \left( \calL h \right)(\tau,X(\tau),Y(\tau)).
\end{equation} 
We will always assume that $\eta_0 < 0$ (by hypothesis, the support of $h_{0}$ lies in this region), which guarantees that $f^{-1} \rd_\eta f > c_0 / 2$ where $c_0 = \rd_yf(0) > 0$ from our choice of the change of variables from $y$ to $\eta$. From \eqref{eq:characteristics-R3}, one sees that $Y$ is independent of $x_0$ and \begin{equation} \label{eq:characteristic-Y}
\eta_0 -\tau \le Y(\tau,\eta_0) \le  \eta_0 - \frac{\tau}{2} .
\end{equation} 
In particular, observe that $Y(\tau, \eta_{0})$ stays in $(-\infty, 0)$ if $\eta_{0} \in (-\infty, 0)$. 
Using that $f(\eta) \le Ce^{c_0\eta}$ with the equation for $\rd_\tau X$ we obtain that \begin{equation} \label{eq:characteristic-X}
x_0 - 2\tau  <  X(\tau,x_0,\eta_0) < x_0 -\tau .
\end{equation}  

\subsubsection*{Analysis of the transport equation}
We now analyze \eqref{eq:wkb_h_sol} and obtain estimates for $h$. First, observe that \eqref{eq:wkb_h_sol} can be simplified using the method of integrating factors. Indeed,
introducing a real-valued function $\alp(\tau, x, \eta)$ defined by
\begin{equation} \label{eq:wkb-alp}
\calL \alp = - \frac{1}{2} \frac{f \rd_{\eta} f}{\sqrt{1 - f^{2}}}, 
\end{equation}
with the initial condition $\alp(\tau = 0) = 0$, we see that
\begin{equation} \label{eq:wkb-h'}
\calL (e^{-\alp} h ) = 0.
\end{equation}
By \eqref{eq:f-exp} and the bound $\abs{f} < \frac{1}{2}$, for any $m \in \bbN_{0}$ observe that
\begin{equation*}
\Abs{\rd_{\eta}^{m} \left( - \frac{1}{2} \frac{f \rd_{\eta} f}{\sqrt{1 - f^{2}}}\right)}
\aleq e^{ 2 c_{0} \eta} \quad \hbox{ for } \eta \in (-\infty, 0).
\end{equation*}
Moreover, again by \eqref{eq:f-exp} and the bound $\abs{f} < \frac{1}{2}$, we have
\begin{equation} \label{eq:characteristic-commute}
[\rd_{\eta}^{m}, \calL] = \sum_{\ell = 0}^{m} c^{m}_{\ell}(\eta) \rd_{\eta}, \quad \abs{c^{m}_{\ell}(\eta)} \aleq e^{ 2 c_{0} \eta}  \quad \hbox{ for } \eta \in (-\infty, 0),
\end{equation}
while $\calL$ commutes with $\rd_{x}$ and $\rd_{\tau}$. 

In view of \eqref{eq:characteristic-Y}, the exponential factor $e^{ 2 c_{0} \eta}$ turns into an exponential decay in $\tau$ along each characteristics. Thus, by the above commutator relations, integration along characteristics and Gronwall's inequality, we immediately obtain the following $L^{\infty}$ bound for $\alp$:
\begin{align*}
\sup_{0 \leq k + \ell \leq m} \sup_{\tau \geq 0} \nrm{\rd_{\tau}^{k} \rd_{x}^{\ell} \rd_{\eta}^{m - \ell - k} \alp(\tau)}_{L^{\infty}_{x, \eta}}
\aleq_{m} & 1.
\end{align*}
For $e^{-\alp} h$, we wish to prove $L^{p}$ bounds for $1 \leq p \leq \infty$. For this purpose, we note tha tthe divergence of $\calL$ with respect to the volume form $\ud x \wedge \ud \eta$ obeys
\begin{equation} \label{eq:characteristic-div}
\Abs{\mathrm{div}_{\ud x \wedge \ud \eta} \calL} 
= \Abs{\frac{f \rd_{\eta} f}{\sqrt{1-f^{2}}}} \aeq e^{ 2 c_{0} \eta} \quad \hbox{ for } \eta \in (-\infty, 0),
\end{equation}
which also decays exponentially along characteristics. Thus, for any $1 \leq p \leq \infty$, we obtain
\begin{equation*}
\max_{0 \leq k, \ell, k+\ell \leq m} \sup_{\tau \geq 0} \nrm{\rd_{\tau}^{k} \rd_{x}^{\ell} \rd_{\eta}^{m - \ell - k}(e^{-\alp} h)(\tau)}_{L^{p}_{x, \eta}}
\aleq_{m} \nrm{\rd_{x}^{\ell} \rd_{\eta}^{m-\ell} h_{0} }_{W^{m, p}_{x, \eta}}.
\end{equation*}
Therefore, we have arrived at the following result: \begin{lemma} \label{lem:wkb-h-est}
	Let $h$ be the solution to \eqref{eq:wkb_h_sol} with smooth initial data $h_0$ supported on $\eta \le 0$. Then for any $1 \leq p \leq \infty$, we have \begin{equation*}
	\begin{split}
	\max_{0 \leq k, \ell, k+\ell \leq m}  \sup_{\tau \ge 0}\nrm{ \rd_\tau^k\rd_x^\ell \rd_\eta^{m-k-\ell} h(\tau)}_{L^{p}_{x, \eta}}   \lesssim_m \nrm{h_0}_{W^{m, p}_{x, \eta}}.
	\end{split}
	\end{equation*}
Moreover, 
\begin{equation*}
	\begin{split}
	 \sup_{\tau \ge 0}\nrm{\rd_\eta^{m} h(\tau)}_{L^{p}_{x,\eta}}   \lesssim_m \nrm{h_0}_{L^{p}_{x} W^{m, p}_{\eta}}.
	\end{split}
\end{equation*}
\end{lemma}  

As another application of \eqref{eq:characteristic-div}, we estimate the size of the $\eta$-support of $h(\tau)$:
\begin{lemma} \label{lem:wkb-h-supp}
Let $h$ be the solution to \eqref{eq:wkb_h_sol} with smooth initial data $h_0$ supported on $(\bbT, \bbR)_{x} \times [\eta_{0}, \eta_{1}] \subseteq (\bbT, \bbR)_{x} \times (-\infty, 0]$. Then $\supp h(\tau, \cdot, \cdot) \subseteq (\bbT, \bbR)_{x} \times [Y(\tau, \eta_{1}), Y(\tau, \eta_{0})]$. Moreover,
\begin{equation*}
	\abs{Y(\tau, \eta_{1}) - Y(\tau, \eta_{0})} \aeq \eta_{1} - \eta_{0}.
\end{equation*}
\end{lemma}
\begin{proof}
The statement concerning $\supp h(\tau, \cdot, \cdot)$ is easily proved using the method of characteristics. To prove the remaining statement, it suffices to restrict our attention to the case $(\bbT, \bbR)_{x} = \bbT_{x}$. Let $\chi(\tau, x, \eta)$ be the solution to $\calL \chi = 0$ with $\chi(0, \cdot, \cdot)$ equal to the characteristic function of $[\eta_{0}, \eta_{1}]$ in $\eta$. By the method of characteristics, we see that $\chi(\tau, x, \eta)$ is independent of $x$ and is equal to the characteristic function of the interval $[Y(\tau, \eta_{0}), Y(\tau, \eta_{1})]$ in $\eta$. Hence, $Y(\tau, \eta_{1}) - Y(\tau, \eta_{0}) = \int \chi(\tau, \eta) \, \ud \eta$, while the latter is comparable to the value at $\tau = 0$ (i.e., $\eta_{1} - \eta_{0}$) thanks to \eqref{eq:characteristic-div} (which, as remarked above, exponentially decay along characteristics). \qedhere
\end{proof}

\subsubsection*{Error in the $\varphi$-equation}
Finally, we estimate the error in the $\varphi$-equation. Let $\errwp_{\varphi}[h_{0}; \lmb](\tau, x, \eta)$ be the LHS of \eqref{eq:e-mhd-eta-conj2} evaluated with $\varphi = \lmb^{-1} e^{i \lmb (x + \Phi(\tau, \eta))} h$. In what follows, we will often abbreviate $\errwp_{\varphi} = \errwp_{\varphi}[h_{0}; \lmb](\tau, x, \eta)$. We compute
\begin{equation} \label{eq:wkb-error}
\begin{split}
&\errwp_{\varphi} = -\lmb^{-1} e^{i\lmb(x +\Phi)}( \frac{1}{2} \rd_\eta\left( {f^{-1}\rd_\eta f} \right) + \frac{1}{4} {(f^{-1}\rd_\eta f)^2}   ) ( -{h} + 2i(\lmb^{-1}\rd_x)h + (\lmb^{-1}\rd_x)^2 h) \\  
&\phantom{\errwp_{\varphi} =}
+ \lmb^{-1} e^{i\lmb(x +\Phi)} \left(\rd_\tau^2 h + \rd_\eta^2(-h + 2i(\lmb^{-1}\rd_x)h + (\lmb^{-1}\rd_x)^2 )h \right. \\  
&\phantom{\errwp_{\varphi} = + \lmb^{-1} e^{i\lmb(x +\Phi)}} 
\left. + i\lmb(\rd_\eta^2\Phi+2\rd_\eta\Phi\rd_\eta)(2i(\lmb^{-1}\rd_x)h + (\lmb^{-1}\rd_x)^2 h)) \right.  \\
&\phantom{\errwp_{\varphi} = + \lmb^{-1} e^{i\lmb(x +\Phi)}}
\left.  - \lmb^2(\rd_\eta\Phi)^2(\lmb^{-1}\rd_x)^2h  + f^2(-6\rd_x^2 h + 4i\lmb^{-1}\rd_x^3 h + \lmb^{-2}\rd_x^4 h ) \right) .
\end{split}
\end{equation} For each fixed $\tau \geq 0$, we see that $\errwp_{\varphi}$ is bounded in $L_{x, \eta}^2$ by \begin{equation} \label{eq:wkb-err-est}
\begin{aligned}
\nrm{\errwp_{\varphi}(\tau)}_{L_{x, \eta}^2} \lesssim& \lmb^{-1}( \nrm{ h}_{L_{x, \eta}^2}  + \nrm{ \rd_\tau^2 h}_{L_{x, \eta}^2}+ \nrm{ \rd_\eta^2 h }_{L_{x, \eta}^2} + \nrm{ \rd_\eta^2\rd_x^2 h}_{L_{x, \eta}^2} + \nrm{\rd_x^4 h}_{L_{x, \eta}^2} )(\tau) \\
\lesssim &\lmb^{-1} \nrm{h_0}_{H^4}.
\end{aligned}
\end{equation} Moreover, when we compute $\rd_x \errwp_{\varphi}$, we only lose at most a constant multiple of $\lmb$ (when $\rd_x$ falls on the phase $e^{i\lmb(x+\Phi)}$). Therefore, for any integer $m \ge 0$, we obtain 
\begin{equation} \label{eq:wkb-err-dx-est}
\sup_{\tau \ge 0}\nrm{(\lmb^{-1}\rd_x)^{m}\errwp_{\varphi}(\tau)}_{L_{x, \eta}^2} \lesssim_m \lmb^{-1} \nrm{h_{0}}_{H^{4+m}}.
\end{equation}
\subsubsection*{Modifications for $\bgB = f(r) \rd_{\tht}$}
Finally, we sketch the necessary modifications needed in case $\bgB = f(r) \rd_{\tht}$, which are all minor. The ansatz now takes the form
\begin{equation*}
\varphi = \lmb^{-1} e^{i \lmb (\tht + \Phi(\tau, \eta))} h(\tau, \eta),
\end{equation*}
with the initial condition $h(0, \eta) = h_{0}(\eta)$ satisfying
\begin{equation*}
h_{0} \in C^{\infty}(\bbR_{\eta}), \quad \supp h_{0} \subseteq (-\infty, 0).
\end{equation*}
Note that $\tht$ plays the role of $x$, and $h(\tau, \eta)$ is chosen to be independent of $\tht$.

Since \eqref{eq:e-mhd-eta-conj2} and \eqref{eq:e-mhd-eta-conj2-axisym} differ only by terms of order $2$ (in space) and lower, the Hamilton--Jacobi and transport equations satisfied by $\Phi$ and $h$ are exactly the same as in the previous case, where the term $-(1 + f^{2}) \rd_{x}$ is dropped. Therefore, Lemma~\ref{lem:wkb-h-est} holds with $x$ replaced by $\tht$. 

In this case, we define $\errwp_{\varphi}[h_{0}; \lmb](\tau, x, \eta)$ to be the LHS of \eqref{eq:e-mhd-eta-conj2-axisym}. Again, since \eqref{eq:e-mhd-eta-conj2} and \eqref{eq:e-mhd-eta-conj2-axisym} differ only by terms of order $\lmb^{-1}$ and lower, it is straightforward to establish the analogues of \eqref{eq:wkb-err-est} and \eqref{eq:wkb-err-dx-est} hold with $x$ replaced by $\tht$ (and without $ {m}$ on the RHS of \eqref{eq:wkb-err-dx-est}, although this point will be irrelevant).

\subsection{Proof of Propositions~\ref{prop:wavepackets} and~\ref{prop:wavepackets-hall}}
We are ready to complete the proofs of the results stated in Section~\ref{subsec:wavepackets-results}.
\begin{proof}[Proof of Proposition~\ref{prop:wavepackets}]
We first handle case~(a), i.e., when $\bgB = f(y) \rd_{x}$. We apply the WKB construction in Section~\ref{subsec:wkb} to
\begin{equation*}
	h_{0}^{(-1)} (x, \eta) = \frac{1}{i \lmb} f^{-\frac{1}{2}} (y(\eta))g_{0}^{(-1; \lmb)}(x, y(\eta)), \quad
	h_{0} (x, \eta) = f^{-\frac{1}{2}} (y(\eta))g_{0}(x, y(\eta)),
\end{equation*}
and denote the resulting amplitudes (i.e., the solution to \eqref{eq:wkb_h}) by $h^{(-1)}$ and $h$, respectively. By construction, we have the relations
\begin{equation*}
	e^{i \lmb x + i \lmb \tau + i \lmb G(\eta)} h(\tau, x, \eta)
	= \rd_{x} \left(e^{i \lmb x + i \lmb \tau + i \lmb G(\eta)} h^{(-1)}(\tau, x, \eta)\right), \quad
	\errwp_{\varphi}[h_{0}; \lmb]
	= \rd_{x} \errwp_{\varphi}[h_{0}^{(-1)}; \lmb].
\end{equation*}

Given $h$, we define the approximate solution by 
\begin{align}
	\tb^{z}_{(\lmb)} 
	= & f^{-\frac{1}{2}} \lmb \, \Re \left(\frac{1}{i} e^{i (\lmb^{2} t + \lmb x + \lmb G(\eta(y)))} (h^{(-1)} + \frac{1}{i \lmb} \rd_{\tau} h^{(-1)}) (\lmb t, x, \eta(y)) \right) , \label{eq:def-tb}\\
	\tpsi_{(\lmb)} =& f^{\frac{1}{2}} \lmb^{-1} \, \Re \left( e^{i (\lmb^{2} t + \lmb x + \lmb G(\eta(y)))} h(\lmb t, x, \eta(y))\right). \label{eq:def-tpsi}
\end{align} From the definition, it follows that the identities \begin{equation*}
\begin{split}
\rd_t \tpsi_{(\lmb)} &= f^{\frac{1}{2}} \lmb^{-1} \Re \left( e^{i (\lmb^{2} t + \lmb x + \lmb G(\eta(y)))} (\lmb \rd_\tau h + i\lmb^2 h) \right)  \\
&=  f^{\frac{1}{2}} \lmb \rd_x \Re \left( i e^{i (\lmb^{2} t + \lmb x + \lmb G(\eta(y)))} (\frac{1}{i\lmb} \rd_\tau h^{(-1)} +   h^{(-1)}) \right) \\
&= -f\rd_x \tb^{z}_{(\lmb)} 
\end{split}
\end{equation*} hold. Next, from the construction, the linearity and $x$-invariance properties (as stated in Proposition~\ref{prop:wavepackets}) are clear. Evaluating the expression  \eqref{eq:def-tpsi} at $t = 0$, we obtain that \begin{equation*}
\begin{split}
\tpsi_{(\lmb)}(t = 0) &=  f^{\frac{1}{2}} \lmb^{-1} \Re \left( e^{i \lmb  (x +  G )} h_0 \right)  = \lmb^{-1} \Re \left( e^{i \lmb  ( x+G)} g_0 \right)  . 
\end{split}
\end{equation*} Next, using the relation between $\tpsi_{(\lmb)}$ and $\tb_{(\lmb)}^z$, we have \begin{equation*}
\begin{split}
\tb^z_{(\lmb)}(t = 0) &= \lmb f^{-\frac{1}{2}} \rd_x^{-1} \Re \left( \frac{1}{i} e^{i\lmb(x+G)} (h_0 + \frac{1}{i\lmb}(\rd_{\tau} h)_0 \right) 
\end{split}
\end{equation*} and the first term is simply \begin{equation*}
\begin{split}
 \lmb f^{-1} \rd_x^{-1} \Re \left( \frac{1}{i} e^{i\lmb(x+G)}g_0  \right)
\end{split}
\end{equation*} whereas the second term is given by \begin{equation*}
\begin{split}
&-f^{-\frac{1}{2}} \rd_x^{-1} \Re \left( e^{i\lmb(x+G)} (\rd_\tau h)_0 \right) \\
&\qquad = -f^{-\frac{1}{2}} \rd_x^{-1} \Re \left( e^{i\lmb(x+G)} ( \sqrt{1-f^2}\rd_{\eta} + (1+f^2)\rd_x - \frac{1}{2} \frac{f\rd_{\eta} f}{\sqrt{1-f^2}} )h_0 \right) \\
&\qquad = -f^{-\frac{1}{2}} \rd_x^{-1} \Re \left( e^{i\lmb(x+G)} ( -\frac{1}{2} \frac{f^{\frac{1}{2}}\rd_{\eta}f}{f^2\sqrt{1-f^2}} g_0 + f^{-\frac{1}{2}}\sqrt{1-f^2}\rd_\eta g_0 + f^{-\frac{1}{2}}(1+f^2)\rd_x g_0 ) \right) \\
&\qquad = f^{-1}\rd_x^{-1} \Re\left( e^{i\lmb(x+G)} ( \frac{1}{2} \frac{\rd_y f}{\sqrt{1-f^2}} g_0 - f\sqrt{1-f^2} \rd_y g_0 - (1+f^2) \rd_x g_0 )   \right). 
\end{split}
\end{equation*} 
To prove the initial data lower bound, it suffices to estimate $\rd_{x} \tpsi_{(\lmb)}$. Note that \begin{equation*}
\begin{split}
\rd_{x} \tpsi_{(\lmb)} =\Re (e^{i \lmb (x + G(y))} g_{0}) + \lmb^{-1} \Re (e^{i\lmb (x+G(y))} \rd_x g_0)
\end{split}
\end{equation*} and clearly the second term in $L^2$ is bounded by $C\lmb^{-1}\nrm{g_0}_{H^1}$. Regarding the first term, we compute 
\begin{align*}
\left( \Re (e^{i \lmb (x + G(y))} g_{0}) \right)^{2} 
 =& \frac{1}{4} (e^{i \lmb (x + G(y))} g_{0} + e^{-i \lmb (x + G(y))} \bar{g_{0}})^{2} \\
=& \frac{1}{2} \abs{g_{0}}^{2} + \frac{1}{4} e^{2 i \lmb (x + G(y))} g_{0}^{2} + \frac{1}{4} e^{-2 i \lmb (x + G(y))} \bar{g_{0}}^{2}.
\end{align*}
Thus, integrating this equation over $M^{2}$ and using integration by parts in $x$ for the last two terms, we obtain the desired initial data lower bound.

To verify the remaining assertions in Proposition~\ref{prop:wavepackets}, we need to transfer the upper bounds proved in Section~\ref{subsec:wkb} to the present context. At $t = \tau = 0$ and any $0 \leq k \leq m$, we have the relations
\begin{align} 
	\nrm{\rd_{x}^{k} \rd_{\eta}^{m-k} h_{0}}_{L^{2}_{x, \eta}}
	=& \nrm{f^{-\frac{1}{2}} \rd_{x}^{k} (f \rd_{y})^{m -k} f^{-\frac{1}{2}} g_{0}}_{L^{2}_{x, y}}
	\aleq_{m} (f'(0) y_{1})^{-1} \nrm{g_{0}}_{H^{m}_{x, y}}, \label{eq:wp-g0->h0}
\end{align} and \begin{align}
	\nrm{\rd_{x}^{k} \rd_{\eta}^{m-k} h_{0}^{(-1)}}_{L^{2}_{x, \eta}}
	=& \lmb^{-1} \nrm{f^{-\frac{1}{2}} \rd_{x}^{k} (f \rd_{y})^{m -k} f^{-\frac{1}{2}} g_{0}^{(-1; \lmb)}}_{L^{2}_{x, y}}
	\aleq_{m} (f'(0) y_{1})^{-1} \lmb^{-1} \nrm{g_{0}^{(-1; \lmb)}}_{H^{m}_{x, y}}. \label{eq:wp-g1->h1} 
\end{align}
Combined with Lemma~\ref{lem:wkb-h-est}, for any $m \in \bbN_{0}$ and $t \geq 0$, we have
\begin{align}
	\max_{0 \leq \ell, k, \ell + k \leq m} \nrm{(\lmb^{-2} \rd_{t})^{\ell} (\lmb^{-1} \rd_{x})^{k} (\lmb^{-1} f \rd_{y})^{m-k-\ell} f^{-1} \tpsi_{(\lmb)}(t)}_{L^{2}_{x, y}}
	\aleq_{m, (f'(0) y_{1})^{-1}} & \lmb^{-1} \nrm{g_{0}}_{H^{m}_{x, y}}, \label{eq:wp-psi-est}\\
	\max_{0 \leq \ell, k \ell + k \leq m} \nrm{(\lmb^{-2} \rd_{t})^{\ell} (\lmb^{-1} \rd_{x})^{k} (\lmb^{-1} f \rd_{y})^{m-k-\ell} \tb_{(\lmb)}^z(t)}_{L^{2}_{x, y}}
	\aleq_{m, (f'(0) y_{1})^{-1}} & \nrm{g_{0}^{(-1; \lmb)}}_{H^{m+1}_{x, y}}. \label{eq:wp-b-est}
\end{align}
The regularity estimates now follow in a straightforward manner; note that in order to estimate $\rd_{x} \tpsi_{(\lmb)}$, we needed to use the fact that $f < \frac{1}{2}$ on the support of $(\tpsi_{(\lmb)}, \tb_{(\lmb)}^z)$. 

We now prove the degeneration properties. 
As a preparation, we begin by noting that each component $\tb_{(\lmb)}^{x} = \rd_{y} \tpsi_{(\lmb)}$, $\tb_{(\lmb)}^{y} = -\rd_{x} \tpsi_{(\lmb)}$ and $\tb_{(\lmb)}^{z}$ of $\tb_{(\lmb)}$ of $\tb_{(\lmb)}$ may be written in the form 
\begin{equation} \label{eq:wavepackets-degen-overall}
	\tilde{b}^{x, y, z}_{(\lmb)} (t, x, y) = f^{-\frac{1}{2}}(y) \Re\left( e^{i (\lmb^{2} t + \lmb x + \lmb G(\eta(y)))} \tilde{h}_{(\lmb)}^{x, y, z}(\lmb t, x, \eta(y)) \right),
\end{equation}
where, for any $p \in [1, \infty]$, $m \in \bbN_{0}$ and $\tau \geq 0$, each of $\tilde{h}_{(\lmb)}^{x}$, $\tilde{h}_{(\lmb)}^{y}$ and $\tilde{h}_{(\lmb)}^{z}$ obeys
\begin{align} 
	\nrm{\rd_{\eta}^{m} \tilde{h}_{(\lmb)}^{x, y, z}(\tau, x, \eta)}_{L^{p}_{x, \eta}}
	& \aleq_{m} \nrm{ (g_{0}, g_{0}^{(-1; \lmb)}(x, y))}_{W^{m+1, p}_{x, y}}, \label{eq:wavepackets-degen-amp} \\
\supp \tilde{h}_{(\lmb)}^{x, y, z}(\tau, \cdot, \cdot) 
	&\subseteq (\bbT, \bbR)_{x} \times (Y(\tau, 0) - Y_{supp}, Y(\tau, 0))_{\eta},
\label{eq:wavepackets-degen-supp}
\end{align}
where $Y(\tau, \eta_{0})$ is the $\eta$-characteristic for $\calL$ introduced in Section~\ref{subsec:wkb}, and $Y_{supp}$ is independent of $\tau$, $\lmb$ (but dependent on $f$). Indeed, from \eqref{eq:def-tb}, \eqref{eq:def-tpsi} and Lemma~\ref{lem:wkb-h-est}, \eqref{eq:wavepackets-degen-amp} with the $W^{m+1, p}_{x, \eta}$-norm on the RHS follows. Recall that $\supp (g_{0}, g_{0}^{(-1; \lmb)})$ (in the variables $x, \eta$) is contained in $(\bbT, \bbR)_{x} \times [\eta(\frac{1}{2} y_{1}), \eta(y_{1})] \subseteq (\bbT, \bbR)_{x} \times [-\frac{2}{c_{0}} \log 2, 0]$, where the the last inclusion follows from the hypothesis $f'(y) \geq \frac{1}{2} c_{0}$ and the choice $\eta(y_{1}) = 0$. Thus, the values of $f$ is comparable on $\supp (g_{0}, g_{0}^{(-1; \lmb)})$, so that $\nrm{(g_{0}, g_{0}^{(-1; \lmb)}}_{W^{m+1, p}_{x, \eta}} \aeq\nrm{(g_{0}, g_{0}^{(-1; \lmb)}}_{W^{m+1, p}_{x, y}}$ (with a constant depending only on $m$ and $f$); hence \eqref{eq:wavepackets-degen-amp} follows.
For \eqref{eq:wavepackets-degen-supp}, we use Lemma~\ref{lem:wkb-h-supp} and preceding assertion about the $\eta$-support of the initial data. 

In what follows, we will only be using \eqref{eq:wavepackets-degen-overall}, \eqref{eq:wavepackets-degen-amp} and \eqref{eq:wavepackets-degen-supp}, and hence the components $\tb_{(\lmb)}^{x, y, z}$ will be treated in the same manner; hence the superscripts $x, y, z$ will often be suppressed. Moreover, the same proof applies to $(\lmb^{-1} \rd_{x})^{\ell} \tb_{(\lmb)}$ for any $\ell \in \bbN_{0}$. 

Let $f_{\lmb t} = f(y(Y(\lmb t, 0)))$. Note that, by \eqref{eq:wkb-f}, \eqref{eq:wavepackets-degen-overall} and \eqref{eq:wavepackets-degen-supp}, for $t \geq 0$ we have
\begin{equation} \label{eq:wavepackets-degen-f}
	f(y) \aeq f(y(Y(\lmb t, 0))) = f_{\lmb t} \quad \hbox{ on } \supp \tb_{(\lmb)}(t, \cdot, \cdot),
\end{equation}
i.e., $f_{\lmb t}$ is the typical value of $f$ on the support of $\tb_{(\lmb)}$. We claim that
\begin{align} 
	\nrm{P_{y; k} \tb_{(\lmb)}(t, x, y)}_{L^{p}_{x, y}}
	& \aleq_{m} (2^{-k} \lmb f_{\lmb t}^{-1})^{m} f_{\lmb t}^{\frac{1}{p}-\frac{1}{2}} \nrm{(g_{0}, g_{0}^{(-1; \lmb)})}_{W^{m+1, p}_{x, y}} 
	& & \hbox{ for } 2^{k} \geq \lmb f_{\lmb t}^{-1}, \label{eq:wavepackets-degen-high} \\
	\nrm{P_{y; k} \tb_{(\lmb)}(t, x, y)}_{L^{p}_{x, y}}
	& \aleq_{m} (2^{k} \lmb^{-1} f_{\lmb t})^{m} f_{\lmb t}^{\frac{1}{p}-\frac{1}{2}} \nrm{(g_{0}, g_{0}^{(-1; \lmb)})}_{W^{m+1, p}_{x, y}} 
	& &\hbox{ for } f_{\lmb t}^{-1} \leq 2^{k} \leq \lmb f_{\lmb t}^{-1}, \label{eq:wavepackets-degen-med} \\
	\nrm{P_{y; k} \tb_{(\lmb)}(t, x, y)}_{L^{p}_{x, y}} 
	& \aleq \lmb^{-1} 2^{(1-\frac{1}{p}) k} f_{\lmb t}^{\frac{1}{2}}  \nrm{(g_{0}, g_{0}^{(-1; \lmb)})}_{W^{2, p}_{x, y}}
	& & \hbox{ for } 2^{k} \leq f_{\lmb t} ^{-1}
	 \label{eq:wavepackets-degen-low},
\end{align}
where $P_{y; k}$ is the inhomogeneous Littlewood--Paley projection\footnote{The precise definition is as follows. Denote by $\calF_{y}[f(y)](\hat{y})$ the Fourier transform in $y$, where $\hat{y}$ is the dual variable. Consider a smooth partition of unity $1 = m_{0}(\hat{y}) + \sum_{k} m_{k}(\hat{y})$ on $\bbR$, where $m_{0} = 1$ on $[-1, 1]$ and vanishes outside of $[-2, 2]$ and $m_{k}(\hat{y}) = m_{\leq 0}(\hat{y}/2^{k}) - m_{\leq 0}(\hat{y}/2^{k-1})$. Correspondingly, we define $\set{P_{y; k}}_{k \in \bbN_{0}}$ by $\calF_{y}[P_{y; k} f](\hat{y}) = m_{k}(\hat{y}) \calF_{y}[f](\hat{y})$.} to $y$-frequencies $\aeq 2^{k}$.

First, we demonstrate how \eqref{eq:wavepackets-degen-upper}--\eqref{eq:wavepackets-degen-small} follow from the preceding estimates.
To prove \eqref{eq:wavepackets-degen-upper}, it suffices to bound $\sum_{k \geq - \log_{2} f_{\lmb t}} 2^{s k} \nrm{P_{y; k} \tb_{(\lmb)}}_{L^{p}_{x, y}}$ by the RHS of \eqref{eq:wavepackets-degen-upper}, via the triangle inequality and the frequency localization property of $P_{y; k}$. We split the $k$-summation into the ranges above and use \eqref{eq:wavepackets-degen-high} with $m > \max\set{s, 0}$, \eqref{eq:wavepackets-degen-med} with $m > \max\set{-s, 0}$ and \eqref{eq:wavepackets-degen-low} in the respective ranges; as a result, we would obtain \eqref{eq:wavepackets-degen-upper} with $f_{\lmb t}$ on the RHS in place of $e^{-c_{f} \lmb t}$. Finally, we use 
\begin{equation} \label{eq:wavepackets-degen-cf}
e^{-C_{f} \lmb t} \aleq f_{\lmb t} = f(y(Y(\tau, 0))) \aleq e^{-c_{f} \lmb t} \quad \hbox{ with } c_{f} = \frac{1}{2} c_{0} = \frac{1}{2} \rd_{y} f(0), \quad C_{f} = 2 c_{f} = \rd_{y} f(0),
\end{equation}
which follows from \eqref{eq:wkb-f} and $-\tau \leq Y(\tau, 0) \leq -\frac{\tau}{2}$ from \eqref{eq:characteristic-Y}, to eliminate $f_{\lmb t}$.
For the proof of \eqref{eq:wavepackets-degen}--\eqref{eq:wavepackets-degen-small}, we decompose $\tb_{(\lmb)}$ into $\tb_{(\lmb)}^{main} + \tb_{(\lmb)}^{small}$, where
\begin{equation*}
	\tb^{main; x, y, z}_{(\lmb)} = \sum_{k \geq - \log_{2} f_{\lmb t}} P_{y; k} \tb_{(\lmb)}^{x, y, z}, \qquad
	\tb^{small; x, y, z}_{(\lmb)} = \sum_{k < - \log_{2} f_{\lmb t}} P_{y; k} \tb_{(\lmb)}^{x, y, z}.
\end{equation*}
Proceeding as before using \eqref{eq:wavepackets-degen-high} with $m > \max\set{s, 0}$, \eqref{eq:wavepackets-degen-med} with $m > \max\set{-s, 0}$ and \eqref{eq:wavepackets-degen-cf}, we obtain \eqref{eq:wavepackets-degen}. From \eqref{eq:wavepackets-degen-low} and \eqref{eq:wavepackets-degen-cf}, \eqref{eq:wavepackets-degen-small} also follows.

To complete the proof of the degeneration properties, it remains to establish \eqref{eq:wavepackets-degen-high}--\eqref{eq:wavepackets-degen-low}. For \eqref{eq:wavepackets-degen-high} in the case $m = 1$, we write $P_{y; k} = 2^{-k} \rd_{y} \tilde{P}_{y; k}$, where $\tilde{P}_{y; k}$ is a convolution operator in $y$ with an integrable kernel (the integral is bounded by an absolute constant), and estimate
\begin{align*}
\nrm{P_{y; k} \tb_{(\lmb)}(t, x, y)}_{L^{p}_{x, y}}
&\aleq 2^{-k} \nrm{\rd_{y} \tilde{b}_{(\lmb)}(t, x, y)}_{L^{p}_{x, y}} \\
&\aleq 2^{-k} \lmb \nrm{f^{\frac{1}{p}-\frac{3}{2}} G'(\eta) \tilde{h}_{(\lmb)}(\lmb t, x, \eta)}_{L^{p}_{x, \eta}} \\
&\phantom{\aleq}
+ 2^{-k} \nrm{f^{\frac{1}{p}-\frac{3}{2}} \tilde{h}_{(\lmb)}(\lmb t, x, \eta)}_{L^{p}_{x, \eta}}
+ 2^{-k} \nrm{f^{\frac{1}{p}-\frac{3}{2}} \rd_{\eta} \tilde{h}_{(\lmb)}(\lmb t, x, \eta)}_{L^{p}_{x, \eta}}.
\end{align*}
For the second inequality, we used \eqref{eq:wavepackets-degen-overall}, $\rd_{y} = f^{-1} \rd_{\eta}$ and $\ud y = f \ud \eta$. Using \eqref{eq:wavepackets-degen-amp}, \eqref{eq:wavepackets-degen-supp} and \eqref{eq:wavepackets-degen-f} (recall also that $G'(\eta) = \sqrt{1 - f^{2}(\eta)}$ and that $f < 1/2$ in $\set{\eta < 0}$), \eqref{eq:wavepackets-degen-high} in the case $m = 1$ follows. The cases $m \geq 2$ are treated similarly.

For \eqref{eq:wavepackets-degen-med}, we use the identity $e^{i \lmb G(\eta(y))} = i^{-1} \lmb^{-1} (G'(\eta(y))^{-1} f \rd_{y} e^{i \lmb G(\eta(y))}$ to rewrite \eqref{eq:wavepackets-degen-overall} as
\begin{align*}
\tilde{b}_{(\lmb)} (t, x, y) 
&= \Re\left( i^{-1} \lmb^{-1} (G'(\eta(y)))^{-1} \rd_{y} e^{i (\lmb^{2} t + \lmb x + \lmb G(\eta(y)))} f^{\frac{1}{2}}(y) \tilde{h}(\lmb t, x, \eta(y))\right) \\
&= \rd_{y} \Re \left(i^{-1} \lmb^{-1} (G'(\eta(y)))^{-1} e^{i (\lmb^{2} t + \lmb x + \lmb G(\eta(y)))} f^{\frac{1}{2}}(y) \tilde{h}(\lmb t, x, \eta(y)) \right)   \\
&\phantom{=}
- \Re\left( i^{-1} \lmb^{-1} e^{i (\lmb^{2} t + \lmb x + \lmb G(\eta(y)))} \rd_{y} \left((G'(\eta(y)))^{-1} f^{\frac{1}{2}}(y) \tilde{h}(\lmb t, x, \eta(y))\right) \right).
\end{align*}
Taking $P_{y; k}$ of both sides and considering their $L^{p}_{x, y}$-norms, we obtain
\begin{equation} \label{eq:wavepackets-degen-dbp}
\begin{aligned}
	\nrm{P_{y; k} \tilde{b}_{(\lmb)} (t, x, y)}_{L^{p}_{x, y}}
	& \aleq 2^{k} \lmb^{-1} \nrm{P_{y; k} (e^{i (\lmb^{2} t + \lmb x + \lmb G(\eta(y)))} f^{\frac{1}{2}}(y) \tilde{h}(\lmb t, x, \eta(y)))}_{L^{p}_{x, y}} \\
	&\phantom{\aleq} + \lmb^{-1} \nrm{P_{y; k} (e^{i (\lmb^{2} t + \lmb x + \lmb G(\eta(y)))} \rd_{y} (G'(\eta(y))^{-1} f^{\frac{1}{2}}(y) \tilde{h}(\lmb t, x, \eta(y))))}_{L^{p}_{x, y}}.
\end{aligned}
\end{equation}
To prove \eqref{eq:wavepackets-degen-med} in the case $m = 1$, we use $\rd_{y} = f^{-1} \rd_{\eta}$ and $\ud y = f \ud \eta$ to eestimate the RHS by
\begin{align*}
	2^{k} \lmb^{-1} \nrm{f^{\frac{1}{2}+\frac{1}{p}} \tilde{h}(\lmb t, x, \eta)}_{L^{p}_{x, \eta}} 
	+ \lmb^{-1} \nrm{f^{-\frac{1}{2}+\frac{1}{p}} \tilde{h}(\lmb t, x, \eta)}_{L^{p}_{x, \eta}}
	+ \lmb^{-1} \nrm{f^{-\frac{1}{2}+\frac{1}{p}} \rd_{\eta} \tilde{h}(\lmb t, x, \eta)}_{L^{p}_{x, \eta}},
\end{align*}
and then apply \eqref{eq:wavepackets-degen-amp}, \eqref{eq:wavepackets-degen-supp} and \eqref{eq:wavepackets-degen-f}. Note that since $2^{k} \geq e^{c_{f} \lmb t}$, the contribution of the first term dominates those of the other two. The cases $m \geq 2$ follows by repeating the above ``differentiation by parts'' procedure.

Finally, to prove \eqref{eq:wavepackets-degen-low}, we resume from \eqref{eq:wavepackets-degen-dbp}. Using Bernstein's inequality in $y$ (i.e., that $P_{y; k} : L^{p}_{y} \to 2^{(1-\frac{1}{p}) k} L^{1}_{y}$ is bounded), changing the variable $y$ to $\eta$ and then applying H\"older's inequality in $\eta$ (making use of \eqref{eq:wavepackets-degen-supp}), we estimate the RHS by
\begin{align*}
& 2^{(2-\frac{1}{p}) k} \lmb^{-1} \nrm{f^{\frac{3}{2}} \tilde{h}(\lmb t, x, \eta)}_{L^{p}_{x} L^{1}_{\eta}} 
	+ 2^{(1-\frac{1}{p}) k} \lmb^{-1} \nrm{f^{\frac{1}{2}} \tilde{h}(\lmb t, x, \eta)}_{L^{p}_{x} L^{1}_{\eta}}
	+ 2^{(1-\frac{1}{p}) k} \lmb^{-1} \nrm{f^{\frac{1}{2}} \rd_{\eta} \tilde{h}(\lmb t, x, \eta)}_{L^{p}_{x} L^{1}_{\eta}} \\
& \aleq
2^{(2-\frac{1}{p}) k} \lmb^{-1} \nrm{f^{\frac{3}{2}} \tilde{h}(\lmb t, x, \eta)}_{L^{p}_{x, \eta}} 
	+ 2^{(1-\frac{1}{p}) k} \lmb^{-1} \nrm{f^{\frac{1}{2}} \tilde{h}(\lmb t, x, \eta)}_{L^{p}_{x, \eta}}
	+ 2^{(1-\frac{1}{p}) k} \lmb^{-1} \nrm{f^{\frac{1}{2}} \rd_{\eta} \tilde{h}(\lmb t, x, \eta)}_{L^{p}_{x, \eta}}.
\end{align*}
Lastly, we apply \eqref{eq:wavepackets-degen-amp}, \eqref{eq:wavepackets-degen-supp} and \eqref{eq:wavepackets-degen-f}, which proves \eqref{eq:wavepackets-degen-low}.

To conclude the proof in case~(a), it only remains to establish the error bounds. By definition $\err_{\psi}[\tb_{(\lmb)}^z, \tpsi_{(\lmb)}] = 0$, and
\begin{equation*}
	\err_{b}[\tb_{(\lmb)}^z, \tpsi_{(\lmb)}] =  \lmb ^{2}\Re \, f^{-\frac{1}{2}}\errwp_{\varphi}[h_{0}^{(-1)}; \lmb],
\end{equation*}
so that
\begin{align*}
	\nrm{\err_{b}[\tb_{(\lmb)}^z, \tpsi_{(\lmb)}](t)}_{L_{x, y}^{2}}
	= \nrm{\lmb^{2} \errwp_{\varphi}[h_{0}^{(-1)}; \lmb](\lmb t)}_{L_{x, \eta}^{2}} \aleq \lmb \nrm{h_{0}^{(-1)}}_{H^{4}_{x, \eta}} \aleq \nrm{g_{0}^{(-1; \lmb)}}_{H^{4}_{x, y}},
\end{align*}
as desired. 

The proof in case~(b) is a minor modification of that in case~(a). Here, as $g_{0}$ is independent of $\tht$, there is no need for an auxiliary function $g_{0}^{(-1; \lmb)}$. We apply the WKB construction in Section~\ref{subsec:wkb} to
\begin{equation*}
	h_{0}(\eta) = f^{-\frac{1}{2}}(r(\eta)) g_{0}(r(\eta)),
\end{equation*}
and define
\begin{align}
	\tb^{z}_{(\lmb)} = & f^{-\frac{1}{2}} \lmb \Re\left(\frac{1}{i \lmb} e^{i (\lmb^{2} t + \lmb \tht + \lmb G(\eta(r)))} h(\lmb t, \eta(r))\right), \label{eq:wp-b} \\
	\tpsi_{(\lmb)} = & f^{-\frac{1}{2}} \lmb^{-1} \Re\left(e^{i (\lmb^{2} t + \lmb \tht + \lmb G(\eta(r)))} h(\lmb t, \eta(r))\right).  \label{eq:wp-psi}
\end{align}
Then the properties stated in Proposition~\ref{prop:wavepackets} are proved in the same manner as in case~(a). We omit the obvious details.
\qedhere
\end{proof}

Next, we turn to the proof of Proposition~\ref{prop:wavepackets-hall} for \eqref{eq:hall-mhd}. As we will see, a (technical) part of the proof is to ensure that $\nb (-\lap)^{-1}$ are well-defined in various contexts. In the case $\bgB = f(y) \rd_{x}$, we always prepare the the RHS to be of the form $\rd_{x} a$, so that we may rely on $L^{2}$-boundedness of the singular integral $\nb (-\lap)^{-1} \rd_{x}$ on $M^{2} = (\bbT, \bbR)_{x} \times (\bbT, \bbR)_{y}$. In the case $\bgB = f(r) \rd_{\tht}$ on $M^{2} = \bbR^{2}$, we use a similar trick with $\rd_{\tht}$ in place of $\rd_{x}$; indeed, by writing $\rd_{\tht} a = \rd_{x} (y a) - \rd_{y} (x a)$, we may handle $\nb (-\lap)^{-1} \rd_{\tht}$.

\begin{proof}[Proof of Proposition~\ref{prop:wavepackets-hall}]

We first handle the case~(a), i.e., $\bgB = f(y) \rd_{x}$.  As in the proof of Proposition~\ref{prop:wavepackets}, we construct $h^{(-1)}$ and $h$ from $g_{0}^{(-1; \lmb)}$ and $g_{0}$, respectively, and define $\tb^{z}_{(\lmb)}$, $\tpsi_{(\lmb)}$ by \eqref{eq:def-tb}, \eqref{eq:def-tpsi}, respectively. Moreover, we define $\tu^{z}_{(\lmb)}$ and $\tomg_{(\lmb)}$ from $\tb^{z}_{(\lmb)}$, $\tpsi_{(\lmb)}$ as in \eqref{eq:wavepackets-hall-u-omg}.
Then the estimates for $\tu^{z}_{(\lmb)}$, $\nb \tu^{z}_{(\lmb)}$ and $\tomg_{(\lmb)}$ claimed in Proposition~\ref{prop:wavepackets-hall} follow from \eqref{eq:wp-psi-est} and \eqref{eq:wp-b-est}. To handle $\nb^{\perp}(-\lap)^{-1} \tomg_{(\lmb)}$, we observe that, by $x$-invariance,
\begin{align*}
	\nb^{\perp}(-\lap)^{-1} \tomg_{(\lmb)}
	= & - \nb^{\perp} (-\lap)^{-1} \rd_{x} \rd_{x}^{-1}\tb^{z}_{(\lmb)} \\
	= &- \nb^{\perp} (-\lap)^{-1} \rd_{x} f^{-\frac{1}{2}} \lmb \left( e^{i (\lmb^{2} t + \lmb x + \lmb G(\eta(y))} (h^{(-2)} + \frac{1}{i \lmb} \rd_{\tau} h^{(-2)})(\lmb t, x, \eta(y)) \right)
\end{align*}
where $h^{(-2)}$ is constructed by the WKB analysis in Section~\ref{subsec:wkb} applied to 
\begin{equation*}
h_{0}^{(-2)} (x, \eta) = \frac{1}{(i \lmb)^{2}} f^{-\frac{1}{2}}(y(\eta)) g_{0}^{(-2;\lmb)} (x, y(\eta)),
\end{equation*}
for which we have, for any $0 \leq k \leq m$,
\begin{equation} \label{eq:wp-g0-2->h0-2}
\begin{aligned}
	\nrm{\rd_{x}^{k} \rd_{\eta}^{m-k} h_{0}^{(-2)}}_{L^{2}_{x, \eta}}
	&= \lmb^{-2}\nrm{f^{-\frac{1}{2}} \rd_{x}^{k} (f \rd_{y})^{m -k} f^{-\frac{1}{2}} g_{0}^{(-2; \lmb)}}_{L^{2}_{x, y}} \\
	&\aleq_{m} (f'(0) y_{1})^{-1} \lmb^{-2} \nrm{g_{0}^{(-2; \lmb)}}_{H^{m}_{x, y}}.
\end{aligned}\end{equation}
Using Lemma~\ref{lem:wkb-h-est} in Section~\ref{subsec:wkb}, we obtain the desired estimate for $\nb^{\perp}(-\lap)^{-1} \tomg_{(\lmb)}$.

It remains to verify the error estimates stated in Proposition~\ref{prop:wavepackets-hall}. Comparing  {\eqref{eq:hall-2.5d-err-parallel} and \eqref{eq:e-2.5d-err-parallel}}, observe that with our choice of $\tu^{z}_{(\lmb)}$ and $\tomg_{(\lmb)}$,
\begin{equation} \label{eq:wavepackets-hall-err}
\begin{aligned}
	\errh_{u}^{(\nu)} [\tu^{z}_{(\lmb)}, \tomg_{(\lmb)}, \tb^{z}_{(\lmb)}, \tpsi_{(\lmb)}] + \nu \lap \tpsi_{(\lmb)} 
	= & -\err_{\psi}[\tb^{z}_{(\lmb)}, \tpsi_{(\lmb)}], \\
	\errh_{\omg}^{(\nu)} [\tu^{z}_{(\lmb)}, \tomg_{(\lmb)}, \tb^{z}_{(\lmb)}, \tpsi_{(\lmb)}] + \nu \lap \tb^{z}_{(\lmb)}
	= & -\err_{b}[\tb^{z}_{(\lmb)}, \tpsi_{(\lmb)}], \\
	\errh_{b}^{(\nu)} [\tu^{z}_{(\lmb)}, \tomg_{(\lmb)}, \tb^{z}_{(\lmb)}, \tpsi_{(\lmb)}] 
	= & \err_{b}[\tb^{z}_{(\lmb)}, \tpsi_{(\lmb)}] + f \rd_{x} \tpsi_{(\lmb)}, \\
	\errh_{\psi}^{(\nu)} [\tu^{z}_{(\lmb)}, \tomg_{(\lmb)}, \tb^{z}_{(\lmb)}, \tpsi_{(\lmb)}]
	= & \err_{\psi}[\tb^{z}_{(\lmb)}, \tpsi_{(\lmb)}] + f \rd_{x} (-\lap)^{-1} \tb^{z}_{(\lmb)}.
\end{aligned}
\end{equation}
The estimates for $\errh_{u}^{(\nu)}$, $\errh_{b}^{(\nu)}$ and $\nb \errh_{\psi}^{(\nu)}$ follow from the error estimates in Proposition~\ref{prop:wavepackets} and the preceding bound for $\nb^{\perp}(-\lap)^{-1} \tb^{z}_{(\lmb)} = - \nb^{\perp}(-\lap)^{-1} \tomg_{(\lmb)}$ (for the last term). For $\nb^{\perp} (-\lap)^{-1} \errh_{\omg}^{(\nu)}$, we need to estimate $\nb^{\perp} (-\lap)^{-1} \err_{b}$. Again by $x$-invariance, note that
\begin{align*}
	\nb^{\perp} (-\lap)^{-1} \err_{b}[\tb_{(\lmb)}, \tpsi_{(\lmb)}]
	=& \lmb^{2} \nb^{\perp} (-\lap)^{-1} \Re f^{-\frac{1}{2}} \errwp_{\varphi}[h_{0}^{(-1)}; \lmb] \\
	=& \lmb^{2} \nb^{\perp} (-\lap)^{-1} \rd_{x} \Re f^{-\frac{1}{2}} \errwp_{\varphi}[h_{0}^{(-2)}; \lmb].
\end{align*}
Then the desired estimate for the $L^{2}$ norm of the last term follows from $L^{2}$-boundedness of $\nb^{\perp} (-\lap)^{-1} \rd_{x}$, \eqref{eq:wkb-err-est} and \eqref{eq:wp-g0-2->h0-2}.

The proof in case~(b) (i.e., $\bgB = f(r) \rd_{\tht}$) is similar, so we only sketch the necessary modifications. We define $(\tb^{z}_{(\lmb)}, \tpsi_{(\lmb)})$ by \eqref{eq:wp-b} and \eqref{eq:wp-psi} as in the proof of Proposition~\ref{prop:wavepackets}, and $(\tu^{z}_{(\lmb)}, \tomg_{(\lmb)})$ as in \eqref{eq:wavepackets-hall-u-omg}. Again, the estimates for $\tu^{z}_{(\lmb)}$, $\nb \tu^{z}_{(\lmb)}$ and $\tomg_{(\lmb)}$ claimed in Proposition~\ref{prop:wavepackets-hall} follow from the preceding proof. To handle $\nb^{\perp} (-\lap)^{-1} \tomg_{(\lmb)}$, we simply write $$\tomg_{(\lmb)} = - \tb^{z}_{(\lmb)} = \rd_{\tht} \left( (i \lmb)^{-1} \tb^{z}_{(\lmb)} \right),$$ and observe that since $\tb^{z}_{(\lmb)}(t)$ is always supported in $\set{r < r_{1}}$, for any $t \geq 0$ we have
\begin{align*}
	\nrm{\nb^{\perp} (-\lap)^{-1} \rd_{\tht} (i\lmb)^{-1} \tb^{z}_{(\lmb)}(t)}_{L^{2}}
	= & \lmb^{-1} \nrm{\nb^{\perp} (-\lap)^{-1} (\rd_{y} (x \tb^{z}_{(\lmb)}) - \rd_{x} (y \tb^{z}_{(\lmb)}))(t)}_{L^{2}} \\
	\aleq & r_{1} \lmb^{-1} \nrm{\tb^{z}_{(\lmb)}(t)}_{L^{2}},
\end{align*}
as desired. Finally, the error estimates follow from  {\eqref{eq:hall-2.5d-err-axi}, \eqref{eq:e-2.5d-err-axi}}, the error estimates in Proposition~\ref{prop:wavepackets}, the preceding bound for $\nb^{\perp}(-\lap)^{-1} \tb^{z}_{(\lmb)}$, and
\begin{equation*}
	\nrm{\nb^{\perp}(-\lap)^{-1} \err_{b}[\tb^{z}_{(\lmb)}, \tpsi_{(\lmb)}]}_{L^{2}} \aleq r_{1} \lmb^{-1} \nrm{\err_{b}[\tb^{z}_{(\lmb)}, \tpsi_{(\lmb)}]}_{L^{2}}, 
\end{equation*}
which is proved again using the trick of pulling out $(i \lmb)^{-1}\rd_{\tht}$ from $\err_{b}[\tb^{z}_{(\lmb)}, \tpsi_{(\lmb)}]$. 
\end{proof}

\begin{remark} \label{rem:wavepackets-hall-err}
Key to the proof was the remarkable simplicity of the error terms under the choice \eqref{eq:wavepackets-hall-u-omg}, for which the fluid variables $\tu^{z}_{(\lmb)}$ and $\tomg_{(\lmb)}$ are also one order smoother than energy. The origin of such a nice structure may be traced back to the existence of a set of ``good variables'' for \eqref{eq:hall-mhd} with $\nu = 0$: Introducing the vector field
\begin{equation*}
\bfZ := \bfB + \bfomg,
\end{equation*} 
\eqref{eq:hall-mhd} with $\nu = 0$ has the following reformulation in terms of $(\bfZ, \bfB)$: \begin{equation}\label{eq:hall-mhd-ZB}
\left\{
\begin{aligned}
&\rd_{t} \bfZ + \bfu\cdot\nabla\bfZ -  \bfZ\cdot\nabla\bfu = 0 ,  \\
&\rd_{t} \bfB + \bfu\cdot\nabla\bfB -  \bfB\cdot\nabla\bfu  + \nb \times ((\nb \times \bfB) \times \bfB)= 0, \\
&\nb \cdot \bfZ = \nb \cdot \bfB = 0, \\
& \nabla \times \bfu = \bfZ - \bfB, \quad  \nabla\cdot\bfu= 0. 
\end{aligned}
\right.
\end{equation} 
For more details on this reformulation we refer to \cite{JO2}, where it plays a central role. This reformulation have already appeared in the work of Chae and Wolf in \cite{CWo1} for the purpose of obtaining partial regularity results for the $2+\frac{1}{2}$ dimensional Hall-MHD system. 

In terms of these variables, our approximate solution for \eqref{eq:hall-mhd} with $\nu = 0$ corresponds to taking the $\bfZ$-perturbation zero, and the $\bfB$-perturbation identical to the \eqref{eq:e-mhd} case. The last div-curl identities for $\bfu$ explains why this choice results in the crucial smoothing of $\tu_{(\lmb)}$ by one order compared to $\tb_{(\lmb)}$. 

\end{remark}

\section{Proof of the linear illposedness results in Sobolev spaces} \label{sec:pf-lin}
In this section, we prove Theorems~\ref{thm:norm-growth} and \ref{thm:inst}.
\subsection{Proof of Theorems~\ref{thm:norm-growth} and \ref{thm:inst} for \eqref{eq:e-mhd}}
In order to apply Proposition~\ref{prop:wavepackets}, we  {begin} by constructing a family of bump functions $p_{0, \lmb}$ on $(\bbT, \bbR)_{x}$ for which we have a uniform control of $p_{0, \lmb}^{(-j; \lmb)}$:
\begin{lemma} \label{lem:bump}
For each $\lmb \in \bbN$ and $n \in \bbN$, there exist nonzero $p_{0, \lmb} \in \calS((\bbT, \bbR)_{x})$ such that each $p_{0, \lmb}^{(-j; \lmb)}$ for $1 \leq j \leq n$ is well-defined and belongs to $\calS((\bbT, \bbR)_{x})$, obeys
\begin{equation} \label{eq:bump-est}
	\nrm{(p_{0, \lmb}, \ldots, p_{0, \lmb}^{(-n; \lmb)})}_{H^{m}_{x}} \aleq_{m, n} \nrm{p_{0, \lmb}}_{L^{2}} \quad \hbox{ for every } m \in \bbN_{0},
\end{equation}
and has one of the following properties:
\begin{itemize}
\item $p_{0, \lmb}, \ldots, p_{0, \lmb}^{(-n; \lmb)}$ are supported in $(-1, 1)$; or
\item $\calF[p_{0, \lmb}], \ldots, \calF[p_{0, \lmb}^{(-n; \lmb)}]$ are supported in $(-1, 1)$.
\end{itemize}
\end{lemma}
We emphasize that the implicit constant is independent of $\lmb$.
\begin{proof}
In $\bbT_{x}$, the simple choice $p_{0, \lmb} = 1$ does the job. 

In $\bbR_{x}$, the second case is easily handled by making a $\lmb$-independent choice $p_{0, \lmb} = p_{0}$, where $p_{0} \neq 0$ and $\supp \calF [p_{0}] \subseteq (-1, 1)$. Indeed, since $\calF [e^{i \lmb x} p_{0}]$ is supported away from $0$, \eqref{eq:bump-est} follows from the formula $$\calF[p_{0}^{(-n; \lmb)}] (\xi) = \left( \frac{\lmb}{\xi+\lmb} \right)^n\calF[p_{0}](\xi).$$

Thus, the only remaining case is the first case in $\bbR_{x}$. We start with a nonnegative function $p_{0} \in C^{\infty}_{c}(-1, 1)$ with $\int p_{0} \, \ud x= 1$. We would like to construct $p_{0, \lmb}$ as a small perturbation of $p_{0}$; however, to make each $p^{(-j; \lmb)}_{0, \lmb}$ is supported in $(-1, 1)$, we need to ensure that $\int x^{k} p_{0, \lmb} \, \ud x = 0$ for $k = 0, \ldots, n-1$. For this purpose, we introduce auxiliary functions $q_{k}$ for $k=0, \ldots, n-1$ that are defined as follows:
 \begin{equation*}
	q_{k}(x) = \frac{2^{k+1}}{k!} (\rd_{x}^{k} p_{0})(2 x ).
\end{equation*}
Then $q_{0} \neq p_{0}$ (by the support property), $\int q_{0} = 1$, $\supp q_{k} \subseteq (-1, 1)$ and
\begin{equation*}
	\int x^{k} q_{k} = 1, \quad \int x^{j} q_{k} = 0 \quad \hbox{ for any } 1 \leq k \leq n, \ 0 \leq j \leq k-1.
\end{equation*}
In other words, the matrix $A_{jk} = \int x^{j} q_{k} \, \ud x$ is upper triangular with diagonal entries all equal to $1$; in particular, $A$ is invertible with $\nrm{A^{-1}} \aleq_{p_{0}} 1$. Now, for any $\lmb \in \bbN$, we define
\begin{equation*}
	p_{0, \lmb} =  {p_{0}(x)} - \sum_{j=0}^{n-1} \alp_{j}(\lmb) q_{j}(x),
\end{equation*}
where $\alp_{j}(\lmb) \in \bbR$'s are chosen so that $\int x^{k} p_{0, \lmb} = 0$ for $k = 0, \ldots, n-1$. Such a choice exists by the invertibility of $A$, and we have the estimate
\begin{equation*}
	\sup_{0 \leq j \leq n-1} \abs{\alp_{j}(\lmb)} \aleq_{p_{0}} \sup_{0 \leq j \leq n-1} \Abs{\int x^{j} e^{i \lmb x} p_{0}(x) \, \ud x}.
\end{equation*}
Finally, by repeated integration by parts, observe that the RHS is bounded by $C_{N, p_{0}} \lmb^{-N}$ for any $N > 0$. From this property, the desired uniform-in-$\lmb$ estimate \eqref{eq:bump-est} follows. \qedhere
\end{proof}

We now complete the proof of Theorem \ref{thm:norm-growth}. 

\begin{proof}[Proof of Theorem \ref{thm:norm-growth}]
	We consider only the translationally-symmetric case, as the proof in the axi-symmetric case requires only minor modifications. The proof is a straightforward application of Proposition \ref{prop:wavepackets}. We divide the argument into three simple steps. 
	
	\medskip
	
	\textit{(i) choice of initial data}
	
	\medskip 
	
	We start with an initial amplitude with single frequency and normalized energy
	\begin{equation*}
	g_{0, \lmb}(x, y) = p_{0, \lmb}(x) q_{0}(y),
	\end{equation*}
	where $p_{0}$ is given by Lemma~\ref{lem:bump} and $q_{0}$ is a fixed smooth function supported in $(\frac{1}{2} y_{1}, y_{1})$. Then we apply Proposition~\ref{prop:wavepackets} to construct the initial data $$\tb_{(\lmb)}(0)  = ( \rd_y\tpsi_{(\lmb)}(0),-\rd_x\tpsi_{(\lmb)}(0) , \tb_{(\lmb)}^z(0)) $$ for \eqref{eq:e-mhd-2.5d-lin}, and normalize its $L^2$ norm by 1. The lower bound stated in Proposition \ref{prop:wavepackets} guarantees that (by taking $\lmb \ge 1$ large if necessary) $\nrm{g_{0,\lmb}}_{L^2}  {\aeq} 1$ uniformly in $\lmb \gg 1$. We denote by $\tb_{(\lmb)} = ( \rd_y\tpsi_{(\lmb)} ,-\rd_x\tpsi_{(\lmb)} , \tb_{(\lmb)}^z) $ the corresponding degenerating wave packet solution, and $b_{(\lmb)}$ be an $L^{2}$-solution with the same initial data. 
	
	\medskip
	
	\textit{(ii) application of the generalized energy identity}
	
	\medskip 
	
	Notice that since $\tilde{b}_{(\lmb)}$ is smooth, Proposition \ref{prop:justify-eq} is applicable. Then using \eqref{eq:en-e-mhd-parallel}, we obtain 
	\begin{equation*}
	\begin{split}
	&\brk{\tilde{b}_{(\lmb)}, b_{(\lmb)}}(t) - \brk{\tilde{b}_{(\lmb)}, b_{(\lmb)}}(0) \\
	&\qquad =  \brk{\tilde{b}_{(\lmb)}, b_{(\lmb)}}(t) - 1 \\ 
	&\qquad = \int_0^t \int -f''(\rd_x\tilde{\psi}_{(\lmb)} b_{(\lmb)}^z + \rd_x\psi_{(\lmb)} \tb^z_{(\lmb)}) + \nabla^\perp \err_{\tpsi} \cdot \nabla^\perp\psi_{(\lmb)} + \err_{\tb}b^z_{(\lmb)} \, \ud x \ud y  \ud s  
	\end{split}
	\end{equation*} and then applying the error bounds from Proposition \ref{prop:wavepackets} gives that  \begin{equation*}
	\begin{split}
	\left|\brk{\tilde{b}_{(\lmb)}, b_{(\lmb)}}(t) - \brk{\tilde{b}_{(\lmb)}, b_{(\lmb)}}(0)\right| & \lesssim  t \nrm{b_{(\lmb)}}_{L^\infty(I;L^2)} \left(  \nrm{\tb_{(\lmb)}}_{L^\infty(I;L^2)} + \nrm{ \nabla^\perp \err_{\tpsi} }_{L^\infty(I;L^2)} + \nrm{ \err_{\tb}}_{L^\infty(I;L^2)} \right) \\
	&\lesssim t \nrm{b_{(\lmb)}}_{L^\infty(I;L^2)}
	\end{split}
	\end{equation*} where the multiplicative constants depend on $f$ but not on $\lmb$. Thus, choosing $0 < T = T(\frac{\nrm{b}_{L^\infty(I;L^2)}}{\nrm{b_0}_{L^2}})$ sufficiently small (independent of $\lmb$),
	\begin{equation}\label{eq:lb}
	\brk{\tilde{b}_{(\lmb)}, b_{(\lmb)}}(t) >  \frac{1}{2} \quad \hbox{ for } t \in [0, T].
	\end{equation} 
		
	\medskip
	
	\textit{(iii) growth of Sobolev norms}
	
	\medskip 

	Let $p' \in [1, \infty]$ and $s \in \bbR$ such that $s + \frac{1}{p'} - \frac{1}{2} \geq 0$. By our normalization of $\tb_{(\lmb)}(0)$ and \eqref{eq:wavepackets-degen}, we have the decomposition $\tb_{(\lmb)} = \tb_{(\lmb)}^{main} + \tb_{(\lmb)}^{small}$, where
\begin{equation*}
	\begin{split}
	\nrm{\tb_{(\lmb)}^{main}(t)}_{L^{p'}_{x} W^{-s, p'}_{y}} 
	\aleq_{s} \lmb^{-s} e^{-c_f(s + \frac{1}{p'}-\frac{1}{2})\lmb t}, \quad
	\nrm{\tb_{(\lmb)}^{small}(t)}_{L^{2}_{x, y}} 
	\aleq \lmb^{-1},
	\end{split}
	\end{equation*} which holds uniformly for $t \in [0,T]$. Let $p$ be the Lebesgue dual of $p$ (i.e., $\frac{1}{p} = 1 - \frac{1}{p'}$). By \eqref{eq:lb}, duality, \eqref{eq:wavepackets-degen} and \eqref{eq:wavepackets-degen-small},
	\begin{align*}
	\frac{1}{2}
	&< \brk{\tb_{(\lmb)}, b_{(\lmb)}}(t)
	\leq \nrm{\tb_{(\lmb)}^{main}}_{L^{p'}_{x} W^{-s, p'}_{y}} \nrm{b_{(\lmb)}}_{L^{p}_{x} W^{s, p}_{y}}
	+ \nrm{\tb_{(\lmb)}^{small}}_{L^{2}_{x, y}} \nrm{b_{(\lmb)}}_{L^{2}_{x, y}} \\
	& \leq C_{s} \lmb^{-s} e^{-c_f(s + \frac{1}{p'}-\frac{1}{2})\lmb t} \nrm{b_{(\lmb)}}_{L^{p}_{x} W^{s, p}_{y}} + C \lmb^{-1}.
	\end{align*} 
	By requiring $\lmb$ to be sufficiently large, we may absorb the last term into the LHS. 
	Due to our construction and normalization, it is easy to see that $\nrm{b_{(\lmb)}(0)}_{W^{s, p}_{x, y}} \aeq \lmb^{s}$; hence we have proved the desired norm growth.  \qedhere
	\end{proof}

We now prove Theorem \ref{thm:inst}. Recall that for the purpose of stating this result, we have assumed that the solution map is uniquely well-defined for $L^2$ initial data. 

\begin{proof}[Proof of Theorem \ref{thm:inst}]
	Again, we only consider the translationally-symmetric case; in the axi-symmetric case, $\theta$ plays the role of $x$ and the arguments are somewhat simpler thanks to the periodicity in $\theta$. Note that it is sufficient to consider an arbitrarily small $s' > 0$; in particular, we may assume that $s' < \frac{1}{2}$.
		
 	We construct initial data $\tb_{(\lmb)}(0)$ for all $\lmb \ge \lmb_{0}$, where $\lmb_{0} \in \bbN$ is sufficiently large with respect to $\bgB$,  as in the proof of Theorem \ref{thm:norm-growth}. Using the $L^{2}$-solution map, the solution $b_{(\lmb)}$ with initial data $b_{(\lmb)}(0) = \tilde{b}_{(\lmb)}(0)$ is well-defined on the time interval $[0,1]$. Each $b_{(\lmb)}(0)$ is normalized to be 1 in $L^2$. The idea is to take the series $$b =\sum_{\lmb } \alp_{\lmb} b_{(\lmb)}$$ with an appropriate choices of $\set{\lmb} \subseteq 2^{\bbN_{0}}$ and $\alp_{\lmb} > 0$. Note that $x$-translation is preserved by uniqueness. By linearity and boundedness, $\rd_{x}$'s are propagated. Furthermore, again by linearity and boundedness, any $x$-frequency support properties are preserved. 
 	
 	When $b_{(\lmb)}$ is chosen so that its $x$-frequency support lies in the region $\lmb + O(1)$ (that is, the first statement of Theorem~\ref{thm:lin-stab} and the second case of Lemma~\ref{lem:bump}), then we simply choose $\alp_{\lmb}$ to be any super-polynomially decaying sequence such that $e^{c \lmb} {\alp_{\lmb}} \to \infty$ for any $c > 0$, and arrange $\lmb$'s so that
 	\begin{equation*}
 	\brk{b_{(\lmb')}(t), \tilde{b}_{(\lmb)}(t)} = 0 \hbox{ if } \lmb' \neq \lmb.
 	\end{equation*} This choice of coefficients $\alpha_\lmb$ guarantees that the initial data is $C^\infty$-smooth. On the other hand, consider $0 < t < \dlt$, where $\dlt > 0$ is sufficiently small so that $\brk{b_{(\lmb)}(t), \tb_{(\lmb)}(t)} > \frac{1}{2}$ (see the proof of Theorem~\ref{thm:norm-growth}). By the orthogonality condition,
	\begin{align*}
	\frac{1}{2} \alp_{\lmb} < \brk{b_{(\lmb)}(t), \tb_{(\lmb)}(t)} = \brk{b(t), \tb_{(\lmb)} (t)} \aleq \nrm{b(t)}_{H^{s'}} \nrm{\tb_{(\lmb)}(t)}_{H^{-s'}}.
\end{align*}
Since $0 < s' < \frac{1}{2}$, by \eqref{eq:wavepackets-degen}, \eqref{eq:wavepackets-degen-small} and our normalization, we have $\nrm{\tb_{(\lmb)}}_{H^{-s'}} \aleq_{s'} \lmb^{-s'} e^{- c_{0}(\bgB) s' \lmb t}$; thus $\nrm{b(t)}_{H^{s'}} \ageq \alp_{\lmb} \lmb^{s'} e^{c_{0}(\bgB) s' \lmb t}$, where the implicit constant is independent of $\lmb$ and $t$. Taking $\lmb \to \infty$, we see that $b(t) \not \in H^{s'}$ for any $0 < t < \dlt$, as desired.

 	
 	 Next, we consider the second statement in Theorem~\ref{thm:lin-stab}, i.e., when we would like $\supp b_{(\lmb)} \subseteq (-1, 1) \times (\bbT, \bbR)_{y} \times \bbT_{z}$. We construct $p_{0, \lmb}(x)$ by Lemma~\ref{lem:bump} with $n > s+1$, and  choose $\alp_{\lmb} = \lmb^{-s}$. For $\lmb' > \lmb$, we simply have
 	\begin{equation*}
 	\sum_{\lmb' > \lmb} \brk{\alpha_{\lmb'} b_{(\lmb')}, \tilde{b}_{(\lmb)}} (t) \aleq \sum_{\lmb' > \lmb} \lmb'^{-s} \aleq \lmb^{-s}.
 	\end{equation*} On the other hand, to treat $\lmb' < \lmb$, we require $\tilde{b}_{(\lmb)}$ to obey
 	\begin{equation*}
 		\nrm{(\lmb^{-1} \rd_{x})^{-k} \tilde{b}_{(\lmb)}}_{L^\infty([0,1];L^{2})} \aleq_{k} 1
 	\end{equation*}
 	for all $0 \le k \le n$. To prove the above estimate, use the $x$-invariance to estimate the LHS in terms of $g_{0, \lmb}^{(-n-1; \lmb)}(x, y) = p_{0, \lmb}^{(-n-1; \lmb)}(x) q(y)$ and then use \eqref{eq:bump-est}. Thus
 	\begin{equation*}
 	\sum_{\lmb' < \lmb} \brk{\alpha_{\lmb'} b_{(\lmb')}, \tilde{b}_{(\lmb)}} (t) 
 	= \sum_{\lmb' < \lmb} \brk{\alpha_{\lmb'} (\lmb^{-1} \rd_{x})^{n} b_{(\lmb')}, (\lmb^{-1} \rd_{x})^{-n} \tilde{b}_{(\lmb)}} (t) 
 	\aleq \sum_{\lmb' < \lmb} \left( \frac{\lmb'}{\lmb} \right)^{n} \lmb'^{-s}
 	\aleq \lmb^{-s}
 	\end{equation*}
 	where the last inequality holds since $n > s + 1$. Choosing $\lmb$'s to be sufficiently separated, we may ensure that the last implicit constant in each estimate is small. Hence for $0 < t < \dlt$, where $\dlt > 0$ is sufficiently small so that $\brk{b_{(\lmb)}(t), \tb_{(\lmb)}(t)} > \frac{1}{2}$ (see the proof of Theorem~\ref{thm:norm-growth}), we obtain
\begin{align*}
	\lmb^{-s}
	\aleq \brk{b(t), \tb_{(\lmb)}(t)} 
	\aleq \nrm{b(t)}_{H^{s'}} \nrm{\tb_{(\lmb)}(t)}_{H^{-s'}}.
\end{align*}
As before, we have $\nrm{\tb_{(\lmb)}}_{H^{-s'}} \aleq_{s'} \lmb^{-s'} e^{- c_{0}(\bgB) s' \lmb t}$. Thus $\nrm{b(t)}_{H^{s'}} \ageq \lmb^{s'-s} e^{c_{0}(\bgB) s' \lmb t}$, where the implicit constant is independent of $\lmb$, $t$. Taking $\lmb \to \infty$, we see that $b(t) \not \in H^{s'}$ for any $0 < t < \dlt$.  

The fact that $b(t)$ is not contained even in the local Sobolev space $H^{s'}_{loc}$ follows from the preceding duality arguments, as the approximate solution $\tb$ is compactly supported in $y$ and either compactly supported or decaying sufficiently fast in $x$. This finishes the proof. \end{proof}

\subsection{Proof of Theorems~\ref{thm:norm-growth} and \ref{thm:inst} for \eqref{eq:hall-mhd}}
The proof is analogous with the case of  {e}lectron-MHD. However, a slight twist from the \eqref{eq:e-mhd} case is to choose
\begin{equation*}
	u_{0(\lmb)} = 0.
\end{equation*}
The idea is that it differs from the initial data in Proposition~\ref{prop:wavepackets-hall} only by $O(\lmb^{-1})$, so it does not matter. When $\nu > 0$, we need to use the dissipation term to control some errors; however, the same scheme works. 

\begin{proof}[Proof of Theorem \ref{thm:norm-growth} for \eqref{eq:hall-mhd}]
	Again, we focus only on the translationally symmetric case. 
	
	\medskip
	
	\textit{(i) choice of initial data}
	
	\medskip 
	
	We take the same function $g_{0,\lmb}$ as in the above proof, and apply Proposition \ref{prop:wavepackets-hall} to construct the degenerating wave packets $\tb_{(\lmb)}, \tu_{(\lmb)}$ associated with $g_{0,\lmb}$. We normalize  $g_{0,\lmb}$ so that the $L^2$ norm of $\tb_{(\lmb)}(0)$ becomes 1 (we still have $\nrm{g_{0,\lmb}}_{L^2}  {\aeq} 1$ uniformly in $\lmb \gg 1$). Now let $(u_{(\lmb)}, b_{(\lmb)})$ be a solution with the initial data $(0, \tb_{(\lmb)}(0))$. 
	
	\medskip
	
	\textit{(ii) application of the generalized energy identity}
	
	\medskip 
	
	Notice that since the functions $\tu_{(\lmb)}, \tilde{b}_{(\lmb)}$ are smooth, Proposition \ref{prop:justify-eq} is applicable. Then using \eqref{eq:en-hall-mhd-parallel}, we obtain 
	\begin{equation*}
	\begin{split}
	&\brk{\tilde{b}_{(\lmb)}, b_{(\lmb)}}(t) - \brk{\tilde{b}_{(\lmb)}, b_{(\lmb)}}(0) + \brk{\tilde{u}_{(\lmb)}, u_{(\lmb)}}(t)  - \brk{\tilde{u}_{(\lmb)}, u_{(\lmb)}}(0) + 2\nu\int_0^t \brk{\nabla\tu_{(\lmb)} , \nabla u_{(\lmb)}} ds  \\
	&\qquad =  \brk{\tilde{b}_{(\lmb)}, b_{(\lmb)}}(t) - 1 + \brk{\tilde{u}_{(\lmb)}, u_{(\lmb)}}(t)  + 2\nu\int_0^t \brk{\nabla\tu_{(\lmb)} , \nabla u_{(\lmb)}} ds   \\ 
	&\qquad = \int_0^t \int -f''(\rd_x\tilde{\psi}_{(\lmb)} b_{(\lmb)}^z + \rd_x\psi_{(\lmb)} \tb^z_{(\lmb)}) -f'( \nabla\tpsi_{(\lmb)} \cdot\nabla^\perp(-\Delta)^{-1}\omega_{(\lmb)} + \nabla\psi_{(\lmb)}\cdot\nabla^\perp(-\Delta)^{-1}\tomg_{(\lmb)}  )   \\
	&\qquad\qquad + \nabla^\perp \errh_{\tpsi} \cdot \nabla^\perp\psi_{(\lmb)} + \errh_{\tb}b^z_{(\lmb)}   + \nabla^\perp(-\Delta)^{-1}\errh_{\tomg}^{(\nu)} \cdot\nabla^\perp(-\Delta)^{-1}\omega + \errh_{\tu}^{(\nu)}u^z_{(\lmb)} \, \ud x \ud y  \ud s  
	\end{split}
	\end{equation*} and then after a bit of rearranging, \begin{equation*}
	\begin{split}
	&\left|\brk{\tilde{b}_{(\lmb)}, b_{(\lmb)}}(t) - 1 \right|  \lesssim \nrm{\tu_{(\lmb)}(t)}_{L^2}\nrm{u_{(\lmb)}(t)}_{L^2} + \nu t^{1/2} \nrm{\nabla\tu_{(\lmb)}}_{L^\infty(I;L^2)}\nrm{u_{(\lmb)}}_{L^2(I;\dot{H}^1)} \\
	&\qquad + t\nrm{b_{(\lmb)}}_{L^\infty(I;L^2)} \left( \nrm{\tb_{(\lmb)}}_{L^\infty(I;L^2)} + \nrm{\nabla^\perp\errh_{\tpsi}}_{L^\infty(I;L^2)} + \nrm{\errh_{\tb}}_{L^\infty(I;L^2)}  \right) \\
	&\qquad + t\nrm{u_{(\lmb)}}_{L^\infty(I;L^2)} \left( \nrm{\tb_{(\lmb)}}_{L^\infty(I;L^2)} + \nrm{\nabla^\perp(-\Delta)^{-1}(\errh_{\tomg}^{(\nu)}+\nu\Delta\tomg_{(\lmb)})}_{L^\infty(I;L^2)} + \nrm{\errh_{\tu}^{(\nu)}+\nu\Delta\tpsi_{(\lmb)}}_{L^\infty(I;L^2)}  \right) \\
	&\qquad +\nu t^{1/2}\nrm{u_{(\lmb)}}_{L^2(I;\dot{H}^1)} \left(  \nrm{\tomg_{(\lmb)}}_{L^\infty(I;L^2)} + \nrm{\nabla\tpsi_{(\lmb)}}_{L^\infty(I;L^2)} 
	 \right) 
	\end{split}
	\end{equation*}
	and applying the error bounds together with the smoothing estimates from Proposition \ref{prop:wavepackets}, for $\lmb \ge 1, 0 < t \le 1$ we obtain \begin{equation*}
	\begin{split}
	&\left|\brk{\tilde{b}_{(\lmb)}, b_{(\lmb)}}(t) - 1 \right|  \lesssim ((1+\nu) t^{1/2} +\lmb^{-1}) \left(   \nrm{b_{(\lmb)}}_{L^\infty(I;L^2)} +\nrm{u_{(\lmb)}}_{L^\infty(I;L^2)} + \nrm{u_{(\lmb)}}_{L^2(I;\dot{H}^1)} \right) 
	\end{split}
	\end{equation*}
	 where the multiplicative constants depend on $f$ but not on $\lmb$. Thus, choosing $0 < T \le 1$ sufficiently small (independent of $\lmb$ and depending on $\nu$ only when $\nu \gg 1$), we obtain for all sufficiently large $\lmb$ that 
	\begin{equation}\label{eq:lb2}
	\brk{\tilde{b}_{(\lmb)}, b_{(\lmb)}}(t) > \frac{1}{2} \quad \hbox{ for } t \in [0, T].
	\end{equation}
	
	\medskip
	
	\textit{(iii) growth of Sobolev norms}
	
	\medskip 
	
	With \eqref{eq:lb2} in hand, the proof of the norm growth estimate proceeds exactly as in the electron-MHD case, via duality and the degeneration estimates. We omit the details. \qedhere
\end{proof}

\begin{proof}[Proof of Theorem \ref{thm:inst} for \eqref{eq:hall-mhd}]
	We simply repeat the proof of Theorem \ref{thm:inst} for the electron-MHD case, using instead the lower bound \eqref{eq:lb2}: take data and solution of the form \begin{equation*}
	\begin{split}
	b = \sum_{\lmb } \alpha_\lmb b_{(\lmb)} , \qquad \tb = \sum_{\lmb } \alpha_\lmb \tb_{(\lmb)}
	\end{split}
	\end{equation*} with appropriately chosen $\alpha_\lmb$ as in the above proof.
\end{proof}

\section{Proof of the nonlinear illposedness results in Sobolev spaces} \label{sec:pf-nonlin}

We are in a position to complete the proof  of Theorem  \ref{thm:illposed-strong}. We emphasize in advance that in the proof below, all the implicit constants are independent  {of} $\lmb$  {as well as the adequate norm of the solution map (which will be finite by a contradiction assumption).} 

\subsection{Proof of Theorem \ref{thm:illposed-strong} for \eqref{eq:e-mhd}}
 
We consider only the case when $\bgB = f(y)\rd_x$; the proof in the axi-symmetric case requires only minor modifications. The proof is by contradiction. That is,  we further assume from now on that for $s_0 \ge 3$, the solution map is bounded, and for $s_0 > \max\{2, 3(1-\alpha)\}$, the solution map is $\alpha$-H\"older continuous.

\medskip

\textit{(i) choice of initial data}

\medskip

We fix a complex-valued Schwartz function $g_{0}(x, y) \in \calS(M^{2})$ with 
\begin{equation*}
\supp g_{0} \subseteq (\bbT, \bbR)_{x} \times (\tfrac{1}{2} y_{1}, y_{1}).
\end{equation*} We may take $g_0$ to be compactly supported in $x$ as well. Then, for $\lmb \in \bbN$, we  {choose the} initial data explicitly  {as} \begin{equation}\label{eq:illposed-strong-id}
\begin{split}
\bfB_{(\lmb)}(0) = \bgB + \eps \lmb^{-s-n} \tilde{b}_{(\lmb)}(0)
\end{split}
\end{equation} where  $\eps > 0, s > 0$ and $f$ are from the statement of the theorem, $n \ge 0$ is a parameter that will be chosen to be depending on $s, \alpha$ below, and \begin{equation*}
\begin{split}
\tilde{b}_{(\lmb)}(0) = (-\rd_y \tilde{\psi}_{(\lmb)}(0) , \rd_x\tilde{\psi}_{(\lmb)}(0) ,\tilde{b}_{(\lmb)}^z(0) ), 
\end{split}
\end{equation*} where the pair $(\tilde{\psi}_{(\lmb)}(0), \tilde{b}_{(\lmb)}^z(0))$ is explicitly given in \eqref{eq:initial-psi}, \eqref{eq:initial-bz}. For each $\lmb$, we normalize  $\nrm{\tilde{b}_{(\lmb)}(0)}_{L^2} = 1$. Since $\nrm{\tb_{(\lmb)}(0)}_{\dot{H}^{s'}} \lesssim_{s'} \lmb^{s'}$ for any $s' > 0$, the initial data $\bfB_{(\lmb)}(0)$ belongs to the ball $\calB_{\eps}(\bgB;H^s_{comp})$, by replacing the coefficient $\eps$ in \eqref{eq:illposed-strong-id} by $\eps/A$ for some large constant $A > 0$ independent of $\lmb$ if necessary.

\medskip

\textit{(ii) application of the generalized energy identity}

\medskip

By the assumption of Theorem~\ref{thm:illposed-strong}, there exists $\delta > 0$ and a unique local solution $\bfB_{(\lmb)}(t)$ in $L^\infty_t([0,\delta],H^{s_0})$. The additional hypothesis guarantees that, for all $\lmb \in \mathbb{N}$, the sequence of solutions $$b_{(\lmb)}(t) := \bgB - \bfB_{(\lmb)}(t)$$ is uniformly bounded in the space $L^\infty_t([0,\delta],H^{s_0})$ (with $s_0$ depending on $\alpha$). Moreover, by the uniqueness assumption, $b_{(\lmb)}$ is independent of $z$, and introducing $\psi_{(\lmb)}$ such that $b_{(\lmb)} = (\nabla^\perp\psi_{(\lmb)},b_{(\lmb)}^z)$, we have that the pair $(\psi_{(\lmb)}, b_{(\lmb)})$ solves the system  {(cf.~\eqref{eq:e-mhd-2.5d-pert})} \begin{equation}\label{eq:emhd-nonlinear}
\left\{ \begin{aligned}
&\rd_tb^z_{(\lmb)} - f\rd_x\Delta\psi_{(\lmb)} + f''\rd_x\psi_{(\lmb)} = -\nabla^\perp\psi_{(\lmb)} \cdot\nabla\Delta\psi_{(\lmb)} \\
&\rd_t\psi_{(\lmb)} + f\rd_xb^z_{(\lmb)} = \nabla^\perp\psi_{(\lmb)} \cdot\nabla b_{(\lmb)}^z 
\end{aligned} \right. 
\end{equation} with initial data $(\tilde{\psi}_{(\lmb)}(0), \tilde{b}_{(\lmb)}^z(0))$. 

Regarding the approximate solution $(\tpsi_{(\lmb)},\tb_{(\lmb)})$, let us recall  the error bounds \begin{equation*}
\begin{split}
\sup_{t \in [0,\delta]} \nrm{\err_{\tb}(t)}_{L^2} &\lesssim 1,\\
\sup_{t \in [0,\delta]} \nrm{\nabla  \err_{\tpsi}  }_{L^2} & \lesssim 1 ,
\end{split}
\end{equation*} where all the implicit constants are independent of $\lmb$. By the definition of the error terms, we have \begin{equation}\label{eq:emhd-nonlinear2}
\left\{ \begin{aligned}
&\rd_t\tb^z_{(\lmb)} - f\rd_x\Delta\tpsi_{(\lmb)} + f''\rd_x\tpsi_{(\lmb)} = \err_{\tb} \\
&\rd_t\tpsi_{(\lmb)} + f\rd_x\tb^z_{(\lmb)} =  \err_{\tpsi}.
\end{aligned} \right. 
\end{equation}  Using \eqref{eq:emhd-nonlinear} and \eqref{eq:emhd-nonlinear2}, we compute that \begin{equation}\label{eq:gei-unbounded}
\begin{split}
\frac{\ud}{\ud t} \brk{b_{(\lmb)}(t),\tb_{(\lmb)}(t)} &= -\brk{f''\rd_x\tpsi_{(\lmb)}, b^z_{(\lmb)}} - \brk{\tb_{(\lmb)}^z,f''\rd_x\psi_{(\lmb)}} \\
&\quad + \brk{\nabla\err_{\tpsi},\nabla\psi_{(\lmb)}} + \brk{\nabla\tpsi_{(\lmb)},\nabla ( \nabla^\perp\psi_{(\lmb)}\cdot\nabla b^z_{(\lmb)})} \\
&\quad + \brk{\err_{\tb},b^z_{(\lmb)}} + \brk{\tb_{(\lmb)}^z, -\nabla^\perp\psi_{(\lmb)}\cdot\nabla \Delta\psi_{(\lmb)}}. 
\end{split}
\end{equation} 

 {First}, we proceed in the case $s_0 \ge 3$ to contradict boundedness of the solution map. We bound the first and the third lines on the  {RHS,} respectively, by \begin{equation*}
\begin{split}
\left| -\brk{f''\rd_x\tpsi_{(\lmb)}, b^z_{(\lmb)}} - \brk{\tb_{(\lmb)}^z,f''\rd_x\psi_{(\lmb)}} \right|  \lesssim \nrm{\tb_{(\lmb)}}_{L^2} \nrm{b_{(\lmb)}}_{L^2} \lesssim \nrm{b_{(\lmb)}}_{L^2} 
\end{split}
\end{equation*} and \begin{equation*}
\begin{split}
\left| \brk{\err_{\tb},b^z_{(\lmb)}} + \brk{\tb_{(\lmb)}^z, -\nabla^\perp\psi_{(\lmb)}\cdot\nabla \Delta\psi_{(\lmb)}} \right| & \lesssim \nrm{\err_{\tb}}_{L^2} \nrm{b_{(\lmb)}}_{L^2} + \nrm{\tb_{(\lmb)}}_{L^2} \nrm{\nabla\Delta\psi_{(\lmb)}}_{L^2} \nrm{b_{(\lmb)}}_{L^\infty} \\
& \lesssim \nrm{b_{(\lmb)}}_{L^2}, 
\end{split}
\end{equation*} where we have used that $\sup_{t \in [0,\delta]}\nrm{\tb_{(\lmb)}}_{L^2} \lesssim 1$, the error bound on $\nrm{\err_{\tb}}_{L^2}$, and the uniform bound  $$\nrm{\nabla\Delta\psi_{(\lmb)}}_{L^2}  \nrm{b_{(\lmb)}}_{L^\infty}  \lesssim \nrm{b_{(\lmb)}}_{L^2}^{1/3}\nrm{b_{(\lmb)}}_{H^{3}}^{2/3} \nrm{b_{(\lmb)}}_{L^{2}}^{2/3} \nrm{b_{(\lmb)}}_{H^{3}}^{1/3}  
\lesssim \nrm{b_{(\lmb)}}_{L^2} \nrm{b_{(\lmb)}}_{H^3}
\lesssim \nrm{b_{(\lmb)}}_{L^2},$$ for $s_0 \ge 3$. Regarding the second line we have \begin{equation*}
\begin{split}
& \brk{\nabla\err_{\tpsi},\nabla\psi_{(\lmb)}} + \brk{\nabla\tpsi_{(\lmb)}, \nabla^\perp\psi_{(\lmb)}\cdot\nabla \nabla b^z_{(\lmb)} }   + \brk{\nabla\tpsi_{(\lmb)}, \nabla^\perp\nabla \psi_{(\lmb)}\cdot\nabla b^z_{(\lmb)} }  
\end{split}
\end{equation*} and the first two terms are bounded by a constant multiple of $\nrm{b_{(\lmb)}}_{L^2}$, again with $s_0 \ge 3$.  Lastly, \begin{equation*}
\begin{split}
\left|  \brk{\nabla\tpsi_{(\lmb)}, \nabla^\perp\nabla \psi_{(\lmb)}\cdot\nabla b^z_{(\lmb)} }   \right| \lesssim \nrm{\nabla^2\psi_{(\lmb)} \nabla b^z_{(\lmb)}}_{L^2} \lesssim \nrm{\nabla b_{(\lmb)}}_{L^4}^2 \lesssim \nrm{b_{(\lmb)}}_{L^2} \nrm{b_{(\lmb)}}_{H^3} 
\end{split}
\end{equation*} by an application of the Sobolev inequality. Collecting the bounds and using the energy identity for $b_{(\lmb)}$, we conclude that \begin{equation*}
\begin{split}
\left| \frac{\ud}{\ud t} \brk{b_{(\lmb)}(t), \tb_{(\lmb)}(t)} \right| \lesssim \nrm{b_{(\lmb)}(t)}_{L^2} \lesssim \nrm{b_{(\lmb)}(0)}_{L^2}
\end{split}
\end{equation*} where the implicit constants are independent on $\lmb$.  Therefore, by taking sufficiently small $0 < T \le \delta$, we can guarantee that \begin{equation*}
\begin{split}
\brk{b_{(\lmb)}(t),\tb_{(\lmb)}(t)} > \frac{1}{2}\nrm{b_{(\lmb)}(0)}_{L^2} = \frac{1}{2}\eps\lmb^{-s-n}, \qquad 0 < t \le T. 
\end{split}
\end{equation*} uniformly for all sufficiently large $\lmb$. 

Now we show how to arrive at the above inequality in the case $0 < \alpha \le 1$ under the $\alpha$-H\"older continuity assumption. While the choice of $n \ge 0$ did not play any role in the above, now we shall take it to be sufficiently large. Using the assumption of H\"older continuity around the stationary solution $\bgB$, we obtain the bound \begin{equation*}
\begin{split}
\nrm{b_{(\lmb)}}_{H^{s_0}} \lesssim \eps^\alpha \lmb^{-n\alpha}.
\end{split}
\end{equation*} Then, we can obtain better bounds on the quadratic terms $\nrm{\nabla b_{(\lmb)}}_{L^4}^2$ and $\nrm{b_{(\lmb)}\nabla^2 b_{(\lmb)}}_{L^2}$. Regarding the former, we bound \begin{equation*}
\begin{split}
\nrm{\nabla b_{(\lmb)}}_{L^4}^2 &\lesssim \nrm{b_{(\lmb)}}_{L^2}^{2\theta} \nrm{b_{(\lmb)}}_{H^{s_0}}^{2(1-\theta)}, \quad \theta = 1 - \frac{3}{2s_0} \\
& \lesssim  \eps^{2(1-\theta)\alpha + 2\theta } \lmb^{-2\theta(n+s)-2(1-\theta)n\alpha} \lesssim \eps \lmb^{-n-s}  
\end{split}
\end{equation*}  for $s_0 > \max\{ \frac{3}{2}, 3(1-\alpha) \}$  {by taking $n$ sufficiently large}. A similar bound can be obtained for $\nrm{b_{(\lmb)}\nabla^2 b_{(\lmb)}}_{L^2}$, now with $s_0 > \max\{ 2, 3(1-\alpha) \}$. 

\medskip

\textit{(iii) growth of Sobolev norms}

\medskip

Proceeding as in the proof of Theorem \ref{thm:norm-growth} using duality and \eqref{eq:wavepackets-degen}--\eqref{eq:wavepackets-degen-small}, we obtain \begin{equation*}
\begin{split}
\nrm{b_{(\lmb)}(t)}_{H^{s_0}} \ageq_{s_{0}} \lmb^{s_{0}} e^{c_{f} s_{0} \lmb t} \nrm{b_{(\lmb)}(0)}_{L^{2}} = \eps \lmb^{s_{0}-s-n} e^{c_{f} s_{0} \lmb t},
\end{split}
\end{equation*} 
which is a contradiction since $\lmb$ may be arbitrarily large. This finishes the proof for the electron-MHD case. \qedsymbol

\subsection{Proof of Theorem \ref{thm:illposed-strong} for \eqref{eq:hall-mhd}}

We shall restrict ourselves to the case $s_0 \ge 3$, necessary changes for the H\"older case of $s_0 > \max\{2,3(1-\alpha)\}$ being obvious. We also fix $\bgB = f(y)\rd_x$ and some $\nu \ge 0$. 

To begin with, take  {the} initial data as in \eqref{eq:illposed-strong-id} together with trivial initial velocity; that is,  $\bfu_{(\lmb)}(0) = 0 $. Then, by the assumption of existence and uniqueness, we obtain a $z$-independent solution quadruple $(u^z_{(\lmb)}, \omega_{(\lmb)}, b^z_{(\lmb)}, \psi_{(\lmb)})$ to the system \eqref{eq:hall-mhd-2.5d-pert}. The solution is uniformly bounded (in $\lmb$) in the space \begin{equation*}
\begin{split}
\nabla u^z_{(\lmb)}, \omega_{(\lmb)} \in L^\infty([0,\delta];H^{s_0-1}),\qquad \nabla\psi_{(\lmb)}, b^z_{(\lmb)} \in L^\infty([0,\delta]; H^{s_0})
\end{split}
\end{equation*} with some constant $\delta > 0$. Appealing to Proposition \ref{prop:wavepackets-hall} with initial data \eqref{eq:initial-bz}, \eqref{eq:initial-psi}, and \eqref{eq:wavepackets-hall-u-omg}, we obtain the approximate solution $(\tu_{(\lmb)}^z, \tomg_{(\lmb)}, \tb_{(\lmb)}^z, \tpsi_{(\lmb)})$ with the estimates \begin{equation*}
\begin{split}
\nrm{\tu^{z}_{(\lmb)}(t)}_{L^{2}} + \nrm{\nb^{\perp} (-\lap)^{-1} \tomg_{(\lmb)}(t)}_{L^{2}} \aleq & \lmb^{-1}   ,\\
\nrm{\nb \tu^{z}_{(\lmb)}(t)}_{L^{2}} + \nrm{\tomg_{(\lmb)}(t)}_{L^{2}} \aleq & 1 ,
\end{split}
\end{equation*}
and 
\begin{equation*}
\begin{split}
\nrm{\errh_{\tu}^{(\nu)}(t) + \nu \lap \tpsi_{(\lmb)}(t)}_{L^{2}} \aleq & \lmb^{-1} , \\
\nrm{\nb^{\perp} (-\lap)^{-1} (\errh_{\tomg}^{(\nu)} + \nu \lap \tomg_{(\lmb)})(t)}_{L^{2}} \aleq & \lmb^{-1} , \\
\nrm{\errh_{\tb}^{(\nu)}(t)}_{L^{2}} \aleq & 1, \\
\nrm{\nb  \errh_{\tpsi}^{(\nu)}  (t)}_{L^{2}} \aleq & 1.
\end{split} 
\end{equation*}  Now using \eqref{eq:en-hall-mhd-parallel}, we obtain 
\begin{equation*}
\begin{split}
&\frac{\ud}{\ud t}\left(\brk{\tilde{b}_{(\lmb)}, b_{(\lmb)}}(t) + \brk{\tilde{u}_{(\lmb)}, u_{(\lmb)}}(t) \right) + 2\nu \brk{\nabla\tu_{(\lmb)} , \nabla u_{(\lmb)}}   \\
&\qquad = - \brk{f'' \rd_{x} \tpsi_{(\lmb)}, b_{(\lmb)}^{z}} - \brk{\tb^{z}_{(\lmb)}, f'' \rd_{x} \psi_{(\lmb)}} 
- \brk{f' \nb \tpsi_{(\lmb)}, u_{(\lmb)}^{x, y}} 
- \brk{\tu_{(\lmb)}^{x,y}, f' \nb \psi_{(\lmb)}} \\
&\qquad \phantom{=}
+\brk{\nb^{\perp} \errh_{\tpsi}, \nb^{\perp} \psi_{(\lmb)}} + \brk{\nb^{\perp} \tpsi_{(\lmb)}, \nb^{\perp} \errh_{\psi}} 
+ \brk{\errh_{\tb}, b^{z}_{(\lmb)}} + \brk{\tb^{z}_{(\lmb)}, \errh_{b}} \\
&\qquad \phantom{=}
- \brk{\nb^{\perp} (-\lap)^{-1} \errh_{\tomg}^{(\nu)} , u_{(\lmb)}^{x,y}} - \brk{\tu_{(\lmb)}^{x,y}, \nb^{\perp} (-\lap)^{-1} \errh_{\omg}^{(\nu)}} 
+ \brk{\errh_{\tu}^{(\nu)}, u_{(\lmb)}^{z}} + \brk{\tu_{(\lmb)}^{z}, \errh_{u}^{(\nu)}}.
\end{split}
\end{equation*} and then, the terms on the first line of the RHS is bounded simply by \begin{equation}\label{eq:last-1}
\begin{split}
&\left| - \brk{f'' \rd_{x} \tpsi_{(\lmb)}, b_{(\lmb)}^{z}} - \brk{\tb^{z}_{(\lmb)}, f'' \rd_{x} \psi_{(\lmb)}} 
- \brk{f' \nb \tpsi_{(\lmb)}, u_{(\lmb)}^{x, y}} 
- \brk{\tu_{(\lmb)}^{x,y}, f' \nb \psi_{(\lmb)}} \right| \\
&\qquad \lesssim  \nrm{\tb_{(\lmb)}}_{L^2}  \nrm{b_{(\lmb)}}_{L^2} +\nrm{\tb_{(\lmb)}}_{L^2}   \nrm{u_{(\lmb)}}_{L^2} + \nrm{\tu_{(\lmb)}}_{L^2} \nrm{b_{(\lmb)}}_{L^2} \lesssim  \nrm{b_{(\lmb)}}_{L^2}  + \nrm{u_{(\lmb)}}_{L^2}.
\end{split}
\end{equation}  To bound the other terms, we recall the form of the error for a solution of the (nonlinear) Hall-MHD equations  {(cf.~\eqref{eq:hall-mhd-2.5d-pert} and \eqref{eq:hall-2.5d-err-parallel})}: \begin{equation*}
\left\{
\begin{aligned}
&\errh_{u}^{(\nu)} = - u_{(\lmb)}^{x,y} \cdot \nb u^{z}_{(\lmb)}, \\
&\errh_\omega^{(\nu)} = - u_{(\lmb)}^{x,y} \cdot \nb \omg_{(\lmb)} + \nb^{\perp} \psi_{(\lmb)} \cdot \nb \lap \psi_{(\lmb)} = \nabla\cdot(-\omg_{(\lmb)}u^{x,y}_{(\lmb)} + \Delta\psi_{(\lmb)}\nabla^\perp\psi_{(\lmb)}),  \\
&\errh_b = - u_{(\lmb)}^{x,y} \cdot \nb b_{(\lmb)}^{z} -  \nb^{\perp} \psi_{(\lmb)} \cdot \nb u_{(\lmb)}^{z} -  \nb^{\perp} \psi_{(\lmb)} \cdot \nb \lap \psi_{(\lmb)} ,\\
&\errh_\psi = - u_{(\lmb)}^{x,y} \cdot \nb \psi_{(\lmb)} + \nb^{\perp} \psi_{(\lmb)} \cdot \nb b^{z}_{(\lmb)}   .
\end{aligned}
\right.
\end{equation*}
{For the approximate solution, we write $\errh_{\tu} = \errh_{\tu}^{(\nu)} + \nu \lap \tu^{z}_{(\lmb)}$ and $\errh_{\tomg} = \errh_{\tomg}^{(\nu)} + \nu \lap \tomg_{(\lmb)}$.}
 Then, \begin{equation}\label{eq:last-2}
\begin{split}
\left| \brk{\nb^{\perp} \errh_{\tpsi}, \nb^{\perp} \psi_{(\lmb)}} + \brk{\nb^{\perp} \tpsi_{(\lmb)}, \nb^{\perp} \errh_{\psi}}  \right| & \lesssim  \nrm{b_{(\lmb)}}_{L^2} + \nrm{ \nabla^\perp u_{(\lmb)} \cdot \nabla\psi_{(\lmb)}  }_{L^2} + \nrm{u_{(\lmb)}\cdot\nabla\nabla^\perp\psi_{(\lmb)}}_{L^2} \\
& \quad  + \nrm{\nabla^\perp\nabla^\perp\psi_{(\lmb)}\cdot\nabla b_{(\lmb)}^z}_{L^2} + \nrm{\nabla^\perp\psi_{(\lmb)}\cdot\nabla\nabla^\perp b_{(\lmb)}^z}_{L^2} \\
&\lesssim \nrm{b_{(\lmb)}}_{L^2} + \nrm{u_{(\lmb)}}_{L^2}
\end{split}
\end{equation} where we have used \begin{equation*}
\begin{split}
\nrm{\nabla^\perp\nabla^\perp\psi_{(\lmb)}\cdot\nabla b_{(\lmb)}^z}_{L^2}  \lesssim \nrm{b_{(\lmb)}}_{L^2} \nrm{b_{(\lmb)}}_{H^3} \lesssim \nrm{b_{(\lmb)}}_{L^2} 
\end{split}
\end{equation*} as in the electron-MHD case. 
Next,
\begin{equation}\label{eq:last-3}
\begin{split}
\left|\brk{\errh_{\tb}, b^{z}_{(\lmb)}} + \brk{\tb^{z}_{(\lmb)}, \errh_{b}}\right| 
& \lesssim \nrm{\errh_{\tb}}_{L^2} \nrm{b_{(\lmb)}}_{L^2} + \nrm{\tb^z_{(\lmb)}}_{L^2} \nrm{u_{(\lmb)}}_{L^2}\nrm{\nabla b^z_{(\lmb)}}_{L^\infty} \\
& \quad + \nrm{\tb^z_{(\lmb)}}_{L^2} \nrm{b_{(\lmb)}}_{L^2} (\nrm{\nabla u_{(\lmb)}^{z}}_{L^\infty} + \nrm{b_{(\lmb)}}_{H^{3}} ) \\ & \lesssim \nrm{b_{(\lmb)}}_{L^2} + \nrm{u_{(\lmb)}}_{L^2} .
\end{split}
\end{equation} 
Now rewriting \begin{equation*}
\begin{split}
&- \brk{\nb^{\perp} (-\lap)^{-1} \errh_{\tomg}^{(\nu)} , u_{(\lmb)}^{x,y}} - \brk{\tu_{(\lmb)}^{x,y}, \nb^{\perp} (-\lap)^{-1} \errh_{\omg}^{(\nu)}} \\
&=  -\brk{\nabla^\perp(-\Delta)^{-1}(\errh_{\tomg} + \nu\Delta\tomg_{(\lmb)}),u^{x,y}_{(\lmb)}} - \nu  \brk{\nabla^\perp\tomg_{(\lmb)},u^{x,y}_{(\lmb)}} \\
& \phantom{=} + \brk{\tu_{(\lmb)}^{x,y}, \nabla^\perp(-\Delta)^{-1}\nabla\cdot(-\omg_{(\lmb)}u^{x,y}_{(\lmb)} + \Delta\psi_{(\lmb)}\nabla^\perp\psi_{(\lmb)})} ,
\end{split}
\end{equation*} we obtain 
 \begin{equation}\label{eq:last-4}
\begin{split}
&\left| - \brk{\nb^{\perp} (-\lap)^{-1} \errh_{\tomg}^{(\nu)} , u_{(\lmb)}^{x,y}} - \brk{\tu_{(\lmb)}^{x,y}, \nb^{\perp} (-\lap)^{-1} \errh_{\omg}^{(\nu)}}  \right| \\
&\lesssim \nrm{\nabla^\perp(-\Delta)^{-1}(\errh_{\tomg}+\nu\Delta\tomg_{(\lmb)})}_{L^2}\nrm{u_{(\lmb)}}_{L^2} + \nu\nrm{\tomg_{(\lmb)}}_{L^2} \nrm{\nabla u_{(\lmb)}}_{L^2} \\
&\quad + \nrm{\tu_{(\lmb)}}_{L^2} (\nrm{\omg_{(\lmb)}}_{L^\infty} \nrm{u_{(\lmb)}}_{L^2} + \nrm{\Delta\psi_{(\lmb)}}_{L^\infty}\nrm{b_{(\lmb)}}_{L^2} ) \\
&\lesssim \nrm{b_{(\lmb)}}_{L^2}  + \nrm{u_{(\lmb)}}_{L^2} + \nu\nrm{\nabla\tu_{(\lmb)}}_{L^2}\nrm{\nabla u_{(\lmb)}}_{L^2} .
\end{split}
\end{equation} Lastly, \begin{equation}\label{eq:last-5}
\begin{split}
&\left| \brk{\errh_{\tu}^{(\nu)}, u_{(\lmb)}^{z}} + \brk{\tu_{(\lmb)}^{z}, \errh_{u}^{(\nu)}} \right|  \lesssim \nrm{\errh_{\tu}+\nu\Delta\tpsi_{(\lmb)}}_{L^2}\nrm{u_{(\lmb)}}_{L^2} + \nu\nrm{\nb\tu_{(\lmb)}}_{L^2}\nrm{\nb u_{(\lmb)}}_{L^2}  . 
\end{split}
\end{equation} Collecting the bounds \eqref{eq:last-1}--\eqref{eq:last-5}  {and recalling that $\nrm{\nb \tu_{(\lmb)}}_{L^{2}} \aleq 1$ (cf.~Proposition~\ref{prop:wavepackets-hall})}, we obtain \begin{equation*}
\begin{split}
\frac{\ud}{\ud t} \left(\brk{\tilde{b}_{(\lmb)}, b_{(\lmb)}}(t) + \brk{\tilde{u}_{(\lmb)}, u_{(\lmb)}}(t) \right) &\lesssim  \nrm{b_{(\lmb)}}_{L^2}  + \nrm{u_{(\lmb)}}_{L^2} +  {\nu \nrm{\nabla u_{(\lmb)}}_{L^2},}
\end{split}
\end{equation*} and integrating in time and using the energy inequality \begin{equation*}
\begin{split}
\nrm{b_{(\lmb)}(t)}_{L^2}^2 + \nrm{u_{(\lmb)}(t)}_{L^2}^2 + 2\nu\int_0^t \nrm{\nabla u_{(\lmb)}}_{L^2}^2 ds \lesssim \nrm{b_{(\lmb)}(0)}_{L^2}^2 + \nrm{u_{(\lmb)}(0)}_{L^2}^2  =  \nrm{b_{(\lmb)}(0)}_{L^2}^2  
\end{split}
\end{equation*} gives \begin{equation*}
\begin{split}
\left|  \brk{\tilde{b}_{(\lmb)}, b_{(\lmb)}}(t) - \brk{\tilde{b}_{(\lmb)}, b_{(\lmb)}}(0) \right| \lesssim (t+ \nu^{1/2}t^{1/2} +  \lmb^{-1})\nrm{b_{(\lmb)}(0)}_{L^2} .
\end{split}
\end{equation*} Since \begin{equation*}
\begin{split}
\brk{\tilde{b}_{(\lmb)}, b_{(\lmb)}}(0) = \nrm{b_{(\lmb)}(0)}_{L^2}(1 + O(\lmb^{-1})),
\end{split}
\end{equation*}  we may take  {a small number} $0 < T \le \delta $ such that for all sufficiently large $\lmb$, \begin{equation*}
\begin{split}
 \brk{\tilde{b}_{(\lmb)}, b_{(\lmb)}}(t)  > \frac{1}{2} \nrm{b_{(\lmb)}(0)}_{L^2}, \qquad t \in [0,T]. 
\end{split}
\end{equation*} The rest of the argument is the same as in the electron-MHD case. The proof is complete.  \qedsymbol

\subsection{Proof of Theorem \ref{thm:illposed-strong2} for \eqref{eq:e-mhd}} \label{subsec:illposedness-strong-e}
In this section, we give the proof of Theorem \ref{thm:illposed-strong2}  {for \eqref{eq:e-mhd}. Compared to \eqref{eq:hall-mhd}, a rather strong localization is possible in this case, and thus the proof works also on $M = \bbT^{3}$. We proceed in several steps. 

\medskip

\textit{(i) choice of initial data and contradiction hypothesis}

\medskip

As described in Section~\ref{subsec:ideas}, the key idea is to superpose many instabilities in physical space. In order to fit everything in a compact interval, which allows us to consider the case $M = \bbT^{3}$, and control the constants involved in the instability argument, we use a simple rescaling argument. 

Let $M$, $s > 3+\frac{1}{2}$ and $\eps > 0$ be given by the statement of Theorem~\ref{thm:illposed-strong2}; in what follows, we suppress the dependence of implicit constants on $s$. We simultaneously give constructions involving the translational- and axi-symmetric stationary magnetic fields; the former construction works for $M = \mathbb{T}_x \times (\bbT,\bbR)_y \times \bbT_z$, and the latter applies to all of $M = (\bbT, \bbR)_x \times (\bbT,\bbR)_y \times \bbT_z$.} To this end, we take \begin{equation}\label{eq:illposed-strong2-stat}
\begin{split}
\mathring{\mathbf{B}}^{(tr,axi)} = \sum_{k = k_0}^{\infty}  {\bgB}_k^{(tr,axi)} := \sum_{k = k_0}^{\infty}  {2^{-sk}} \widetilde{\mathbf{B}}^{(tr,axi)}(2^{k}x , 2^{k} (y-y_k) ), \quad y_k = 2^{-\frac{k}{2}} 
\end{split}
\end{equation} where \begin{equation*}
\begin{split}
\widetilde{\mathbf{B}}^{tr}(x,y) = f^{tr}_0(y)\rd_x 
\end{split}
\end{equation*} in the translationally symmetric case, and \begin{equation*}
\begin{split}
\widetilde{\mathbf{B}}^{axi}(x,y) = f^{axi}_{0} (\sqrt{x^2 + y^2}) (x \partial_{y} - y \partial_{x})
\end{split}
\end{equation*} in the axi-symmetric case. Here,  {$k_0 \ge 10$} is some large positive integer (to be specified later),  $f_0^{tr}(y) \in C^\infty_{comp}$ supported in $|y| \le 1/10$ and $f_0^{tr}(y) = y$ for $|y| \le 1/20$, and $f^{axi}_0 \in C^\infty_{comp}$ is supported in $r \le 1/5$ and satisfies $f^{axi}_0(r) = r - 1/20$ for $1/40 \le r \le 1/10$. We further assume that $f^{axi}_0$ is a constant for $r$ small so that $\widetilde{\mathbf{B}}^{axi}$ defines a smooth vector field on $\mathbb{R}^2_{x,y}$. Note that in both cases, $ {\bgB^{(tr, axi)}_k}$ is a stationary solution and the supports of $ {\bgB^{(tr, axi)}_k}$ are disjoint, so $ {\bgB^{(tr, axi)}}$ defines a stationary solution to \eqref{eq:e-mhd}. The  {coefficient $2^{-sk}$ guarantees that $\bgB^{tr}_{k} \in H^s_{comp}(M)$ with $\nrm{\bgB^{tr}_{k}}_{H^{s}} \aleq 2^{-\frac{1}{2} k}$  when $M = \mathbb{T}_x \times (\bbT,\bbR)_y \times \bbT_z$, and $\bgB^{axi}_{k} \in H^{s}_{comp}(M)$ with $\nrm{\bgB^{tr}_{k}}_{H^{s}} \aleq 2^{-\frac{1}{2} k}$ when $M = (\bbT, \bbR)_{x} \times (\bbT, \bbR)_{y} \times \bbT_{z}$. In all cases, $\bgB^{(tr, axi)} \in H^{s}_{comp}(M)$ and we may ensure that $\nrm{\bgB^{(tr, axi)}}_{H^{s}} < \frac{1}{2} \eps$ by taking $k_{0}$ large enough.}

We now fix the initial data. In the translationally symmetric case, take some compactly supported function $g_0^{tr} \in C^\infty(M)$  {that is independent of $x$} and whose $y$-support is contained in $1/40 < y < 1/20$ so that the hypotheses of Proposition \ref{prop:wavepackets} is satisfied for $g_0 = g_0^{tr}$,  {$\bgB = \widetilde{\bfB}^{tr} = f^{tr}_{0}(y) \rd_{x}$ and $y_{1} = \frac{1}{20}$.} Then define $\tb_{(\lmb)} := (\nabla^\perp \tpsi_{(\lmb)}, \tb^z_{(\lmb)})$ to be the associated degenerating wave packet solution provided by Proposition~\ref{prop:wavepackets}. For
\begin{equation*}
 \lmb_k =  { 2^{Nk} }
\end{equation*}
for some $N \gg 1$ sufficiently large  to be specified later (depending only on $s$), we define \begin{equation} \label{eq:rescaled-wp-trans}
\begin{aligned}
\tb_k^{tr}( {t},x,y,z) =& 2^{\frac{k}{2}}\tb_{( {2^{-k}} \lmb_k)}( {2^{-sk} 2^{2k} t,2^{k}} x,2^k(y-y_k)), \\
 {\tpsi^{tr}_{k}(t, x, y, z) = }&  {2^{-\frac{k}{2}} \tpsi_{(2^{-k} \lmb_{k})} (2^{-sk} 2^{2k} t, 2^{k} x, 2^{k} (y - y_{k})),}
\end{aligned}
\end{equation} 
 {By construction, $e^{-i \lmb_{k} x} \tb_k^{tr}(x, y, z)$ is independent of both $x$ and $z$, and by suitably normalizing $g^{tr}_{0}$, we may take $\nrm{\tb_k^{tr}}_{L^{2}} = 1$. Note that these definitions are consistent with the relation $(\tb_{k}^{x, y})^{tr} = \nb^{\perp} \tpsi_{k}^{tr}$.
}
Then, we define the initial data to be  \begin{equation}\label{eq:illposed-strong2-id-tr}
\begin{split}
\mathbf{B}^{tr}_0 & = \bgB^{tr} + \sum_{k = k_0}^{\infty}  {2^{-k}} \lmb_k^{-s} \tilde{b}_k^{tr}(t = 0).
\end{split}
\end{equation} Recalling that $\tb_k^{tr}$ is uniformly bounded in $L^2$, we see that 
 {$\nrm{\lmb_k^{-s} \tilde{b}_k^{tr}(t=0)}_{H^{s}} \aleq 1$; thanks to the factor $2^{-k}$, we have $\mathbf{B}^{tr}_{0} \in H^s_{comp}$ and we may ensure that $\nrm{\mathbf{B}^{tr}_{0}}_{H^s} < \eps$ by taking $k_0$ large.}

 We proceed similarly in the axi-symmetric case. Take $g_0^{axi}(r) \in C^\infty_{comp}$ with $\supp(g_0^{axi}) \subset (3/40,1/10)$. With this choice, the hypotheses of Proposition \ref{prop:wavepackets} in the axi-symmetric case is satisfied with $\widetilde{\mathbf{B}}^{axi} = f_0^{axi}(r)\rd_\theta$ and $g_0 = g_0^{axi}$. Applying Proposition \ref{prop:wavepackets} with this data gives $\tb_{(\lmb)} = (\nabla^\perp\tpsi_{(\lmb)}, \tb^z_{(\lmb)})$ and define \begin{equation} \label{eq:rescaled-wp-axi}
 {\tb_k^{axi}( {t,}x,y,z) =   2^k \tb_{( {2^{-k}} \lmb_k)}( {2^{-sk} 2^{2k} t,} 2^k x, 2^k(y-y_k)). } 
\end{equation} We then define the initial data as  \begin{equation}\label{eq:illposed-strong2-id-axi}
\begin{split}
\mathbf{B}^{axi}_0 & = \bgB^{axi} + \sum_{k = k_0}^{\infty}  {2^{-k}} \lmb_k^{-s} \tilde{b}_k^{axi}(t = 0).
\end{split}
\end{equation} 
 {Again, $\bfB^{axi}_{0} \in H^{s}_{comp}(M)$ and we may take $\nrm{\bfB^{axi}_{0}}_{H^{s}} < \eps$ by taking $k_{0}$ adequately large. }

 {At this point, from \eqref{eq:rescaled-wp-trans} and \eqref{eq:rescaled-wp-axi}, it is easy to check that $\tb^{(tr, axi)}_{k}$ obeys the following types of boundedness, error and degeneration estimates, respectively:
\begin{align}
 {\nrm{\tb_{k}(t)}_{L^{2}} } \aleq 2^{C_{s} k}, \label{eq:bdd-k} \\
\nrm{\err_{\tb,k}(t)}_{L^{2}} + \nrm{\nb \err_{\tpsi,k}(t)}_{L^{2}} \aleq 2^{C_{s} k}, \label{eq:error-k} \\
\nrm{\tb^{(tr, axi)}(t)}_{L^{2}_{x} H^{-\frac{1}{4}}_{y}} \aleq 2^{C_{s} k} \exp\left( - 2^{-c_{s} k} \lmb_{k} t\right),  \label{eq:degen-k}
\end{align}
where $c_{s}, C_{s} > 0$ are constants depending only on $s$ (that may change from line to line) and
\begin{align*}
\err_{\tb,k}(t) = \err_{b}[\tb^{z}_{k}, \tpsi_{k}], \quad
\err_{\tpsi,k}(t) = \err_{\psi}[\tb^{z}_{k}, \tpsi_{k}]
\end{align*}
are defined according to \eqref{eq:e-2.5d-err-parallel} and \eqref{eq:e-2.5d-err-axi} respect to $\bgB^{(tr, axi)}_{k}$, $(\tb^{z}_{k})^{(tr, axi)}$ and $\tpsi_{k}^{(tr, axi)}$.
}

Towards a contradiction, we assume that there exist $\delta > 0$ and a solution $\mathbf{B}^{(tr,axi)} \in L^\infty_t([0,\delta];H^s)$ to \eqref{eq:e-mhd} with initial data \eqref{eq:illposed-strong2-id-tr} and \eqref{eq:illposed-strong2-id-axi}, and set $b(t) := \mathbf{B}^{(tr,axi)}(t) - \bgB^{(tr,axi)}$, respectively. Since we do not assume uniqueness of the solution, $b(t)$ may depend on $z$ as well, and it satisfies\footnote{Even in this case, \eqref{eq:e-mhd} can be reformulated in terms of $b^z$ and $\psi$, where $-\Delta_{x,y}\psi = (\nabla \times b)^z$ with $\Delta_{x,y} := \rd_{xx} + \rd_{yy}$. But now the expression for $\rd_t\psi$ involves non-local terms.} 
\begin{equation}\label{eq:e-mhd-pert}
\begin{split}
& \quad \rd_t b + (b\cdot\nabla)(\nabla\times\bgB) - (\nabla\times\bgB)\cdot\nabla b + (\bgB\cdot\nabla)(\nabla\times b) - (\nabla\times b)\cdot\nabla\bgB  \\
& \quad = \nabla \times ((\nabla\times b)\times b). 
\end{split}
\end{equation} 

\medskip

 {\textit{(ii) localization of the energy identity}}

\medskip

 Before we continue, let us briefly give an outline of the argument.  As discussed in Section~\ref{subsec:ideas}, we would like to localize the energy identity for $b$ (as well as the generalized energy identity between $b$ and $\tb_k$ in the next step) to the support of $ {\bgB^{(tr, axi)}_k}$. If the corresponding localized statements were exactly true, the proof will be completed immediately since near $ {\bgB^{(tr, axi)}_k}$, the $H^s$-norm of the perturbation will grow at a rate of $\lmb_k t$, which clearly dominate losses of $2^{A k}$ coming from various normalizations, by taking $\lmb_k = 2^{Nk}$ with $N$ large. Not surprisingly, the main enemy is the loss of one derivative coming from the commutator $[\chi, (\bgB \cdot \nabla)(\nabla\times b)]$ where $\chi$ is a cutoff. A derivative on $b$ localized to the support of $\bgB_k$ could in principle cost $\lmb_k$, but we gain a little bit by a time integration, which gives a necessary control of the local energy in the time scale $\aeq \lmb_k^{s/(s+1)}$, which is still sufficient for unbounded growth in $k$. 

Now we proceed to prove a localized version of the energy estimate for the perturbation $b$. 
 {The proof is similar for both the translationally- and axi-symmetric cases, and for simplicity we only consider the translationally symmetric case from now on.}
Ideally we would like to show that the $L^2$-norm of $b$ localized to the support of $\bgB_k$ admits an energy inequality by itself, but since there will be some contribution from neighboring pieces, we use cutoff functions with fast decaying tails which can accommodate such interactions. To this end, we prepare a $C^\infty$ positive function $\chi : \mathbb{R} \rightarrow \mathbb{R}_+$ with the following properties: \begin{itemize}
	\item $\chi(y) =1$ for $y \in [-1/4,1/4]$\footnote{This property is not essential but for convenience of the estimates below.},
	\item $|\chi'(y)| \le |\chi(y)|$ for all $y \in \mathbb{R}$, and
	\item $\chi$ decays exponentially; i.e. $\chi(y) \le 2^{-|y|}$ for $y > 1/2$. 
\end{itemize}  
Then, in the case $(\mathbb{T}, \mathbb{R})_y = \mathbb{R}_y$, we simply define $\chi_k(y):= \chi(2^k(y-y_k))$. In the case of $\mathbb{T}_y = \mathbb{R}/(2\pi\mathbb{Z})$, which we view as the interval $[-\pi,\pi]$ with the endpoints identified, we proceed as follows: for $k$ sufficiently large, take the $2\pi$-periodic function \begin{equation*}
\begin{split}
\sum_{n \in \mathbb{Z}} \chi_k(y + 2\pi n) 
\end{split}
\end{equation*} and one may modify this function only on the interval $[y_k - 2^{-1-k}, y_k + 2^{-1-k}]$ so that it is identically 1 on $[y_k-2^{-2-k},y_k+2^{-2-k}]$, which we redefine as $\chi_k$. Note that in this process we can guarantee that $|\chi_k'(y)|\lesssim 2^k|\chi(y)|$ on $\mathbb{T}_y$ with a constant independent of $k$. Regarding the decay, we shall only need $|\chi_k(y)| \lesssim 2^{-2^k|y-y_k|}$ for $|y- y_k| \le 1/10$, which holds for all $k $ sufficiently large, in the case of $\mathbb{T}_y$ as well. For the simplicity of the argument we shall proceed in the case of $\mathbb{R}_y$. Multiplying both sides of \eqref{eq:e-mhd-pert} by $\chi_k(y)$ and taking the $L^2$ inner product in $M$ with $\chi_k(y) b$, we obtain \begin{equation*}
\begin{split}
\left|\brk{\chi_k(b\cdot\nabla)(\nabla\times \bgB), \chi_k b}\right| \lesssim \nrm{\nabla^2\bgB}_{L^\infty} \nrm{\chi_k b}_{L^2}^2, 
\end{split}
\end{equation*}  \begin{equation*}
\begin{split}
\brk{\chi_k  (\nabla\times\bgB)\cdot\nabla b ,\chi_kb } = 0
\end{split}
\end{equation*} (after integrating by parts as $ (\nabla\times\bgB)\cdot\nabla =-f'(y)\rd_z$ commutes with $\chi_k$), and \begin{equation*}
\begin{split}
&\left|\brk{ \chi_k ((\bgB\cdot\nabla)(\nabla\times b) - (\nabla\times b)\cdot\nabla\bgB ), \chi_k b} \right| = 4\left| \brk{\chi_k' fb^z,\rd_x(\chi_k b^x)}\right| \\
&\qquad\lesssim 2^k \nrm{\chi_kb}_{L^2} \nrm{\nabla(\chi_kb)}_{L^2} \\
&\qquad \lesssim_s 2^k \nrm{\chi_{k} b}_{L^2}^{2-\frac{1}{s}} \nrm{\chi_kb}_{H^s}^{\frac{1}{s}} \lesssim 2^{2k} \nrm{b}_{H^s}^{\frac{1}{s}} \nrm{\chi_{k} b}_{L^2}^{2-\frac{1}{s}}
\end{split}
\end{equation*} after observing cancellations using integration by parts. We then used the algebra property of $H^s$ with Sobolev embedding. (Actually it is clear that the previous inner product should be of the form $\int \chi_k'\chi_k b\nabla b  $ since unless a derivative falls on the cutoff we obtain complete cancellations.) Finally we treat the nonlinearity \begin{equation*}
\begin{split}
\brk{\chi_k \nabla\times((\nabla\times b)\times b), \chi_k b},
\end{split}
\end{equation*} which has terms of the type $\brk{\chi_k' b\nabla b, \chi_k b}$ after integration by parts since unless the curl falls on $\chi_k$, we obtain a cancellation. The bound $|\chi_k'| \lesssim 2^k |\chi_k|$ allows us to estimate such terms by $\lesssim 2^k\nrm{\nabla b}_{L^\infty} \nrm{\chi_kb}_{L^2}^2 $. 
Collecting all the estimates, \begin{equation}\label{eq:local-energy-estimate}
\begin{split}
\frac{\ud}{\ud t}\nrm{\chi_kb}_{L^2} \lesssim 2^{2k}\left( \nrm{\chi_k b}_{L^2}^{1 - \frac{1}{s}} + \nrm{\chi_kb}_{L^2} \right) 
\end{split}
\end{equation} where the  {implicit} constant now depends on $\nrm{b}_{L^\infty_t H^s}$. Let us estimate, at the initial time, the local energy \begin{equation*}
\begin{split}
\nrm{\chi_kb(t = 0)}_{L^2}^2 =  {2^{-2k}} \lmb_k^{-2s} + \sum_{k_0 \le k' , k' \ne k }  {2^{-2k'}} \lmb_{k'}^{-2s}\nrm{\chi_k \tb_{k'}(t = 0)}^2_{L^2}  . 
\end{split}
\end{equation*}  Note that the contribution to the above sum for $ k' > k$ is negligible relative to $ {2^{-2k}} \lmb_k^{-2s}$ from the decay of $ {2^{-2k'}} \lmb_{k'}^{-2s}$ in $k'$. On the other hand, for $k' < k$, we use the decay of $\chi_k$: for any $k' < k$, the support of $\tb_{k'}$ is separated from $y_k$ by at least $c2^{-\frac{k}{2}}$ with $c > 0$ independent of $k$. Hence, \begin{equation*}
\begin{split}
|\chi_k| \lesssim 2^{-c2^{\frac{k}{2}}} \lesssim_{N,s} 2^{-4Nsk} \ll   {2^{-2k}} \lmb_k^{-2s} 
\end{split}
\end{equation*} on the support of $\tb_{k'}$ with any $k' > k \ge k_0$, and we obtain that $\nrm{\chi_kb(t = 0)}_{L^2}   {\aleq 2^{-2k}} \lmb_k^{-s}$ by choosing $k_0 $ sufficiently large with respect to $N, s$. Using this together with \eqref{eq:local-energy-estimate} yields that \begin{equation}\label{eq:localenergy-time}
\begin{split}\nrm{\chi_kb {(t)}}_{L^2} \lesssim_s \left(  {2^{-\frac{k}{s}}} \lmb_k^{-1} + 2^{2k}t \right)^s \lesssim_s  {2^{-k}} \lmb_k^{-s} + 2^{2ks}t^s. 
\end{split}
\end{equation}

\medskip

 {\textit{(iii) localization of the generalized energy identity and conclusion of the proof}}

\medskip

We shall now need a version of the generalized energy inequality which is localized in space.  {As before, since the argument is similar for both the translationally- and axi-symmetric cases, we only consider the translationally symmetric case for simplicity.}

 {Recall from the construction in Section~\ref{sec:wavepackets}} that the support of  {the rescaled and translated degenerating wave packet} $\tb_k$ is contained in $[y_k - 2^{-2-k},y_k+2^{-2-k}]$, and in particular $\chi_k\tb_k = \chi_k^2\tb_k$.  We now compute \begin{equation*}
\begin{split}
\frac{\ud}{\ud t} \brk{\chi_k b, \tb_k} &= \brk{-\chi_k f'\rd_zb^z + \chi_k f''b^y -\chi_k f\rd_x(\rd_xb^y-\rd_yb^x) ,\tb^z_k}  \\
&\qquad + \brk{\chi_k b^z,  f\rd_x\lap_{x,y}\tpsi_k -   f''\rd_x\tpsi_k + \err_{\tb^z,k}} \\
&\qquad + \brk{-\chi_k f'\rd_zb^y -\chi_k f\rd_x(\rd_zb^x-\rd_xb^z)   , \chi_k \tb^y } \\
&\qquad + \brk{\chi_k b^y,   f \rd_{xx}\tb^z_k -  \rd_x \err_{\tpsi,k} } \\
&\qquad + \brk{ \chi_k f'\rd_zb^x + \chi_k\rd_y(f(\rd_zb^x-\rd_xb^z)) +\chi_k f\rd_z(\rd_xb^y-\rd_yb^x)   ,  \tb^x_k }  \\
&\qquad + \brk{ \chi_k b^x, -\chi_k \rd_y (f\rd_x\tb^z) + \chi_k \rd_y \err_{\tpsi,k}}  + \brk{\chi_k \nabla\times ((\nabla\times b)\times b), \tb_k}.
\end{split}
\end{equation*} Taking absolute values, the terms containing $\err_{\tb^z,k}$ and $\err_{\psi,k}$ are bounded by \begin{equation*} 
\begin{split}
\lesssim \nrm{\chi_k b }_{L^2}\left( \nrm{\err_{\tb^z,k}}_{L^2} + \nrm{  \nabla\err_{\tpsi,k}}_{L^2}\right). 
\end{split}
\end{equation*} Next, the terms involving a $z$-derivative vanish after moving the derivative to the other side, from $z$-independence of $\tb_k$. This leaves us with \begin{equation*} 
\begin{split}
I:= &\brk{\chi_k f(y)\rd_x(\rd_yb^x-\rd_xb^y) + \chi_k f''b^y  ,\tb^z_k } +  \brk{\chi_k b^z, f\rd_x\lap_{x,y}\tpsi_k -  f''\rd_x\tpsi_k } \\
&\quad + \brk{-\chi_k f\rd_{xx}b^z  ,  \rd_x\tpsi_k } + \brk{-\chi_k b^y, - f \rd_{xx}\tb^z_k } + \brk{-\chi_k \rd_y(f \rd_x b^z)    , \rd_y \tpsi_k  } + \brk{ \chi_k b^x, -\rd_y (f\rd_x\tb^z_k) }
\end{split}
\end{equation*} and  \begin{equation*} 
\begin{split}
II:= &\brk{\chi_k \nabla\times ((\nabla\times b)\times b), \tb_k}. 
\end{split}
\end{equation*} After observing cancellations, we see that \begin{equation*} 
\begin{split}
|I| &= \left|2\brk{b^x, \chi_k' \rd_x\tb^z_k} + 2\brk{\chi'_k\rd_y\tpsi_k, f \rd_x b^z} + \brk{\chi_k f'' b^y,\tb^z_k} - \brk{\chi_k b^z,f''\rd_x\tpsi_k} \right|\lesssim \nrm{\chi_kb}_{L^2} \nrm{\tb_k}_{L^2} 
\end{split}
\end{equation*} where we have used the fact that $\tb_k$ vanishes on the support of $\chi'_k$. Finally, using that $\chi'_k \tb_k  \equiv 0$ and $\chi_k\tb_k = \chi_k^2\tb_k$, we bound \begin{equation*}
\begin{split}
\left|II \right| &\lesssim \nrm{\nabla^2 b}_{L^\infty} \nrm{\chi_kb}_{L^2}\nrm{\tb_k}_{L^2} + \nrm{\nabla(\chi_k b)}_{L^4}^2\nrm{\tb_k}_{L^2} \\ 
&\lesssim_s 2^{sk}\nrm{b}_{H^s} \nrm{\chi_kb}_{L^2}\nrm{\tb_k}_{L^2}. 
\end{split}
\end{equation*}  We have therefore arrived at the following inequality: \begin{equation}\label{eq:gei-local}
\begin{split}
\left| \frac{\ud}{\ud t} \brk{\chi_k b, \tb_k} \right| &\lesssim_s 2^{ks}(1 + \nrm{b}_{H^{s}})\nrm{ \chi_k b }_{L^2}\left(  \nrm{\err_{\tb^z,k}}_{L^2} +  \nrm{ \nabla\err_{\tpsi,k}}_{L^2} + \nrm{ \tb_k}_{L^2} \right) \\
&\lesssim_s 2^{C_{s} k} \nrm{\chi_kb}_{L^2} ,
\end{split}
\end{equation} where we have used  {\eqref{eq:bdd-k} and \eqref{eq:error-k}}; the final constant depends also on $\nrm{b}_{L^\infty_tH^s}$.

We are in a position to complete the proof. Combining \eqref{eq:localenergy-time} with \eqref{eq:gei-local} gives \begin{equation*} 
\begin{split}
\left|\brk{\chi_k b, \tb_{k}(t)} - \lmb_k^{-s}\right| \lesssim_s   {2^{C_{s} k}} \int_0^t \left(  {2^{-k}} \lmb_k^{-s} + 2^{2ks}t'^s \right) \ud t' \lesssim_s  {2^{C_{s} k}}( {2^{-k}} \lmb_k^{-s}t + 2^{2ks}t^{s+1}), 
\end{split}
\end{equation*} so that we are able to obtain \begin{equation*} 
\begin{split}
\brk{\chi_k b, \tb_{(\lmb_k)}}(t) \ge \frac{1}{2}  {2^{-k}} \lmb_k^{-s}
\end{split}
\end{equation*} on the time interval $[0,t^*_k]$ with 
\begin{equation} \label{eq:t-ast-k}
 {t^{*}_{k} = 2^{-c_{s} k} \lmb_{k}^{-\frac{s}{s+1}}}
\end{equation}
where $\lmb_k = 2^{Nk}$ with $N = N(s) \gg 1$ and $k_{0}$ is sufficiently large.
Taking $k_0 $ larger if necessary, it is easy to guarantee that $t^*_k \le \delta$ for all $k \ge k_0$. 

Next, using interpolation in $y$, we have  
\begin{equation*}
	2^{-k} \lmb_{k}^{-s} \aleq \nrm{\tb_{(\lmb_{k})}(t^{*}_{k})}_{L^{2}_{x} H^{-\frac{1}{4}}_{y}} \nrm{\chi_{k} b(t^{*}_{k})}_{L^{2}_{x} H^{\frac{1}{4}}_{y}} \aleq \nrm{\tb_{(\lmb_{k})}(t^{*}_{k})}_{L^{2}_{x} H^{-\frac{1}{4}}_{y}} \nrm{\chi_{k}b(t^{*}_{k})}_{L^{2}}^{1-\frac{1}{4 s}} \nrm{\chi_{k} b(t^{*}_{k})}_{H^{s}}^{\frac{1}{4 s}}.
\end{equation*}
By the degeneration property \eqref{eq:degen-k}, \eqref{eq:localenergy-time} and \eqref{eq:t-ast-k}, it follows that
\begin{equation*}
	\nrm{\chi_{k} b(t^{*}_{k})}_{H^{s}} 
	\ageq_{s} 2^{-C_{s} k} \lmb_{k}^{- 3\frac{s^{2}}{s+1}} \exp\left( 2^{-c_{s} k} \lmb_{k}^{\frac{1}{s+1}}\right).
\end{equation*}
By the algebra property of $H^{s}$, we may replace the LHS by $\nrm{b(t^{*}_{k})}_{H^{s}}$ by altering $C_{s}'$ on the RHS. Now recall that $\lmb_{k} = 2^{Nk}$; thus by taking $N = N(s)$ sufficiently large, we may ensure that
\begin{equation*}
	\nrm{b(t^{*}_{k})}_{H^{s}} \ageq_{s} 2^{c_{s} k}
\end{equation*}
for some $c_{s} > 0$ independent of $k \geq k_{0}$. This clearly contradicts boundedness of $\nrm{b}_{L^\infty_t H^s}$.  \qedsymbol

\subsection{Proof of Theorem~\ref{thm:illposed-strong2} for \eqref{eq:hall-mhd}} \label{subsec:illposedness-strong-hall}
 {Here we indicate the necessary modifications for the case of \eqref{eq:hall-mhd}. 
In this case, the energy identity for $\bfu$ does not obey as favorable localization properties as in \eqref{eq:e-mhd} due to the pressure (see \eqref{eq:p-si}). Instead, we require $M$ to be noncompact (more specifically, $(\bbT, \bbR)_{y} = \bbR_{y}$) and place the instabilities at dyadic loci $y_{k} \aeq \mu^{k}$. 

\medskip

\textit{(i) choice of initial data and contradiction hypothesis}

\medskip}

Again,  {two constructions using} the translation- and axi-symmetric  {building blocks can be described} almost simultaneously.  {We borrow the definitions of $\widetilde{\bfB}^{(tr, axi)}$, $f_{0}^{(tr, axi)}$ and $g_{0}^{(tr, axi)}$ from the previous proof.} In the Hall-MHD case, the stationary magnetic field is taken  {to be
\begin{equation*}
\mathring{\mathbf{B}}^{(tr,axi)} = \sum_{k = k_0}^{\infty}  {\bgB}_k^{(tr,axi)} := \sum_{k = k_0}^{\infty}  {2^{-k}} \widetilde{\mathbf{B}}^{(tr,axi)}(x , y-y_k ), \quad y_k = y_{k-1} + \mu^k,
\end{equation*}
where $y_{1} = 1$ and $\mu \gg 1$ depending only on $N$ and $s$.
Compared to  \eqref{eq:illposed-strong2-stat}, note that there are no spatial rescalings, and the requirement that $(\bbT, \bbR)_{y} = \bbR_{y}$ is used to justify the choices of $y_{k}$.}

 {Next, we apply Proposition~\ref{prop:wavepackets-hall} for $g_{0} = g_{0}^{(tr, axi)}$ and $\bgB = \widetilde{\bfB}^{(tr, axi)}$ (with $y_{1} = \frac{1}{20}$ in the translationally symmetric case and $(r_{0}, r_{1}) = (\frac{1}{20}, \frac{1}{10})$ in the axi-symmetric case), from which we obtain
$\tb_{(\lmb)} = (\nb^{\perp} \tpsi_{(\lmb)}, \tb^{z}_{(\lmb)})$ and $\tu_{(\lmb)} = (\nb^{\perp}(-\lap)^{-1} \tb^{z}_{(\lmb)}, -\tpsi_{(\lmb)})$. For $\lmb_{k} = 2^{Nk}$ with $N$ to be chosen later, we set 
\begin{align*}
	\tb_{k}^{(tr, axi)}(t, x, y, z) = \tb_{(\lmb_{k})}(2^{-k} t, x, y-y_{k}), \quad
	\tpsi_{k}^{(tr, axi)}(t, x, y, z) = \tpsi_{(\lmb_{k})}(2^{-k} t, x, y-y_{k}).
\end{align*}
By construction, it may be checked that $(\tb^{(tr, axi)}_{k}, \tu^{(tr, axi)}_{k})$ obeys all estimates claimed in Propositions~\ref{prop:wavepackets} and \ref{prop:wavepackets-hall} (formulated in terms of $\bgB^{(tr, axi)}_{k}$, $(\tb^{z}_{k})^{(tr, axi)}$ and $\tpsi_{k}^{(tr, axi)}$) with implicit constants of size $O(2^{Ck})$.

We take as the initial data 
\begin{equation*}
	\bfB_{0}^{(tr, axi)} = \bgB^{(tr, axi)} + \sum_{k=k_{0}}^{\infty} 2^{-k} \lmb_{k}^{-s} \tb_{k}^{(tr, axi)}(t=0), \quad \bfu_{0}^{(tr, axi)} = 0.
\end{equation*}
Clearly, $(\bfu_{0}, \bfB_{0}) \in H^{s-1}_{comp} \times H^{s}_{comp}$, and its $H^{s-1} \times H^{s}$ can be smaller than $\eps > 0$ by taking $k_{0}$ sufficiently large.
}

Towards contradiction, assume there exists a solution $(\mathbf{u},\mathbf{B}) \in L^\infty_t([0,\delta];H^{s-1} \times H^s)$ for some $\delta > 0$ with initial data $(\mathbf{u}_0, \mathbf{B}_0)$ with $\mathbf{u}_0 = 0$ and $\mathbf{B}_0$ is defined as in \eqref{eq:illposed-strong2-id-tr}. Set $b(t) = \mathbf{B}(t) - \bgB$. 
 {The system of equations for $\mathbf{u}$ and $b$ are:} \begin{equation}\label{eq:hall-mhd-pert-vel}
\begin{split}
\rd_t\bfu+ \bfu\cdot\nabla\bfu +\nabla\bfp - \nu\lap\bfu  = (\nabla\times \bfB)\times \bfB 
\end{split}
\end{equation} and \begin{equation}\label{eq:hall-mhd-pert-mag}
\begin{split}
&\rd_t b + (b\cdot\nabla)(\nabla\times\bgB) - (\nabla\times\bgB)\cdot\nabla b + (\bgB\cdot\nabla)(\nabla\times b) - (\nabla\times b)\cdot\nabla\bgB  \\
&\quad = \nabla \times ((\nabla\times b)\times b) + \bfB\cdot\nabla \mathbf{u} - \mathbf{u} \cdot\nabla \bfB . 
\end{split}
\end{equation} 

\medskip

 {\textit{(ii) localization of the energy identity}}

\medskip

 {For simplicity, we proceed in the case of translationally symmetric case, and leave the similar axi-symmetric case to the reader.} 

 {Let us first obtain a simple $L^2$-estimate for $\bfp$. Recall from \eqref{eq:p-grad} that $\bfp$ has been only fixed up to a constant; we fix this ambiguity by \emph{defining}
$\bfp$ as \begin{equation} \label{eq:p-si}
\begin{split}
\bfp = \sum_{i,j} R_{i} R_{j} (\bfu^i\bfu^j) - \sum_{i,j} R_{i} R_{j} (\bfB^i\bfB^j) - \frac{|\bfB|^2}{2}.
\end{split}
\end{equation} Then} we obtain\footnote{In fact, the same estimate justifies the choice of $\bfp$ as above for our solution.} \begin{equation*}
\begin{split}
\nrm{\bfp}_{L^2} \lesssim \nrm{|\bfu|^2}_{L^2} + \nrm{|\bfB|^2}_{L^2} \lesssim \nrm{\bfu}_{H^{s-1}}^2 + \nrm{\bfB}_{H^s}^2 
\end{split}
\end{equation*} using the embeddings $\nrm{|\bfu|^2}_{L^2} \lesssim \nrm{\bfu}_{L^\infty}\nrm{\bfu}_{L^2} \lesssim \nrm{\bfu}_{H^{s-1}}^2$ and similarly for $\bfB$.

 {We now introduce the cutoff functions. This time, we fix some smooth function $\chi(y) \ge 0$ supported on $[-1,1]$, $\chi(y) = 1$ on $[-1/2,1/2]$ and define \begin{equation*}
\begin{split}
\chi_k(y) = \chi(2\mu^{-k}(y-y_k)).
\end{split}
\end{equation*} We have that $\chi_k$ is supported on $[y_k-\mu^k/2,y_k+\mu^k/2]$, and $|\chi_k'(y)| \lesssim \mu^{-k}$.}

 {W}e multiply both sides of \eqref{eq:hall-mhd-pert-vel} by $\chi_k$ and take the $L^2$ inner product with $\chi_k {u}$. We handle the  {RHS} as \begin{equation*}
\begin{split}
\brk{\chi_k(\nabla\times \bfB)\times \bfB , \chi_k \bfu }&=\brk{\chi_k \bfB \cdot\nabla \bfB ,\chi_k \bfu } - \brk{\chi_k \nabla \frac{|\bfB|^2}{2}, \chi_k\bfu }\\
& = \brk{\chi_k(\bfB\cdot\nabla)b,\chi_k\bfu}   + \brk{|\bfB|^2\nabla\chi_k ,\chi_k \bfu},
\end{split}
\end{equation*} where we have used that $\bgB \cdot\nabla \bgB = 0$ and an integration by parts. Applying integration by parts to the other terms, we obtain \begin{equation}\label{eq:lei-1}
\begin{split}
&\left| \frac{1}{2} \frac{\ud}{\ud t}\nrm{\chi_k \bfu}_{L^2}^2 -  \brk{\chi_k(\bfB\cdot\nabla)b,\chi_k\bfu}   \right| \\
&\qquad \lesssim \nrm{\chi_k'}_{L^\infty} \nrm{\chi_k \bfu}_{L^2} \left( \nrm{|\bfu|^2}_{L^2} + \nrm{\bfp}_{L^2} + \nrm{|\bfB|^2}_{L^2} + \nu\nrm{\nabla\bfu}_{L^2} \right)   \\
& \qquad \lesssim  \nrm{\chi_k'}_{L^\infty}    \nrm{\chi_k \bfu}_{L^2} 
\end{split}
\end{equation} on $t \in [0,\delta]$ with a constant depending on $\nu \ge  0$ and the norm of $(\bfu,\bfB)$ in $L^\infty_t(H^{s-1}\times H^s)$. We now multiply both sides of \eqref{eq:hall-mhd-pert-mag} by $\chi_k$ and take the $L^2$-inner product with $\chi_kb$. We proceed similarly as in the \eqref{eq:e-mhd} case, except that we simply use the quantity $\nrm{\chi_k'}_{L^\infty}$ whenever a derivative falls on $\chi_k$. Then we obtain this time \begin{equation}\label{eq:lei-2}
\begin{split}
 \left| \frac{1}{2} \frac{\ud}{\ud t}\nrm{\chi_k b}_{L^2}^2 - \brk{\chi_k (\bfB\cdot\nabla)\bfu , \chi_k b} \right|  &\lesssim \nrm{\nabla^2\bgB}_{L^\infty}\nrm{\chi_kb}_{L^2}^2  + \nrm{\chi_k'}_{L^\infty}\nrm{\nabla b}_{L^2} \nrm{\chi_k b}_{L^2} \\
&\qquad  + \nrm{\chi_k'}_{L^\infty}\nrm{b\nabla b}_{L^2}\nrm{\chi_kb}_{L^2} + \nrm{\nabla\bfB}_{L^\infty}\nrm{\chi_k\bfu}_{L^2}\nrm{\chi_kb}_{L^2} \\
&\lesssim \left( \nrm{\chi_k'}_{L^\infty} + \nrm{\chi_k b}_{L^2} + \nrm{\chi_k\bfu}_{L^2} \right) \nrm{\chi_kb}_{L^2} ,
\end{split}
\end{equation} where the last implicit constant depends on $\nrm{\bfB}_{L^\infty_tH^s}$. Then, putting \eqref{eq:lei-1} and \eqref{eq:lei-2} together and applying integration by parts, we have \begin{equation*}
\begin{split}
\left| -  \brk{\chi_k(\bfB\cdot\nabla)b,\chi_k\bfu}  - \brk{\chi_k (\bfB\cdot\nabla)\bfu , \chi_k b} \right| \lesssim \nrm{\chi_k'}_{L^\infty}\nrm{|\bfB|b}_{L^2}\nrm{\chi_k\bfu}_{L^2}
\end{split} 
\end{equation*} and hence \begin{equation*}
\begin{split}
\left| \frac{\ud}{\ud t}(\nrm{\chi_k\bfu}_{L^2}^2 + \nrm{\chi_kb}_{L^2}^2) \right| \lesssim \nrm{\chi_k'}_{L^\infty}    \nrm{\chi_k \bfu}_{L^2} +\left( \nrm{\chi_k'}_{L^\infty} + \nrm{\chi_k b}_{L^2} + \nrm{\chi_k\bfu}_{L^2} \right) \nrm{\chi_kb}_{L^2}. 
\end{split}
\end{equation*} Note that $\nrm{\chi_k b(t=0)}_{L^2} =  {2^{-k}} \lmb_k^{-s} =  {2^{-(Ns+1)k}}$. We now choose $\mu \gg 1 $ in a way that (depending only on $N$ and $s$) $\nrm{\chi_k'}_{ {L^{\infty}}}\lesssim \mu^{-k} \lesssim  {2^{-(Ns+1)k}}$. Then using Gronwall's inequality, we obtain, for $t \in [0,\delta]$, \begin{equation} \label{eq:loc-en-hall}
\begin{split}
  \nrm{\chi_k\bfu {(t)}}_{L^2}^2 + \nrm{\chi_k b {(t)}}_{L^2}^2 \lesssim  {2^{-2k}} \lmb_k^{-2s}  {= \nrm{\chi_{k} b(t=0)}_{L^{2}}^{2}}
\end{split}
\end{equation} with  {an implicit} constant independent of $k$. 

\medskip

 {\textit{(iii) localization of the generalized energy identity and conclusion of the proof}}

\medskip

Using the same cutoff $\chi_k$ as in the previous step, it is straightforward to obtain a localized version of the generalized energy inequality in this case: taking the $L^2$-inner product with the  {degenerating wave packet solution $(\tb_k, \tu_{k})$ from step~(i), we may prove} \begin{equation*}
\begin{split}
 {\left|\frac{\ud}{\ud t} \brk{\chi_kb, \tb_k} + \frac{\ud}{\ud t} \brk{\chi_k \bfu, \tu_k}\right| \lesssim_s 2^{C_{s} k} \nrm{\chi_k b(t=0)}_{L^2} .}
\end{split}
\end{equation*}
The rest of the argument is parallel with the  {\eqref{eq:e-mhd}} case,  {using interpolation and the degeneration property. In fact, the proof in this case is simpler since the energy bound \eqref{eq:loc-en-hall} is stronger and valid for a longer time. We omit the straightforward details.} 
 \qedsymbol

\section{Proof of Gevrey space illposedness} \label{sec:gevrey}
\subsection{Reduction to construction of degenerating wave packets}
We shall treat the electron- and Hall-MHD cases simultaneously.   The main step of the proof is a version of Proposition \ref{prop:wavepackets} and Proposition \ref{prop:wavepackets-hall} applicable for Gevrey (in particular, analytic) class of data. Note that the statement of Proposition \ref{prop:wavepackets-Gevrey} is essentially the same with Propositions \ref{prop:wavepackets} and \ref{prop:wavepackets-hall} except for the form of the initial data and the phase is now taken to be $e^{i\lmb(x+y)}$. For simplicity, we restrict ourselves to the $x$-independent case. This allows  {us to only consider} degenerating wave packets  {that} are pure functions of $y$, modulo the phase $e^{i\lmb x}$.  

\begin{proposition}\label{prop:wavepackets-Gevrey}
	Let $\bgB = f(y)\rd_x$ as in Theorem~\ref{thm:illposed-gevrey}, and assume without loss of generality that $f(0) = 0$ and $f'(0) > 0$. Let $g_0(y) \in C^\infty(\mathbb{T})$ be a complex-valued function such that $f^{-1}g_0 \in C^\infty(\mathbb{T})$ as well. Assume further that $g_0$ is supported in $[-y_1,y_1]$ for some $y_1 > 0$. For any such $g_0$ and $\lmb \in \mathbb{N}_0$, we may associated a pair $(\tb^{z}_{(\lmb)}, \tpsi_{(\lmb)})[g_{0}]$ satisfying the following properties:
	\begin{itemize}
		\item (linearity) the map $g_{0} \mapsto (\tb^{z}_{(\lmb)}, \tpsi_{(\lmb)})[g_{0}]$ is   linear;
		\item ($x$-separation) for all $t$, $e^{-i\lmb x}\tb^{z}_{(\lmb)}$ and $e^{-i\lmb x}\tpsi_{(\lmb)}$ are functions of $y$ only; 
		\item (initial data) at $t =0$, we have 
		\begin{equation}\label{eq:initial-bz-gev}
		\begin{split}	
		&\tb^{z}_{(\lmb)}(0) =  -e^{i\lmb(x+y)}\left( \sqrt{2}g_0 + \frac{1}{\sqrt{2}i\lmb}(\rd_yg_0 - \frac{1}{2}f^{-1}\rd_y f g_0)    \right),
		\end{split}
		\end{equation}
		\begin{equation}\label{eq:initial-psi-gev} 
		\tpsi_{(\lmb)}(0) = \lmb^{-1}  e^{i \lmb (x +y)} g_{0} ,
		\end{equation} and 
		\begin{equation*}
		\nrm{\Re [\tb^{z}_{(\lmb)}(0) ]}_{L^{2}} + \nrm{\Re [\nb \tpsi_{(\lmb)}(0) ]}_{L^{2}}
		\geq c \nrm{g_{0}}_{L^{2}} - C \lmb^{-1} \nrm{g_{0}}_{H^{1}};
		\end{equation*}
		
		\item (regularity estimates) for any $m \in \bbN_{0}$ and $t \in [0, 1]$,
		\begin{align*}
		\sup_{0 \leq k  \leq m} 
		\nrm{(\lmb^{-2} \rd_{t})^{k} (\lmb^{-1} f \rd_{y})^{m - k } \tb^{z}_{(\lmb)}(t)}_{L^{2}} 
		\aleq & \nrm{g_{0}}_{H^{m+1}}, \\
		\sup_{0 \leq k \leq m} \nrm{(\lmb^{-2} \rd_{t})^{k} (\lmb^{-1} f \rd_{y})^{m - k} \nb \tpsi_{(\lmb)}(t)}_{L^{2}} 
		\aleq & \nrm{g_{0}}_{H^{m+1}};
		\end{align*}
		\item ($L^{p}$-degeneration) for any $1 \leq p \leq 2$ and $t \in [0, 1]$, with some $c_f > 0$,
		\begin{equation*}
		\nrm{\tb^{z}_{(\lmb)}(t)}_{L^{2}_{x} L^{p}_{y}}
		+ \nrm{\nb \tpsi_{(\lmb)}(t)}_{L^{2}_{x} L^{p}_{y}}
		\aleq e^{- c_{f} (\frac{1}{p} - \frac{1}{2})\lmb t} \nrm{ g_{0} }_{H^{1}};
		\end{equation*}
		\item (error bounds) for $t \in [0, 1]$, $\err_{\psi}[\tb^{z}_{(\lmb)}, \tpsi_{(\lmb)}](t) = 0$ and
		\begin{equation*}
		\nrm{\err_{b}[\tb^{z}_{(\lmb)}, \tpsi_{(\lmb)}](t)}_{L^{2}} \aleq \nrm{g_{0} }_{H^{2}}.
		\end{equation*}
	\end{itemize} 
	
	In the case of Hall-MHD, in addition to $(\tb^{z}_{(\lmb)}, \tpsi_{(\lmb)})$, we take
	\begin{equation} \label{eq:wavepackets-hall-u-omg-Gev}
	\tu^{z}_{(\lmb)}[g_{0}] = - \tpsi_{(\lmb)}[g_{0}], \quad \tomg_{(\lmb)}[g_{0}] = -\tb^{z}_{(\lmb)}[g_{0}] 
	\end{equation}
	and then we have 
	\begin{itemize}
		\item (smoothing for fluid components) for $t \in [0, 1]$, we have
		\begin{align*}
		\nrm{\tu^{z}_{(\lmb)}(t)}_{L^{2}} + \nrm{\nb^{\perp} (-\lap)^{-1} \tomg_{(\lmb)}(t)}_{L^{2}} \aleq & \lmb^{-1} \nrm{ g_{0} }_{H^{1}} ,\\
		\nrm{\nb \tu^{z}_{(\lmb)}(t)}_{L^{2}} + \nrm{\tomg_{(\lmb)}(t)}_{L^{2}} \aleq & \nrm{ g_{0} }_{H^{1}} ;
		\end{align*}
		\item (error estimates) for $t \in [0, 1]$, we have
		\begin{align*}
		\errh_{u}^{(\nu)}[\tu^{z}_{(\lmb)}, \tomg_{(\lmb)}, \tb^{z}_{(\lmb)}, \tpsi_{(\lmb)}] + \nu \lap \tpsi = & 0, \\
		\nrm{\nb^{\perp} (-\lap)^{-1} (\errh_{\omg}^{(\nu)}[\tu^{z}_{(\lmb)}, \tomg_{(\lmb)}, \tb^{z}_{(\lmb)}, \tpsi_{(\lmb)}] + \nu \lap \tomg)(t)}_{L^{2}} \aleq & \lmb^{-1} \nrm{g_0}_{H^{2}}, \\
		\nrm{\errh_{b}^{(\nu)}[\tu^{z}_{(\lmb)}, \tomg_{(\lmb)}, \tb^{z}_{(\lmb)}, \tpsi_{(\lmb)}](t)}_{L^{2}} \aleq & \nrm{g_0}_{H^{2}}, \\
		\nrm{\nb \errh_{\psi}^{(\nu)}[\tu^{z}_{(\lmb)}, \tomg_{(\lmb)}, \tb^{z}_{(\lmb)}, \tpsi_{(\lmb)}](t)}_{L^{2}} \aleq & \nrm{g_{0} }_{H^{2}}.
		\end{align*}
	\end{itemize}
\end{proposition}

\begin{remark}
	Note that $g_0$ itself cannot belong to all Gevrey classes (in particular, analytic) since it has compact support. In the proof below, we shall take some Gevrey class function $\tilde{g}_0$ which is supported near $y = 0$ and truncate it to obtain $g_0$. 
\end{remark}

\begin{remark}
	While stronger degeneration properties on par with \eqref{eq:wavepackets-degen-upper}--\eqref{eq:wavepackets-degen-small} are expected to hold, in order to reduce the amount of technicality, we chose to state and prove only the simpler $L^{p}$-degeneration property that is sufficient for the proof of Theorem~\ref{thm:illposed-gevrey}.
\end{remark}

We now give the proof of Theorem~\ref{thm:illposed-gevrey} assuming the above statement. 

\begin{proof}[Proof of Theorem~\ref{thm:illposed-gevrey} from Proposition \ref{prop:wavepackets-Gevrey}]
	We first consider \eqref{eq:e-mhd-2.5d-lin}, the linearized electron-MHD equations at $\bgB= f(y)\rd_x$, with data having a single frequency in $x$. We may take some function $\tilde{g}_0(y)$ such that $\tilde{g}_0$ and $f^{-1}\tilde{g}_0$ belong to $G^\sigma$ given $\sigma > 0$, since we have assumed that $f(y) \in G^\sigma$ and $G^\sigma$ is closed under multiplication. 
	
	Take some $y_1 > 0$ and let $\chi$ be a smooth bump function in $y$ with $\chi = 1$ on $|y| \le y_1/2$ and vanishes for $|y| > y_1$. We could have assumed that the support of $\tilde{g}_0$ contains the interval $[-y_1/2,y_1/2]$. Take $g_0 := \tilde{g} \chi$ and define \begin{equation*} 
	\begin{split}
	\tilde{b}_{(\lmb)}(0) = (\rd_y \tilde{\psi}_{(\lmb)}(0), -\rd_x\tilde{\psi}_{(\lmb)}(0), \tb_{(\lmb)}^z(0))
	\end{split}
	\end{equation*} with $(\tb_{(\lmb)}^z, \tpsi_{(\lmb)})[g_0]$ which is provided by Proposition \ref{prop:wavepackets-Gevrey}. On the other hand, define the initial data $b_{(\lmb)}(0)$ from \begin{equation*} 
	\begin{split}
	b^z_{(\lmb)}(0) := -e^{i\lmb(x+y)}\sqrt{2}g_0, \quad \psi_{(\lmb)}(0) := \lmb^{-1}e^{i\lmb(x+y)}g_0 
	\end{split}
	\end{equation*} which clearly belongs to $G^\sigma$. We may take the real parts of $\tb_{(\lmb)}(0), b_{(\lmb)}(0)$ to ensure that the data are real-valued, and further normalize the $L^2$-norm of $\tb_{(\lmb)}(0)$ by 1. We then proceed as in the proof of Theorem \ref{thm:norm-growth}: denoting the unique solution of \eqref{eq:e-mhd-2.5d-lin} with initial data $b_{(\lmb)}(0)$ by $b_{(\lmb)}(t)$, we have \begin{equation*} 
	\begin{split}
	\left|\brk{\tilde{b}_{(\lmb)}, b_{(\lmb)}}(t) - \brk{\tilde{b}_{(\lmb)}, b_{(\lmb)}}(0)\right| & \lesssim t \nrm{b_{(\lmb)}}_{L^\infty([0,t];L^2)} \lesssim t \nrm{b_{(\lmb)}(0)}_{L^2}.
	\end{split}
	\end{equation*} On the other hand, \begin{equation*} 
	\begin{split}
	\brk{\tilde{b}_{(\lmb)}, b_{(\lmb)}}(0) \gtrsim \nrm{b_{(\lmb)}(0)}_{L^2}
	\end{split}
	\end{equation*} independently of $\lmb$ so that for sufficiently small $\delta > 0$, we obtain \begin{equation} \label{eq:gevrey-b-tb-lb} 
	\begin{split}
	\brk{\tilde{b}_{(\lmb)}, b_{(\lmb)}}(t) \gtrsim \nrm{b_{(\lmb)}(0)}_{L^2},\quad t \in [0,\delta]
	\end{split}
	\end{equation} again with a constant independent of $\lmb$. 
	
	 {We now claim that there exists $c_{\ast} \in \bbR$ such that for each positive integer $n$,
\begin{equation} \label{eq:dn-b-growth}
	\begin{split}
	\nrm{\rd_y^n b_{(\lmb)}(t)}_{L^2} \gtrsim e^{(c_{\ast} + c_f \lmb t) n} \nrm{b_{(\lmb)}(0)}_{L^2},
	\end{split}
	\end{equation} 
	We emphasize that the implicit constant is independent of $n$. When $n = 1$, the claim follows by \eqref{eq:gevrey-b-tb-lb}, the degeneration property and the Sobolev inequality as before. Next, for any $n \geq 1$, we have
	\begin{equation*}
	\nrm{\rd_{y} b_{(\lmb)}}_{L^{2}} \leq \nrm{b_{(\lmb)}}_{L^{2}}^{\frac{n-1}{n}} \nrm{\rd_{y}^{n} b_{(\lmb)}}_{L^{2}}^{\frac{1}{n}}
\end{equation*}
	which can be seen easily by using the Fourier transform. Then \eqref{eq:dn-b-growth} follows from the case $n = 1$ and the bound $\sup_{t \in [0, \dlt]} \nrm{b_{(\lmb)}(t)}_{L^{2}} \leq e^{c_{\ast}} \nrm{b_{(\lmb)}(0)}_{L^{2}}$ for some $c_{\ast} \in \bbR$ independent of $\lmb$.
	}
	
	 {To conclude the proof in the electron-MHD case,} we consider initial data of the form \begin{equation*} 
	\begin{split}
	b(0) = \sum_{\lmb \in \mathbb{N} } c_\lmb \tilde{b}_{(\lmb)}(0) 
	\end{split}
	\end{equation*} where we normalize each $\tilde{b}_{(\lmb)}(0) $ in $L^2$ and $c_\lmb = e^{-\lmb^{1/\sigma}}$. This ensures that the initial data belongs to $G^\sigma(\mathbb{T}^3)$. Again by the assumption of uniqueness, we deduce that the solution satisfies \begin{equation} \label{eq:gevrey-key-lb} 
	\begin{split}
	&\nrm{\rd_y^n b(t)}_{L^2(\mathbb{T}^3)}^2 \gtrsim \sum_{\lmb \in \mathbb{N}} c_\lmb^2 e^{ {(c_{\ast} + 2c_f \lmb t) n }} = \sum_{\lmb \in \mathbb{N}} e^{-2\lmb^{1/\sigma}  +  {2(c_{\ast} + c_f \lmb t) n }}.
	\end{split}
	\end{equation} 
{At this point, we divide the argument into two cases:
\pfstep{Case~1: $\sgm \geq 1$} When $\sgm \geq 1$, this series simply does not converge for any $t > 0$ for $n$ large depending on $t, \sigma$, which concludes the proof.
\medskip

\pfstep{Case~2: $0 < \sgm < 1$} Fix a small parameter $0 < \eps \ll \frac{\sgm^{2}}{1-\sgm}$. Since each summand on the RHS of \eqref{eq:gevrey-key-lb} is nonnegative, we may show that, for sufficiently large $n$ depending on $c_{f}$, $c_{\ast}$ and $t$,
\begin{equation} \label{eq:gevrey-n-lb}
	\nrm{\rd_{y}^{n} b(t)}_{L^{2}(\bbT^{3})} \ageq e^{c_{\sgm} (c_{f} t)^{\frac{1}{1-\sgm}} n^{\frac{1}{1-\sgm}} + c_{\ast} n}, \quad c_{\sgm} = \sgm^{\frac{\sgm}{1-\sgm}} - \sgm^{\frac{1}{1-\sgm}},
\end{equation}
by keeping only the summand with $\lmb = \lfloor (\sgm c_{f} t)^{\frac{\sgm}{1-\sgm}} n^{\frac{\sgm}{1-\sgm}} \rfloor$.\footnote{The choice of $\lmb$ is motivated by the Laplace method for deriving asymptotics of an exponential integral.}
Observe the crucial properties that $\frac{1}{1- \sgm} > 1$ and $c_{\sgm} > 0$, since $0 < \sgm < 1$. Recalling \eqref{eq:gevrey-rad} and the crude  bound $n! \lesssim  e^{n \log n }$, we see that \eqref{eq:gevrey-n-lb} implies that the $G^{\sgm'}$ radius of convergence of $b(t)$ is zero (i.e., $b(t) \not \in G^{\sgm'}$) for every $\sgm' > 0$. This finishes the proof for the electon-MHD case.

}
	
	We now indicate the necessary modifications for the Hall-MHD case. In addition to the initial data $b_{(\lmb)}(0)$ defined above, we simply take $u_{(\lmb)}(0) = 0$. Let $(u_{(\lmb)}, b_{(\lmb)})$ be the solution with  {the initial data} $(0, b_{(\lmb)}(0))$ which exists by assumption. Then, applying the generalized energy inequality from Proposition \ref{prop:justify} with the degenerating wave packet from Proposition \ref{prop:wavepackets-Gevrey} yields \begin{equation*}
	\begin{split}
	&\left| \brk{\tilde{b}_{(\lmb)}, b_{(\lmb)}}(t) - \brk{\tilde{b}_{(\lmb)}, b_{(\lmb)}}(0) \right| \\
	&\qquad \lesssim ((1+\nu)t^{1/2} + \lmb^{-1})\left(   \nrm{b_{(\lmb)}}_{L^\infty(I;L^2)} +\nrm{u_{(\lmb)}}_{L^\infty(I;L^2)} + \nrm{u_{(\lmb)}}_{L^2(I;\dot{H}^1)} \right) 
	\end{split}
	\end{equation*}   applying the error bounds together with the smoothing estimates from Proposition \ref{prop:wavepackets-Gevrey}. Thus, choosing $0 < \dlt \le 1$ sufficiently small, we obtain for all sufficiently large $\lmb$ that 
	\begin{equation*} 
	\brk{\tilde{b}_{(\lmb)}, b_{(\lmb)}}(t) \gtrsim \nrm{b_{(\lmb)}(0)}_{L^2} \quad \hbox{ for } t \in [0, \dlt].
	\end{equation*} The rest of the argument is the same with the electron-MHD case. \end{proof}

\subsection{Analysis of the Hamilton-Jacobi equation}

The heart of the matter in establishing Proposition \ref{prop:wavepackets-Gevrey} is to repeat the WKB-type analysis for the phase function which is simply given initially by $y$. That is, we seek a solution of \eqref{eq:wkb_phi} with initial data $\Phi(0,\eta) = y(\eta)$. Before we proceed, we recall the renormalized form of \eqref{eq:e-mhd-2.5d-lin}: after change of variables \begin{equation*} 
\begin{split}
\tau = \lmb t, \quad \eta'(y) = \frac{1}{f(y)}, \quad \varphi = f^{-\frac{1}{2}}\psi =: e^{i\lmb x} \phi 
\end{split}
\end{equation*} we arrive at \begin{equation}\label{eq:e-mhd-eta-conj3}
\begin{split}
\rd_\tau^2\phi - \rd_\eta^2 \phi + \lmb^2 f^2\phi + \left[\frac{1}{2}\rd_\eta ( f^{-1} \rd_{\eta} f) + \frac{1}{4} f^{-2} (\rd_{\eta} f)^{2} \right]\phi = 0. 
\end{split}
\end{equation} The ansatz for $\phi$ will be \begin{equation*} 
\begin{split}
\phi(\tau,\eta) = \lmb^{-1}e^{i\lmb\Phi(\tau,\eta)} h(\tau,\eta) 
\end{split}
\end{equation*} where $h_0 = h(0,\cdot)$ is simply given by $f^{-\frac{1}{2}}g_0$ with $g_0$ given in the statement of Proposition \ref{prop:wavepackets-Gevrey}. We recall the system of equations \begin{equation}\label{eq:wkb_phi-Gev}
\begin{split}
(\rd_\tau\Phi)^2 - (\rd_\eta\Phi)^2 = f^2, \quad \Phi(0,\eta) = y(\eta)
\end{split}
\end{equation} and \begin{equation}\label{eq:wkb_h-Gev}
\begin{split}
(\rd_\tau\Phi\rd_\tau - \rd_\eta\Phi\rd_\eta) h = -\frac{1}{2} (\rd_\tau^2\Phi - \rd_\eta^2\Phi)h . 
\end{split}
\end{equation}

\subsubsection*{Explicitly solvable model case} It will be instructive to take a look at the simplest model case of $f(y) = y$, to get an idea of the behavior of the solutions to \eqref{eq:wkb_phi-Gev} and \eqref{eq:wkb_h-Gev}.  In this case, $f(\eta) = e^\eta$ and hence $\Phi(0,\eta) = e^\eta$. Let us also set $h_0(\eta) = e^\eta$ (for $\eta \le 0$). Here and in the following, we shall use the notation $A \sim B$ to denote that the ratio $A/B$ converges to some positive constant in the limit $\eta \rightarrow -\infty$, and use $A \approx B$ when the constant is $1$. Taking the ansatz $\Phi(\tau,\eta) = e^\eta H(\tau)$, we are led to solve \begin{equation*}
\begin{split}
(H')^2 - H^2 = 1, \quad H(0) = 1 
\end{split}
\end{equation*} and we have the solution (unique up to sign) \begin{equation*}
\begin{split}
H(\tau) = \sinh(\tau -c_0),\quad c_0 = \sinh^{-1}(1).
\end{split}
\end{equation*} With this $\Phi$,  {\eqref{eq:wkb_h-Gev}} becomes simply \begin{equation*}
\begin{split}
\left(\rd_\tau - \tanh(\tau - c_0) \rd_\eta \right) h = 0, 
\end{split}
\end{equation*} noticing the cancellation $\rd_\tau^2\Phi - \rd_\eta^2 \Phi = 0$. The solution is then explicitly given by \begin{equation*}
\begin{split}
h(\tau,\eta) = h_0\left(\eta + \log(\frac{\cosh(\tau-c_0)}{\cosh(-c_0)}) \right) = e^\eta \, \frac{\cosh(\tau - c_0)}{\cosh(-c_0)}.
\end{split}
\end{equation*} 
One sees that in this case, all the characteristic curves are parallel in the $(\tau,\eta)$-plane and moves to $\eta \rightarrow -\infty$ with asymptotically unit speed, and {that $h(\tau, \eta) / e^{\eta + \tau} \to 1$ along each characteristic curve in the region $\eta < -\log(\frac{\cosh(\tau-c_0)}{\cosh(-c_0)}) \approx  -\tau $}. Assuming that the support of $h_0$ is contained in $\{ \eta < C \}$ for some $C > 0$, one sees that all the Sobolev norms of $h(\tau)$ are uniformly bounded in terms of the corresponding norm of the initial data in the limit $\tau \rightarrow +\infty$. 

\subsubsection*{Initial data} We now take some general smooth $f$ with $\rd_y f(0) = c_0 > 0$. We recall that \begin{equation*}
\begin{split}
\eta(y) \approx c + \frac{1}{c_0}\ln y 
\end{split}
\end{equation*} for some constant $c$ (which can be normalized to be 0), and hence \begin{equation*}
\begin{split}
f(\eta) \approx c_0e^{c_0\eta} .
\end{split}
\end{equation*} We had the following asymptotic expressions for the derivatives: \begin{equation} \label{eq:f-exp-Gev}
|\rd_\eta^{(n)}f|(\eta) \lesssim_n f(\eta) \lesssim e^{c_0\eta},\quad \forall \eta \le 0.  
\end{equation}  We note that from $\Phi(0,\eta) = y(\eta)$, $\rd_\eta\Phi(0,\eta) = f(\eta) \approx c_0 e^{c_0\eta}$. Therefore, $\rd_\tau\Phi(0,\eta) = \sqrt{2}f(\eta)  \approx \sqrt{2}c_0e^{c_0\eta}$.   Similarly, one may check that $\rd_{\tau\eta}\Phi(0,\eta) = \sqrt{2}\rd_\eta f(\eta)$, $\rd_{\tau}^2\Phi(0,\eta) = \rd_{\eta}^2\Phi(0,\eta) = \rd_\eta f(\eta)$. 

\subsubsection*{Characteristics}  Define the characteristic curves by \begin{equation*}
\begin{split}
\frac{\ud}{\ud \tau}Y(\tau,\eta_0) = - \frac{\rd_\eta\Phi}{\rd_\tau\Phi}(\tau,Y(\tau,\eta_0)),\quad Y(0,\eta_0) = \eta_0. 
\end{split}
\end{equation*} Then, differentiating the equation for $\Phi$ in $\tau$ and $\eta$, we respectively obtain \begin{equation}\label{eq:Phi-tau}
\begin{split}
\frac{\ud}{\ud \tau} \rd_\tau\Phi(\tau,Y(\tau,\eta_0)) = 0
\end{split}
\end{equation} and \begin{equation}\label{eq:Phi-eta}
\begin{split}
\frac{\ud}{\ud \tau} \rd_\eta\Phi(\tau,Y(\tau,\eta_0)) = \frac{ff'}{\rd_\tau\Phi}(\tau,Y(\tau,\eta_0)). 
\end{split}
\end{equation} Therefore, we deduce that \begin{equation*}
\begin{split}
\frac{\ud}{\ud \tau} \left(- \frac{\rd_\eta\Phi}{\rd_\tau\Phi}(\tau,Y(\tau,\eta_0))\right) = - \frac{ff'}{(\rd_\tau\Phi)^2}(\tau,Y(\tau,\eta_0)) < 0. 
\end{split}
\end{equation*} Since $\rd_\eta\Phi/\rd_\tau\Phi|_{\tau =0 } = 1/\sqrt{2}$ and $\rd_\eta\Phi/\rd_\tau\Phi \le 1$ from the equation, we obtain \begin{equation}\label{eq:Y-Gev}
\begin{split}
\eta_0 - \tau  < Y(\tau,\eta_0) \le   \eta_0 -  \frac{1}{\sqrt{2}} \tau  
\end{split}
\end{equation} for all $\tau \ge 0$ and $\eta_0 $ large negative.  In particular $f(\tau,Y(\tau,\eta)) \le e^{c_0(\eta-c\tau)}$ and similarly for $f'$, so that \eqref{eq:Phi-eta} implies $\rd_\eta\Phi(\tau,Y(\tau,\eta))  {\aeq}  e^{c_0\eta}$. In the following we shall always assume that $\eta_0 \le 0$ is taken to be sufficiently negative so that the above estimates hold.

\subsubsection*{Second derivatives of $\Phi$} 
We compute  {
\begin{equation}\label{eq:Phi-tautau}
\begin{split}
 \frac{\ud}{\ud \tau} \left( \frac{\rd_{\tau\tau}\Phi}{\rd_\tau\Phi} (\tau,Y(\tau,\eta_{0})) \right) = 
\frac{f^{2}}{(\rd_{\tau} \Phi)^{2}} \left(\frac{\rd_{\tau} \Phi}{\rd_{\eta} \Phi} \right)^{2} \left(\frac{\rd_{\tau \tau} \Phi}{\rd_{\tau} \Phi}\right)^{2}(\tau,Y(\tau,\eta_{0})).
\end{split}
\end{equation} 
For each fixed $\eta_{0} \ll -1$, we introduce $Q(\tau) = ({\rd_{\tau\tau}\Phi}/{\rd_\tau\Phi}) (\tau,Y(\tau,\eta_{0}))$. Since the RHS of \eqref{eq:Phi-tautau} is nonnegative, for all $\tau \geq 0$ it follows that 
\begin{equation*}
Q(\tau) \geq Q(0)  > 0.
\end{equation*}
Next, note that $\rd_{\tau} \Phi$ is invariant and $\frac{\rd_{\tau} \Phi}{\rd_{\eta} \Phi}$ is decreasing along characteristics; at $\tau = 0$ they are equal to $\sqrt{2} f(\eta_{0})$ and $\sqrt{2}$, respectively. Thus,
\begin{equation*}
	\frac{\ud}{\ud \tau} Q(\tau) \leq \frac{f^{2}(Y(\tau, \eta_{0}))}{f^{2}(\eta_{0})} Q(\tau)^{2}.
\end{equation*}
Solving this differential inequality, we see that
\begin{equation*}
	Q(\tau) \leq \frac{Q(0)}{1 - Q(0) \int_{0}^{\tau} \frac{f^{2}(Y(\tau, \eta_{0}))}{f^{2}(\eta_{0})} \, \ud \tau'},
\end{equation*}
as long as the denominator is positive. Recall that $Q(0) \approx \frac{1}{\sqrt{2}} c_{0}$ and $f(\eta_{0}) \approx c_{0} e^{c_{0} \eta_{0}}$ as $\eta_{0} \to -\infty$. Moreover, $Y(\tau, \eta_{0}) \leq \eta_{0} - \frac{1}{\sqrt{2}} \tau$ by \eqref{eq:Y-Gev}. Thus,
\begin{equation*}
	Q(0) \int_{0}^{\tau} \frac{f^{2}(Y(\tau, \eta_{0}))}{f^{2}(\eta_{0})} \, \ud \tau'
	\approx \frac{c_{0}}{\sqrt{2}} \int_{0}^{\infty} e^{-\sqrt{2} c_{0} \tau'} \, \ud \tau' 
	\leq \frac{1}{2},
\end{equation*}
which ensures that the above denominator is $\ageq 1$ for all sufficiently negative $\eta_{0}$.

In conclusion, we have proved that
 \begin{equation*}
\begin{split}
\frac{\rd_{\tau\tau}\Phi}{\rd_\tau\Phi}(\tau,Y(\tau,\eta))  {\aeq} 1
\end{split}
\end{equation*} for all $\eta $ sufficiently negative. In turn, using \eqref{eq:Phi-tau} and \eqref{eq:Phi-eta}, we respectively deduce that  \begin{equation*}
\begin{split}
\frac{\rd_{\eta\tau} \Phi}{\rd_\tau\Phi}= \frac{\rd_\tau\Phi}{\rd_\eta\Phi}\frac{\rd_{\tau\tau}\Phi}{\rd_\tau\Phi}  {\aeq} 1, \quad
\frac{\rd_\eta^2\Phi}{\rd_\tau\Phi} = \frac{\rd_\tau\Phi}{\rd_\eta\Phi} \frac{\rd_{\tau\eta}\Phi}{\rd_\tau\Phi} -  \frac{ff'}{\rd_\eta\Phi\rd_\tau\Phi}  {\aeq} 1  
\end{split}
\end{equation*} along the characteristics, for sufficiently negative $\eta$. }

\subsubsection*{Higher derivatives}  To begin with, differentiating  {\eqref{eq:wkb_phi-Gev} twice}, we obtain a linear system of equations in third order derivatives of $\Phi$: \begin{equation}\label{eq:Phi-third}
\begin{split}
\begin{cases}
(\rd_{\tau\tau}\Phi)^2 + \rd_\tau\Phi\rd_{\tau\tau\tau}\Phi - (\rd_{\tau\eta}\Phi)^2 - \rd_\eta\Phi \rd_{\tau\tau\eta}\Phi = 0, \\
\rd_{\tau\eta}\Phi\rd_{\tau\tau}\Phi + \rd_\tau\Phi \rd_{\tau\tau\eta}\Phi - \rd_{\eta\eta}\Phi\rd_{\tau\eta}\Phi - \rd_\eta\Phi \rd_{\tau\eta\eta}\Phi = 0 , \\
\rd_\tau\Phi\rd_{\tau\eta\eta}\Phi + (\rd_{\tau\eta}\Phi)^2 - (\rd_{\eta\eta}\Phi)^2 - \rd_\eta\Phi\rd_{\eta\eta\eta}\Phi = (f')^2 + ff'' . 
\end{cases}
\end{split}
\end{equation} An estimate on a single third-order term along the characteristics lead to the corresponding estimates for all the other third-order derivatives, using \eqref{eq:Phi-third}. To this end we shall estimate $\rd_{\tau\tau\tau}\Phi$:  differentiating \eqref{eq:Phi-tau} twice in $\tau$ and using that $\rd_\tau\Phi$ is constant along the characteristics,
\begin{equation*}
\begin{split}
\frac{\ud}{\ud \tau} \frac{\rd_{\tau\tau\tau}\Phi}{\rd_\tau\Phi}(\tau,Y(\tau,\eta_0)) &=  \frac{\rd_{\tau\tau}\Phi}{\rd_\tau\Phi} \frac{2f^2}{\rd_\tau\Phi\rd_\eta\Phi}   \frac{\rd_{\tau\tau\eta}\Phi}{\rd_\tau\Phi} + \rd_\tau \left(\frac{\rd_{\tau\tau}\Phi}{\rd_\tau\Phi} \frac{f^2}{\rd_\tau\Phi\rd_\eta\Phi} \right) \frac{\rd_\tau\Phi}{\rd_\eta\Phi} \frac{\rd_{\tau\tau}\Phi}{\rd_\tau\Phi} \\
& =  \frac{\rd_{\tau\tau}\Phi}{\rd_\tau\Phi} \frac{2f^2}{\rd_\tau\Phi\rd_\eta\Phi}  \left(\frac{(\rd_{\tau\tau}\Phi)^2}{\rd_\tau\Phi\rd_\eta\Phi} + \frac{\rd_{\tau\tau\tau}\Phi}{\rd_\tau\Phi} \frac{\rd_\tau\Phi}{\rd_\eta\Phi} - \frac{(\rd_{\tau\eta}\Phi)^2}{\rd_\tau\Phi\rd_\eta\Phi}\right) \\
&\quad +  \frac{\rd_{\tau\tau}\Phi}{\rd_\tau\Phi} \frac{f^2}{\rd_\tau\Phi\rd_\eta\Phi}\left( \frac{\rd_{\tau\tau\tau}\Phi}{\rd_\tau\Phi} - \frac{\rd_{\tau\tau}\Phi(2\rd_{\tau\tau}\Phi\rd_\eta\Phi+ \rd_\tau\Phi\rd_{\tau\eta}\Phi)}{(\rd_\tau\Phi)^2\rd_\eta\Phi} \right) \frac{\rd_\tau\Phi}{\rd_\eta\Phi}  ,
\end{split}
\end{equation*}  where the expressions on the  {RHS's} are evaluated at $(\tau,Y(\tau,\eta_0))$. Apart from the expression ${\rd_{\tau\tau\tau}\Phi}/{\rd_\tau\Phi}$ which we need to estimate, all the ratios appearing on the last expression are of $ {\aeq} 1$, except that $f^2/(\rd_\tau\Phi\rd_\eta\Phi)$ decays exponentially in $\tau$. From this we conclude that \begin{equation*}
\begin{split}
\frac{\rd_{\tau\tau\tau}\Phi}{\rd_\tau\Phi}(\tau,Y(\tau,\eta_0))  {\aeq} 1 
\end{split}
\end{equation*} for all $\eta_0$ sufficiently negative. Using \eqref{eq:Phi-third} we also deduce \begin{equation*}
\begin{split}
\frac{\rd^3\Phi}{\rd_\tau\Phi}(\tau,Y(\tau,\eta_0))  {\aeq} 1 . 
\end{split}
\end{equation*} It is clear now that a similar estimates hold for higher derivatives of arbitrary order. For instance, to obtain such estimates for the fourth order derivatives, it is sufficient to prove $\rd_\tau^4\Phi/\rd_\tau\Phi  {\aeq} 1$, and for this purpose one simply needs to differentiate the above expression for ${\rd_{\tau\tau\tau}\Phi}/{\rd_\tau\Phi}$ in $\tau$ and observe that the RHS can be written in the form where all the expressions are of order 1 except for the quantity $\rd_\tau^4\Phi/\rd_\tau\Phi$ itself which is multiplied with a temporally decaying factor $f^2/(\rd_\tau\Phi\rd_\eta\Phi)$.

\subsubsection*{Analysis of the transport equation}
We consider 
\begin{equation*}
\calL = \rd_{\tau} - \frac{\rd_\eta\Phi}{\rd_\tau\Phi} \rd_{\eta} ,
\end{equation*}
towards the goal of estimating $h$ via the transport equation \eqref{eq:wkb_h-Gev}. Before we begin,  {note that the divergence of $\calL$ with respect to $\ud \eta$ is 
\begin{equation} \label{eq:div-L-Gev}
	- \rd_{\eta} \frac{\rd_\eta\Phi}{\rd_\tau\Phi}
	= \frac{\rd_{\eta}^{2} \Phi - \rd_{\tau}^{2} \Phi}{\rd_{\tau} \Phi}.
\end{equation}
Comparing this expression with the RHS of \eqref{eq:wkb_h-Gev}, we see that}
 the $L^2$-norm is conserved: $\nrm{h(\tau)}_{L^2} = \nrm{h_0}_{L^2}$. We shall now proceed to show that actually all $W^{s,p}$-norms of $h$ are uniformly  bounded in $\tau$ as well.

First, observe that \eqref{eq:wkb_h-Gev} can be simplified using the method of integrating factors: introducing a real-valued function $\alp(\tau, \eta)$ defined by
\begin{equation} \label{eq:wkb-alp-Gev}
\calL \alp = - \frac{1}{2} \frac{\rd_\tau^2\Phi-\rd_\eta^2\Phi}{\rd_\tau\Phi}, 
\end{equation}
with the initial condition $\alp(\tau = 0) = 0$, we see that
\begin{equation} \label{eq:wkb-h'-Gev}
\calL (e^{-\alp} h ) = 0.
\end{equation}
For any $m \in \bbN_{0}$ we claim that
\begin{equation*}
\Abs{\rd_{\eta}^{m} \left( - \frac{1}{2} \frac{\rd_\tau^2\Phi-\rd_\eta^2\Phi}{\rd_\tau\Phi} \right)}(\tau,Y(\tau,\eta_0)) \lesssim_m e^{-2cc_0\tau}
\end{equation*} holds for $\eta \ll -1$. It follows directly from \eqref{eq:f-exp} and the estimates for derivatives of $\Phi$ along the characteristics obtained in the above. To see this in the case $m = 0$, note that \begin{equation*}
\begin{split}
\rd_\tau^2\Phi - \rd_\eta^2\Phi & = \frac{\rd_\eta\Phi}{\rd_\tau\Phi} \rd_{\tau\eta}\Phi - \frac{\rd_\tau\Phi}{\rd_\eta\Phi} \rd_{\tau\eta}\Phi + \frac{ff'}{\rd_\eta\Phi} 
\end{split}
\end{equation*} so that \begin{equation*}
\begin{split}
- \frac{1}{2} \frac{ \rd_\tau^2\Phi - \rd_\eta^2\Phi }{\rd_\tau\Phi} = \frac{1}{2} \frac{\rd_{\tau\eta}\Phi}{\rd_\tau\Phi} \frac{f^2}{\rd_\tau\Phi\rd_\eta\Phi} - \frac{1}{2} \frac{ff'}{\rd_\eta\Phi\rd_\tau\Phi} ,
\end{split}
\end{equation*} which is $ {\aeq} e^{-2cc_0\tau}$ when evaluated along a characteristic. It is now straightforward to extend the above estimate to $m \ge 1$.  {Then by \eqref{eq:div-L-Gev}}, for any $\ell \ge 1$ we have decay of the coefficients \begin{equation*}
\begin{split}
\left|\rd_\eta^\ell \left(\frac{\rd_\eta\Phi}{\rd_\tau\Phi}\right)\right|(\tau,Y(\tau,\eta_0)) \lesssim_\ell e^{-2cc_0\tau}.
\end{split}
\end{equation*} In the case $\ell = 1$, it shows that the divergence of $\calL$ with respect to $\ud \eta$ decays exponentially in $\tau$ along characteristics. 

Using the above observations, we obtain the following $L^{\infty}$-bound for $\alp$:
\begin{align*}
\sup_{0 \leq k \leq m} \sup_{\tau \geq 0} \nrm{\rd_{\tau}^{k}  \rd_{\eta}^{m  - k} \alp(\tau)}_{L^{\infty}_{\eta}}
\aleq_{m} & 1 
\end{align*} from which it follows that \begin{lemma} \label{lem:wkb-h-est-gev}
	Let $h$ be the solution of \eqref{eq:wkb_h} with smooth initial data $h_0$ supported on $\eta \le 0$. Then we have the estimates \begin{equation*}
	\begin{split}
	\max_{0 \le k \le m}  \sup_{\tau \ge 0}\nrm{ \rd_\tau^k \rd_\eta^{m-k } h(\tau)}_{L^p(\mathbb{R}_\eta)}   \lesssim_m \nrm{h_0}_{W^{m,p}(\mathbb{R}_\eta)}
	\end{split}
	\end{equation*} for any integer $m \ge 0$ and $1 \leq p \leq \infty$. 
\end{lemma}

\begin{remark}\label{rem:gevrey-other-domain}
		One may consider  the $x$-dependent case of the transport system as in \eqref{eq:wkb_h} with $\Phi$ solving \eqref{eq:wkb_phi-Gev}. From
		\begin{equation} \label{eq:wkb_h-Gev2}
		\calL = \rd_{\tau} - \frac{\rd_\eta\Phi}{\rd_\tau\Phi} \rd_{\eta} - \frac{(\rd_\eta\Phi)^2 + 2f^2}{\rd_\tau\Phi} \rd_{x} ,
		\end{equation} and \begin{equation*} 
		\begin{split}
		\frac{\ud}{\ud \tau} X(\tau,x_0,\eta_0) :=	- \frac{(\rd_\eta\Phi)^2 + 2f^2}{\rd_\tau\Phi}(\tau,Y(\tau,\eta_0)) \sim e^{c_0\eta_0} ,
		\end{split}
		\end{equation*} we see that $\eta_0$-gradient of the speed of the $X$-characteristics does not decay in $\tau$. This inevitably gives rise to a linear in $\tau$ growth for $\eta$-derivatives of $h$; indeed in the expression for $\calL[\rd_\eta h]$, we have \begin{equation*} 
		\begin{split}
		\rd_\eta\left( \frac{(\rd_\eta\Phi)^2 + 2f^2}{\rd_\tau\Phi} \right)\rd_x h 
		\end{split}
		\end{equation*} on the right hand side, which does not decay in $\tau$ along the characteristics whereas all the other coefficients are exponentially decaying. Therefore we cannot hope for a uniform-in-$\lambda$ error estimates for our WKB ansatz after returning to the $t$-variable. 
		
		In view of this, the explicit choice of the phase function in \eqref{eq:wkb_phi_sol} is not just for simplicity, but it is precisely the choice which allows uniform estimates for the $\eta$-derivatives of $h$ in $\tau$. 
	\end{remark}

\subsection{Degenerating wave packet approximate solutions}

In this subsection, we complete the proof of Proposition \ref{prop:wavepackets-Gevrey}. We note here that we shall only consider the region $y > 0$, but a parallel argument can be given for $y < 0$ with a similar change of coordinates. (Strictly speaking, the wave packets can be simply defined to be zero for $y \le 0$ and still the proof of Theorem \ref{thm:illposed-gevrey} goes through.) 

\medskip

\textit{(i) case of \eqref{eq:e-mhd}}

\medskip

The first step is to estimate the error in the $\phi$-equation. Given $g_0$, we apply the WKB construction from the previous subsection with \begin{equation*} 
\begin{split}
&h_0 (\eta) := f^{-\frac{1}{2}}(y(\eta)) g_0(y(\eta)) 
\end{split}
\end{equation*} to obtain $(\Phi, h)$ and define \begin{equation*} 
\begin{split}
&\tilde{\psi}_{(\lmb)} = f^{\frac{1}{2}} \lmb^{-1}e^{i\lmb(x+\Phi(\lmb t, \eta(y)))}  h(\lmb t, \eta(y)), \quad \tb_{(\lmb)}^z = -(f\rd_x)^{-1}(\rd_t \tpsi_{(\lmb)}). 
\end{split}
\end{equation*} Then it is clear that \begin{equation*} 
\begin{split}
&\tpsi_{(\lmb)}(t = 0) = \lmb^{-1}e^{i\lmb(x+y)} g_0 
\end{split}
\end{equation*} and \begin{equation*} 
\begin{split}
\tb_{(\lmb)}^z(t = 0) & =  -(f\rd_x)^{-1}(\rd_t \tpsi_{(\lmb)}) |_{t = 0}   =  -f^{-1}(i\lmb)^{-1} (i\lmb\rd_\tau\Phi(t=0) g_0 + \rd_\tau g_0 )e^{i\lmb(x+y)} \\
& = -f^{-1}\left( \sqrt{2}f g_0 + \frac{1}{i\lmb} \rd_\tau g_0  \right)e^{i\lmb(x+y)} \\
&= -\left( \sqrt{2}g_0 + \frac{1}{\sqrt{2}i\lmb}(\rd_yg_0 - \frac{1}{2}f^{-1}\rd_y f g_0)    \right)e^{i\lmb(x+y)}
\end{split}
\end{equation*} using \begin{equation*} 
\begin{split}
&(\rd_\tau g)_0 = \frac{1}{\sqrt{2}} f^{\frac{3}{2}} \rd_y\left(f^{-\frac{1}{2}} g_0 \right)
\end{split}
\end{equation*} which follows from evaluating \eqref{eq:wkb-h'-Gev} at $ \tau = 0$. The claimed lower bound on the initial data can be checked in a straightforward manner, and the upper bounds for $\tpsi_{(\lmb)}, \tb_{(\lmb)}^z$ follow from the corresponding bounds on $h$ as in the proof of Proposition \ref{prop:wavepackets}. The $L^{p}$-degeneration property follows since in the $(\tau,\eta)$-coordinates, the support of $(\tpsi_{(\lmb)}, \tb_{(\lmb)}^z)$ moves to $\eta \rightarrow -\infty$ with speed at least $1/\sqrt{2}$. 

The last step is to estimate the error. Denote  {by} $\errwp_\phi = \errwp_\phi[h_0;\lmb](\tau,\eta)$ the LHS of \eqref{eq:e-mhd-eta-conj3} evaluated with $\phi = \lmb^{-1} e^{i\lmb\Phi(\tau,\eta)} h(\tau,\eta)$. It is a straightforward computation to see that for each $\tau \ge 0$, \begin{equation*} 
\begin{split}
&\nrm{ \errwp_\phi(\tau) }_{L^2_\eta} \lesssim \lmb^{-2}\left( \nrm{h}_{L^2} + \nrm{\rd_\tau^2 h}_{L^2} + \nrm{\rd_\eta^2h }_{L^2} \right)(\tau) \lesssim \lmb^{-2}\nrm{h_0}_{H^2_\eta}. 
\end{split}
\end{equation*} Then from $\err_b[\tb_{(\lmb)}^z,\tpsi_{(\lmb)}] = \lmb^2  e^{i\lmb x} f^{-\frac{1}{2}}\errwp_\phi$, it follows that  \begin{equation*} 
\begin{split}
&\nrm{\err_b[\tb_{(\lmb)}^z,\tpsi_{(\lmb)}]}_{L^2_{x,y}} \lesssim \nrm{h_0}_{H^2_\eta} \lesssim \nrm{g_0}_{H^2_y}. 
\end{split}
\end{equation*}

\medskip

\textit{(ii) case of \eqref{eq:hall-mhd}}

\medskip

First, defining $\tu^{z}_{(\lmb)}$ and $\tomg_{(\lmb)}$ from $\tb^{z}_{(\lmb)}$, $\tpsi_{(\lmb)}$ as in \eqref{eq:wavepackets-hall-u-omg-Gev}, the claimed estimates for $\tu^{z}_{(\lmb)}$, $\nb \tu^{z}_{(\lmb)}$, $\tomg_{(\lmb)}$, and $\nb^{\perp}(-\lap)^{-1} \tomg_{(\lmb)}$ follow directly from the regularity estimates for $\tb^z_{(\lmb)}$ and $\tpsi_{(\lmb)}$, using $L^2$-boundedness of the operator $\nabla^\perp(-\lap)^{-1}\rd_x$ and the fact that $\rd_x^{-1}$ gives simply division by $i\lmb$. Moreover, the error estimates follow in a similar way, using  {the relation \eqref{eq:wavepackets-hall-err}}
and that $\err_b[\tb_{(\lmb)}^z,\tpsi_{(\lmb)}] = \lmb^2  e^{i\lmb x} f^{-\frac{1}{2}}\errwp_\phi$. \qedsymbol

\section{Proof of illposedness for fractionally dissipative systems}\label{sec:fradiss}

The goal of this section is to give the proof of Theorem~\ref{thm:illposed-fradiss}. As discussed in the introduction, we only consider the case $M = \bbT^{3}$.
The proof is parallel to the proof of Theorems~\ref{thm:norm-growth} and \ref{thm:illposed-strong}. We shall use the exact same degenerating wave packets constructed earlier; the only difference is that there are additional error terms arising from the dissipative term and the time-dependence of the background magnetic field, which is also induced by the dissipative term.

\subsection{Background magnetic fields}

Upon taking $\bgB = f(t,y)\rd_x$ into \eqref{eq:e-mhd-fradiss}, or $(\bgu, \bgB) = (0, f(t, y) \rd_{x})$ into \eqref{eq:hall-mhd-fradiss}, the nonlinearity vanishes and we are left with \begin{equation}\label{eq:fradiss-background}
\begin{split}
\rd_t f(t,y) = -\eta (-\lap)^\alp f(t,y).
\end{split}
\end{equation} 
From now on, we shall fix the initial data $f_0$ to be a function satisfying the following assumptions: \begin{itemize}
	\item $f_0$ is $C^\infty$-smooth, 
	\item $f_0$ meets all the requirements stated in Proposition \ref{prop:wavepackets} with $y_1 = \frac{1}{2}$; that is, \begin{equation*}
	\begin{split}
	f_0(0) = 0, \quad f_0'(0) > 0, \quad f_0'(y) > \frac{1}{2}f_0'(0) \mbox{ and } 0<f_0(y)<\frac{1}{2}\mbox{ for } y \in [0,\frac{1}{2}],
	\end{split}
	\end{equation*} 
	\item $f_0$ is odd at $y = 0$, and 
	\item $(-\lap)^\alp f_0 \equiv  0$ on $[0,\frac{1}{2}]$.\footnote{This is a cheap way to avoid error terms of order $O(t)$ in the generalized energy estimate. Presumably, a more appropriate way to proceed is to repeat the entire WKB analysis with time-dependent coefficient $f(t)$.}
\end{itemize} 
It is not difficult to see that by first taking $g \in C^\infty(\bbT_y)$ to be some odd function which is supported outside of $[-1,1]$, we can arrange $f_0 := (-\lap)^{-\alp}g$ to satisfy all the required properties above. 

Using Fourier series, it is easy to see that there is a unique smooth solution $f(t)$ to \eqref{eq:fradiss-background} with initial data $f_0$, which satisfies $f(t,0)=0$ for all $t\ge0$. We shall need the following simple \begin{lemma}
	We have, for $f_0$ satisfying the assumptions above and for any $\dlt>0$, \begin{equation}\label{eq:background-timesmall}
	\begin{split}
	\nrm{t^{-2}y^{-1}(f(t,y)-f_0(y))}_{L^\infty([0,\dlt];W^{1,\infty}(0,\frac{1}{2}))} + \nrm{t^{-2} (f(t,y)-f_0(y))}_{L^\infty([0,\dlt];W^{2,\infty}(0,\frac{1}{2}))} \lesssim_{f_0,\dlt} \eta. 
	\end{split}
	\end{equation}
\end{lemma}
\begin{proof} Since \eqref{eq:fradiss-background} preserves the odd symmetry, we have $f(t,0)=0$ for all $t\ge0$. Similarly, $\rd_{t} f(t,\cdot)$ and $\rd_{t}^{2} f(t,\cdot)$ vanish at $y = 0$ for all $t\ge0$. 
	We then estimate 
	\begin{equation*}
	\begin{split}
	\frac{|f(t,y)-f_0(y)-t(\rd_tf)|_{t=0}(y)|}{|y|} = \left| \int_0^t\int_0^s \frac{\rd_{t}^{2} f(\tau)}{|y|} d\tau ds  \right| \le \frac{t^2}{2} \nrm{\rd_{t}^{2} \rd_yf}_{L^\infty([0,\dlt]\times\bbT_y)}. 
	\end{split}
	\end{equation*} From standard energy estimates, we have that \begin{equation*}
	\begin{split}
	\nrm{\rd_{t}^{2} \rd_yf}_{L^\infty([0,\dlt]\times\bbT_y)} \lesssim_{f_0,\dlt} \eta. 
	\end{split}
	\end{equation*} Since $(\rd_tf)|_{t=0} = -\eta(-\lap)^\alp f_0$ vanishes on $(-\frac{1}{2},\frac{1}{2})$ from the assumption on $f_0$, we have that \begin{equation*}
	\begin{split}
	\nrm{t^{-2}y^{-1}(f(t,y)-f_0(y))}_{L^\infty([0,\dlt]\times (0,\frac{1}{2}))} \lesssim_{f_0,\dlt}\eta. 
	\end{split}
	\end{equation*} The other estimates can be obtained similarly.  
\end{proof} 
In the remainder of this section, unless otherwise specified, we suppress the dependence of implicit constants on $f_{0}$.

\subsection{Proof of Theorem~\ref{thm:illposed-fradiss} for \eqref{eq:e-mhd-fradiss}} \label{subsec:e-mhd-fradiss}

We now prove Theorem~\ref{thm:illposed-fradiss} in the electron-MHD case, assuming $0\le\alp<\frac{1}{2}$ and $\eta > 0$. Towards a contradiction, assume that there exist $\dlt,\eps>0$ and $s\ge s_0 \ge 3$ such that the solution operator is well-defined and \textit{bounded} as a  map \begin{equation*}
\begin{split}
\calB_{\eps}(0;H^s) \rightarrow L^\infty_t([0,\dlt];H^{s_0}). 
\end{split}
\end{equation*} Note that we are also assuming $s< 3s_0$ and $s < \frac{s_0}{2\alp}$. Under this assumption, we consider the sequence of initial data parameterized by $\lmb$: \begin{equation}\label{eq:initial-data-fradiss}
\begin{split}
\bfB_0^{(\lmb)} = f_0(y)\rd_x + b_{(\lmb)}(0),
\end{split}
\end{equation} 
where $f_{0}$ satisfies all the assumptions from above and $b_{(\lmb)}(0)$ will be specified below.
(As before, by rescaling the data we can assume without loss of generality that $\eps=1$, although strictly speaking, the value of $\eta$ will now depend on $\eps$ as well. The following argument works for any large constant $\eta$.) Then, for each $\lmb\ge 1$, there is a solution $\bfB^{(\lmb)}(t) \in L^\infty([0,\dlt];H^s)$ to \eqref{eq:e-mhd-fradiss} with initial data $\bfB_0^{(\lmb)}$. We have that, from the assumption of boundedness of the solution operator, \begin{equation} \label{eq:contra-fradiss}
\sup_{\lmb} \sup_{t\in[0,\dlt]} \nrm{\bfB^{(\lmb)}(t)}_{H^{s_0}} \le A 
\end{equation} for some constant $A>0$. In the following, the constants $C, c, \cdots$ may depend on $A$. For simplicity of the notation, we shall drop the dependence of the solution in $\lmb$ from now on. 

We now specify $b_{(\lmb)}(0)$. As in Theorem~\ref{thm:norm-growth}, let $\tilde{b}_{(\lmb)}$ be the degenerating wave packet provided by Proposition \ref{prop:wavepackets} with $f = f_0$ associated with a nonzero smooth compactly supported profile $g_{0}$. Taking the support of $g_{0}$ to be sufficiently close to $y=0$, we may ensure that the constants $C_{f_{0}}$, $c_{f_{0}}$ in \eqref{eq:wavepackets-degen-upper}--\eqref{eq:wavepackets-degen-small} obey $C_{f_{0}} = c_{f_{0}} + \veps_{0}$, where $\veps_{0} = \veps_{0}(\alp)$ will be specified below. In what follows, we suppress the dependence of implicit constants on $g_{0}$. We take 
\begin{equation*}
b_{(\lmb)}(0) = \lmb^{-s} \tilde{b}_{(\lmb)}(0).
\end{equation*}

We now define $b_{(\lmb)}(t):= \bfB^{(\lmb)}(t) - \bgB(t):= \bfB(t)-f(t,y)\rd_x$, where $f(t)$ is the solution of \eqref{eq:fradiss-background} with initial data $f_0$. We also set $\bgB_0 = f_0(y)\rd_x$. To lighten the notation, in what follows we simply write $b(t) = b_{(\lmb)}(t)$.

The equation for $b$ is given by \begin{equation}\label{eq:e-mhd-fradiss-pert}
\left\{
\begin{split}
&\rd_t b + \nb\times((\nb\times b)\times\bgB) + \nb\times((\nb\times\bgB)\times b) = -\eta(-\lap)^\alp b - \nb\times((\nb\times b)\times b) ,\\
&\nb\cdot b = 0. 
\end{split}
\right.
\end{equation} Taking the inner product of the first equation with $b$ and integrating over $[0, t] \times \bbT^{3}$ for $0 \leq t < \dlt$, we have on $[0,\dlt]$ 
\begin{equation}\label{eq:pert-energy}
\begin{split}
\nrm{b (t)}_{L^2} \lesssim \nrm{b_{(\lmb)}(0)}_{L^2} \lesssim \lmb^{-s}\nrm{\tb_{(\lmb)}(0)}_{L^2}\lesssim \lmb^{-s} 
\end{split}
\end{equation} where the implicit constants depend on $\bgB(t)$ but not on $\lmb$. Moreover, by \eqref{eq:contra-fradiss} \begin{equation*}
\begin{split}
\nrm{b(t)}_{H^{s_0}} \le \nrm{\bfB(t)}_{H^{s_0}} + \nrm{\bgB(t)}_{H^{s_0}} \lesssim 1 
\end{split}
\end{equation*} uniformly in $\lmb$ for $t \in [0,\dlt]$. 

We write $\err_{\tb} = (-\nb^\perp\err_{\tpsi},\err_{\tb^z})$ to denote the error associated with $\tb_{(\lmb)}$ defined in Proposition \ref{prop:wavepackets} (see \eqref{eq:emhd-nonlinear2}). Then, \begin{equation}\label{eq:gei-fradiss}
\begin{split}
\frac{\ud}{\ud t} \brk{b(t),\tb_{(\lmb)}(t)} & = \brk{b,\err_{\tb}} - \brk{f''_0b^z,\tb^y_{(\lmb)}} - \brk{f''_0b^y,\tb^z_{(\lmb)}} - \brk{\nb\times((\nb\times b)\times b) , \tb_{(\lmb)}} \\
&\quad -\eta \brk{(-\lap)^\alp b, \tb_{(\lmb)} } - \brk{\bfG, \tb_{(\lmb)}}, 
\end{split}
\end{equation} where \begin{equation*}
\begin{split}
\bfG(t) := \nb\times((\nb\times b)\times(\bgB-\bgB_0)) + \nb\times((\nb\times(\bgB-\bgB_0))\times b). 
\end{split}
\end{equation*} Comparing this identity with \eqref{eq:gei-unbounded}, the only additional terms on the right hand side are from fractional dissipation and time dependence of the background magnetic field. 

We estimate \begin{equation*}
\begin{split}
\left| \brk{(-\lap)^\alp b, \tb_{(\lmb)} } \right| = \left| \brk{b, (-\lap)^\alp \tb_{(\lmb)} } \right| \le \nrm{b}_{L^2}\nrm{ (-\lap)^\alp\tb_{(\lmb)}  }_{L^2}
\end{split}
\end{equation*} where, by \eqref{eq:wavepackets-degen-upper} in Proposition~\ref{prop:wavepackets}, we have \begin{equation*}
\begin{split}
\nrm{ (-\lap)^\alp\tb_{(\lmb)}  }_{L^2} \lesssim \nrm{\tb_{(\lmb)}}_{L^2}^{1-2\alp}\nrm{\tb_{(\lmb)}}_{H^1}^{2\alp} \lesssim \lmb^{2\alp} e^{C_{f_{0}} (2\alp)\lmb t}. 
\end{split}
\end{equation*}
Next, \begin{equation*}
\begin{split}
\brk{  \nb\times((\nb\times b)\times(\bgB-\bgB_0)), \tb_{(\lmb)} }  =- \brk{ b, \nb\times ((\nb\times  \tb_{(\lmb)})\times(\bgB-\bgB_0) ) }
\end{split}
\end{equation*} and using \eqref{eq:background-timesmall} together with the fact that $\lmb^{-1}y\rd_y$ acts as a bounded operator on $\tb_{(\lmb)}$, \begin{equation*}
\begin{split}
\nrm{\nb\times ((\nb\times  \tb_{(\lmb)})\times(\bgB-\bgB_0) ) }_{L^2} &\lesssim t^2(\nrm{\rd_x\tb_{(\lmb)}}_{L^2} + \nrm{\rd_{x}^{2} \tb_{(\lmb)}}_{L^2} + \nrm{y\rd_y\tb_{(\lmb)}}_{L^2} + \nrm{y\rd_y\rd_x\tb_{(\lmb)}}_{L^2}  ) \\
& \lesssim t^2 (1 + \lmb^2). 
\end{split}
\end{equation*} 
Similarly, we can write  \begin{equation*}
\begin{split}
\brk{\nb\times((\nb\times(\bgB-\bgB_0))\times b),\tb_{(\lmb)}} = -\brk{ b , (\nb\times(\bgB-\bgB_0))\times (\nb\times\tb_{(\lmb)})}
\end{split}
\end{equation*} and estimate \begin{equation*}
\begin{split}
\nrm{(\nb\times(\bgB-\bgB_0))\times (\nb\times\tb_{(\lmb)})}_{L^2}\lesssim t^2\lmb e^{C_{f_{0}} \lmb t}. 
\end{split}
\end{equation*}This gives \begin{equation*}
\begin{split}
\left| \brk{\bfG,\tb_{(\lmb)}} \right| \lesssim t^2(1+\lmb^2+ \lmb e^{C_{f_{0}} \lmb t}) 
\end{split}
\end{equation*} Recalling the error estimate $\nrm{\err_b}_{L^2}\lesssim 1$, 
collecting the terms, and applying \eqref{eq:pert-energy}, \begin{equation*}
\begin{split}
\left| \frac{\ud}{\ud t} \brk{b(t),\tb_{(\lmb)}(t)}  \right| \lesssim \left( 1 +\lmb^{2\alp} e^{C_{f_{0}} (2\alp)\lmb t}+  t^2(1+\lmb^2+ \lmb e^{C_{f_{0}} \lmb t}) \right)  \nrm{b(0)}_{L^2}. 
\end{split}
\end{equation*} We integrate in $t\in[0,t^*]$, where $t^* :=\lmb^{-1} \ln(\lmb^{\veps})$ with $\veps>0$ to be determined: \begin{equation*}
\begin{split}
&\int_0^{t^*} \left( 1 +\lmb^{2\alp} e^{C_{f_{0}} (2\alp)\lmb t}+  t^{2}(1+\lmb^2+ \lmb e^{C_{f_{0}} \lmb t}) \right)  \ud t \lesssim_\alp t^* + \lmb^{2\alp-1}e^{C_{f_{0}} (2\alp)\lmb t^*} + (t^*)^3(\lmb^2 + \lmb e^{C_{f_{0}} \lmb t^*}) \\
&\quad \lesssim \lmb^{-1} \ln(\lmb^\veps) + \lmb^{(2\alp-1) + C_{f_{0}} (2\alp)\veps } + \lmb^{-1}\ln(\lmb^\veps) + \lmb^{C_{f_{0}} \veps-2} (\ln(\lmb^\veps))^3 \ll 1
\end{split}
\end{equation*} since it is easy to pick $\veps>0$ small (depending only on $\alp, C_{f_{0}} $) so that \begin{equation*}
\begin{split}
(2\alp-1) + (2\alp)C_{f_{0}} \veps < 0 , \quad C_{f_{0}} \veps - 2 < 0 ,
\end{split}
\end{equation*} recalling that $2\alp-1<0$. Therefore, with such a choice of $\veps>0$, we conclude that on $[0,t^*]$, \begin{equation} \label{eq:gei-conc-fradiss}
 \brk{b(t),\tb_{(\lmb)}(t)}  > \frac{1}{2}  \brk{b(0),\tb_{(\lmb)}(0)} 
\end{equation} for $\lmb$ large enough. Using the degeneration estimates \eqref{eq:wavepackets-degen}--\eqref{eq:wavepackets-degen-small} for $\tb_{(\lmb)}(t)$ at $t = t^*$ (see also the proof of Theorem~\ref{thm:norm-growth}), we then obtain \begin{equation*}
\begin{split}
\nrm{b(t^*)}_{H^{s_0}} \gtrsim \nrm{b(0)}_{L^{2}} \lmb^{s_{0}} \exp(c_{f_0}s_0\lmb t^*) \gtrsim \lmb^{\veps c_{f_0}s_0}\lmb^{s_0-s} .
\end{split}
\end{equation*} Recalling the assumptions on $s$ and $s_0$, we may pick $\veps$  and $\veps_{0} = C_{f_{0}} - c_{f_{0}}$ so that \begin{equation*}
\begin{split}
(c_{f_0}\veps +1 )s_0>s 
\end{split}
\end{equation*} which gives a contradiction for $\lmb$ sufficiently large since $\nrm{b(t^*)}_{H^{s_0}} \aleq 1$. \qedsymbol

\begin{remark}[Modifications for $\alp =\frac{1}{2}$] \label{rem:fradiss-crit-pf}
We sketch the necessary modifications of the preceding argument for establishing Remark~\ref{rem:fradiss-crit}. Without loss of generality, we normalize $\eta = 1$. Consider the sequence of initial data $\bfB_{0}^{(\lmb)} = \frac{A}{C_{0}} (f_{0}(y) \rd_{x} + b_{(\lmb)}(0))$, which is obtained by rescaling the data from \eqref{eq:initial-data-fradiss} by $\frac{A}{C_{0}}$; $C_{0}$ is chosen so that $\nrm{\bfB_{0}^{(\lmb)}}_{H^{s}} \leq A$. Towards a contradiction, assume that $\liminf_{\lmb \to \infty} \dlt^{(\lmb)} = \dlt_{0} > 0$ and $\limsup_{\lmb \to \infty} \sup_{t \in [0, \dlt_{0}]} \frac{\nrm{\bfB^{(\lmb)}(t)}_{H^{s}}}{A^{(1-\veps) s+1}} = 0$. Using Proposition~\ref{prop:wavepackets} in combination with the rescaling of time $t \mapsto \frac{A}{C_{0}} t$ (which takes into account the rescaling of the initial data) and proceeding as above, it is possible to show that 
\begin{equation*}
\brk{b^{\ast}_{(\lmb)}, \tb^{\ast}_{(\lmb)}}(t^{\ast})> \frac{1}{2} \quad \hbox{ for } t^{\ast} = \frac{C_{0}}{C_{f_{0}} \lmb A} \ln \left(C_{1} A\right) 
\end{equation*}
for an appropriate constant $C_{1} > 0$, where $b^{\ast}_{(\lmb)}(t) = \bfB^{(\lmb)}(t) - \frac{A}{C_{0}} f(t, y) \rd_{x}$ and $\tb^{\ast}_{(\lmb)}(t, x, y) = \tb_{(\lmb)}(\frac{A}{C_{0}} t, x, y)$. By duality and the degeneration property in Proposition~\ref{prop:wavepackets}, it follows that 
\begin{equation*}
\nrm{b^{\ast}(t^{\ast})}_{H^{s}} \ageq \nrm{b^{\ast}(0)}_{L^{2}} \lmb^{s} \exp(c_{f_{0}} A s \lmb t^{\ast}) \ageq A^{\frac{c_{f_{0}}}{C_{f_{0}}} s +1}.
\end{equation*}
Choosing $C_{f_{0}} - c_{f_{0}} < \veps C_{f_{0}}$ (see Remark~\ref{rem:wavepackets-const}), we attain the desired contradiction.
\end{remark}

\subsection{Proof of Theorem~\ref{thm:illposed-fradiss} for \eqref{eq:hall-mhd-fradiss}} \label{subsec:hall-mhd-fradiss}
We indicate the necessary modifications for the fractionally dissipative Hall-MHD \eqref{eq:hall-mhd-fradiss}. As before, without loss of generality, we take $\eps = 1$ (the argument below works for any $\nu, \eta \geq 0$). Let $\tb_{(\lmb)}(t)$, $b_{(\lmb)}(0)$ and $\bfB_{0}^{(\lmb)}$ be defined as in Section~\ref{subsec:e-mhd-fradiss}. In addition, we take $\bfu^{(\lmb)}_{0} = 0$. By hypothesis, for each $\lmb \geq 1$, there exists a solution $(\bfu^{(\lmb)}, \bfB^{(\lmb)}) \in L^{\infty}([0, \dlt]; H^{r} \times H^{s})$ to \eqref{eq:hall-mhd-fradiss} with initial data $(\bfu_{0}^{(\lmb)}, \bfB_{0}^{(\lmb)})$. For the purpose of contradiction, assume that there exists $A > 0$ such that
solution operator, \begin{equation} \label{eq:contra-fradiss-hall}
\sup_{\lmb} \sup_{t\in[0,\dlt]} \left(\nrm{\bfu^{(\lmb)}(t)}_{H^{s_0-1}} + \nrm{\bfB^{(\lmb)}(t)}_{H^{s_0}}\right) \le A .
\end{equation}
As before, in what follows, we suppress the dependence of implicit constants on $A$.

Let $(u, b)(t) = (u_{(\lmb)}, b_{(\lmb)})(t) = (\bfu^{(\lmb)}, \bfB^{(\lmb)} - \bgB)(t)$, where $\bgB = f(t, y) \rd_{x}$ as in Section~\ref{subsec:e-mhd-fradiss}. The equation for $(u, b)$ is given by
\begin{equation*}
\left\{
\begin{aligned}
&\rd_{t} u - \bbP\left((\nb \times \bgB) \times b + (\nb \times b) \times \bgB \right) = - \nu (-\lap)^{1+\bt} u - \bbP\left((\nb \times u) \times u - (\nb \times b) \times b \right) \\
&\rd_t b + \nb \times (u \times \bgB) + \nb\times((\nb\times b)\times\bgB) + \nb\times((\nb\times\bgB)\times b) = -\eta(-\lap)^\alp b - \nb\times((\nb\times b)\times b) ,\\
& \nb\cdot u = \nb\cdot b = 0. 
\end{aligned}
\right.
\end{equation*}
By taking the inner product of the first two equations with $u$ and $b$, respectively, and integrating over $[0, t] \times \bbT^{3}$ for $0 \leq t < \dlt$, we have on $[0, \dlt]$,
\begin{equation} \label{eq:en-fradiss-hall}
	\nrm{u(t)}_{L^{2}} + \nrm{b(t)}_{L^{2}} + \nu \int_{0}^{t} \nrm{(-\lap)^{\frac{1+\bt}{2}} u(t')}_{L^{2}} \, \ud t' \aleq \nrm{b(0)}_{L^{2}} \aleq \lmb^{-s},
\end{equation}
where the implicit constant depends on $\bgB(t)$ but not on $\lmb$. Moreover, 
\begin{equation*}
	\nrm{u(t)}_{H^{s_{0}-1}} + \nrm{b(t)}_{H^{s_{0}}} \aleq 1
\end{equation*}
uniformly in $\lmb$ for $t \in [0, \dlt]$.

Writing $\bfdlt_{\tu} = (-\nb^{\perp}(-\lap)^{-1}\bfdlt^{(0)}_{\tomg}, \bfdlt^{(0)}_{\tu^{z}})$ and $\bfdlt_{\tb} = (-\nb^{\perp }\bfdlt_{\tpsi}, \bfdlt_{\tb^{z}})$ for the error associated with $(\tu_{(\lmb)}, \tb_{(\lmb)})$ defined in Proposition~\ref{prop:wavepackets-hall}, the generalized energy identity in this case reads:
\begin{equation} \label{eq:gei-fradiss-hall}
\begin{aligned}
& \frac{\ud}{\ud t} \left( \brk{b(t), \tb_{(\lmb)}(t)} + \brk{u(t), \tu_{(\lmb)}(t)} \right)\\
& = \brk{u, \bfdlt_{\tu}} + \brk{b, \bfdlt_{\tb}} - \brk{f_{0}' u^{x, y}, (\tb^{x, y})^{\perp}} - \brk{f_{0}' (b^{x, y})^{\perp}, \tu^{x, y}}- \brk{f_{0}'' b^{z}, \tb^{y}_{(\lmb)}} - \brk{f_{0}'' b^{y}, \tb^{z}_{(\lmb)}} \\
&\phantom{=} - \nu \brk{(-\lap)^{1 +\bt} u, \tu_{(\lmb)}} - \eta \brk{(-\lap)^{\alp} b, \tb_{(\lmb)}} - \brk{\bfH, \tu_{(\lmb)}} - \brk{\bfG, \tb_{(\lmb)}},
\end{aligned}
\end{equation}
where $\bfG$ is as in Section~\ref{subsec:e-mhd-fradiss} and 
\begin{equation*}
	\bfH := - (\nb \times (\bgB - \bgB_{0})) \times b - (\nb \times b) \times (\bgB - \bgB_{0}).
\end{equation*}
Comparing this identity with the proof of Theorem~\ref{thm:illposed-strong} and \eqref{eq:gei-fradiss}, note that the only additional terms on the RHS are from the fractional dissipation $\nu (-\lap)^{1+\bt} u$ and $\bfH(t)$. For the former, we use Proposition~\ref{prop:wavepackets-hall} to estimate
\begin{equation*}
	\abs{\nu \brk{(-\lap)^{1 +\bt} u, \tu_{(\lmb)}}}
	= \nu \nrm{(-\lap)^{\frac{1 +\bt}{2}} u}_{L^{2}} \nrm{(-\lap)^{\frac{1 +\bt}{2}} \tu_{(\lmb)}}_{L^{2}} 
	\aleq \nu \nrm{(-\lap)^{\frac{1 +\bt}{2}} u}_{L^{2}} \lmb^{\bt} e^{C_{f_{0}} \bt \lmb t},
\end{equation*}
while for the latter, we move $\nb \times$ away from $b$ by an integration by parts and estimate
\begin{align*}
	\abs{\brk{\bfH, \tu_{(\lmb)}}}
	\leq \abs{\brk{(\nb \times (\bgB - \bgB_{0})) \times b, \tu_{(\lmb)}}} + \abs{\brk{b , \nb\times(\tu_{(\lmb)}\times (\bgB - \bgB_{0}))}}
	\aleq t^{2}.
\end{align*}
In both bounds, we used the property that $\nb \tu_{(\lmb)}$ obeys the same estimates as $\tb_{(\lmb)}$. In conclusion,
\begin{align*}
& \Abs{\frac{\ud}{\ud t} \left( \brk{b(t), \tb_{(\lmb)}(t)} + \brk{u(t), \tu_{(\lmb)}(t)} \right)} \\
&\aleq \left( 1 + \lmb^{2 \alp} e^{C_{f} (2 \alp) \lmb t} + t^{2} (1 + \lmb^{2} + \lmb e^{C_{f_{0}} \lmb t}) \right) \nrm{b(0)}_{L^{2}}
+ \nu \lmb^{\bt} e^{C_{f} \bt \lmb t} \nrm{(-\lap)^{\frac{1+\bt}{2}} u(t)}_{L^{2}}.
\end{align*}
As before, we integrate this inequality over the time interval $[0, t^{\ast}] = [0, \lmb^{-1} \ln(\lmb^{\veps})]$ where $\veps > 0$ is to be determined. The contribution of the first term on the RHS is handled as in Section~\ref{subsec:e-mhd-fradiss}. For the second term, we use Cauchy--Schwarz and the last term on the LHS of \eqref{eq:en-fradiss-hall}. Then for $\veps$ satisfying 
\begin{equation*}
(2 \alp - 1) + (2 \alp) C_{f_{0}} \veps < 0, \quad
(2 \bt - 1) + (2 \bt) C_{f_{0}} \veps < 0, \quad
C_{f_{0}} \veps - 2 < 0,
\end{equation*}
and $\lmb$ large enough, we obtain 
\begin{equation*}
\brk{b(t^{\ast}), \tb_{(\lmb)}(t^{\ast})} + \brk{u(t^{\ast}), \tu_{(\lmb)}(t^{\ast})} > \frac{3}{4} \brk{b(0), \tb_{(\lmb)}(0)}.
\end{equation*}
Observe, moreover, that $\brk{u(t), \tu_{(\lmb)}(t)} \aleq \lmb^{-1} \brk{b(0), \tb_{(\lmb)}(0)} $ by \eqref{eq:en-fradiss-hall} and Proposition~\ref{prop:wavepackets-hall}. Hence, we arrive at \eqref{eq:gei-conc-fradiss} for sufficiently large $\lmb$, after which the proof proceeds in the same way as in Section~\ref{subsec:e-mhd-fradiss}. \qedsymbol

\appendix
\section{Existence of an $L^2$-solution for the linearized systems} \label{sec:L2-exist}

In this section, we give a sketch of the proof of existence of an $L^2$-solution for the linearized Hall-MHD and electron-MHD systems, which are recalled here for convenience. In the case of Hall-MHD ($\nu \ge 0$), we seek a solution $(u,b) \in C_w(I;L^2)$ satisfying 
\begin{equation} \label{eq:hall-mhd-lin-again}
\left\{
\begin{aligned}
& \rd_{t} u - \nu \lap u = \bbP ((\nb \times \bgB) \times b + (\nb \times b) \times \bgB) \\
& \rd_{t} b + \nb \times (u \times \bgB) + \nb \times ((\nb \times b) \times \bgB) + \nb \times ((\nb \times \bgB) \times b) = 0, \\
& \nb \cdot u = \nb \cdot b = 0,
\end{aligned}
\right.
\end{equation} in the sense of distributions with the extra requirement $u \in L^2_t(I;\dot{H}^1)$ in the case of $\nu> 0$, and in the electron-MHD case, we simply need $b \in C_w(I;L^2)$ to satisfy 
\begin{equation} \label{eq:e-mhd-lin-again}
\left\{
\begin{aligned}
& \rd_{t} b + \nb \times ((\nb \times b) \times \bgB) + \nb \times ((\nb \times \bgB) \times b) = 0, \\
& \nb \cdot b = 0.
\end{aligned}
\right.
\end{equation}

\begin{proposition}\label{prop:existence-aubinlions}
	Let $M = (\bbT, \bbR)_{x} \times (\bbT, \bbR)_{y} \times \bbT_{z}$ and $\bgB$ be a sufficiently smooth stationary magnetic field. For any divergence-free initial data $(u_0,b_0) \in L^2(M)$, there exists a solution $(u,b) \in C_w([0,\infty);L^2)$ to \eqref{eq:hall-mhd-lin-again} with initial data $(u_0,b_0)$ satisfying \begin{equation*}
	\begin{split}
	\frac{1}{2}\left(\nrm{u(t)}_{L^2(M)}^2 + \nrm{b(t)}_{L^2(M)}^2\right) + \nu \nrm{u}_{L^2([0,t];\dot{H}^1)}^2 \le \frac{1}{2}\left(\nrm{u_0}_{L^2(M)}^2 +  \nrm{b_0}_{L^2(M)}^2 \right) e^{ Ct\nrm{\nabla\bgB}_{C^1(M)}}
	\end{split}
	\end{equation*} for all $t > 0$. In the case of \eqref{eq:e-mhd-lin-again}, there is a solution $b \in  C_w([0,\infty);L^2)$ corresponding to any divergence-free data $b_0 \in L^2(M)$ satisfying \begin{equation*}
	\begin{split}
	\frac{1}{2} \nrm{b(t)}_{L^2(M)}^2 \le \frac{1}{2}  \nrm{b_0}_{L^2(M)}^2   e^{ Ct\nrm{\nabla^2\bgB}_{L^\infty(M)}}
	\end{split}
	\end{equation*} for all $t > 0$. 
\end{proposition}
\begin{proof}
	The proof is standard; see for instance \cite{K_inv,KL}. An alternative way is to mollify the equations as well as the data by truncating high frequencies while preserving the necessary structure for energy estimates, as it is done in \cite{CDL}. We consider viscous regularizations of \eqref{eq:hall-mhd-lin-again} for $\eps > 0$, solve the regularized system \begin{equation} \label{eq:hall-mhd-lin-reg}
	\left\{
	\begin{aligned}
	& \rd_{t} u^{(\eps)} - \nu \lap u^{(\eps)} = \bbP ((\nb \times \bgB) \times b^{(\eps)} + (\nb \times b^{(\eps)}) \times \bgB) - \eps\lap^2 u^{(\eps)}  \\
	& \rd_{t} b^{(\eps)} + \nb \times (u^{(\eps)} \times \bgB) + \nb \times ((\nb \times b^{(\eps)}) \times \bgB) + \nb \times ((\nb \times \bgB) \times b^{(\eps)}) = -\eps \lap^2 b^{(\eps)}, 
	\end{aligned}
	\right.
	\end{equation} with the same initial data $(u_0,b_0)$, subject to divergence-free conditions. For each fixed $\eps > 0$, there is a unique global solution $(u^{(\eps)},b^{(\eps)})$ to \eqref{eq:hall-mhd-lin-reg} with initial data $(u_0,b_0) \in L^2$, which is smooth once $t > 0$. The energy identity \eqref{eq:lin-en-hall} with the extra term $\eps\nrm{\lap u^{(\eps)}}_{L^2(M)}$ on the  {LHS} can be justified for this solution. This shows that the sequence of solutions $\{ (u^{(\eps)},b^{(\eps)}) \}_{\eps>0}$ is uniformly bounded in $C_t(I;L^2)$ and $u^{(\eps)}$ is bounded uniformly in $L^2_t(I;\dot{H}^1)$ in the case of $\nu > 0$ for any fixed finite time interval $I = [0,T]$ with $T > 0$. In the same vein, the solution sequence is uniformly bounded in $\mathrm{Lip}_t(I;H^{-4})$. Applying the Aubin-Lions lemma (see \cite[Theorem II.5.16]{BF} for a proof), we can extract a subsequence (still denoted by $\{(u^{(\eps)},b^{(\eps)})  \}$) which converges to some $(u,b)$ in $C^0(I;H^{-s})$ for all $s < 0$. Since the space $L^\infty(I;L^2)$ is weak-* compact, we can guarantee that $(u,b) \in L^\infty(I;L^2)$ as well.
	
	Clearly we have $(u,b)|_{t = 0} = (u_0,b_0)$, and the fact that $(u,b)$ is a solution of \eqref{eq:hall-mhd-lin-again} and weakly continuous in time follows readily from strong convergence in $C^0(I;H^{-s})$. 
	
	The case of \eqref{eq:e-mhd-lin-again} is only simpler and we omit the proof.
\end{proof}

\begin{remark}
	We observe that when the stationary magnetic field $\bgB$ and the initial data enjoy  a set of symmetries respected by the Hall-MHD system (or electron-MHD system), the above proof actually guarantees existence of a solution satisfying the same set of symmetries as well. 
\end{remark}

\bibliographystyle{abbrv}
\bibliography{hallmhd}

\end{document}